\documentclass{article}
\usepackage[utf8]{inputenc}
\usepackage{amsmath}
\usepackage{geometry}
\usepackage{amsfonts}
\usepackage{mathbbol}
\usepackage{blindtext}
\usepackage[maxbibnames=5, backend=biber]{biblatex}
\usepackage[title]{appendix}

\input{customCommands.sty}

\usepackage{tikz}
\usepackage{xcolor}
\input{tikzstring.sty}
\usetikzlibrary{math}
\usetikzlibrary{shapes.geometric}
\usetikzlibrary { decorations.pathmorphing, decorations.pathreplacing, decorations.shapes, }

\usetikzlibrary{calc}
\usetikzlibrary{positioning}
\definecolor{green}{RGB}{0,200,0}
\definecolor{brown}{RGB}{100,50,0}

\newtheorem{theorem}{Theorem}[subsection]
\newtheorem{lemma}[theorem]{Lemma}
\newtheorem{prop}[theorem]{Proposition}
\newtheorem{coro}[theorem]{Corollary}

\theoremstyle{definition}
\newtheorem{defi}[theorem]{Definition}

\newtheorem{rema}[theorem]{Remark}
\newtheorem{example}[theorem]{Example}
\newtheorem{notation}[theorem]{Notation}

\numberwithin{equation}{subsection}

\newcounter{visible}
\newcounter{auxl}
\newcommand{\excludeDiagrams}[2]
{\setcounter{auxl}{#1}
\addtocounter{auxl}{\value{visible}}
\ifnum\value{auxl}>0 {#2} \fi
}

\title{Diagramatics of $A_1 \times A_1$ Folded Soergel Bimodules}
\author{Nicolas Jaramillo Torres}

\begin{document}

\maketitle

\begin{abstract}
	There is an action of $\ZZ/2$ on the category of Soergel Bimodules of type $A_1\times A_1$ induced by
	the nontrivial automorphism of its Dynkin diagram.
	We give an isotopy presentation by local generators and relations
	of the equivariantization of the category of Soergel Bimodules of type $A_1\times A_1$ under this action.
	This is the first step in a program to describe and study the categories
	of equivariant Soergel Bimodules that appear in \cite{elias2017folding}.
\end{abstract}

	\section{Introduction}
	The category of Soergel Bimodules $\Hcal(W)$ is a categorification of the Hecke algebra $\Hbold(W)$ associated to a Coxeter group $W$. 
	In \cite{elias2010diagrammatics}, Elias and Khovanov gave a diagrammatic presentation
	(by generators and relations) of this category for Coxeter groups of type $A$.
	This was generalized to other Coxeter groups 
	in \cite{elias2011soergel} and \cite{elias2016soergel}.

	Let $W$ be a finite Coxeter group with Coxeter system $(W,S)$ and let $G$ be a finite group of automorphisms of $W$ that preserves $S$.
	To fix notation, let $W^G$ be the $G$-invariant subgroup of $S$,
	let $S/G$ be the set of $G$-orbits of $S$,
	and for each $s\in S$ let $O_s$ be its $G$-orbit.
	Also let  $w_{O_s}$ the longest element of the parabolic subgroup generated by $O_s$.
	Under the identification $\widehat{\iota}:S/G\longrightarrow W$ sending $O_s\mapsto w_{O_s}$,
	we may regard $(W^G,S/G)$ as a Coxeter system,
	and the inclusion $\iota:W^G\longrightarrow W$ becomes a Coxeter embedding that extends $\widehat{\iota}$.
	Using $\iota$, we may restrict the length function $\ell_W$ of $W$ 
	to a weight function $L:W^G\longrightarrow \NN$ defined by $L(w)= \ell_W(w)$.
	The triple $(W^G,S/G,L)$ defines a Hecke algebra with unequal parameters $H(W^G,L)$ \cite[Ch. 3]{lusztig2002hecke}.
	Note that if we take the Dynkin diagram of $(W,S)$, 
	identify the vertices on each orbit, and relabel the edges appropriately,
	we obtain the Dynkin diagram of $(W^G,S/G)$.
	Because of this, the Coxeter system $(W^G,S/G)$ is called a \textit{folding} of $(W,S)$.
	
	The following is the smallest example of folding.
	Let \begin{equation}\label{labels A2}
		A_1 \times A_1=	
		\begin{unboxedstrings}
			\node[line width=0.1mm, label={[yshift=-4.5mm]$s$}, circle,draw=black, fill=black, minimum size=1.5mm] (s) at (0,0) {};
			\node[line width=0.1mm, label={[yshift=-4.5mm]$t$}, circle,draw=black, fill=black, minimum size=1.5mm] (t) at (0.5,0) {};
		\end{unboxedstrings}
	\end{equation}
	be the Dynkin diagram above with simple reflections $s$ and $t$.
	Let $\tau$ be the automorphism of $A_1 \times A_1$ swapping 
	the simple reflections $s$ and $t$.
	The invariant subgroup $W(A_1 \times A_1)^\tau$ is generated by $st$ and is isomorphic to $W(A_1)$.
	In this case $L(st)=2$ and the Hecke algebra with unequal parameters
	\begin{equation}
		H(W(A_1 \times A_1)^\tau,L)=
		\bigrightslant{\ZZ[v^{\pm 1}][b_{st}]}{\left<b_{st}^2=v^2 b_{st} + v^{-2}b_{st}\right>}.
	\end{equation}

	Let us now consider
	\begin{equation}\label{labels A3}
		A_3=	
		\begin{unboxedstrings}
			\node[line width=0.1mm, label={[yshift=-4.5mm]$s_1$}, circle,draw=black, fill=red, minimum size=1.5mm] (s1) at (0,0) {};
			\node[line width=0.1mm, label={[yshift=-4.5mm]$s_2$}, circle,draw=black, fill=blue, minimum size=1.5mm] (s2) at (0.5,0) {};
			\node[line width=0.1mm, label={[yshift=-4.5mm]$s_3$}, circle,draw=black, fill=red, minimum size=1.5mm] (s3) at (1,0) {};
			\draw [line width=0.5mm] (s1) -- (s2) -- (s3);
		\end{unboxedstrings}
	\end{equation}
	be the Dynkin diagram above, and let $s_1,s_2,$ and $s_3$ be its simple reflections.
	Let $\tau$ be the automorphism of $A_3$
	given by swapping the simple reflections $s_1$ and $s_3$.
	In this case, $\tau$ induces an automorphism of the group $W(A_3)$,
	whose invariant subgroup $W(A_3)^\tau$ is isomorphic to $W(B_2)$.
	Moreover, if we set $t_1,t_2$ to be the simple reflections of $B_2$,
	the morphism $\iota:W(B_2)\longrightarrow W(A_3)$
	sending $t_1\mapsto s_1s_3$ and $t_2\mapsto s_2$
	(the longest elements of the parabolic subgroups of the orbits) is a Coxeter embedding.
	In  this case, we may obtain the Dynkin diagram of $B_2$ by folding the Dynkin diagram of $A_1$.
	That is, we fold the diagram
	\begin{equation}\label{folding A3}	
		\begin{unboxedstrings}
			\node[line width=0.1mm, label={[yshift=-4.5mm]$s_1$}, circle,draw=black, fill=red, minimum size=1.5mm] (s1) at (0,0) {};
			\node[line width=0.1mm, label={[yshift=-4.5mm]$s_2$}, circle,draw=black, fill=blue, minimum size=1.5mm] (s2) at (0.5,0) {};
			\node[line width=0.1mm, label={[yshift=-4.5mm]$s_3$}, circle,draw=black, fill=red, minimum size=1.5mm] (s3) at (1,0) {};
			\draw [line width=0.5mm] (s1) -- (s2) -- (s3);
			\draw [stealth-stealth,line width=0.4mm]
			(s1) edge [bend left=60] (s3);
			\node[label={[yshift=3.5mm]$\tau$}] (tau) at (0.5,0) {};
		\end{unboxedstrings}
		\text{ into }
		\begin{unboxedstrings}
			\node[line width=0.1mm, label={[yshift=-4.5mm]$s_1s_3$}, circle,draw=black, fill=red, minimum size=1.5mm] (s1) at (0,0) {};
			\node[line width=0.1mm, label={[yshift=-4.5mm]$s_2$}, circle,draw=black, fill=blue, minimum size=1.5mm] (s2) at (0.8,0) {};
			\node[label={[yshift=0.5mm]{\footnotesize $4$}}] (mid) at (0.4,0) {};
			\draw [line width=0.5mm] (s1) -- (s2) ;
		\end{unboxedstrings}
		.
	\end{equation}
	The induced weight function $L:W(B_2)\longrightarrow \NN$ is determined by $L(t_1)=2$ and $L(t_2)=1$.
	Finally, the Hecke algebra with unequal parameters
	\begin{equation}
		H(W(B_2),L)=
		\bigrightslant{\ZZ[v^{\pm 1}]\left<b_{t_1},b_{t_2}\right>}
		{\left<
			\begin{matrix}
				b_{t_1}^2=v^2 b_{t_1} + v^{-2}b_{t_1} \\
				b_{t_2}^2=v b_{t_1} + v^{-1}b_{t_2} \\
				b_{t_1}b_{t_2}b_{t_1}b_{t_2}=b_{t_2}b_{t_1}b_{t_2}b_{t_1}
			\end{matrix}
			\right>}.\vspace{0.4cm}
	\end{equation}
	
	Note that the embedding of $W(B_2)$ into $W(A_3)$ does not induce a morphism of algebras
	from $H(W(B_2),L)$ to $H(W(A_3))$.
	Instead, the relation between these two algebras only becomes clear through the use of categorification.

	The categorification of Hecke algebra with unequal parameters has a long history.
	It was first done geometrically in \cite{lusztig2002hecke},
	and then algebraically in \cite{elias2017folding}.
	In the algebraic case, we consider an action of $G$ on $\Hcal(W)$ by autoequivalences and take the equivariantization $\Hcal(W)^G$.
	Both $H(W^G,S/G,L)$ and $H(W,S)^G$ 
	(the subalgebra of $G$-invariants of the Hecke algebra)
	are obtained as a specialization of the	Grothendieck group of $\Hcal(W)^G$
	\cite[Prop. 5.1, Thm. 5.4]{elias2017folding}.
	One can calculate this specialization indirectly,
	by looking at the action of $G$ on various $\Hom$ spaces and taking their characters. This was the method employed by Lusztig \cite{lusztig2002hecke}.
	
	Nevertheless, there has not been a systematic study of the equivariant category itself.
	In particular, we want to study if the equivariant category has an intrinsic
	diagrammatic presentation, independent of its description as an equivariant category.
	Given one such description, we may ask:
	How does the diagrammatic presentation of the category of Soergel Bimodules
	compare to the presentation of the equivariant category?
	Which desirable properties of the category of Soergel Bimodules are preserved by the equivariantization?

	Let $\HcalAxA^\tau$, be the equivariantization of $\HcalAxA$ with respect to the action of $\tau$, the natural involution of $A_1\times A_1$.
	In this paper we give a diagrammatic presentation of $\HcalAxA^\tau$ by generators and relations,
	independent of its description as an equivariantization of $\HcalAxA$.
	Even though this is the smallest possible example of equivariant Soergel bimodules,
	it is far more complex than the original category. 
	We encourage the reader to take a glance at Definition \ref{relations equiv category},
	where we lay out the relations in our presentation.
	We want to use as a first step to present other equivariant categories.
	For instance, $\HcalAxA^\tau$ is a full subcategory of the equivariant category
	$\Hcal(A_3)^\tau$ coming from the folding in \eqref{folding A3}.

	Our main results are to construct the diagrammatic category $\Dcal_{A_1\times A_1}$,
	to build a functor $\Fbold:\Dcal_{A_1\times A_1}\longrightarrow\HcalAxA^\tau$,
	and then show that $\Fbold$ is, in fact, an equivalence of categories.
	Here are a few words about the proof.
	We show that $\Fbold$ is an essentially surjective well-defined functor by verifying that our relations still hold in $\HcalAxA$ after applying $\Fbold$.
	By finding enough idempotent decompositions, we classify the indecomposable objects of $\Dcal_{A_1\times A_1}$.
	Using adjunctions, we show that to check $\Fbold$ is full and faithful
	it is enough to check it on $\Hom(\widehat{A},\onecalbar)$,
	where $\widehat{A}$ is an indecomposable object of $\Dcal_{A_1\times A_1}$.
	Theorem \ref{diagram reduction} finds a small spanning set for $\Hom(\widehat{A},\onecalbar)$, using a diagrammatic induction argument.
	These spanning sets are a basis after applying $\Fbold$, implying that $F$ is full and faithful.
	It is important to note that to use this presentation
	it really helps to have a comprehensive list of simplifying relations
	including the idempotent decompositions in Proposition \ref{prop idemp decomp}
	and the relations in Appendix \ref{appendix-relations}.

	This article is meant to be the first in a series of three articles  describing the diagrammatics of folded Soergel Bimodules of type $A_1\times A_1$.
	The second article will provide a diagrammatic presentation of the localization of $\HcalAxA^\tau$ by inverting $\alpha_s\alpha_t$.
	It will use this localization to provide cellular bases for the
	$\Hom$ spaces of $\HcalAxA^\tau$, analogous to
	Libedinsky's light leaves \cite{libedinsky2015light} in the non-equivariant case.
	The third article will study the categorical diagonalizability of the full twist chain complex,
	corresponding the equivariantization of the full twist on $A_1\times A_1$.

	The structure of this article is as follows:
	
	In Sections \ref{sec 0 graded and pivotal} and \ref{sec 1 equivar general} we present some background on graded categories, pivotal categories, and their equivariantization.
	These results shown are known to experts,
	especially in the context of semi-simple categories,
	but finding them out in the literature is difficult,
	so we prove them carefully here.

	In Sections \ref{subsection diagrammatics of A1}
	and \ref{subsection diagrammatics of AxA} we give a quick introduction to
	the structure of $\HcalA$ and $\HcalAxA$.
	Then, in Sections \ref{subsection Equivariantization of Soergel} and \ref{subsection hom spaces equivar} we describe $\HcalAxA^\tau$,
	its Grothendieck group, and its general properties.
	
	Section \ref{sec 3 diagrammatic folded category}
	are about the diagrammatic category $\Dcal_{A_1\times A_1}$.
	In Sections \ref{subsection free category} and \ref{subsection-presentation},
	we define $\Dcal_{A_1\times A_1}$ by generators and relations,
	construct the functor $\Fbold:\Dcal_{A_1\times A_1}\longrightarrow \Hcaltau$,
	and show it is well defined.

	In Section \ref{subsection poly forcing 1} and \ref{subsection idempotent decomposition},
	we deduce basic polynomial forcing relations and the idempotent decompositions in $\Dcal_{A_1\times A_1}$ from its defining relations.
	In Section \ref{subsection orange strand} we show that the orange strands
	slide over other diagrams.
	Finally, in Section \ref{subsection-diagram reduction},
	we prove Theorem \ref{diagram reduction} by means of diagrammatic induction.
	
	\textbf{Acknowledgments.}
	The author would like to thank Ben Elias for all the advice, support, and corrections throughout the writing of this article.
	The author was supported by several NSF grants awarded to Elias, namely DMS-2201387 and DMS-1800498 and DMS-1553032.

\section{Graded and Pivotal Categories}\label{sec 0 graded and pivotal}

\subsection{Graded Monoidal Categories}

\begin{defi}\label{defi graded monoidal cat}
	We say $\Ccal$ is a \textit{$\ZZ$-graded monoidal category} 
	(or just a \textit{graded monoidalcategory}),
	if the $\Hom$ spaces in $\Ccal$ have a $\ZZ$-graduation compatible with both the composition of morphisms and the tensor product.
	
	For $X,Y\in\Ccal$ we use the notation $\Hom^i(X,Y)$ for the $i$-th graded component of  $\Hom(X,Y)$, so that
	\begin{equation}
		\Hom(X,Y)=\bigoplus_{i\in\ZZ}\Hom^i(X,Y).
	\end{equation}
\end{defi}

We assume all our graded monoidal categories $\Ccal$ are \textit{closed under grading shift}.
That is to say that $\Ccal$ has a degree $1$ graded autoequivalence $(-)[1]:\Ccal\longrightarrow \Ccal$
with a graded pseudo-inverse $(-)[-1]:\Ccal\longrightarrow \Ccal$ (of degree $-1$).
We will denote by
$(-)[d]:\Ccal\longrightarrow \Ccal$ to the shift functor of degree $d$.
Define
\begin{equation}
	\Hom^i(X,Y)[d]:= \Hom^{i+d}(X,Y)[d],
\end{equation}
and 
\begin{equation}
	\Hom(X,Y)[d]:= \bigoplus_{i\in\ZZ} \Hom^i(X,Y)[d].
\end{equation}
Unless specified otherwise, we will use the notation
$\Hom(X,Y)$ (without any superindex) for the whole graded ring in $\Ccal$.
Note that
\begin{equation}
	\Hom(X[k],Y[m])= \Hom(X,Y)[m-k].
\end{equation}

\begin{defi}
	We say that a graded ring $R=\bigoplus_{i\in\ZZ}R^i$ is a \textit{local graded ring}
	if any sum of two homogeneous non-units is again a non-unit.
\end{defi}

\begin{rema}\label{local graded ring}
	Let $\Ccal$ be a graded monoidal category and
	let $X\in \Ccal$. 
	Then $\End(X)$ is a graded algebra and its idempotents only live in degree $0$.
	In particular, if $\End(X)$ is non-negatively graded and $\End^0(X)$ is local,
	we have that $\End(X)$ is a local graded ring and $X$ is indecomposable.
\end{rema}

\subsection{Pivotal Categories and Tensor Adjunction}

\begin{defi}
	Let $\Ccal$ be a strict monoidal category and $X\in\Ccal$.
	A \textit{right tensor adjoint} (or just \textit{right adjoint}) $(X^*,\eval_X,\coeval_X)$ of $X$ is a triple where $X^*\in\Ccal$, 
	$\eval_X:X^*\otimes X \longrightarrow \onecal$,
	and 
	$\coeval_X:\onecal\longrightarrow X\otimes X^*$
	such that the compositions
	\begin{equation}\label{eqn dual object 1}
		X \overlongrightarrow{\coeval_X\otimes\id_{X}}
		X \otimes X^* \otimes X
		\overlongrightarrow{\id_{X} \otimes \eval_X}
		X
	\end{equation}
	and
	\begin{equation}\label{eqn dual object 2}
		X^* \overlongrightarrow{\id_{X^*}\otimes\coeval_X}
		X^* \otimes X \otimes X^*
		\overlongrightarrow{\eval_X \otimes \id_{X^*}}
		X^*
	\end{equation}
	are $\id_{X}$ and $\id_{X^*}$ respectively.
	We call $\eval_X$ the \textit{evaluation} or \textit{cap} morphism.
	We call $\coeval_X$ the \textit{coevaluation} or \textit{cup} morphism.
	
	A  \textit{left adjoint}
	$(\rstar{X},\sideset{_X}{}{\eval},\sideset{_X}{}{\coeval})$ of $X$
	is defined analogously except 
	$\sideset{_X}{}{\eval}:X\otimes \rstar{X}\longrightarrow \onecal$ and
	$\sideset{_X}{}{\coeval}:\onecal \longrightarrow \rstar{X} \otimes X$.
	
	A  \textit{biadjoint}
	$(X^*,\eval_X,\coeval_X,\sideset{_X}{}{\eval},\sideset{_X}{}{\coeval})$ of $X$
	is an object such that $(X^*,\eval_X,\coeval_X)$ is right adjoint and
	$(X^*,\sideset{_X}{}{\eval},\sideset{_X}{}{\coeval})$ is a left adjoint of $X$.	
\end{defi}

The following standard result shows the name 'right/left adjoint' is well deserved.
\begin{prop}[{\cite[Lemma 2.1.6]{bakalov2001lectures}}]\label{prop tensor adjunction 1}
	Let $X\in \Ccal$. The object $(X^*,\eval_X,\coeval_X)$ is a right adjoint of $X$
		if and only if the functor $(-)\otimes X^*$ is a right adjoint of $(-)\otimes X$
		with $\coeval_X$ and $\eval_X$ being the unit and counit of adjunction respectively.
		We have an analogous result for left adjoint objects.
\end{prop}

\begin{rema}
	In \cite[Sec. 1-2]{turaev2017monoidal} and most of the rigid and fusion category literature, right adjoint objects are called \textit{left duals} and left adjoint are called \textit{right duals}.
\end{rema}

\begin{defi}
	A \textit{two-sided adjoint structure} on $\Ccal$ is a family
\begin{equation}
		\{(X^*,\eval_X,\coeval_X,\sideset{_X}{}{\eval},\sideset{_X}{}{\coeval})\}_{X\in\Ccal}
\end{equation}
	consisting of a choice of a biadjoint for each $X\in\Ccal$.
	We call it a \textit{pivotal structure} if,
	in addition, the natural adjunction functors $(-)^*$ and $\rstar{(-)}$ are equal 
	(see \cite[Sec 1.7]{turaev2017monoidal}).
\end{defi}

By Proposition \ref{prop tensor adjunction 1} we have the following result.
\begin{prop}\label{prop adjunction general}
	Let $\Ccal$ be a pivotal category and let $X\in\Ccal$.
	The functors $(-)\otimes X$ and $(-)\otimes X^*$ are biadjoint.
	The functors $X\otimes (-)$ and $X^* \otimes (-)$ are also biadjoint.
\end{prop}

The following corollary is a direct consequence of Proposition \ref{prop adjunction general}, and is frequently used to calculate the $\Hom$ spaces of additive pivotal categories.
\begin{coro}\label{coro Hom spaces general}
	Let $\Ccal$ be an additive monoidal category 
	with a right (or left) adjunction structure,
	in which every object
	is a finite direct sum of indecomposable objects.
	Then the $\Hom$ spaces of $\Ccal$ are completely determined
	by the spaces $\Hom(X,\onecal)$ (alternatively $\Hom(\onecal, X)$)
	for all indecomposable objects $X\in \Ccal$.
	That is, for any $X,Y\in \Ccal$ the space
	\begin{equation}
		\Hom(X,Y)\cong \bigoplus_{i=1}^k \Hom(X_i,\onecal)
	\end{equation}
	for some indecomposable objects $X_1,...,X_k\in\Ccal$.
\end{coro}

Let $\Ical$ be a set of objects in $\Ccal$ such that every object of $\Ccal$ is a finite direct sum of the objects in $\Ical$.
Corollary \ref{coro Hom spaces general} suggests that to check a given strict monoidal additive functor is a full or faithful, we only need to check it on $\Hom(X,\onecal)$ for all $X\in\Ical$.
The following proposition shows this.

\begin{prop}\label{prop criterion for equivalence}
	Let $\Ccal$ and $\Dcal$ be strict additive monoidal categories.
	Let $\Ical$ be a set of objects in $\Ccal$ such that every object of $\Ccal$ is a finite direct sum of the objects in $\Ical$.
	Suppose $\Ccal$ has a right (or left) adjunction structure.
	Let $\Fbold:\Ccal\longrightarrow\Dcal$ be a strict monoidal functor.
	\begin{enumerate}[(a)]
		\item If the induced map
		 $\Fbold:\Hom(X,\onecal)\longrightarrow
		 \Hom(\Fbold(X),\onecal)$
		  is injective for all 
		  $X\in\Ical$ then $\Fbold$ is faithful.
		 \item If the induced map
		 $\Fbold:\Hom(X,\onecal)\longrightarrow
		 \Hom(\Fbold(X),\onecal)$
		 is surjective for all 
		 $X\in\Ical$ then $\Fbold$ is full.
		 \item If the induced map
		 $\Fbold:\Hom(X,\onecal)\longrightarrow
		 \Hom(\Fbold(X),\onecal)$
		 is bijective for all 
		 $X\in\Ical$
		 and $\Fbold$ is essentially surjective,
		 then $\Fbold$ is an equivalence.
	\end{enumerate}
\end{prop} 
\begin{proof}
	Clearly (a) and (b) imply (c).
	We will only show (a), as the proof of (b) is analogous.
	Note that if $(X^*,\eval_X,\coeval_X)$ is a right adjoint of $X$ then
	$(\Fbold(X^*),\Fbold\left(\eval_X\right),\Fbold\left(\coeval_X\right))$
	is a right adjoint of $\Fbold(X)$.
	This induces a right adjoint structure on $\Fbold(\Ccal)$
	where $\Fbold(X^*)=\Fbold(X)^*$.
	Under this identification, for all $X,Y\in\Ccal$ the diagram
	\begin{equation}
		\xymatrix{
			\Hom(X,Y) \ar[r]^{\Fbold} \ar[d]^{\cong}
			& \Hom(\Fbold(X),\Fbold(Y)) \ar[d]^{\cong}\\
			\Hom(X\otimes Y^*,\onecal) \ar[r]^{\Fbold} 
			& \Hom(\Fbold(X)\otimes \Fbold(Y)^*,\onecal) 
		}
	\end{equation}
	commutes.
	Here, the vertical isomorphisms are given by the right adjoint structures.
	This shows that the top morphism is injective if and only if the bottom morphism is injective.
	Writing 
	\begin{equation}
		X\otimes Y^*\cong X_1\oplus ... \oplus X_k
	\end{equation}
	for some $X_i\in\Ical$ we reduce the problem to
	the assumption that $\Fbold$ is injective on each $\Hom(X_i,\onecal)$.	
\end{proof}

\section{Equivariantization}\label{sec 1 equivar general}

\subsection{Action of a Group on a Monoidal Category}\label{sec-action-def}

\begin{defi}
Let $\Ccal$ be a category and $G$ a group.
An \textit{action of $G$ on $\Ccal$} is a triple
$(G,\Fbold, \gamma, \epsilon)$, where:
\begin{itemize}
	\item $\Fbold$ is a collection of functors $\Fbold_g:\Ccal\longrightarrow \Ccal$ for $g\in G$.
	\item $\gamma_{g,h}:\Fbold_{gh}\longrightarrow \Fbold_g\circ \Fbold_h$ is collection of natural isomorphisms of functors for all $g,h\in G$ called the associator of the action.
	\item $\eta:\Fbold_1\longrightarrow \Id_{\Ccal}$ is an isomorphism of functors called the unit of the action.
\end{itemize}
This collection has to satisfy the axioms in \cite[Sec. 2.1]{burciu2013fusion}.
In this case we say that $\Ccal$ is a \textit{$G$-category}.

If both the associator and unit are the identity morphism, we say that the $G$-action is \textit{strict}.

If $\Ccal$ is a strict monoidal category, we say that the action of $G$ is \textit{strict monoidal} if it is strict
and all the functors $\Fbold_g$ are strict monoidal functors.

If $\Ccal$ is additive, we say that the action of $G$ is \textit{additive} if all the functors $\Fbold_g$ are additive.

If $\Ccal$ is a $\kk$-linear category for some field $\kk$, we say that the action of $G$ is \textit{$\kk$-linear} if all the functors $\Fbold_g$ are $\kk$-linear.

If $\Ccal$ is a strict graded monoidal, we say that the action of $G$ is \textit{graded monoidal} if it is strict and
all the functors $\Fbold_g$ are strict graded monoidal functors.
\end{defi}

In this article we will suppose all our categories are strict monoidal categories and all actions are strict group actions by strict functors, meaning that $\Fbold_g$ is a strict monoidal functor for all $g\in G$.
These conditions simplify our previous definition into the following.
\begin{prop}
Let $\Ccal$ be a strict monoidal category and $G$ a group.
A strict monoidal of $G$ on $\Ccal$ is equivalent to a pair $(G,\Fbold)$ where $F_g:\Ccal\longrightarrow \Ccal$ is a strict monoidal functor for all $g\in G$ such that $\Fbold_1=\Id_{\Ccal}$ and
$\Fbold_g\circ \Fbold_h = \Fbold_{gh}.$
\end{prop}

We also assume that when $\Ccal$ is additive, $\kk$-linear, or graded, so is the group action.

\begin{notation}\label{g is for functor}
Let $g\in G$, $X,Y\in \Ccal$ and a morphism $T:X\longrightarrow Y$. We will write $g(X)$ and $g(T)$ as a shorthand for $\Fbold_g(X)$ and $\Fbold_g(T)$ respectively.
\end{notation}

\begin{rema}
An action can be defined in a more general way as a monoidal functor $\rho:\underline{G}\longrightarrow \Aut_\otimes(\Ccal)$.
Here $\underline{G}$ is the monoidal category whose objects are the set $G$,
with only identity morphisms, and $\otimes$ is given by the product in $G$.
In this context, the action is strict if and only if $\rho$ is a strict monoidal functor (see \cite[Section 2.7]{etingof2016tensor}).
\end{rema}

\subsection{Definition of Equivariantization}\label{sec-equiv-1}
From now on let $G$ be a group and fix a field $\kk$ such that
$\chr(\kk) \nmid |G|$.
Let $\Ccal$ be a strict monoidal $\kk$-linear category with a strict $G$-action by strict $\kk$-linear functors.

\begin{defi}
A \textit{$G$-equivariant object} of $\Ccal$ is a pair $(X,f)$ where $X\in \Ccal$ and $f=(f_g)_{g\in G}$ is a collection of isomorphisms $f_g:g(X)\longrightarrow X$ where $f_1=\Id_{X}$ and such that the diagram
\begin{equation}\label{equivariant objects}
    \xymatrix{
    gh(X)\ar[d]^{f_{gh}} \ar[r]^{g(f_h)} &
    g(X)\ar[dl]^{f_g}\\
    X 
    }
\end{equation}
commutes for all $g,h\in G$.

A \textit{$G$-equivariant morphism} $T:(X,f)\longrightarrow(Y,f')$ is a morphism $T:X\longrightarrow Y$ that intertwines $f$ and $f'$, i.e. a morphism such that the diagram
\begin{equation}\label{equivariant mor}
    \xymatrix{
    g(X) \ar[d]^{f_g} \ar[r]^{g(T)} &
    g(Y)\ar[d]^{f'_g}\\
    X \ar[r]^{T}&
    Y\
    }   
\end{equation}
commutes for all $g\in G$.

The \textit{equivariantization $\Ccal^G$} is the category whose objects are $G$-equivariant objects and whose morphisms are $G$-equivariant morphisms.
\end{defi}

\begin{prop}
    Let $(X,f)$ and $(Y,f')$ be equivariant objects in $\Ccal$.
    The tensor product $(X,f)\otimes (Y,f'):= (X\otimes Y, f\otimes f')$ makes
    $\Ccal^G$ into a strict monoidal category.
\end{prop}
\begin{proof}
	This is a direct consequence of the functoriality of $\otimes$.
\end{proof}

\begin{rema}
For the definition of equivariantization in the general (non-strict) setting see \cite[Sec. 2.7]{etingof2016tensor}.
\end{rema}

As the next example shows, an object $X\in \Ccal$ may be equipped with multiple non-isomorphic equivariant structures.

\begin{example}\label{exmp repg}
	Let $\kk$ be a field.
	Let us consider the case where our category is $\VEC(\kk)$,
	the category of finite dimensional vector spaces over $\kk$,
	and $G$ acts trivially on $\VEC(\kk)$.
	The equivariantization $\VEC(\kk)^G$ is isomorphic to
	$\Rep(G)$, the category of finite dimensional representations of $G$.
	In fact, we can see the object $(V,f)\in \VEC(\kk)^G$ as the
	 representation where $g(v)=f_g(v)$ (for $v\in V$ and $g\in G$).
	As such, we can see $G$-equivariant categories as a generalization of the category $\Rep(G)$.
\end{example}

\begin{defi}
	Let $\Ccal$ be an additive category $\Ccal$. We say $\Ccal$ is
	\textit{Karoubian} if for every object $X\in\Ccal$ and every idempotent
	$f\in \Hom(X,X)$ there is an object $Y$ and maps 
	$\iota_Y:Y\longrightarrow X$ and $p_Y:X\longrightarrow Y$ such that
	$f=\iota_Y\circ p_Y$ and $p_Y\circ \iota_Y = \id_Y$.
	That is, $\Ccal$ is closed under direct summands.
\end{defi}

In \cite[Sec. 4.1.3]{drinfeld2010braided}, it is noted without a proof that the equivariantization 
of a Karoubian category will again be Karoubian.
Here we provide a proof.

\begin{prop}\label{equiv karoubian}
	Let $\kk$ be a field. Let $\Ccal$ be a Karoubian $\kk$-linear category
	with an additive $G$-action.
	The category $\Ccal^G$ is Karoubian.
\end{prop}
\begin{proof}
	Suppose that $\Ccal$ is Karoubian.
	Let $(X,f)\in \Ccal^G$ and let $T:X\longrightarrow X$ be an idempotent.
	Since $\Ccal$ is Karoubian, there exists a $Y\in\Ccal$,
	$\iota_Y:Y\longrightarrow X$ and
	$p_Y:X\longrightarrow Y$ such that
	$T= \iota_Y \circ p_Y$ and $\Id_Y= p_Y\circ \iota_Y$.
	For each $g\in G$ define $f'_g:g(Y)\longrightarrow Y$ by the composition
	\begin{equation}
		g(Y) \overset{g(\iota_Y)}{\longrightarrow} g(X) 
		\overset{f_g}{\longrightarrow} X 
		\overset{p_Y}{\longrightarrow} Y.
	\end{equation}
	We want to show that $(Y,f')$ is a $G$-equivariant object and that
	$\iota_Y:(Y,f')\longrightarrow (X,f)$ and 
	$p_Y:(X,f)\longrightarrow (Y,f')$ 
	are $G$-equivariant morphisms.
	For the first statement, check that the diagram
	\begin{equation}
		\xymatrix{
			gh(Y) \ar[r]^{gh(\iota_Y)} & gh(X) \ar[r]^{g(f_h)} \ar[ddrr]^{f_{gh}}&
			g(X) \ar[r]^{g(p_Y)} \ar[dr]^{\Id_{g(X)}}& g(Y) \ar[d]^{g(\iota_Y)} \\
			& & & g(X) \ar[d]^{f_g} \\ 
			& & & X \ar[d]^{p_Y} \\ 
			& & & Y 
		}
	\end{equation}
	commutes.
	
	To check that $\iota_Y$ is a $G$-equivariant morphism, observe that
	\begin{equation}
		\xymatrix{
			g(Y) \ar[dr]_{g(\iota_Y)} \ar[r]^{g(\iota_Y)}&
			g(X)\ar[r]^{f_g} \ar[d]^{\Id_{g(X)}}& X \ar[r]^{p_Y} \ar[d]^{\Id_X} & Y \ar[dl]^{\iota_Y}\\
			& g(X)\ar[r]^{f_g} & X
		}
	\end{equation}
	commutes.
	The diagram for $p_Y$ is similar.
\end{proof}

\subsection{Induction and Restriction Functors}

Since $G$-equivariant categories generalize the category of finite
dimensional representations of a group, we may expect induction and
restriction functors like in the usual case.
We will base our approach on \cite[Sec. 4]{drinfeld2010braided}.
We will discuss these functors because they provide a way to understand the indecomposable objects of the equivariant category
from the indecomposable objects of the original category.

\begin{defi}
Define the \textit{induction functor} $\Ind:\Ccal\longrightarrow \Ccal^G$ as follows:
\begin{itemize}
    \item We will map an object $X\in \Ccal$ to $X\longmapsto (\bigoplus_{h\in G} h(X), \sigma_g)$,
    where $\sigma_g$ just permutes the components of $\bigoplus_{h\in G} h(X)$ by multiplication by $g$.
    
    \item We will map a morphism $T:X\longmapsto Y$ to $\bigoplus_{h\in G}h(T)$.
\end{itemize}

Define the \textit{restriction functor} $\Res:\Ccal^G \longrightarrow \Ccal$ to be the functor sending $(X,f)\mapsto X$ and sending $T:(X,f)\longrightarrow (Y,f')$ to $T:X\longrightarrow Y$ (in $\Ccal$).
\end{defi}

Just like in $\Rep(G)$, we will show that the induction and restriction functors are biadjoint. Let $X\in \Ccal$ and $(Y,f')\in \Ccal^G$.
Define the map 
\begin{equation}
	\Phi_{X,(Y,f')}:\Hom(X,\Res(Y,f'))=\Hom(X,Y)\longrightarrow \Hom(\Ind(X),(Y,f'))
\end{equation}
by 
\begin{equation}
	\Phi_{X,(Y,f')}:\varphi\mapsto \bigoplus_{g\in G} f'_g\circ g(\varphi).
\end{equation}
Also define the map 
\begin{equation}
	\Psi_{X,(Y,f')}:\Hom(\Ind(X),(Y,f'))\longrightarrow \Hom(X,\Res(Y,f'))=\Hom(X,Y)
\end{equation}
by the restriction
\begin{equation}
	\Psi_{X,(Y,f')}:\psi=\left(\bigoplus_{g\in G}\psi_g\right) \mapsto \psi_1.
\end{equation}
\begin{prop}
	Let $\phi:X\longrightarrow Y$. The map $\Phi_{X,(Y,f')}(\phi)$ is $G$-equivariant. Therefore, $\Phi_{X,(Y,f')}$ is a well defined map. 
\end{prop}
\begin{proof}
	To check that the map $\Phi(\varphi)$ is $G$-equivariant, we can follow the component $gh(X)$ of $\Ind(X)$ for $g,h\in G$ on the diagram \eqref{equivariant mor}.
	That is the same as checking that the diagram
	\begin{equation}
		\xymatrix{
			gh(X)\ar[r]^{gh(\varphi)} \ar[d]^{\Id} & 
			gh(Y) \ar[r]^{g(f'_{h})} \ar[dr]^{f'_{gh}}&
			g(Y) \ar[d]^{f'_{g}} \\
			gh(X) \ar[rr]^{f'_{gh}\circ gh(\varphi)} & &
			Y
		}
	\end{equation}
	commutes. The square clearly commutes and the triangle commutes by equation \eqref{equivariant objects}.
\end{proof}

\begin{prop}\label{ind-res adjunction}
The maps $\Phi_{X,(Y,f')}$ and $\Psi_{X,(Y,f')}$ are functorial and inverse to each other.
Therefore, the functor $\Ind:\Ccal\longrightarrow \Ccal^G$ is left adjoint to $\Res:\Ccal^G \longrightarrow \Ccal$.
\end{prop}
\begin{proof}
Let $X\in\Ccal$, and $(Y,f),(Y',f')\in \Ccal^G$.
Let $T:(Y,f)\longrightarrow (Y',f')$, $\varphi:X\longrightarrow Y$, and
$\psi:\Ind(X)\longrightarrow (Y,f)$.
The functoriality of most components is straightforward. 
We will show the functoriality of $\Phi_{X,(Y,f)}$ on the $(Y,f)$ component and leave the others to the reader.

Using \eqref{equivariant mor},
\begin{align}
	\Phi_{X,(Y',f')}(T\circ \varphi) 
	&= \bigoplus_{g\in G} f'_g\circ g(T\circ\varphi) \\
	&= \bigoplus_{g\in G} f'_g\circ g(T)\circ g(\varphi) \\
	&= \bigoplus_{g\in G} T\circ f_g\circ g(\varphi).
\end{align}

The following computations show that $\Phi$ and $\Psi$ are mutually inverse.
\begin{align}
	\Psi_{X,(X,f)}\circ \Phi_{X,(Y,f)} (\varphi) 
	&= \Psi_{X,(X,f)}\left(\bigoplus_{g\in G} f_g\circ g(\varphi) \right)\\
	&= f_1\circ 1(\varphi) \\
	&= \varphi.
\end{align}
\begin{align}
	\Phi_{X,(Y,f)}\circ \Psi_{X,(Y,f)} (\psi) 
	&= \bigoplus_{g\in G} f_g\circ g(\psi)_1.
\end{align}
Using \eqref{equivariant mor}, we obtain
\begin{align}
	\Phi_{X,(Y,f)}\circ \Psi_{X,(Y,f)} (\psi) 
	&= \bigoplus_{g\in G} \psi_g = \psi.
\end{align}
\end{proof}

Now suppose that $G$ is finite. 
The induction and restriction functors on 
$\Rep(G)$ are bi-adjoint. We will show the same for $\Ind$ and $\Res$.
Define 
\begin{equation}
	\Phi'_{(X,f),Y}:\Hom(\Res(X,f),Y)=\Hom(X,Y)\longrightarrow \Hom((X,f),\Ind(Y))
\end{equation}
by 
\begin{equation}
	\Phi'_{(X,f),Y}:\varphi\mapsto 
	\begin{pmatrix}
		1(\varphi\circ f_{1})\\
		\vdots \\
		g(\varphi\circ f_{g^{-1}})
	\end{pmatrix}_{g\in G}
	.
\end{equation}
Above and elsewhere, we will describe the maps to and from $\Ind(X)$ by matrices indexed by $G$.

Also define
\begin{equation}
	\Psi'_{(X,f),Y}:\Hom((X,f),\Ind(Y))\longrightarrow \Hom(\Res(X,f),Y)=\Hom(X,Y)
\end{equation}
by 
\begin{equation}
	\Psi'_{(X,f),Y}:\psi\mapsto \pi_1\circ\psi,
\end{equation}
where $\pi_1:\bigoplus_{g\in G}g(Y) \longrightarrow Y$ is the projection to the identity component.
Note that the definition of $\Phi'$ requires $G$ to be finite. Otherwise, it might have infinitely many non-zero components.

\begin{prop}\label{ind-res coadjunction}
	The maps $\Phi'_{(X,f),Y}$ and $\Psi'_{(X,f),Y}$ are functorial and mutually inverse.
	Therefore, when $G$ is finite, the functor $\Ind:\Ccal\longrightarrow \Ccal^G$ is right adjoint to $\Res:\Ccal^G \longrightarrow \Ccal$.
\end{prop}
The proof of this proposition is similar to the one of Proposition \ref{ind-res adjunction}.

Now, for every $(X,f)\in \Ccal^G$ define the morphisms
$p_{(X,f)}:\Ind(X)\longrightarrow (X,f)$ given by 
\begin{equation}
	p_{(X,f)}:= 
	\begin{pmatrix}
		f_1 & \cdots & f_g
	\end{pmatrix}_{g\in G}
\end{equation}
and $\iota_{(X,f)}: (X,f)\longrightarrow\Ind(X)$ given by
\begin{equation}
	\iota_{(X,f)}:= 
	\begin{pmatrix}
		1(f_{1})\\
		\vdots \\
		g(f_{g^{-1}})
	\end{pmatrix}_{g\in G}
	.
\end{equation}
\begin{prop}
	The morphisms $p_{(X,f)}:\Ind(X)\longrightarrow (X,f)$ and 
	$\iota_{(X,f)}:\Ind(X)\longrightarrow (X,f)$ are $G$-equivariant
	and satisfy $p_{(X,f)} \circ \iota_{(X,f)}= |G|\id_{X}$.
\end{prop}
\begin{proof}
	The map $p_{(X,f)}=\Phi_{X,(X,f)}(\id_{X})$ is the counit of the adjunction from Proposition \ref{ind-res adjunction}.
	Similarly, $\iota_{(X,f)}=\Phi'_{(X,f),X}(\id_{(X,f)})$ 
	is the unit of the adjunction from Proposition \ref{ind-res coadjunction}.
	Alternatively, we can show $p_{(X,f)}$ is equivariant by following the component $gh(X)$ of $\Ind(X)$ in equation \eqref{equivariant mor}.
	That is, checking that the diagram
	\begin{equation}
		\xymatrix{
			gh(X) \ar[r]^{g(f_h)} \ar[d]^{\id}&
			g(X)\ar[d]^{f_g}\\
			gh(X) \ar[r]^{f_{gh}}&
			X
		}
	\end{equation}
	commutes.
	But that is just equation \eqref{equivariant objects}.
	The check for $\iota_{(X,f)}$ is similar.
	
	Also note that 
	\begin{equation}\label{condition equiv mor}
		g(f_{g^{-1}})=f_g^{-1}
	\end{equation}
	by equation \eqref{equivariant mor}.
	Hence
	\begin{align}
		p_{(X,f)} \circ \iota_{(X,f)}
		&= \sum_{g\in G} f_g \circ g(f_{g^{-1}})\\
		&= \sum_{g\in G} f_g \circ f_g^{-1}\\
		&= |G| \id_{X}.
	\end{align}
\end{proof}

The following corollary relies on the fact that $\chr(\kk)\nmid |G|$.
\begin{coro}\label{coro summands}
	The map $\iota_{(X,f)}\circ p_{(X,f)}$ 
	is a pseudo-idempotent.
	Moreover, in this case $(X,f)$ is a direct summand of $\Ind(X)$
	with inclusion $\iota_{(X,f)}$ and projection $\frac{1}{|G|}p_{(X,f)}$.
\end{coro}
\begin{proof}
	By the previous proposition,
	\begin{equation}
		(\iota_{(X,f)}\circ p_{(X,f)})^2 = |G|\iota_{(X,f)}\circ p_{(X,f)}.
	\end{equation}
	The scalar $|G|$ is invertible if and only if the characteristic of $\kk$ does not divide the order of $G$.
\end{proof}

\subsection{Unique Decomposition in $\Ccal^G$}

\begin{defi}
	Let $\Ccal$ be an additive category and let $X\in \Ccal$.
	The object $X$ is said to have \textit{unique decomposition} if
	it satisfies the following properties:
	\begin{itemize}
		\item $X$ has a 
		finite direct sum decomposition
		\begin{equation}
			X\cong X_1\oplus X_2\oplus ... \oplus X_k
		\end{equation}
		for some
		indecomposable objects $X_1,...,X_k \in \Ccal$.
		\item If 
		\begin{equation}
			X\cong Y_1\oplus Y_2 \oplus ... \oplus Y_m
		\end{equation}
		is another decomposition of $X$ into indecomposable objects,
		then $m=k$ and $Y_i\cong X_{\sigma(i)}$ for some 
		permutation $\sigma\in S_k$.
	\end{itemize}
	The category $\Ccal$ is said to have the \textit{unique decomposition property} if all of its objects have unique decomposition.
	Alternatively we may say that $\Ccal$ \textit{has unique decomposition}.
\end{defi}

One of our goals is to show that all the categories we care about,
$\HcalA$, $\HcalAxA$, and $\HcalAxA^\tau$,
have unique decomposition.
The unique decomposition property for these categories will come from the general theory for Krull-Schmidt categories.
This will require us to have a description of all their indecomposable objects
and show they have (graded) local endomorphism rings.
The next result will allow us to get a full list of indecomposable objects
for the equivariantization in all the cases we care about.

\begin{prop}\label{Prop unique decomp indX}
	Let $(Y,f)\in\Hcal$. Suppose that both $Y$ and $\Ind(Y)$
	have unique decomposition.
	Let 
	\begin{equation}
		Y\cong Y_1\oplus Y_2 \oplus ... \oplus Y_k
	\end{equation}
	be the unique decomposition of $Y$ into indecomposables.
	Then $(Y,f)$ is a summand of $\Ind(Y_i)$ for some $i$.
	Hence, listing all the summands of $\Ind(X)$ 
	for all indecomposable $X\in\Ccal$
	will give us a full list of the indecomposable objects of $\Ccal^G$.
	
\end{prop}
\begin{proof}
	Let $(Y,f)$ be an indecomposable object of $\Ccal^G$.
	By Corollary \ref{coro summands}, $(Y,f)$ is a summand of $\Ind(Y)$.
	If $Y$ is indecomposable, $k=1$ and we are done. 
	Otherwise, $\Ind(Y)\cong \Ind(Y_1)\oplus ... \oplus \Ind(Y_k)$.
	Since $\Ind(Y)$ has unique decomposition, $(Y,f')$ has to be a summand of $\Ind(Y_i)$ for some $i$.
\end{proof}

\begin{defi}
	Let $\Ccal$ be an additive category and let $X\in \Ccal$.
	We say $X$ is a has a \textit{local decomposition} if it there exists a direct sum decomposition
	\begin{equation}
	X\cong X_1\oplus X_2\oplus ... \oplus X_k
	\end{equation}
	 for some
	indecomposable objects $X_1,...,X_k \in \Ccal$ with local endomorphism rings.
	If case $\Ccal$ is a graded category,
	$X$ has a \textit{graded local decomposition}  if the endomorphism rings of the $X_i$'s are graded local.
	
	We say $\Ccal$ is a \textit{Krull-Schmidt category} 
	(or has the \textit{Krull-Schmidt} property) 
	if all of its objects have local decompositions.
	A graded category $\Ccal$ is a \textit{graded Krull-Schmidt category} if all of its objects have graded local decompositions.
\end{defi}

It is important to note that all the results here apply both in the graded and non-graded cases.
The following is a standard result about Krull-Schmidt categories.
\begin{lemma}[{\cite[Thm. 11.50]{elias2020introduction}}]\label{krullSchmidt objects}
	Let $\Ccal$ be a Karoubian category and suppose that $X\in\Ccal$
	admits a local decomposition. Then $X$ has unique decomposition.
	Hence, any Krull-Schmidt category has unique decomposition.
\end{lemma}

To show that an additive (graded) category $\Ccal$ is Krull-Schmidt,
one usually finds a set $P$ of objects of $\Ccal$ such that:
\begin{itemize}
	\item All the objects in $\mathcal{P}$ admit a local decomposition.
	\item All objects in $\Ccal$ are a summand of an object of the form $P_1\oplus P_2\oplus ... \oplus P_k$
	for some $P_1,...,P_k\in P$ (or are a grading shift of an object in $P$).
	In other words,
	\begin{equation}
		\mathsf{Kar}(\mathsf{add}(A))=\Ccal.
	\end{equation}
\end{itemize}
This is enough because a direct summand of an object with a local decomposition also has a local decomposition.
In the case of $\Ccal^G$ this yields the following result.

\begin{prop}\label{prop CcalG is KS}
	Suppose $\Ccal$ is a (graded) category with unique decomposition and for every indecomposable object $X\in\Ccal$
	the object $\Ind(X)$ has a (graded) local decomposition. Then $\Ccal^G$ is a Krull-Schmidt category.
\end{prop}
\begin{proof}
	Let $(Y,f)\in\Ccal^G$. By Corollary \ref{coro summands}, $(Y,f)$ is a summand of
	$\Ind(Y)$.
	By unique decomposition, there exist indecomposable objects $X_1,...,X_k$ such that
	$Y\cong X_1\oplus...\oplus X_k$.
	Hence $\Ind(Y)\cong \Ind(X_1)\oplus...\oplus \Ind(X_k)$ has unique decomposition.
	Since $(Y,f)$ is a summand of $\Ind(Y)$, it also has a local decomposition.
\end{proof}

\subsection{Actions of $\ZZ/2$ on a Monoidal Category}\label{sec-equiv-2}

One very important special case is when the group $G=\ZZ/2$. Let $\tau$ be the generator of $G$. 
From here on, we will suppose that $\chr(\kk)\neq 2$.

A strict action of $\ZZ/2$ on a monoidal strict category $\Ccal$ is the same as a
strict monoidal functor
$\Fbold_\tau$ such that $\Fbold_\tau^2=\Id_{\Ccal}$.
Using Notation \ref{g is for functor} $\tau$ will represent both the functor and the group element.
In this case, an equivariant object $(X,f)$
only depends on $f_\tau$,
as $f_1=\id_{X}$.
We will denote our equivariant objects by $(X,f_\tau)$, listing the morphism $f_\tau$ in the second coordinate.
For an equivariant object $(X,f_\tau)$, equation \eqref{condition equiv mor} shows that we need 
\begin{equation}\label{condition Z2 equiv structure}
	\tau(f_{\tau})= f_{\tau}^{-1}.
\end{equation}
Also note that for a given $(X,f_\tau)$ satisfying \eqref{condition Z2 equiv structure}, $(X,-f_\tau)$ is also an equivariant object. 

\begin{example}\label{example End1}
	Consider the unit object $\onecal\in \Ccal$.
	Since $\tau$ fixes $\onecal$, we have that both 
	$(\onecal, \id_{\onecal})$ and $(\onecal, -\id_{\onecal})$.
	We will note these objects by $(\onecal,1)$ and $(\onecal,-1)$ respectively. Let $R:=\End(\onecal)$.
	By our strictness assumption, $\tau(\onecal)=\onecal$.
	Thus, functor $\tau$ acts on $R$ in a natural way.
	Define
	\begin{equation}\label{definition rtau}
		\Rtau:=\{r\in R | \tau(r)=r\}
	\end{equation}
	and
	\begin{equation}
		\Rantitau:=\{r\in R | \tau(r)=-r\}.
	\end{equation}
	Using equation \eqref{equivariant mor} it is easy to see that
	\begin{equation}\label{eqn end 1}
		\Hom((\onecal,1),(\onecal,1))=\Rtau,
	\end{equation}
	the set of  $\tau$-invariant elements of $R$.
	Similarly, 
	\begin{equation}
		\Hom((\onecal,-1),(\onecal,-1))=\Rtau,
	\end{equation}
	and
	\begin{equation}
		\Hom((\onecal,1),(\onecal,-1))=\Rantitau.
	\end{equation}
	Since $\Hom((\onecal,1),(\onecal,1))\ncong \Hom((\onecal,1),(\onecal,-1))$ we can conclude
	that $(\onecal,1)$ and $(\onecal,-1)$ are not isomorphic.	
\end{example}

\begin{example}
	Now let $X\in \Ccal$ and $\Ind(X) = \left(X\oplus\tau(X), \IndMatrix{X}\right)$.
	In contrast with Example \ref{example End1}, we claim that 
	\begin{equation}\label{eqn isomorphism sngXInd}
		\Ind(X) \cong (\onecal,-1)\otimes\Ind(X)  =\left(X\oplus\tau(X), -\IndMatrix{X}\right).
	\end{equation}
	Let 	
	\begin{equation}
		T:  \xymatrix{(\onecal,-1)\otimes\Ind(X)
		\ar[rr]^(0.58){
			\smallMatrixgen{\id_{X}}{0}{0}{-\id_{\tau(X)}}
		} 
		&&
		\Ind(X)
	}.
	\end{equation}
	It is easy to see that $T^2=\smallMatrixgen{\id_{X}}{0}{0}{\id_{\tau(X)}}$.
	To check that $T$ is a $\tau$-equivariant morphism, observe that the diagram
	\begin{equation}
		\xymatrix{
			\tau(X\oplus \tau(X))
			\ar @{} [r] |{=}
			&
			\tau(X)\oplus X
			\ar[rr]^(0.5){
				\smallMatrixgen{\id_{\tau(X)}}{0}{0}{-\id_{X}}
			} 
			\ar[d]_{
			-\IndMatrix{X}
			}
			&&
			\tau(X)\oplus X
			\ar @{} [r] |{=}
			\ar[d]^{
				\IndMatrix{X}
			}
			&
			\tau(X\oplus \tau(X))
			\\
			& 
			X\oplus \tau(X)
			\ar[rr]^(0.5){
				\smallMatrixgen{\id_{X}}{0}{0}{-\id_{\tau(X)}}
			} 
			&&
			X\oplus \tau(X)
		}
	\end{equation}
	commutes.

\end{example}

Let $\Ccal$ be a graded monoidal $\kk$-linear category with unique decomposition for an algebraically closed field $\kk$.
Let $X\in\Ccal$ be an indecomposable object such that
 $\End^0(X)=\kk \id_{X}$.
Let $G_X$ to the stabilizer of $X$, that is,
\begin{equation}
	G_X:=\{g\in G| g(X)\cong X\}.
\end{equation}
 Since $G=\ZZ/2$, $G_X$ is either trivial or $G$.
 We have two cases for decomposing $\Ind(X)$ into indecomposables.
\begin{itemize}
	\item If $G_X=1$ then $\Ind(X)=\left(X\oplus \tau(X),
	\IndMatrix{X}
	\right) $
	is indecomposable.
	That is because, by unique decomposition,
	$X$ and $\tau(X)$ are the only summands of $\Ind(X)$
	and they can not have an equivariant structure.

	\item If $G_X=G$, let $T:\tau(X)\longrightarrow X$ be an isomorphism.
	Then the morphism $T\circ \tau(T)= a \id_{X}$ for some $a\in\kk^*$.
	The field $\kk$ is algebraically closed, so there is $b\in\kk^*$ such that $b^2=a$. Let $f_\tau=b^{-1}T$.
	Since $(b^{-1}T)\circ (\tau(b^{-1}T)) = \id_{X}$, the object
	$(X,f_{\tau})$ is an equivariant structure of $X$.
	Furthermore,
	we have a decomposition $\Ind(X)\cong (X,f_{\tau})\oplus (X,-f_{\tau})$ with splittings
	\begin{equation}\label{splitting ind X}
		\xymatrix{
			(X,f_\tau) 
			\ar@/^1pc/[r]^{
				\begin{pmatrix}
					1 \\
					\tau(f_\tau)
				\end{pmatrix}
			}
			& \Ind(X) 
			\ar@/^1pc/[l]^{
				\frac{1}{2}
				\begin{pmatrix}
					1 &
					f_\tau
				\end{pmatrix}
			}
			\ar@/^1pc/[r]^{
				\begin{pmatrix}
					1 \\
					-\tau(f_\tau)
				\end{pmatrix}
			}
			& (X,-f_\tau) \ar@/^1pc/[l]^{
				\frac{1}{2}
				\begin{pmatrix}
					1 &
					-f_\tau
				\end{pmatrix}
			}
		}
	\end{equation}
	(as in Corollary \ref{coro summands}).
\end{itemize}

\section{Diagrammatic Hecke Category of type $A_1\times A_1$.}\label{sec 2 diag hecke}

\subsection{Diagrammatic Hecke Category of type $A_1$}\label{subsection diagrammatics of A1}
\begin{defi}
Let $R:= \kk[\alpha_s]$.
A one color \textit{$(m,n)$-diagram} consists of the following.
\begin{enumerate}[(a)]
    \item A graph $G$ embedded on the planar strip $[0,1]\times [0,1]$ with $m$ strands on its lower boundary and $n$ strands on its upper boundary, made out of trivalent vertices 
    $\left(
    \Trivup{red}{red}{red}
    \right)
    $
    and univalent vertices (also called dots)
    $\left(
    \Dotup{red}
    \right)
    $.
    
    \item Floating polynomial boxes in the regions of $[0,1]\times [0,1]$ partitioned by $G$, labeled by polynomials $f\in R$.
    
\end{enumerate}

We say two $(n,m)$-diagrams are \textit{isotopic}
if the graph embeddings are isotopic relative to the boundary 
(both upper and lower) 
and we have the we have the same polynomial boxes in each region of $[0,1]\times [0,1]$.
\end{defi}

\begin{defi}\label{A1 Hecke}
We define the Diagrammatic Hecke Category of type $A_1$, $\HcalA$,  to be the graded $R$-linear monoidal category presented in the following way.
The objects of $\HcalA$  sums of tensor powers of $B_s$ and its grading shifts.

The morphisms  $B_s^{\otimes m}\longrightarrow B_s^{\otimes n}$
are $R$-linear combination of $(m,n)$-diagrams.
The morphisms of $\HcalA$ are generated by
\begin{enumerate}[(a)]
    \item The identity of $B_s$
        \begin{equation}
        \begin{strings}
        \idfig{(1,1)}{red}
        \end{strings}
        : B_s \longrightarrow B_s,
        \end{equation}
    of degree $0$.
    \item The trivalent vertices
        \begin{equation}
        \begin{strings}
            \trivu{(0,0)}{red}
            \idfigparam{(nd-1-1)}{-0.25}{red}
            \idfigparam{(nd-1-2)}{0.25}{red}
            \idfigparam{(nd-1-3)}{0.25}{red}
        \end{strings}
        : B_s\otimes B_s \longrightarrow B_s[-1] \hspace{0.1cm}
        \text{ and } \hspace{0.1cm}
        \begin{strings}
            \trivd{(0,0)}{red}
            \idfigparam{(nd-1-1)}{-0.25}{red}
            \idfigparam{(nd-1-2)}{-0.25}{red}
            \idfigparam{(nd-1-3)}{0.25}{red}
        \end{strings}
        : B_s[1]\longrightarrow B_s\otimes B_s,
        \end{equation}
    of degree $-1$.
    \item The dots 
    \begin{equation}
            \begin{strings}
            \dotuparam{(0,0)}{0.6}{0.06}{red}
            \path (nd-1-1) ++ (0,1);
            \end{strings}
            :B_s\longrightarrow \onecal[1]
            \hspace{0.1cm}
            \text{ and }
            \hspace{0.1cm}
            \begin{strings}
            \dotdparam{(0,1)}{0.6}{0.06}{red}
            \path (0,0) ++ (0,1);
            \end{strings}
            :\onecal\longrightarrow B_s[1]
            ,
    \end{equation}
    of degree $1$.
    \item The polynomial boxes
    \begin{equation}
            \begin{strings}
            \polybox{(0.1,0.5)}{$f$}
            \path (0,0) ++ (0,1);
            \end{strings} 
            :\onecal\longrightarrow \onecal
            ,
    \end{equation}
    of degree $\deg(f)$ (where $\deg(\alpha_s)=2$).
\end{enumerate}

These morphisms are subject to the isotopy relations and subject to the following local relations.

    \begin{equation}\label{red unit rel}
    \begin{boxedstrings}
    \idfig{(1,1)}{red}
    \dotl{(nd-1-3)}{red}
    \end{boxedstrings}
    =
    \begin{boxedstrings}
    \idfig{(1,1)}{red}
    \end{boxedstrings}
    =
    \begin{boxedstrings}
    \idfig{(1,1)}{red}
    \dotr{(nd-1-3)}{red}
    \end{boxedstrings}
    ,
    \end{equation}

    \begin{equation}\label{red H=I}
    \begin{boxedstrings}
        \Hlongparam{(0,0)}{0.5}{1}{red}{red}{red}{red}{red}
    \end{boxedstrings}
    =
    \begin{boxedstrings}
        \Ilongparam{(0,0)}{0.5}{1}{red}{red}{red}{red}{red}
    \end{boxedstrings},
    \end{equation}

	\begin{equation}
		\begin{boxedstrings}
			\node at (0,1) {};
		\node (a) at (-0.2,0.5) {};
		\def\coltop{red}
		\def\colbot{red}
		\capfigparam{(a)}{0.4}{0.2}{\coltop}
		\cupfigparam{(a)}{0.4}{0.2}{\coltop}
		\idfigparam {(nd-2-3)}{0.3}{\colbot}
		\end{boxedstrings}
		=0,
	\end{equation}

    \begin{equation}\label{red barbell}
    \begin{boxedstrings}
    \vbarb{(0,0.7)}{red}
    \path (0,0) ++ (0,1);
    \end{boxedstrings}
    =
    \begin{boxedstrings}
    \polybox{(0,0.5)}{$\alpha_s$}
    \path (0,0) ++ (0,1);
    \end{boxedstrings},
    \end{equation}
    
    \begin{equation}\label{left poly rel}
    	\begin{boxedstrings}
    		\polybox{(0,0.5)}{$f$}
    		\path (0,0) ++ (0,1);
    	\end{boxedstrings}
    	=
    	f
    	\begin{boxedstrings}
    		\path (0,0) ++ (0,1);
    	\end{boxedstrings},
    \end{equation}
    
    \begin{equation}\label{polynomial sum}
    \begin{boxedstrings}
    \polybox{(0,0.5)}{$f+g$}
    \path (0,0) ++ (0,1);
    \end{boxedstrings}  
    =
    \begin{boxedstrings}
    \polybox{(0,0.5)}{$f$}
    \path (0,0) ++ (0,1);
    \end{boxedstrings}  
    +
    \begin{boxedstrings}
    \polybox{(0,0.5)}{$g$}
    \path (0,0) ++ (0,1);
    \end{boxedstrings}  
    \end{equation}

    \begin{equation}\label{polynomial mult}
    \begin{boxedstrings}
    \polybox{(0,0.5)}{$fg$}
    \path (0,0) ++ (0,1);
    \end{boxedstrings}  
    =
    \begin{boxedstrings}
    \polybox{(0,0.5)}{$f$}
    \polybox{(0.65,0.5)}{$g$}
    \path (0,0) ++ (0,1);
    \end{boxedstrings}
    \end{equation}
    
    \begin{equation}\label{polynomial forcing}
    \begin{boxedstrings}
     \idfigparam{(0,1)}{1}{red}
     \polybox{(0.4,0.5)}{$f$}
     \end{boxedstrings} 
     =
     \begin{boxedstrings}
     \polybox{(-0.6,0.5)}{$s(f)$}
     \idfigparam{(0,1)}{1}{red}
     \end{boxedstrings} 
     +
    \begin{boxedstrings}
    \dotd{(0.75,1)}{red}
    \dotu{(0.75,0)}{red}
    \polybox{(0,0.5)}{$\partial_s(f)$}
    \end{boxedstrings}.
    \end{equation}
    
    Here, the dashed boxes on the boundary indicate that this relation takes place within a neighborhood homeomorphic to $D^2$.
    
    The tensor product of morphisms is given by horizontal concatenation 
    of diagrams (resizing them to fit into $[0,1]\times [0,1]$),
    while the composition is given by vertical concatenation (of morphisms with matching boundaries).
\end{defi}

Here are some important consequences of the relations in $\HcalA$:
\begin{enumerate}
    \item Any empty cycle is zero. To see this use \eqref{red H=I} repeatedly until you get a needle.
    \item We can always move polynomials to the right by using relation \ref{polynomial forcing}.
    \item One idempotent decomposition for $B_s^{\otimes 2}$ is
    \begin{equation}\label{idem decomp Bs}
        \begin{strings}
        \idfigparam{(0,2)}{1}{red}
        \idfigparam{(0.5,2)}{1}{red}
        \end{strings}
        =
        \frac{1}{2}
        \begin{strings}
        \idfigparam{(0,1)}{0.25}{red}
        \Ifig{(nd-1-2)}{red}
        \idfigparam{(nd-2-3)}{0.25}{red}
        \idfigparam{(nd-2-2)}{-0.25}{red}
        \idfigparam{(nd-2-4)}{0.25}{red}
        \hbarbparam{(0.125,0.875)}{0.25}{0.04}{red}
        \end{strings}
        +
        \frac{1}{2}
        \begin{strings}
        \idfigparam{(0,1)}{0.25}{red}
        \Ifig{(nd-1-2)}{red}
        \idfigparam{(nd-2-3)}{0.25}{red}
        \idfigparam{(nd-2-2)}{-0.25}{red}
        \idfigparam{(nd-2-4)}{0.25}{red}
        \hbarbparam{(0.125,0.125)}{0.25}{0.04}{red}
        \end{strings}.
    \end{equation}
    Since every idempotent in this decomposition factors through $B_s$, then 
    \begin{equation}
    	B_s\otimes B_s\cong B_s[1]\oplus B_s[-1]. 
    \end{equation}
    Note that this requires $\chr(\kk)\neq 2$.
    \item Using the previous idempotent decomposition
    every object $X\in\HcalA$ can be written as
    \begin{equation}\label{A1decompIntoIndec}
    	X\cong \bigoplus_{i\in\NN} 
    	\left(
    	B_s^{\oplus a_i}
    	\oplus
    	\onecal^{\oplus b_i}\right)[i],
    \end{equation}
    where $a_i,b_i\in \NN$ 
    and only finitely many of them are non-zero.
    
    This makes $B_s$ and $\onecal$ (and their grading shifts) the only indecomposable objects. 
    In particular, this shows that $\HcalA$
    is Karoubian.
\end{enumerate}

\begin{rema}
For a general Coxeter group, one typically defines the diagrammatic category of Bott-Samelson bimodules as in Definition \ref{A1 Hecke}
and takes its Karoubi envelope to get the category of Soergel Bimodules (our true object of study).
However, by enumerating their indecomposable objects we obtain that categories $\HcalA$ and $\HcalAxA$ are already Karoubian.  
For the general definition of the Diagrammatic Hecke Category see \cite[Ch. 10]{elias2020introduction}.
\end{rema}

We refer to
$
\Capp{red}
$
and
$
\Cupp{red}
$
as the \textit{cap} and \textit{cup} morphisms respectively.

Note that by isotopy and \eqref{red unit rel} we have
\begin{equation}
	\begin{strings}
		\capfigparam{(0,0)}{0.5}{0.5}{red}
		\path (0,0)--(0,1);
	\end{strings}
	=
	\begin{strings}
		\trivu{(0,0.75)}{red}
		\dotfig{(nd-1-1)}{red}
		\idfigparam{(nd-1-2)}{0.25}{red}
		\idfigparam{(nd-1-3)}{0.25}{red}
		\path (0,0)--(0,1);
	\end{strings}
	\text{ and }
	\begin{strings}
		\cupfigparam{(0,1)}{0.5}{0.5}{red}
		\path (0,0)--(0,1);
	\end{strings}
	=
	\begin{strings}
		\trivd{(0,0.75)}{red}
		\idfigparam{(nd-1-1)}{-0.25}{red}
		\idfigparam{(nd-1-2)}{-0.25}{red}
		\dotfig{(nd-1-3)}{red}
		\path (0,0)--(0,1);
	\end{strings}.
\end{equation}

Also note that $\left(B_s, 
\Capp{red}
,
\Cupp{red}
\right)$ is biadjoint to $B_s$.

The next proposition is a consequence of Proposition
 \ref{prop adjunction general}, 
 but we will prove it diagrammatically so that the reader becomes more familiar with the graphical calculus of $\HcalA$.

\begin{prop}\label{Bs selft adjoint}
	The functor $B_s\otimes(-)$ is self adjoint, i.e. for all $X,Y\in\HcalA$ we have a natural isomorphism
	$\Hom(B_s\otimes X,Y)\cong \Hom(X,B_s\otimes Y)$.
	Moreover, the functor $B_s[k]\otimes (-)$ is bi-adjoint to $B_s[-k]\otimes(-)$.
	Hence, $\HcalA$ is a rigid category.
\end{prop}
\begin{proof}
To describe this isomorphism, let
\begin{equation}
    \begin{strings}
        \draw (0.3,1) -- (0.3,1.5);
        \draw (0.7,1) -- (0.7,1.5);
        \draw (0,0.5) rectangle (1,1);
        \draw[red] (0.1,0) -- (0.1,0.5);
        \draw (0.4,0) -- (0.4,0.5);
        \draw (0.8,0) -- (0.8,0.5);
        \node (a) at (0.5,0.75) {$T$};
        \node (b) at (0.6,0.25) {$X$};
        \node (c) at (0.5,1.25) {$Y$};
    \end{strings}
\end{equation}
represent a map from $T:B_s\otimes X\longrightarrow Y$. 
We can pre-compose $T$ with the cup morphism 
$
\begin{miniunboxedstrings}
    \cupfigparam{(0,1)}{0.5}{0.5}{red}
    \path (0,0) -- ++(0.5,-0.2);
    \end{miniunboxedstrings}
    \otimes \Id_X
$
and get the mapping
\begin{equation}
    \begin{strings}
        \draw (0.3,1) -- (0.3,1.5);
        \draw (0.7,1) -- (0.7,1.5);
        \draw[red] (0.1,0) -- (0.1,0.5);
        \draw (0.4,0) -- (0.4,0.5);
        \draw (0.8,0) -- (0.8,0.5);
        \draw (0,0.5) rectangle (1,1);
        \node (a) at (0.5,0.75) {$T$};
        \node (b) at (0.6,0.25) {$X$};
        \node (c) at (0.5,1.25) {$Y$};
    \end{strings}
    \longmapsto
    \begin{strings}
        \draw (0.3,1) -- (0.3,1.5);
        \draw (0.7,1) -- (0.7,1.5);
        \cupfigparam{(0.1,0.5)}{-0.4}{0.2}{red}
        \idfigparam{(-0.3,0.5)}{-1}{red}
        \draw (0.4,0) -- (0.4,0.5);
        \draw (0.8,0) -- (0.8,0.5);
        \draw (0,0.5) rectangle (1,1);
        \node (a) at (0.5,0.75) {$T$};
        \node (b) at (0.6,0.25) {$X$};
        \node (c) at (0.5,1.25) {$Y$};
    \end{strings}.
\end{equation}
Its inverse consists of applying the corresponding cap morphism on top.

Since all the objects of $\HcalA$ are finite direct sums of grading shifts of $\onecal$ and $B_s$, we get that $\HcalA$ is rigid.
\end{proof}

\begin{defi}
		Define the rotation functor $(-)^\vee:\HcalA\longrightarrow \HcalA^{\op,\otimesop}$ to be the strict additive contravariant and tensor contravariant functor such that:
	\begin{itemize}
		\item $(-)^\vee$ sends the objects $B_s[k]\mapsto B_s[-k]$,
		$\onecal[k]\mapsto \onecal[-k]$ and extends linearly to all objects of $\HcalA$.
		\item Sends a degree $d$ morphism $T:X\longrightarrow Y$
		to a degree $-d$ morphism $T^\vee:Y^\vee\mapsto X^\vee$ given by rotating $T$ by $180^\circ$.		
	\end{itemize}
\end{defi}

Let $\eval_{B_s^{\otimes k}}:B_s^{\otimes k}\otimes B_s^{\otimes k} \longrightarrow \onecal$ 
and 
$\coeval_{B_s^{\otimes k}}:\onecal\longrightarrow B_s^{\otimes k}\otimes B_s^{\otimes k}$ 
be the morphism corresponding to nesting $k$ caps or cups respectively.
Given that $\left(B_s, 
\Capp{red}
,
\Cupp{red}
\right)$ is biadjoint to $B_s$,
we can easily show that 
$\left(B_s^{\otimes k},\eval_{B_s^{\otimes k}},\coeval_{B_s^{\otimes k}}\right)$
is a biadjoint to $B_s^{\otimes k}$.
The next proposition shows that this, in fact, determines a pivotal structure on $\Ccal$.

\begin{prop}\label{prop HA1 is pivotal}
	With the two-sided adjunction structures above we have that the functors 
	$(-)^*=\rstar{(-)}=(-)^\vee$.
	Hence, $\HcalA$ is pivotal.
\end{prop}
\begin{proof}
	It suffices to verify this equality
	on the generating morphisms
	\begin{equation}
		\begin{strings}
			\trivu{(0,0)}{red}
			\idfigparam{(nd-1-1)}{-0.25}{red}
			\idfigparam{(nd-1-2)}{0.25}{red}
			\idfigparam{(nd-1-3)}{0.25}{red}
		\end{strings}
		,
		\hspace{0.2cm}
		\begin{strings}
			\trivd{(0,0)}{red}
			\idfigparam{(nd-1-1)}{-0.25}{red}
			\idfigparam{(nd-1-2)}{-0.25}{red}
			\idfigparam{(nd-1-3)}{0.25}{red}
		\end{strings}
		,
		\hspace{0.2cm}
		 \begin{strings}
			\dotuparam{(0,0)}{0.6}{0.06}{red}
			\path (nd-1-1) ++ (0,1);
		\end{strings},
		\hspace{0.1cm}
		\text{ and }
		\hspace{0.1cm}
		\begin{strings}
			\dotdparam{(0,1)}{0.6}{0.06}{red}
			\path (0,0) ++ (0,1);
		\end{strings}.
	\end{equation}

	For the first generator,
	\begin{equation}
		\left(
		\begin{strings}
			\trivu{(0,0)}{red}
			\idfigparam{(nd-1-1)}{-0.375}{red}
			\idfigparam{(nd-1-2)}{0.375}{red}
			\idfigparam{(nd-1-3)}{0.375}{red}
		\end{strings}
		\right)^*
		=
		\begin{strings}
			\trivuparamcol{(0,0)}{0.375}{0.375}{red}{red}{red}
			\capfigparam{(nd-1-1)}{-0.375}{0.25}{red}
			\cupfigparam{(nd-1-2)}{0.875}{0.4}{red}
			\cupfigparam{(nd-1-3)}{0.25}{0.125}{red}
			\idfigparam{(nd-2-2)}{ 0.875}{red}
			\idfigparam{(nd-3-2)}{-0.875}{red}
			\idfigparam{(nd-4-2)}{-0.875}{red}
		\end{strings}
		=
		\begin{strings}
			\trivd{(0,0)}{red}
			\idfigparam{(nd-1-1)}{-0.375}{red}
			\idfigparam{(nd-1-2)}{-0.375}{red}
			\idfigparam{(nd-1-3)}{0.375}{red}
		\end{strings}.
	\end{equation}
	The rest of the calculations are just as straightforward.

\end{proof}

In \cite{elias2010diagrammatics}, Elias and Khovanov used diagrammatic arguments to calculate the $\Hom$ spaces of
$\Hcal(W)$ for any Coxeter group $W$ of type $A$.
In particular, we have the following proposition.
\begin{prop}
	[{\cite[Coro. 4.26]{elias2010diagrammatics}}]\label{prop HomBs,1 in A1}
	We have that
	\begin{equation}
		\Hom(\onecal,\onecal)= R \text{ and }
		\Hom(B_s,\onecal)= R \cdot
		\begin{strings}
			\dotuparam{(0,0)}{0.6}{0.06}{red}
			\path (0,0) ++ (0,1);
		\end{strings}.
	\end{equation}
	Both of them are free graded left $R$-modules with graded dimension
	\begin{equation}
		\grdim_R(\Hom(\onecal,\onecal))=1  \text{ and }
		\grdim_R(\Hom(B_s,\onecal))= v.
	\end{equation}
\end{prop}

We will now use this to calculate the graded dimension of the endomorphism ring of $B_s$. This is an example on how to calculate $\Hom$ spaces in $\HcalA$.
\begin{coro}\label{coro grdim Bs}
	We have that
	\begin{equation}
		\grdim_R(\End(B_s))= 1 + v^2.
	\end{equation}
\end{coro}
\begin{proof}
	Using the adjunction in Proposition \ref{Bs selft adjoint}, we get
	\begin{align}
		\Hom(B_s,B_s) &\cong \Hom(B_s\otimes B_s,\onecal)
		\\
		&\cong \Hom(B_s[-1],\onecal)\oplus \Hom(B_s[1],\onecal)
		\\
		&\cong \Hom(B_s,\onecal)[1]\oplus \Hom(B_s,\onecal)[-1].
	\end{align}
	Hence,
	\begin{align}
		\grdim_R(\End(B_s))&= \grdim_R(B_s[1])+\grdim_R(B_s[-1])
		\\
		&= \grdim_R(B_s)v+\grdim_R(B_s)v^{-1}
		\\
		&= 1+v^2.
	\end{align}
\end{proof}

\begin{coro}
	The category $\HcalA$ is a graded Krull-Schmidt category.
\end{coro}
\begin{proof}
	The ring $\End(\onecal)=R$ is non-negatively graded with $\End^0(\onecal)=\kk$, thus local. 
	By Corollary \ref{coro grdim Bs},  $\End(B_s)\cong R \oplus R[2]$ as a graded left $R$-module.
	Since $\End(B_s)$ is non-negatively graded and $\End^0(B_s)=\kk$,
	it also is graded local.
\end{proof}

\begin{coro}\label{coro free R modules A1}
	All the $\Hom$ spaces of $\HcalA$ are free graded left $R$-modules.
\end{coro}
\begin{proof}
	By Corollary \ref{coro Hom spaces general}, all $\Hom$ spaces are
	finite direct sums of grading shifts of
	$\Hom(\onecal,\onecal)$ and $\Hom(B_s,\onecal)$.
	By Proposition \ref{prop HomBs,1 in A1}, these are all free left $R$-modules.
\end{proof}

\subsection{Diagrammatic Hecke Category of type $A_1\times A_1$}
\label{subsection diagrammatics of AxA}
The Diagrammatic Hecke Category of type $A_1\times A_1$ is defined analogously to $\HcalA$, but adding a generating object we will call $B_t$.
Let us define the diagrams in this category.

\begin{defi}
Let $w_1=(i_1,i_2,...,i_m)$ and $w_2=(j_1,j_2,...,j_n)$ be words in letters $s$ and $t$. Let $R:= \kk[\alpha_s,\alpha_t]$.
A \textit{$(w_1,w_2)$-Hecke diagram} 
(of type $A_1\times A_1$) consists of the following.
\begin{enumerate}[(a)]
    \item A graph $G$ embedded on the planar strip $[0,1]\times (0,1)$ whose edges are labeled by $s$ (red) or $t$ (blue), with the following properties:
    \begin{enumerate}[(i)]
        \item $G$ has $m$ strands on the lower boundary, labeled with $w_1=(i_1,...,i_m)$ (from left to right).
        \item $G$ has $n$ strands on the upper boundary, labeled with $w_2=(j_1,...,j_n)$ (from left to right).
        \item $G$ is made of one-colored trivalent vertices
            $\left(
            \Trivup{red}{red}{red}
            ,
            \Trivup{blue}{blue}{blue}
            \right)
            $
        , one-colored dots 
            $\left(
            \Dotup{red},
            \Dotup{blue}
            \right)
            $
        , and two color crossings
        $\left( \Cross{red}{blue}, \Cross{blue}{red} \right)$.
    \end{enumerate}

    \item Floating polynomial boxes in the regions of $[0,1]\times (0,1)$ partitioned by $G$, labeled by polynomials $f\in R$.
    \end{enumerate}
    
    We say two $(w_1,w_2)$-diagrams are \textit{isotopic}
    if the graph embeddings are isotopic (relative to their boundary) and we have the we have the same polynomial boxes in each region of $[0,1]\times (0,1)$.
\end{defi}

\begin{defi}
The \textit{diagrammatic Hecke category} of type $A_1\times A_1$, $\HcalAxA$ (abbreviated into $\Hcal$),  to be the graded $R$-linear monoidal category presented in the following way.

The objects of $\Hcal$ are generated by $B_s$ and $B_t$ and their grading shifts, under the operations of direct sum and tensor product.

The generating morphisms are:
        \begin{enumerate}[(a)]
            \item The identity of $B_s$ and $B_t$
                \begin{equation}
                \begin{strings}
                \idfig{(1,1)}{red}
                \end{strings}
                : B_s \longrightarrow B_s,\hspace{0.2cm}
                \text{ and }
                \hspace{0.2cm}
                \begin{strings}
                \idfig{(1,1)}{blue}
                \end{strings}
                : B_t \longrightarrow B_t
                \end{equation}
            and have degree $0$.
            \item The trivalent vertices in red and blue
                \begin{equation}
                \begin{strings}
                    \trivu{(0,0)}{red}
                    \idfigparam{(nd-1-1)}{-0.25}{red}
                    \idfigparam{(nd-1-2)}{0.25}{red}
                    \idfigparam{(nd-1-3)}{0.25}{red}
                \end{strings}
                : B_s\otimes B_s \longrightarrow B_s[-1],
                \hspace{0.1cm}
                \begin{strings}
                    \trivd{(0,0)}{red}
                    \idfigparam{(nd-1-1)}{-0.25}{red}
                    \idfigparam{(nd-1-2)}{-0.25}{red}
                    \idfigparam{(nd-1-3)}{0.25}{red}
                \end{strings}
                : B_s[1]\longrightarrow B_s\otimes B_s,
                \end{equation}
                \begin{equation}
                \begin{strings}
                    \trivu{(0,0)}{blue}
                    \idfigparam{(nd-1-1)}{-0.25}{blue}
                    \idfigparam{(nd-1-2)}{0.25}{blue}
                    \idfigparam{(nd-1-3)}{0.25}{blue}
                \end{strings}
                : B_t\otimes B_t \longrightarrow B_t,
                \hspace{0.1cm}
                \text{ and }
                \hspace{0.1cm}
                \begin{strings}
                    \trivd{(0,0)}{blue}
                    \idfigparam{(nd-1-1)}{-0.25}{blue}
                    \idfigparam{(nd-1-2)}{-0.25}{blue}
                    \idfigparam{(nd-1-3)}{0.25}{blue}
                \end{strings}
                : B_t\longrightarrow B_t\otimes B_t,
                \end{equation}
            of degree $-1$.
            \item The dots 
            \begin{equation}
                    \begin{strings}
                    \dotuparam{(0,0)}{0.6}{0.06}{red}
                    \path (nd-1-1) ++ (0,1);
                    \end{strings}
                    :B_s\longrightarrow \onecal[1],
                    \hspace{0.2cm}
                    \begin{strings}
                    \dotdparam{(0,1)}{0.6}{0.06}{red}
                    \path (0,0) ++ (0,1);
                    \end{strings}
                    :\onecal\longrightarrow B_s[1]
                    ,
            \end{equation}
            \begin{equation}
                    \begin{strings}
                    \dotuparam{(0,0)}{0.6}{0.06}{blue}
                    \path (nd-1-1) ++ (0,1);
                    \end{strings}
                    :B_t\longrightarrow \onecal[1],
                    \hspace{0.1cm}
                    \text{ and }
                    \hspace{0.1cm}
                    \begin{strings}
                    \dotdparam{(0,1)}{0.6}{0.06}{blue}
                    \path (0,0) ++ (0,1);
                    \end{strings}
                    :\onecal\longrightarrow B_t[1]
                    ,
            \end{equation}
            of degree $1$.
            \item The polynomial boxes
            \begin{equation}
                    \begin{strings}
                    \polybox{(0,0.5)}{$f$}
                    \path (0,0) ++ (0,1);
                    \end{strings} 
                    :\onecal\longrightarrow \onecal
                    ,
            \end{equation}
            of degree $\deg(f)$ (where $\deg(\alpha_s)=2$ and $\deg(\alpha_t)=2$).

		\item The $4$-valent vertices
    \begin{equation}
        \begin{strings}
        \Xfig{(0,0)}{red}{blue}
        \idfigparam{(nd-1-1)}{-0.25}{red}
        \idfigparam{(nd-1-2)}{-0.25}{blue}
        \idfigparam{(nd-1-4)}{0.25}{red}
        \idfigparam{(nd-1-3)}{0.25}{blue}
        \end{strings},
        \hspace{0.1cm}
        \text{ and }
        \hspace{0.1cm}
        \begin{strings}
        \Xfig{(0,0)}{blue}{red}
        \idfigparam{(nd-1-1)}{-0.25}{blue}
        \idfigparam{(nd-1-2)}{-0.25}{red}
        \idfigparam{(nd-1-4)}{0.25}{blue}
        \idfigparam{(nd-1-3)}{0.25}{red}
        \end{strings}
    \end{equation}
    of degree $0$.
   
   	        \end{enumerate}
    These morphisms are subject to isotopy relations, the local relations 
    \eqref{red unit rel}-\eqref{polynomial forcing} both in red and blue colors, and the $2$ color relations
    \begin{equation}\label{2 color relations A1xA1}
        \begin{boxedstrings}
            \Xfig{(0,0)}{red}{blue}
            \Xfig{(nd-1-3)}{blue}{red}
        \end{boxedstrings}
            =
         \begin{boxedstrings}
            \idfigparam{(0,0)}{1}{red}
            \idfigparam{(0.5,0)}{1}{blue}
            \end{boxedstrings} 
            ,\hspace{0.5cm}
            \begin{boxedstrings}
            \Xfig{(0,0)}{blue}{red}
            \idfigparam{(nd-1-1)}{-0.25}{blue}
            \idfigparam{(nd-1-3)}{0.25}{red}
            \idfigparam{(nd-1-4)}{0.25}{blue}
            \dotfig{(nd-1-2)}{red}
            \end{boxedstrings}
        =
        \begin{boxedstrings}
            \idfig{(1,1)}{blue}
            \dotu{(0.75,0)}{red}
        \end{boxedstrings}    
            ,\hspace{0.2cm}
            \text{and}
            \hspace{0.2cm}
            \begin{boxedstrings}
            \Xfig{(0,0)}{red}{blue}
            \idfigparam{(nd-1-1)}{-0.25}{red}
            \idfigparam{(nd-1-3)}{0.25}{blue}
            \idfigparam{(nd-1-4)}{0.25}{red}
            \dotfig{(nd-1-2)}{blue}
            \end{boxedstrings}
        =
        \begin{boxedstrings}
            \idfig{(1,1)}{red}
            \dotu{(0.75,0)}{blue}
        \end{boxedstrings}    .
    \end{equation}
    
\end{defi}
Note that by \eqref{2 color relations A1xA1} and \eqref{red barbell},
\begin{equation}\label{poly slide red blue}
	\alpha_s
	\begin{boxedstrings}
		\idfigparam{(0,1)}{1}{blue}
	\end{boxedstrings}
	=
	\begin{boxedstrings}
		\idfigparam{(0,1)}{1}{blue}
		\vbarb{(-0.2,0.7)}{red}
	\end{boxedstrings}
	=
	\begin{boxedstrings}
		\idfigparam{(0,1)}{1}{blue}
		\hbarbparam{(-0.2,0.5)}{0.4}{0.06}{red}
	\end{boxedstrings}
	=
	\begin{boxedstrings}
		\idfigparam{(0,1)}{1}{blue}
		\vbarb{(0.2,0.7)}{red}
	\end{boxedstrings}
	=
	\begin{boxedstrings}
		\idfigparam{(0,1)}{1}{blue}
	\end{boxedstrings}
	\alpha_s.
\end{equation}
Therefore, any polynomial on $\alpha_s$ slides over the blue strands and any polynomial on $\alpha_t$ slides over the red strands.

As before, we have $B_s^2\cong B_s[1]\oplus B_s[-1]$ and $B_t^2\cong B_t[1]\oplus B_t[-1]$. Let $A:=\{\onecal,B_s,B_t,B_sB_t\}$.
Since $B_sB_t\cong B_tB_s$, we have that
\begin{align}
	B_sB_t\otimes B_s & \cong B_tB_s\otimes B_s 
	\cong B_sB_t[-1] \oplus B_sB_t[1],
	\\
	B_sB_t\otimes B_t & \cong  B_sB_t[-1] \oplus B_sB_t[1]
	,
	\\
	B_sB_t\otimes B_sB_t &\cong B_s\otimes B_tB_s \otimes B_t
	\\
	&\cong B_s\otimes B_sB_t \otimes B_t
	\\
	&\cong 
	B_sB_t[-2] \oplus B_sB_t^{\oplus 2}
	\oplus B_sB_t[2]
	.
\end{align}
Hence, the tensor products of any two objects in $A$ decomposes again as direct sums of objects in $A$.
Therefore we have the following proposition.
\begin{prop}\label{decomposition on HAxA}
	Any object $X\in\Hcal$ is a finite direct sum of the objects
	in $A$ and their grading shifts.
\end{prop}

Like in the $\HcalA$ case, $\left(B_s, 
\Capp{red}
,
\Cupp{red}
\right)$ is biadjoint to $B_s$ and
$\left(B_t, 
\Capp{blue}
,
\Cupp{blue}
\right)$
is biadjoint to $B_t$.
By Proposition \ref{prop adjunction general}, we conclude that
\begin{prop}
	The functors $B_s[k]\otimes (-)$ and $B_t[k]\otimes (-)$
	are bi-adjoint to the functors $B_s[-k]\otimes (-)$ and $B_t[-k]\otimes (-)$.
\end{prop}

\begin{defi}
	Define the rotation functor
	$(-)^\vee:\Hcal\longrightarrow \Hcal^{\op,\otimesop}$ to be the strict additive contravariant and tensor contravariant functor such that
	\begin{itemize}
		\item  Sends $\onecal[k]\mapsto \onecal[-k]$,
		$B_s[k]\mapsto B_s[-k]$, and $B_t[k]\mapsto B_t[-k]$.
		\item Sends a degree $d$ morphism $T:X\longrightarrow Y$
		to a degree $-d$ morphism $T^\vee:Y^\vee\mapsto X^\vee$ given by rotating $T$ by $180^\circ$.
	\end{itemize} 
		
\end{defi}

By nesting cups/caps, we can extend the biadjoints of $B_s$ and $B_t$ to a two-sided adjoint structure on $\Hcal$.

\begin{prop}\label{prop HAxA is pivotal}
	Let us fix the two-sided adjoint structure defined above.
	The functors
	$(-)^*=\rstar{(-)}=(-)^\vee$.
	Hence, $\Hcal$ is pivotal.	
\end{prop}
\begin{proof}
	In the proof of Proposition \ref{prop HA1 is pivotal}, 
	we showed that these functors agree on the one-color generators.
	For the 2-color generators,
	notice that
	\begin{equation}
		\begin{strings}
			\Xfig{(0,0)}{red}{blue}
			\idfigparam{(nd-1-1)}{-0.375}{red}
			\idfigparam{(nd-1-2)}{-0.375}{blue}
			\idfigparam{(nd-1-4)}{0.375}{red}
			\idfigparam{(nd-1-3)}{0.375}{blue}
		\end{strings}
		=
		\begin{strings}
			\Xfig{(0,0)}{red}{blue}
			\capfigparam{(nd-1-1)}{-0.25}{0.1}{red}
			\capfigparam{(nd-1-2)}{-1.25}{0.2}{blue}
			\cupfigparam{(nd-1-3)}{1.25}{0.2}{blue}
			\cupfigparam{(nd-1-4)}{0.25}{0.1}{red}
			\idfigparam{(nd-2-2)}{ 0.875}{red}
			\idfigparam{(nd-3-2)}{ 0.875}{blue}
			\idfigparam{(nd-4-2)}{-0.875}{blue}
			\idfigparam{(nd-5-2)}{-0.875}{red}
		\end{strings}
		=
		\begin{strings}
			\Xfig{(0,0)}{red}{blue}
			\idfigparam{(nd-1-1)}{-0.375}{red}
			\idfigparam{(nd-1-2)}{-0.375}{blue}
			\idfigparam{(nd-1-4)}{ 0.375}{red}
			\idfigparam{(nd-1-3)}{ 0.375}{blue}
		\end{strings}.
	\end{equation}
		
\end{proof}

As before, we need the diagrammatic argument in \cite{elias2010diagrammatics} to get a description of the $\Hom$ spaces in $\Hcal$.
\begin{prop}[{\cite[Coro. 4.26]{elias2010diagrammatics}}]\label{prop HomBs,1 in AxA}
	We have that \vspace{-0.25cm}
	\begin{align}
		\Hom(\onecal,\onecal)&= R, 
		&
		\Hom(B_s,\onecal)&= R \cdot
		\begin{strings}
			\dotuparam{(0,0)}{0.6}{0.06}{red}
			\path (0,0) ++ (0,1);
		\end{strings},
		\\
		\Hom(B_t,\onecal)&= R \cdot
		\begin{strings}
			\dotuparam{(0,0)}{0.6}{0.06}{blue}
			\path (0,0) ++ (0,1);
		\end{strings}, 
		&
		\Hom(B_sB_t,\onecal)&= R \cdot
		\begin{strings}
			\dotuparam{(0,0)}{0.6}{0.06}{red}
			\dotuparam{(0.25,0)}{0.6}{0.06}{blue}
			\path (0,0) ++ (0,1);
		\end{strings}.
	\end{align}
	All of them are free graded left $R$-modules with graded dimension
	\begin{align}
		\grdim_R(\Hom(\onecal,\onecal))&=1,
		\\
		\grdim_R(\Hom(B_s,\onecal))&= v,
		\\
		\grdim_R(\Hom(B_t,\onecal))&= v,
		\\
		\grdim_R(\Hom(B_sB_t,\onecal))&= v^2.
	\end{align}
\end{prop}

We can check like in Corollary \ref{coro grdim Bs}, that the endomorphism rings of $\onecal, B_s,B_t,$ and $B_sB_t$ are graded local.
Thus we have the following corollary.
\begin{coro}
	The category $\HcalAxA$ is a Krull-Schmidt category.
\end{coro}

Like Corollary \ref{coro free R modules A1},
the following is a consequence of Corollary \ref{coro Hom spaces general} 
and Proposition \ref{prop HomBs,1 in AxA}.
\begin{coro}\label{coro free R modules AxA}
	All the $\Hom$ spaces of $\Hcal$ are free graded left $R$-modules.
\end{coro}

\subsection{Equivariantization of 
	$\Hcal(A_1\times A_1)$ with respect to $\tau$.
}\label{subsection Equivariantization of Soergel}

\begin{defi}
Let $R=\kk[\alpha_s,\alpha_t]$. 
Define \textit{$\tau_R:R\longrightarrow R$} to be the algebra morphism
permuting $\alpha_s$ and $\alpha_t$.
It is clear that $\tau_R$ has order 2.
\end{defi}

\begin{defi}
Define the functor $\tau:\Hcal\longrightarrow \Hcal$ to be the strict monoidal functor that:
\begin{enumerate}
    \item sends $B_s\mapsto B_t$ and $B_t\mapsto B_s$.
    \item $\tau$ acts on diagrams by exchanging the two colors and acts on polynomials by $\tau_R$.
\end{enumerate}
\end{defi}

\begin{prop}
The functor $\tau$ is its own inverse and it defines a strict action of $\ZZ/2$ on $\Hcal$.
\end{prop}
\begin{proof}
This is clear.
\end{proof}

Given the functor $\tau$, we can think of the equivariantization
$\Hcaltau$.
As seen in Section \ref{sec-equiv-2}, an object in 
$\Hcaltau$ is an object $A\in \Hcal$ together with a morphism
$f:\tau(A)\longrightarrow X$ where $f\circ \tau(f)= \Id_A$.
We will start by decomposing $\Ind(\onecal), \Ind(B_s)=\Ind(B_t)$, and $\Ind(B_sB_t)$ into indecomposables.
Consider the equivariant structures $(\onecal,1)$ and 
$\left(B_s B_t,
\Cross{blue}{red}
\right)$.
From the computation in \eqref{splitting ind X}, we get
\begin{equation}\label{decomp Ind1}
	\Ind(\onecal)\cong (\onecal,1)\oplus (\onecal, -1)
\end{equation} 
and
\begin{equation}\label{decom IndBsBt}
	\Ind(B_s B_t)\cong 
	\left(B_s B_t,	\Cross{blue}{red}\right)
	\oplus 
	\left(B_s B_t,-	\Cross{blue}{red}\right)
	.
\end{equation}
Since $B_s$ and $B_t$,
are not  $\tau$-invariant, 
the object
\begin{equation}
	\Ind(B_s)=\Ind(B_t)=
	\left(B_s\oplus B_t,
	\Ymatrix : B_t\oplus B_s \longrightarrow B_s\oplus B_t
	\right)
\end{equation}
is indecomposable.

From now on we will use the following notation:
\begin{align}
	\onecal&:=(\onecal,1),
\\
\Xobj&:=(\onecal,-1),
\\
\Yobj&:=\left(B_s\oplus B_t,
\Ymatrix
\right),
\\
\Zobj&:=
\left(B_s B_t,\Cross{blue}{red}\right)
,
\\
\XZobj&:=
\left(B_s B_t,-\Cross{blue}{red}\right)
.
\end{align}
This makes sense since $\Xobj\otimes \Zobj \cong \XZobj$.
Also define
\begin{equation}
	\Ical:=\{\onecal, \Xobj,\Yobj,\Zobj,\XZobj\}.
\end{equation}

\begin{lemma}\label{Ind unique decomposition}
	The objects $\Ind(\onecal), \Ind(B_s),$ and $\Ind(B_sB_t)$
	have a graded local decomposition, and therefore have unique decomposition.
	Moreover, $\Ind(A)$ has a graded local decomposition for all $A\in\Hcal$.
\end{lemma}
\begin{proof}
	In Proposition \ref{prop end spaces Htau} we will show that the endomorphism rings
	of objects in $\Ical$ are graded local.	
	The object $\Ind(B_s)= Y$ is already indecomposable and has a graded local endomorphism ring.
	In \eqref{decomp Ind1} and \eqref{decom IndBsBt} we showed
	graded local decompositions for $\Ind(\onecal)$ and $\Ind(B_sB_t)$.
	Since $\Hcal$ is a graded Krull-Schmidt category,
	$\Ind(A)$ is a direct sum of copies of $\Ind(\onecal), \Ind(B_s),$ and $\Ind(B_sB_t)$ (and their grading shifts).
	Hence $\Ind(A)$ also has a graded local decomposition.
\end{proof}

\begin{prop}\label{prop full indecomposable set Htau}
	The set $\Ical:=\{\onecal, \Xobj,\Yobj,\Zobj,\XZobj\}$ and
	its grading shifts make a full set of indecomposable objects of
	$\Hcaltau$.
\end{prop}
\begin{proof}
	By Proposition \ref{Prop unique decomp indX},
	every indecomposable object of $\Hcaltau$ is a summand
	of $\Ind(\onecal), \Ind(B_s),$ or $\Ind(B_sB_t)$.
	By Lemma \ref{Ind unique decomposition}, these objects have unique decomposition, so they can not have any summands not in 
	$\Ical$.
\end{proof}

\begin{coro}\label{coro Htau is Krull schmidt}
	Every object in $\Hcaltau$ is a direct sum of objects in $\Ical$
	(and their grading shifts).
	Hence, the category $\Hcaltau$ is a graded Krull-Schmidt category.
\end{coro}
\begin{proof}
	Let $(A,f)\in\Hcaltau$.
	By Lemma \ref{Ind unique decomposition}, $\Ind(A)$ has a graded local decomposition.
	By Corollary \ref{coro summands}, $(A,f)$ is a direct summand of 
	$\Ind(A)$ so it also has a local decomposition.
\end{proof}

\subsection{The Grothendieck Ring of $\Hcaltau$}\label{subsection Grothendieck group}

Our next goal is to describe the Grothendieck ring of $\Hcaltau$. 
We will do this by writing each tensor products of two indecomposable objects in $\Hcaltau$ as a direct sum of indecomposables.

We will write all the morphisms involved in the following notation.
\begin{notation}
	The tensor products of $\onecal, B_s\oplus B_t,$ and  $B_sB_t$ can be decomposed naively into a direct sum of objects of the form
	$B_{i_1}\otimes ... \otimes B_{i_m}$ for $i_j\in \{s,t\}$ and $m\in\NN$.
	For example,
	\begin{equation}
		(B_s\oplus B_t)\otimes (B_s\oplus B_t) = 
		B_s\otimes B_s \oplus
		B_s\otimes B_t \oplus
		B_t\otimes B_s \oplus
		B_t\otimes B_t 	.
	\end{equation}
	By associating the object $B_{i_1}\otimes ... \otimes B_{i_m}$ with the tuple $(i_1,...,i_m)$ 
	we can order summands with the same number of tensor factors lexicographically (with $s>t$).
	This order allows us to represent morphisms between direct sums of objects of this form by matrices of morphisms between the summands.
	For instance, a morphism $\phi:\Yobj\otimes \Yobj\longrightarrow \Yobj\otimes \Yobj$ can be represented by a matrix of the form
	\begin{equation}
		\begin{pmatrix}
			 B_s\otimes B_s \rightarrow B_s\otimes B_s & 
			 B_s\otimes B_t \rightarrow B_s\otimes B_s &
			 B_t\otimes B_s \rightarrow B_s\otimes B_s &
			 B_t\otimes B_t \rightarrow B_s\otimes B_s \\
			 B_s\otimes B_s \rightarrow B_t\otimes B_s & 
			 B_s\otimes B_t \rightarrow B_t\otimes B_s &
			 B_t\otimes B_s \rightarrow B_t\otimes B_s &
			 B_t\otimes B_t \rightarrow B_t\otimes B_s \\
			 B_s\otimes B_s \rightarrow B_s\otimes B_t & 
			 B_s\otimes B_t \rightarrow B_s\otimes B_t &
			 B_t\otimes B_s \rightarrow B_s\otimes B_t &
			 B_t\otimes B_t \rightarrow B_s\otimes B_t \\
			 B_s\otimes B_s \rightarrow B_t\otimes B_t & 
			 B_s\otimes B_t \rightarrow B_t\otimes B_t &
			 B_t\otimes B_s \rightarrow B_t\otimes B_t &
			 B_t\otimes B_t \rightarrow B_t\otimes B_t 		
		\end{pmatrix}.
	\end{equation}
	To simplify our notation further, we will only write the non-zero entries of our matrices.
\end{notation}

\begin{example}\label{ex decomp YY}
	We will start by calculating the decomposition of $(B_s\oplus B_t)\otimes(B_s\oplus B_t)$ inside $\Hcal$.
	\begin{align}
		(B_s\oplus B_t)\otimes(B_s\oplus B_t) =&
		(B_s^{\otimes 2} \oplus B_t^{\otimes 2})
		\oplus (B_sB_t  \oplus B_tB_s) 
		\\
		\begin{split}
			=&(B_s[-1]\oplus B_s[1] \oplus B_t[-1] \oplus B_t[1])
			\\
			&\oplus (B_sB_t  \oplus B_tB_s) 
		\end{split}
		\\
		\begin{split}
			=&(B_s\oplus B_t)[-1] \oplus (B_s\oplus B_t)[1]
			\\
			&\oplus (B_sB_t  \oplus B_tB_s) .
		\end{split}
	\end{align}
	This suggests that $\Yobj\otimes \Yobj \cong \Yobj[-1] \oplus Y[+1] \oplus \Ind(B_sB_t)$.

Let us consider the morphism
\begin{equation}
	\id_{\Yobj}=
	\begin{pmatrix}
		\iddoubletextmatrix{red}{red}  &  &  &  \\
		 & \iddoubletextmatrix{red}{blue} &  &  \\
		 &  & \iddoubletextmatrix{blue}{red} &  \\
		 &  &  & \iddoubletextmatrix{blue}{blue}
	\end{pmatrix}
	.
\end{equation}
Observe that the restriction
\begin{equation}
	\id_{\Yobj}|_{B_sB_t\oplus B_tB_s}
	=
	\begin{pmatrix}
		\iddoubletextmatrix{red}{blue}  \\
		&  \iddoubletextmatrix{blue}{red} 
	\end{pmatrix}
	=
	\id_{\Ind(B_sB_t)}.
\end{equation}
This is an idempotent factoring through $\Ind(B_sB_t)$, so $\Ind(B_sB_t)$ is a summand of $\Yobj\otimes \Yobj$.

Now consider
\begin{equation}
	\id_{\Yobj}|_{B_sB_s\oplus B_tB_t}
	=
	\begin{pmatrix}
		\iddoubletextmatrix{red}{red}& \\
		& \iddoubletextmatrix{blue}{blue}
	\end{pmatrix}
	.
\end{equation}
Notice that \eqref{idem decomp Bs} gives us a decomposition of $\id_{B_s\otimes B_s}$, but its idempotents are 
not $\tau$-invariant. 
To fix this, an extra 
$\alpha_t\hspace{-0.15cm}
\scaledstrings[0.5]{
\idfigparam{(0,1.5)}{0.5}{red}
\Ifigparamcol{(nd-1-2)}{1}{0.5}{red}{red}{red}{red}{red}
\idfigparam{(nd-2-3)}{0.5}{red}
\idfigparam{(nd-2-2)}{-0.5}{red}
\idfigparam{(nd-2-4)}{0.5}{red}
}
$
and as follows:
\begin{align}
	\begin{strings}
		\idfigparam{(0,1.5)}{1.5}{red}
		\idfigparam{(0.5,1.5)}{1.5}{red}
	\end{strings}
	=&
	\frac{1}{2}
	\left(
	\begin{strings}
		\idfigparam{(0,1.5)}{0.5}{red}
		\Ifigparamcol{(nd-1-2)}{1}{0.5}{red}{red}{red}{red}{red}
		\idfigparam{(nd-2-3)}{0.5}{red}
		\idfigparam{(nd-2-2)}{-0.5}{red}
		\idfigparam{(nd-2-4)}{0.5}{red}
		\polyboxscale{(0.5,1.25)}{$\alpha_s$}{1}
	\end{strings}
	+
	\alpha_t
	\begin{strings}
		\idfigparam{(0,1.5)}{0.5}{red}
		\Ifigparamcol{(nd-1-2)}{1}{0.5}{red}{red}{red}{red}{red}
		\idfigparam{(nd-2-3)}{0.5}{red}
		\idfigparam{(nd-2-2)}{-0.5}{red}
		\idfigparam{(nd-2-4)}{0.5}{red}
	\end{strings}
	-\alpha_t
	\begin{strings}
		\idfigparam{(0,1.5)}{0.5}{red}
		\Ifigparamcol{(nd-1-2)}{1}{0.5}{red}{red}{red}{red}{red}
		\idfigparam{(nd-2-3)}{0.5}{red}
		\idfigparam{(nd-2-2)}{-0.5}{red}
		\idfigparam{(nd-2-4)}{0.5}{red}
	\end{strings}
	+
	\begin{strings}
		\idfigparam{(0,1.5)}{0.5}{red}
		\Ifigparamcol{(nd-1-2)}{1}{0.5}{red}{red}{red}{red}{red}
		\idfigparam{(nd-2-3)}{0.5}{red}
		\idfigparam{(nd-2-2)}{-0.5}{red}
		\idfigparam{(nd-2-4)}{0.5}{red}
		\polyboxscale{(0.5,0.25)}{$\alpha_s$}{1}
	\end{strings}
	\right)
	\\
	=&
	\frac{1}{2}
	\left(
	\begin{strings}
		\idfigparam{(0,1.5)}{0.5}{red}
		\Ifigparamcol{(nd-1-2)}{1}{0.5}{red}{red}{red}{red}{red}
		\idfigparam{(nd-2-3)}{0.5}{red}
		\idfigparam{(nd-2-2)}{-0.5}{red}
		\idfigparam{(nd-2-4)}{0.5}{red}
		\polyboxscale{(0.5,1.25)}{$\alpha_s$}{1}
	\end{strings}
	+
	\begin{strings}
		\idfigparam{(0,1.5)}{0.5}{red}
		\Ifigparamcol{(nd-1-2)}{1}{0.5}{red}{red}{red}{red}{red}
		\idfigparam{(nd-2-3)}{0.5}{red}
		\idfigparam{(nd-2-2)}{-0.5}{red}
		\idfigparam{(nd-2-4)}{0.5}{red}
		\polyboxscale{(0.5,1.25)}{$\alpha_t$}{1}
	\end{strings}
	-
	\begin{strings}
		\idfigparam{(0,1.5)}{0.5}{red}
		\Ifigparamcol{(nd-1-2)}{1}{0.5}{red}{red}{red}{red}{red}
		\idfigparam{(nd-2-3)}{0.5}{red}
		\idfigparam{(nd-2-2)}{-0.5}{red}
		\idfigparam{(nd-2-4)}{0.5}{red}
		\polyboxscale{(0.5,0.25)}{$\alpha_t$}{1}
	\end{strings}
	+
	\begin{strings}
		\idfigparam{(0,1.5)}{0.5}{red}
		\Ifigparamcol{(nd-1-2)}{1}{0.5}{red}{red}{red}{red}{red}
		\idfigparam{(nd-2-3)}{0.5}{red}
		\idfigparam{(nd-2-2)}{-0.5}{red}
		\idfigparam{(nd-2-4)}{0.5}{red}
		\polyboxscale{(0.5,0.25)}{$\alpha_s$}{1}
	\end{strings}
	\right)
	\\
	=&
	\frac{1}{2}
	\left(
	\begin{strings}
		\idfigparam{(0,1.5)}{0.5}{red}
		\Ifigparamcol{(nd-1-2)}{1}{0.5}{red}{red}{red}{red}{red}
		\idfigparam{(nd-2-3)}{0.5}{red}
		\idfigparam{(nd-2-2)}{-0.5}{red}
		\idfigparam{(nd-2-4)}{0.5}{red}
		\polyboxscale{(0.5,1.25)}{$\alpha_s+\alpha_t$}{1}
	\end{strings}
	+
	\begin{strings}
		\idfigparam{(0,1.5)}{0.5}{red}
		\Ifigparamcol{(nd-1-2)}{1}{0.5}{red}{red}{red}{red}{red}
		\idfigparam{(nd-2-3)}{0.5}{red}
		\idfigparam{(nd-2-2)}{-0.5}{red}
		\idfigparam{(nd-2-4)}{0.5}{red}
		\polyboxscale{(0.5,0.25)}{$\alpha_s-\alpha_t$}{1}
	\end{strings}
	\right)
	.
	\label{eqn red Y decomp}
\end{align}
Note that we used \eqref{poly slide red blue} to slide $\alpha_t$ inside of
$
\scaledstrings[0.5]{
	\idfigparam{(0,1.5)}{0.5}{red}
	\Ifigparamcol{(nd-1-2)}{1}{0.5}{red}{red}{red}{red}{red}
	\idfigparam{(nd-2-3)}{0.5}{red}
	\idfigparam{(nd-2-2)}{-0.5}{red}
	\idfigparam{(nd-2-4)}{0.5}{red}
}
$.

Similarly,
\begin{align}
	\begin{strings}
		\idfigparam{(0,1.5)}{1.5}{blue}
		\idfigparam{(0.5,1.5)}{1.5}{blue}
	\end{strings}
	=&
	\frac{1}{2}
	\left(
	\begin{strings}
		\idfigparam{(0,1.5)}{0.5}{blue}
		\Ifigparamcol{(nd-1-2)}{1}{0.5}{blue}{blue}{blue}{blue}{blue}
		\idfigparam{(nd-2-3)}{0.5}{blue}
		\idfigparam{(nd-2-2)}{-0.5}{blue}
		\idfigparam{(nd-2-4)}{0.5}{blue}
		\polyboxscale{(0.5,1.25)}{$\alpha_s+\alpha_t$}{1}
	\end{strings}
	-
	\begin{strings}
		\idfigparam{(0,1.5)}{0.5}{blue}
		\Ifigparamcol{(nd-1-2)}{1}{0.5}{blue}{blue}{blue}{blue}{blue}
		\idfigparam{(nd-2-3)}{0.5}{blue}
		\idfigparam{(nd-2-2)}{-0.5}{blue}
		\idfigparam{(nd-2-4)}{0.5}{blue}
		\polyboxscale{(0.5,0.25)}{$\alpha_s-\alpha_t$}{1}
	\end{strings}
	\right)
	.
	\label{eqn blue Y decomp}
\end{align}

Using \eqref{eqn red Y decomp} and \eqref{eqn blue Y decomp} we get
\begin{align}
	\begin{pmatrix}
		\iddoubletextmatrix[1.125]{red}{red} & \\
		& \iddoubletextmatrix[1.125]{blue}{blue}
	\end{pmatrix}
	=&
	\begin{pmatrix}
		\frac{1}{2}
		\left(
		\scaledstrings{
			\idfigparam{(0,1.5)}{0.5}{red}
			\Ifigparamcol{(nd-1-2)}{1}{0.5}{red}{red}{red}{red}{red}
			\idfigparam{(nd-2-3)}{0.5}{red}
			\idfigparam{(nd-2-2)}{-0.5}{red}
			\idfigparam{(nd-2-4)}{0.5}{red}
			\polyboxscale{(0.5,1.25)}{$\alpha_s+\alpha_t$}{1}
		}
		+
		\scaledstrings{
			\idfigparam{(0,1.5)}{0.5}{red}
			\Ifigparamcol{(nd-1-2)}{1}{0.5}{red}{red}{red}{red}{red}
			\idfigparam{(nd-2-3)}{0.5}{red}
			\idfigparam{(nd-2-2)}{-0.5}{red}
			\idfigparam{(nd-2-4)}{0.5}{red}
			\polyboxscale{(0.5,0.25)}{$\alpha_s-\alpha_t$}{1}
		}
		\right)
		& \\
		& 
		\frac{1}{2}
		\left(
		\scaledstrings{
			\idfigparam{(0,1.5)}{0.5}{blue}
			\Ifigparamcol{(nd-1-2)}{1}{0.5}{blue}{blue}{blue}{blue}{blue}
			\idfigparam{(nd-2-3)}{0.5}{blue}
			\idfigparam{(nd-2-2)}{-0.5}{blue}
			\idfigparam{(nd-2-4)}{0.5}{blue}
			\polyboxscale{(0.5,1.25)}{$\alpha_s+\alpha_t$}{1}
		}
		-
		\scaledstrings{
			\idfigparam{(0,1.5)}{0.5}{blue}
			\Ifigparamcol{(nd-1-2)}{1}{0.5}{blue}{blue}{blue}{blue}{blue}
			\idfigparam{(nd-2-3)}{0.5}{blue}
			\idfigparam{(nd-2-2)}{-0.5}{blue}
			\idfigparam{(nd-2-4)}{0.5}{blue}
			\polyboxscale{(0.5,0.25)}{$\alpha_s-\alpha_t$}{1}
		}
		\right)
	\end{pmatrix}
	.
\end{align}

After expanding right hand side, we may cancel  
$
\scaledstrings[0.5]{
	\idfigparam{(0,1.5)}{0.5}{red}
	\Ifigparamcol{(nd-1-2)}{1}{0.5}{red}{red}{red}{red}{red}
	\idfigparam{(nd-2-3)}{0.5}{red}
	\idfigparam{(nd-2-2)}{-0.5}{red}
	\idfigparam{(nd-2-4)}{0.5}{red}
	\polyboxscale{(0.5,1.25)}{$\alpha_t$}{1}
}
$
with
$
-
\scaledstrings[0.5]{
	\idfigparam{(0,1.5)}{0.5}{red}
	\Ifigparamcol{(nd-1-2)}{1}{0.5}{red}{red}{red}{red}{red}
	\idfigparam{(nd-2-3)}{0.5}{red}
	\idfigparam{(nd-2-2)}{-0.5}{red}
	\idfigparam{(nd-2-4)}{0.5}{red}
	\polyboxscale{(0.5,0.25)}{$\alpha_t$}{1}
}
$
 and 
$
\scaledstrings[0.5]{
	\idfigparam{(0,1.5)}{0.5}{blue}
	\Ifigparamcol{(nd-1-2)}{1}{0.5}{blue}{blue}{blue}{blue}{blue}
	\idfigparam{(nd-2-3)}{0.5}{blue}
	\idfigparam{(nd-2-2)}{-0.5}{blue}
	\idfigparam{(nd-2-4)}{0.5}{blue}
	\polyboxscale{(0.5,1.25)}{$\alpha_s$}{1}
}
$
with
$
-
\scaledstrings[0.5]{
	\idfigparam{(0,1.5)}{0.5}{blue}
	\Ifigparamcol{(nd-1-2)}{1}{0.5}{blue}{blue}{blue}{blue}{blue}
	\idfigparam{(nd-2-3)}{0.5}{blue}
	\idfigparam{(nd-2-2)}{-0.5}{blue}
	\idfigparam{(nd-2-4)}{0.5}{blue}
	\polyboxscale{(0.5,0.25)}{$\alpha_s$}{1}
}
$.
In the end we get the equation
\begin{align}
	\begin{pmatrix}
		\iddoubletextmatrix[1.125]{red}{red} & \\
		& \iddoubletextmatrix[1.125]{blue}{blue}
	\end{pmatrix}
	=&	
	\frac{1}{2}
	\begin{pmatrix}
		\scaledstrings{
			\idfigparam{(0,1.5)}{0.5}{red}
			\Ifigparamcol{(nd-1-2)}{1}{0.5}{red}{red}{red}{red}{red}
			\idfigparam{(nd-2-3)}{0.5}{red}
			\idfigparam{(nd-2-2)}{-0.5}{red}
			\idfigparam{(nd-2-4)}{0.5}{red}
			\polyboxscale{(0.5,1.25)}{$\alpha_s+\alpha_t$}{1}
		}
		& \\
		& 
		\scaledstrings{
			\idfigparam{(0,1.5)}{0.5}{blue}
			\Ifigparamcol{(nd-1-2)}{1}{0.5}{blue}{blue}{blue}{blue}{blue}
			\idfigparam{(nd-2-3)}{0.5}{blue}
			\idfigparam{(nd-2-2)}{-0.5}{blue}
			\idfigparam{(nd-2-4)}{0.5}{blue}
			\polyboxscale{(0.5,1.25)}{$\alpha_s+\alpha_t$}{1}
		}
	\end{pmatrix}
	+
	\frac{1}{2}
	\begin{pmatrix}
		\scaledstrings{
			\idfigparam{(0,1.5)}{0.5}{red}
			\Ifigparamcol{(nd-1-2)}{1}{0.5}{red}{red}{red}{red}{red}
			\idfigparam{(nd-2-3)}{0.5}{red}
			\idfigparam{(nd-2-2)}{-0.5}{red}
			\idfigparam{(nd-2-4)}{0.5}{red}
			\polyboxscale{(0.5,0.25)}{$\alpha_s-\alpha_t$}{1}
		}
		& \\
		& 
		-
		\scaledstrings{
			\idfigparam{(0,1.5)}{0.5}{blue}
			\Ifigparamcol{(nd-1-2)}{1}{0.5}{blue}{blue}{blue}{blue}{blue}
			\idfigparam{(nd-2-3)}{0.5}{blue}
			\idfigparam{(nd-2-2)}{-0.5}{blue}
			\idfigparam{(nd-2-4)}{0.5}{blue}
			\polyboxscale{(0.5,0.25)}{$\alpha_s-\alpha_t$}{1}
		}
	\end{pmatrix}
	.
\end{align}
Each of these terms is a $\tau$-invariant idempotent factoring through $\Yobj$.

This shows that 
\begin{equation}
	\Yobj\otimes \Yobj \cong Y[-1]\oplus Y[1] \oplus \Ind(B_sB_t) \cong Y[-1]\oplus Y[1] \oplus \Zobj \oplus \XZobj.
\end{equation}	
\end{example}

\begin{example}\label{ex decomp YZ}
	We will now calculate the decomposition of $\Yobj \otimes \Zobj$.
	Inside $\Hcal$ we have
		\begin{align}
		(B_s\oplus B_t)\otimes(B_sB_t) \cong&
		B_s\otimes B_s \otimes B_t\oplus B_t \otimes B_s \otimes B_t
		\\
		\cong&
		B_s\otimes B_s \otimes B_t\oplus B_t \otimes B_t \otimes B_s
		\\
		\cong&
		B_sB_t[-1] \oplus B_sB_t[1] \oplus B_tB_s[-1] \oplus B_tB_s[1]
		\\
		\cong&
		(B_sB_t \oplus B_tB_s)[-1] \otimes (B_sB_t \oplus B_tB_s)[1].
	\end{align}
	This suggests that 
	\begin{equation}\label{eqn decomp YZ}
		\Yobj\otimes \Zobj \cong
		\Ind(B_sB_t)[-1] \oplus \Ind(B_sB_t)[1] \cong
		(\Zobj \oplus \XZobj)[-1] \otimes (\Zobj \oplus \XZobj)[1].
	\end{equation}
	
	We may write
	\begin{align}
		\id_{\Yobj\otimes \Zobj} =&
		\begin{pmatrix}
			\scaledstrings{
				\idfigparam{(0,1)}{1.5}{red}
				\idfigparam{(0.25,1)}{1.5}{red}
				\idfigparam{(0.5,1)}{1.5}{blue}
			} 
			& \\
			& 
			\scaledstrings{
				\idfigparam{(0,1)}{1.5}{blue}
				\idfigparam{(0.25,1)}{1.5}{red}
				\idfigparam{(0.5,1)}{1.5}{blue}
			}
		\end{pmatrix}
		\\
		=&
		\begin{pmatrix}
			\frac{1}{2}
			\left(
			\scaledstrings{
				\idfigparam{(0,1.5)}{0.5}{red}
				\Ifigparamcol{(nd-1-2)}{1}{0.5}{red}{red}{red}{red}{red}
				\idfigparam{(nd-2-3)}{0.5}{red}
				\idfigparam{(nd-2-2)}{-0.5}{red}
				\idfigparam{(nd-2-4)}{0.5}{red}
				\idfigparam{(0.75,1.5)}{1.5}{blue}
				\polyboxscale{(0.375,1.25)}{$\alpha_s$}{1}
			}
			+
			\scaledstrings{
				\idfigparam{(0,1.5)}{0.5}{red}
				\Ifigparamcol{(nd-1-2)}{1}{0.5}{red}{red}{red}{red}{red}
				\idfigparam{(nd-2-3)}{0.5}{red}
				\idfigparam{(nd-2-2)}{-0.5}{red}
				\idfigparam{(nd-2-4)}{0.5}{red}
				\idfigparam{(0.75,1.5)}{1.5}{blue}
				\polyboxscale{(0.375,0.25)}{$\alpha_s$}{1}
			}
			\right)
			& \\
			& 
			\frac{1}{2}
			\left(
			\scaledstrings{
				\idfigparam{(0,1.5)}{0.5}{blue}
				\Ifigparamcol{(nd-1-2)}{1}{0.5}{blue}{blue}{blue}{blue}{blue}
				\idfigparam{(nd-2-3)}{0.5}{blue}
				\idfigparam{(nd-2-2)}{-0.5}{blue}
				\idfigparam{(nd-2-4)}{0.5}{blue}
				\idfigparam{(0.75,1.5)}{1.5}{red}
				\polyboxscale{(0.375,1.25)}{$\alpha_t$}{1}
			}
			-
			\scaledstrings{
				\idfigparam{(0,1.5)}{0.5}{blue}
				\Ifigparamcol{(nd-1-2)}{1}{0.5}{blue}{blue}{blue}{blue}{blue}
				\idfigparam{(nd-2-3)}{0.5}{blue}
				\idfigparam{(nd-2-2)}{-0.5}{blue}
				\idfigparam{(nd-2-4)}{0.5}{blue}
				\idfigparam{(0.75,1.5)}{1.5}{red}
				\polyboxscale{(0.3755,0.25)}{$\alpha_t$}{1}
			}
			\right)
		\end{pmatrix}
		\\
		=&
		\frac{1}{4}
		\begin{pmatrix}
			\scaledstrings{
				\idfigparam{(0,1.5)}{0.5}{red}
				\Ifigparamcol{(nd-1-2)}{1}{0.5}{red}{red}{red}{red}{red}
				\idfigparam{(nd-2-3)}{0.5}{red}
				\idfigparam{(nd-2-2)}{-0.5}{red}
				\idfigparam{(nd-2-4)}{0.5}{red}
				\idfigparam{(0.75,1.5)}{1.5}{blue}
				\polyboxscale{(0.375,1.25)}{$\alpha_s$}{1}
			}
			& \\
			&
			\scaledstrings{
				\idfigparam{(0,1.5)}{0.5}{blue}
				\Ifigparamcol{(nd-1-2)}{1}{0.5}{blue}{blue}{blue}{blue}{blue}
				\idfigparam{(nd-2-3)}{0.5}{blue}
				\idfigparam{(nd-2-2)}{-0.5}{blue}
				\idfigparam{(nd-2-4)}{0.5}{blue}
				\idfigparam{(0.75,1.5)}{1.5}{red}
				\polyboxscale{(0.375,1.25)}{$\alpha_t$}{1}
			}
		\end{pmatrix}
		+
		\frac{1}{4}
		\begin{pmatrix}
			\scaledstrings{
				\idfigparam{(0,1.5)}{0.5}{red}
				\Ifigparamcol{(nd-1-2)}{1}{0.5}{red}{red}{red}{red}{red}
				\idfigparam{(nd-2-3)}{0.5}{red}
				\idfigparam{(nd-2-2)}{-0.5}{red}
				\idfigparam{(nd-2-4)}{0.5}{red}
				\idfigparam{(0.75,1.5)}{1.5}{blue}
				\polyboxscale{(0.375,0.25)}{$\alpha_s$}{1}
			}
			& \\
			&
			\scaledstrings{
				\idfigparam{(0,1.5)}{0.5}{blue}
				\Ifigparamcol{(nd-1-2)}{1}{0.5}{blue}{blue}{blue}{blue}{blue}
				\idfigparam{(nd-2-3)}{0.5}{blue}
				\idfigparam{(nd-2-2)}{-0.5}{blue}
				\idfigparam{(nd-2-4)}{0.5}{blue}
				\idfigparam{(0.75,1.5)}{1.5}{red}
				\polyboxscale{(0.375,0.25)}{$\alpha_t$}{1}
			}
		\end{pmatrix}
		.
	\end{align}
	Each of these is an idempotent factoring through $\Ind(B_sB_t)$,
	so \eqref{eqn decomp YZ} follows.	
\end{example}

\begin{example}\label{ex decomp ZZ}
	Similar to Examples \ref{ex decomp YY} and \ref{ex decomp YZ}, we can show that 
	\begin{equation}
		Z\otimes Z \cong Z[-2] \oplus Z \oplus Z[2] \oplus \XZobj.
	\end{equation}
\end{example}

\begin{rema}
	Equations \eqref{Y2 decomp}, \eqref{Z2 decomp}, and \eqref{YZ decomp} give explicit idempotent decompositions of
	$\Yobj\otimes \Yobj$, $\Zobj\otimes \Zobj$, and $\Yobj\otimes \Zobj$ in the diagrammatic presentation we propose for $\Hcaltau$.
	We may use the functor $\Fbold:\Dcal_{A_1\times A_1}\longrightarrow \Hcaltau$ to write explicit idempotent decompositions
	for these tensor products, that do not involve $\Ind(B_sB_t)$.
\end{rema}

The following lemma summarizes all the decompositions we have.
\begin{lemma}\label{lemma tensor products in Htau}
We have the following isomorphisms:
\begin{equation}\label{gr group 1}
	\Xobj\otimes \Xobj\cong \onecal,  
\end{equation}
\begin{equation}\label{gr group 2}
	\Xobj\otimes \Yobj \cong \Yobj \cong \Yobj\otimes \Xobj ,
\end{equation} 
\begin{equation}\label{gr group 3}
	\Xobj\otimes \Zobj \cong \XZobj \cong \Zobj\otimes \Xobj,
\end{equation} 
\begin{equation}\label{gr group 4}
	\Yobj\otimes \Yobj \cong \Yobj[-1]\oplus \Yobj[1]\oplus \Zobj \oplus \XZobj ,
\end{equation}
\begin{equation}\label{gr group 5}
	\Yobj\otimes \Zobj\cong \Zobj[-1]\oplus \Zobj[1] \oplus \XZobj[-1] \oplus \XZobj[1] \cong \Zobj \otimes \Yobj,
\end{equation}
\begin{equation}\label{gr group 6}
	\Zobj\otimes \Zobj\cong \Zobj[-2]\oplus \Zobj \oplus \Zobj[2] \oplus \XZobj .
\end{equation}

In particular, the tensor product in $\Hcaltau$ is commutative.
\end{lemma}
\begin{proof}
The isomorphisms \eqref{gr group 1} and \eqref{gr group 3} are clear.
The equation \eqref{gr group 2} follows from \eqref{eqn isomorphism sngXInd}.
The remaining isomorphisms correspond to Examples \ref{ex decomp YY}, \ref{ex decomp YZ}, and \ref{ex decomp ZZ}.
\end{proof}

Hence, we have the following theorem.
\begin{theorem}\label{thm grothendieck ring}
The Grothendieck ring of $\Hcaltau$ is
\begin{equation}
    \Acal:= \bigrightslant{\ZZ[v^{\pm 1}][X,Y,Z]}
    {K},
\end{equation}
where $K$ is the ideal generated by the relations
\begin{equation}
    X^2= 1
\end{equation}
\begin{equation}
    XY=Y
\end{equation}
\begin{equation}
    Y^2 = (v+v^{-1})Y + Z + XZ
\end{equation}
\begin{equation}
    YZ = (v+v^{-1})(Z+XZ)
\end{equation}
\begin{equation}
    Z^2 = (v^2+1+v^{-2})Z + XZ.
\end{equation}
\end{theorem}

\begin{rema}
We want to note two important specializations of $\Acal$:
\begin{itemize}
	\item Setting $X=1$, where $\Acal$ specializes to the $\tau$-invariant ring of 
	Hecke algebra of type $A_1\times A_1$.
	\item Setting $X=-1$,
	where 
	$\Acal$ specializes to Hecke algebra with unequal parameters of $A_1\times A_1$ folded by $\tau$ (see \cite{lusztig2002hecke}).
\end{itemize}
This is a special case of \cite[Prop. 5.1 \& Thm. 5.4]{elias2017folding}.
\end{rema}

\subsection{
$\Hom$ spaces in $\Hcaltau$
}\label{subsection hom spaces equivar}

In this section we will calculate the 
$\Hom$ spaces of $\Hcaltau$ from its irreducible objects to $\onecal$.
By adjunction, this will completely determine the morphisms between any two objects of $\Hcaltau$.
We will then calculate the endomorphism rings of 
$\onecal, X, Y, Z,$ and $XZ$, which were needed for the proof of 
Lemma \ref{Ind unique decomposition}.

Recall $R=\kk[\alpha_s,\alpha_t]$. Using the notation in 
Example \ref{example End1}, we get 
\begin{equation}
	\Rtau=\kk[\alpha_s+\alpha_t,\alpha_s\alpha_t]
\end{equation}
is the ring of symmetric polynomials in $\alpha_s$ and $\alpha_t$,
and that
\begin{equation}
	\Rantitau= (\alpha_s-\alpha_t)\Rtau
\end{equation}
is the $\tau$ anti-invariant polynomials in $\alpha_s$ and $\alpha_t$.
We also have that
\begin{equation}
	R= \Rtau \oplus (\alpha_s-\alpha_t)\Rtau = \Rtau\oplus \Rantitau.
\end{equation}

\begin{prop}\label{prop HomXto1 in Hcaltau}
	The following $\Hom$ spaces are free $\Rtau$-modules with the generators given below:
		\begin{equation}
			\Hom(\onecal,\onecal) = \Rtau\cdot
			\left(
			\begin{strings}
				\path (0,0) -- (0.3,1);
			\end{strings}
			\right),
			\hspace{0.25cm}
			\Hom(X,\onecal)= \Rtau\cdot \left( (\alpha_s-\alpha_t)
			\begin{strings}
				\path (0,0) -- (0.3,1);
			\end{strings}
			 \right),
		\end{equation}
		\begin{equation}
			\Hom(Y,\onecal)= 
			 \Rtau \cdot
			{
				\begingroup
				\setlength\arraycolsep{0pt}
				\begin{pmatrix}
					\begin{unboxedstrings}
						\dotuparam{(0,0)}{0.5}{0.06}{red}
					\end{unboxedstrings}
					&
					\begin{unboxedstrings}
						\dotuparam{(0,0)}{0.5}{0.06}{blue}
					\end{unboxedstrings}
				\end{pmatrix}
				\endgroup
			}
			\oplus
			\Rtau \cdot
			\left(
			(\alpha_s-\alpha_t)
			{
				\begingroup
				\setlength\arraycolsep{0pt}
				\begin{pmatrix}
					\begin{unboxedstrings}
						\dotuparam{(0,0)}{0.5}{0.06}{red}
					\end{unboxedstrings}
					&
					\begin{unboxedstrings}
						\dotuparam{(0,0)}{0.5}{0.06}{blue}
					\end{unboxedstrings}
				\end{pmatrix}
				\endgroup
			}
			\right)
			,
		\end{equation}
		\begin{equation}
				\Hom(Z,\onecal)= \Rtau \cdot
		\begin{strings}
			\dotuparam{(0,0)}{0.6}{0.06}{red}
			\dotuparam{(0.25,0)}{0.6}{0.06}{blue}
			\path (0,0) ++ (0,1);
		\end{strings}
		, \hspace{0.125cm} \text{ and }
		\Hom(XZ,\onecal)= \Rantitau 
		\begin{strings}
			\dotuparam{(0,0)}{0.6}{0.06}{red}
			\dotuparam{(0.25,0)}{0.6}{0.06}{blue}
			\path (0,0) ++ (0,1);
		\end{strings}.
		\end{equation}
	The graded dimensions of these spaces over $\Rtau$ are
	\begin{align}
		\grdim_\Rtau(\Hom(\onecal,\onecal))&=1, \\
		\grdim_\Rtau(\Hom(X,\onecal))&=v^2, \\
		\grdim_\Rtau(\Hom(Y,\onecal))&=1 +v^2,\\
		\grdim_\Rtau(\Hom(Z,\onecal))&= v^2 ,\\
		\grdim_\Rtau(\Hom(XZ,\onecal))&= v^4.
	\end{align}
\end{prop}
\begin{proof}
	We already showed this for $\Hom(\onecal,\onecal)$ and
	$\Hom(X,\onecal)$ in Example \ref{example End1}.
	By Proposition \ref{ind-res adjunction} and
	Proposition \ref{prop HomBs,1 in A1}
	we have
	\begin{equation}
		\Hom(Y,\onecal)= \Hom(\Ind(B_s),\onecal) \cong \Hom(B_s,\onecal) \cong R.
	\end{equation}
	Since $\grdim_{\Rtau}(R)=1+v^2$, the morphisms
	$	
	\begin{pmatrix}
		\Dotup{red}
		&
		\Dotup{blue}
	\end{pmatrix}
	$
	 and 
	$(\alpha_s-\alpha_t)
	\begin{pmatrix}
		\Dotup{red}
		&
		\Dotup{blue}
	\end{pmatrix}
	$
	are free generators of $\Hom(Y,\onecal)$ as an $\Rtau$-module.
	
	From Proposition \ref{prop HomBs,1 in AxA}, we have that
	$\Hom(B_sB_t,\onecal)= R \cdot
	\begin{miniunboxedstrings}
		\dotuparam{(0,0)}{0.6}{0.06}{red}
		\dotuparam{(0.25,0)}{0.6}{0.06}{blue}
	\end{miniunboxedstrings}
	$.
	By \eqref{equivariant mor}, we get that 
	a morphism in $\Hom(Z,\onecal)$ must be a $\tau$-invariant
	morphism in $\Hom(B_sB_t,\onecal)$,
	hence 
	\begin{equation}
		\Hom(Z,\onecal)=\Rtau \cdot
		\begin{strings}
			\dotuparam{(0,0)}{0.6}{0.06}{red}
			\dotuparam{(0.25,0)}{0.6}{0.06}{blue}
			\path (0,0) ++ (0,1);
		\end{strings}.
	\end{equation}
	Similarly, a morphism in $\Hom(XZ,\onecal)$ must be a $\tau$-anti invariant and 
	\begin{equation}
		\Hom(XZ,\onecal)=\Rantitau \cdot
		\begin{strings}
			\dotuparam{(0,0)}{0.6}{0.06}{red}
			\dotuparam{(0.25,0)}{0.6}{0.06}{blue}
			\path (0,0) ++ (0,1);
		\end{strings}.
	\end{equation}
\end{proof}

Here we give a biadjoint to each object in $\Ical$
not equal to $\onecal$:
\begin{itemize}
	\item $\left(X,\id_{\onecal},\id_{\onecal}\right)$
	is biadjoint to $X$.
	\item $\left(Y,
	\begin{pmatrix}
		\Capp{red}
		\hspace{-0.2cm}
		&
		\Capp{blue}
	\end{pmatrix}
	,
	\begin{pmatrix}
		\Cupp{red}
		\hspace{-0.2cm}
		&
		\Cupp{blue}
	\end{pmatrix}
	\right)$
	is biadjoint to $Y$.
	\item $\left(Z,
	\begin{miniunboxedstrings}
		\capfigparam{(0.25,0)}{0.5}{0.5}{blue}
		\capfigparam{(0,0)}{0.5}{0.5}{red}
	\end{miniunboxedstrings}
	,
	\begin{miniunboxedstrings}
		\cupfigparam{(0.25,1)}{0.5}{0.5}{blue}
		\cupfigparam{(0,1)}{0.5}{0.5}{red}
	\end{miniunboxedstrings}
	\right)$
	is biadjoint to $Z$.
	\item $\left(XZ,
	\begin{miniunboxedstrings}
		\capfigparam{(0.25,0)}{0.5}{0.5}{blue}
		\capfigparam{(0,0)}{0.5}{0.5}{red}
	\end{miniunboxedstrings}
	,
	\begin{miniunboxedstrings}
		\cupfigparam{(0.25,1)}{0.5}{0.5}{blue}
		\cupfigparam{(0,1)}{0.5}{0.5}{red}
	\end{miniunboxedstrings}
	\right)$
	is biadjoint to $XZ$.
\end{itemize}

By Proposition \ref{prop adjunction general}, we have
the following result.
\begin{prop}\label{prop adjunction equiv}
	The functors $X[k]\otimes(-)$, $Y[k]\otimes(-)$, and $Z[k]\otimes(-)$ are bi-adjoint to $X[-k]\otimes(-)$, $Y[-k]\otimes(-)$, and $Z[-k]\otimes(-)$ respectively.
\end{prop}

\begin{prop}\label{prop end spaces Htau}
	We have that 
		\begin{equation}
			\End(\onecal)\cong
			\End(X)\cong \Rtau,
		\end{equation}
		\begin{equation}
			\End(Y)\cong R \oplus R^{\oplus 2}[2],
		\end{equation}
		\begin{equation}
			\End(Z)\cong \End(XZ)\cong \Rtau \oplus \Rtau[2] \oplus \Rtau[4]^{\oplus 2}.
		\end{equation}	
		As free graded $\Rtau$-modules.
		They are all graded local rings.
\end{prop}
\begin{proof}
	In Example \ref{example End1}, we calculated that
	\begin{equation}
		\End(\onecal)\cong
		\End(X)\cong \Rtau.
	\end{equation}
	Using Proposition \ref{ind-res adjunction},
	we get that
	\begin{align}
		\Hom(Y,Y)&=\Hom(\Ind(B_s),\Ind(B_s)) \\
		&\cong \Hom(B_s,B_s\oplus B_t)\\
		&\cong \Hom(B_s,B_s)\oplus \Hom(B_s,B_t) \\
		&\cong R \oplus R[2] \oplus R[2] \\
		&= R \oplus R^{\oplus 2}[2].
	\end{align}
			
	Since, all these endomorphism rings are non-negatively graded and their degree $0$ component is $\kk$, they are graded local rings.
\end{proof}

\excludeDiagrams{0}
{
The graded dimension of these $\Hom$ spaces as free $\Rtau$-modules is as follows.
\[
\grdim(\Hom(\onecal,\onecal))=1
\]
\[
\grdim(\Hom(\onecal,X))=v^2
\]
\[
\grdim(\Hom(\onecal,Y))=v+v^3
\]
\[
\grdim(\Hom(\onecal,Z))=v^2
\]
\[
\grdim(\Hom(\onecal,XZ))=v^4.
\]

\begin{defi}
Define the standard trace $\epsilon:\Acal\longrightarrow \ZZ[v^{\pm 1}]$
to be the linear map sending $\onecal\mapsto 1$, 
$X\mapsto v^2$,
$Y\mapsto v+v^3$,
$Z\mapsto v^2$,
and 
$XZ\mapsto v^4$.
\end{defi}

\begin{prop}
Let $A,B\in \Hcaltau$
The graded dimension
\[
\grdim(\Hom(A,B))=\epsilon(AB).
\]
\end{prop}}

\section{The Diagrammatic Folded Hecke Category of type $A_1\times A_1$}\label{sec 3 diagrammatic folded category}

In Section \ref{subsection Equivariantization of Soergel} we showed that the category $\Hcaltau$ has indecomposable objects
$\onecal:=(\onecal, 1)$,
$\Xobj=(\onecal, -1)$, $\Yobj=(B_s\oplus B_t, {\tiny \begin{pmatrix}
0&1 \\ 1&0
\end{pmatrix}} )$,
 $\Zobj=
 \left(B_s B_t,
 \Cross{blue}{red}
 \right)$
 ,
  and
$\XZobj=\left(B_s B_t,
- \Cross{blue}{red} \right)$
 (up to grading shift).
We also showed $\Hcaltau$ is a Krull-Schmidt category.

In Section \ref{subsection free category}, we will define a `free' model for our category $\Dcalpre$ and a functor $\Fbold: \Dcalpre\longrightarrow \Hcal^\tau$.
In Section \ref{subsection-presentation} we define our diagrammatic category $\Dcal_{A_1\times A_1}$, as a quotient of $\Dcalpre$ by a given set of relations.
We then show that the relations defining $\Dcal_{A_1\times A_1}$ are satisfied through the functor $\Fbold$, so this functor descends to a functor $\Fboldtilde:\Dcal_{A_1\times A_1}\longrightarrow \Hcaltau$.
In Sections \ref{subsection poly forcing 1} and \ref{subsection idempotent decomposition} we show the general polynomial forcing formulas, as well as the idempotent decompositions in $\Dcal_{A_1\times A_1}$.
In Section \ref{subsection orange strand}, we show some manipulations to be able to slide orange strands (corresponding to the identity of the invertible element $\Xobjbar$),
which simplify our diagrams.
Finally, in Section \ref{subsection-diagram reduction} we prove Theorem \ref{diagram reduction}.
Some proofs may require the relations derived in Appendix \ref{appendix-relations}, derived from the ones given in
Definition \ref{relations equiv category}.

\subsection{The Free Diagrammatic Category $\Dcalpre$}\label{subsection free category}

\begin{defi}
Define the \textit{free diagrammatic category}, $\Dcalpre$ 
to be the graded strict monoidal $\Rtau$-linear category,
generated by the self-adjoint objects
$\Xobjbar$, $\Yobjbar$, and $\Zobjbar$, 
with respect to direct sums, tensor products and degree shifts.
The unit object of $\Dcalpre$ will be denoted by $\onecalbar$.
The identity morphisms of  $\Xobjbar$, $\Yobjbar$, and $\Zobjbar$ are represented
by the diagrams $\idtext{orange}$, $\idtext{green}$, and $\idtext{brown}$ respectively.
The cup morphisms are represented by a cup diagram in the appropriate color,
and so are the cap morphisms.
The following is a list of the local generators for the morphism diagrams of $\Dcalpre$ listed by their degree:
\excludeDiagrams{1}
{
\begin{center}
	\begin{minipage}[t]{0.44\textwidth}
		Degree -2:
		\begin{equation}
			\begin{boxedstrings}
				\trivu{(0,0)}{brown}
				\idfigparam{(nd-1-1)}{-0.25}{brown}
				\idfigparam{(nd-1-2)}{0.25}{brown}
				\idfigparam{(nd-1-3)}{0.25}{brown}
			\end{boxedstrings}.
		\end{equation}
		
		Degree -1:
		\begin{equation}
			\begin{boxedstrings}
				\trivu{(0,0)}{green}
				\idfigparam{(nd-1-1)}{-0.25}{green}
				\idfigparam{(nd-1-2)}{0.25}{green}
				\idfigparam{(nd-1-3)}{0.25}{green}
			\end{boxedstrings},
			\begin{boxedstrings}
				\trivuparamcol{(0,0)}{0.5}{0.5}{green}{brown}{brown}
				\idfigparam{(nd-1-1)}{-0.25}{green}
				\idfigparam{(nd-1-2)}{0.25}{brown}
				\idfigparam{(nd-1-3)}{0.25}{brown}
			\end{boxedstrings}.
		\end{equation}

		Degree 0:
		\begin{equation}\label{trivalent vetices simple}
			\begin{boxedstrings}
				\trivuparamcol{(0,0)}{0.5}{0.5}{brown}{green}{green}
				\idfigparam{(nd-1-1)}{-0.25}{brown}
				\idfigparam{(nd-1-2)}{0.25}{green}
				\idfigparam{(nd-1-3)}{0.25}{green}
			\end{boxedstrings}
			,
			\begin{boxedstrings}
				\trivuparamcol{(0,0)}{0.5}{0.5}{orange}{green}{green}
				\idfigparam{(nd-1-1)}{-0.25}{orange}
				\idfigparam{(nd-1-2)}{0.25}{green}
				\idfigparam{(nd-1-3)}{0.25}{green}
			\end{boxedstrings}
			,
		\end{equation}
		\begin{equation}\label{crossings vertices}
			\begin{boxedstrings}
				\Xfig{(0,0)}{brown}{orange}
				\idfigparam{(nd-1-1)}{-0.25}{brown}
				\idfigparam{(nd-1-2)}{-0.25}{orange}
				\idfigparam{(nd-1-4)}{0.25}{brown}
				\idfigparam{(nd-1-3)}{0.25}{orange}
			\end{boxedstrings}
			,
			\begin{boxedstrings}
				\Xfig{(0,0)}{green}{orange}
				\idfigparam{(nd-1-1)}{-0.25}{green}
				\idfigparam{(nd-1-2)}{-0.25}{orange}
				\idfigparam{(nd-1-4)}{0.25}{green}
				\idfigparam{(nd-1-3)}{0.25}{orange}
			\end{boxedstrings}
			,
			\begin{boxedstrings}
				\Xfig{(0,0)}{orange}{orange}
				\idfigparam{(nd-1-1)}{-0.25}{orange}
				\idfigparam{(nd-1-2)}{-0.25}{orange}
				\idfigparam{(nd-1-4)}{0.25}{orange}
				\idfigparam{(nd-1-3)}{0.25}{orange}
			\end{boxedstrings}
			.
		\end{equation}
	\end{minipage}
	\hspace{0.5cm}
	\begin{minipage}[t]{0.44\textwidth}

		Degree 1:
		\begin{equation}\label{bivalent vertices}
			\begin{boxedstrings}
				\dotuparam{(0,0)}{0.6}{0.06}{green}
				\path (nd-1-1) ++ (0,1);
			\end{boxedstrings}
			,
			\begin{boxedstrings}
				\idfigparam{(0,1)}{0.5}{green}
				\idfigparam{(nd-1-2)}{0.5}{brown}
			\end{boxedstrings}
			,
			\begin{boxedstrings}
				\idfigparam{(0,1)}{0.5}{orange}
				\idfigparam{(nd-1-2)}{0.5}{green}
			\end{boxedstrings}
			.      
		\end{equation}

		Degree 2:
		\begin{equation}
			\begin{boxedstrings}
				\dotuparam{(0,0)}{0.6}{0.06}{orange}
				\path (nd-1-1) ++ (0,1);
			\end{boxedstrings},
			\begin{boxedstrings}
				\dotuparam{(0,0)}{0.6}{0.06}{brown}
				\path (nd-1-1) ++ (0,1);
			\end{boxedstrings}.
		\end{equation}
		
		Degree $\deg(f)$ for $f\in \Rtau$:
		\begin{equation}
			\begin{boxedstrings}
				\polybox{(0,0.5)}{$f$}
				\path (0,0) ++ (0,1);
			\end{boxedstrings}.
		\end{equation}
	\end{minipage}
\end{center}

}
We also require polynomial boxes to satisfy relations \eqref{polynomial sum} and \eqref{polynomial mult}.
We set  $\deg(\alpha_s)=2$ and $\deg(\alpha_t)=2$.

The diagrams in \eqref{bivalent vertices} joining two different colored strands will be called \textit{bivalent vertices}.
The green and orange trivalent vertex in \eqref{trivalent vetices simple} will
be called an \textit{orange landing}.
The diagrams in \eqref{crossings vertices} will be called \textit{orange crossings}.
The trivalent (resp. bivalent) vertices in green or brown (with no orange strands) will be called \textit{GB trivalent (resp. bivalent) vertices}.

Here composition is given by vertical stacking and tensor product by horizontal pasting of diagrams.

A diagram in $\Dcalpre$ will be called an \textit{equivariant Hecke diagram of type $A_1\times A_1$}.
\end{defi}

\begin{rema}\label{rema not needed generators}
	In Section \ref{subsection-presentation} we will see that the generators 
	\begin{equation}
		\begin{boxedstrings}
			\trivuparamcol{(0,0)}{0.5}{0.5}{green}{brown}{brown}
			\idfigparam{(nd-1-1)}{-0.25}{green}
			\idfigparam{(nd-1-2)}{0.25}{brown}
			\idfigparam{(nd-1-3)}{0.25}{brown}
		\end{boxedstrings}
	,
		\begin{boxedstrings}
			\Xfig{(0,0)}{brown}{orange}
			\idfigparam{(nd-1-1)}{-0.25}{brown}
			\idfigparam{(nd-1-2)}{-0.25}{orange}
			\idfigparam{(nd-1-4)}{0.25}{brown}
			\idfigparam{(nd-1-3)}{0.25}{orange}
		\end{boxedstrings}
	,
		\begin{boxedstrings}
			\Xfig{(0,0)}{green}{orange}
			\idfigparam{(nd-1-1)}{-0.25}{green}
			\idfigparam{(nd-1-2)}{-0.25}{orange}
			\idfigparam{(nd-1-4)}{0.25}{green}
			\idfigparam{(nd-1-3)}{0.25}{orange}
		\end{boxedstrings}
	,
	\begin{boxedstrings}
		\Xfig{(0,0)}{orange}{orange}
		\idfigparam{(nd-1-1)}{-0.25}{orange}
		\idfigparam{(nd-1-2)}{-0.25}{orange}
		\idfigparam{(nd-1-4)}{0.25}{orange}
		\idfigparam{(nd-1-3)}{0.25}{orange}
	\end{boxedstrings}
	,
		\begin{boxedstrings}
			\idfigparam{(0,1)}{0.5}{green}
			\idfigparam{(nd-1-2)}{0.5}{brown}
		\end{boxedstrings}
	,
		\text{ and }
		\begin{boxedstrings}
			\idfigparam{(0,1)}{0.5}{orange}
			\idfigparam{(nd-1-2)}{0.5}{green}
		\end{boxedstrings},
	\end{equation}
	can be expressed in terms of the other generators in the quotient category $\Dcal_{A_1\times A_1}$. 
	In fact, their definition is given in relations \ref{vertex definitions} and \ref{bivalent definition right}
	.
	We choose to include these generators because they make the notation easier and facilitate the diagrammatic argument used in Section \ref{subsection-diagram reduction}.
\end{rema}

From now on, we will use the hat notation to differentiate the objects $\onecalbar,\Xobjbar,\Yobjbar,\Zobjbar\in \Dcalpre$  from the objects $\onecal,\Xobj,\Yobj,\Zobj\in \Hcaltau$. We will also define $\XZobjbar:=\Xobjbar \otimes \Zobjbar$ and 
$\Icalbar:=\{\onecalbar,\Xobjbar,\Yobjbar,\Zobjbar,\XZobjbar\}$.

\begin{defi}
We can now define a strict monoidal functor $\Fboldtilde: \Dcalpre\longrightarrow \Hcal^\tau$
that sends 
$\Xobjbar\mapsto \Xobj$, $\Yobjbar\mapsto \Yobj$, and $\Zobjbar\mapsto \Zobj$, and sends \vspace{-0.2cm}
\excludeDiagrams{1}
{

\begin{center}
	\begin{minipage}[t]{0.45\textwidth}
		\begin{align}
		\begin{boxedstrings}
			\capfigparam{(0,0)}{0.5}{0.5}{orange}
			\path (0,1)-- ++(0,-1);
		\end{boxedstrings}
			&\mapsto
			\begin{boxedstrings}
				\path (0,0) -- (0.3,1);
			\end{boxedstrings}
			\\
		\begin{boxedstrings}
			\capfigparam{(0,0)}{0.5}{0.5}{green}
			\path (0,1)-- ++(0,-1);
		\end{boxedstrings}
			&\mapsto
			{
				\begingroup
				\setlength\arraycolsep{0pt}
			\begin{pmatrix}
				\begin{unboxedstrings}
					\capfigparam{(0,0)}{0.5}{0.5}{red}
				\end{unboxedstrings}
				&
				\begin{unboxedstrings}
					\capfigparam{(0,0)}{0.5}{0.5}{blue}
				\end{unboxedstrings}
			\end{pmatrix}
				\endgroup
		}
			\\
		\begin{boxedstrings}
			\capfigparam{(0,0)}{0.5}{0.5}{brown}
			\path (0,1)-- ++(0,-1);
		\end{boxedstrings}
		&\mapsto
		\begin{boxedstrings}
			\capfigparam{(0.25,0)}{0.5}{0.5}{blue}
			\capfigparam{(0,0)}{0.5}{0.5}{red}
			\path (0,1) -- ++(0,-1);
		\end{boxedstrings}
			\\
		\begin{boxedstrings}
				\trivu{(0,0)}{brown}
				\idfigparam{(nd-1-1)}{-0.25}{brown}
				\idfigparam{(nd-1-2)}{0.25}{brown}
				\idfigparam{(nd-1-3)}{0.25}{brown}
			\end{boxedstrings}
			&\mapsto
		\begin{boxedstrings}
				\trivu{(0,0)}{red}
				\trivu{(0.25,0)}{blue}
				\idfigparam{(nd-1-1)}{-0.25}{red}
				\idfigparam{(nd-1-2)}{0.25}{red}
				\idfigparam{(nd-1-3)}{0.25}{red}
				\idfigparam{(nd-2-1)}{-0.25}{blue}
				\idfigparam{(nd-2-2)}{0.25}{blue}
				\idfigparam{(nd-2-3)}{0.25}{blue}
			\end{boxedstrings}
			\\
		\begin{boxedstrings}
				\trivu{(0,0)}{green}
				\idfigparam{(nd-1-1)}{-0.25}{green}
				\idfigparam{(nd-1-2)}{0.25}{green}
				\idfigparam{(nd-1-3)}{0.25}{green}
			\end{boxedstrings}
			&\mapsto
			{
			\begingroup
			\setlength\arraycolsep{0pt}
			\begin{pmatrix}
				\hspace{-0.1cm}
				\scaledfitstrings{
					\trivu{(0,0)}{red}
					\idfigparam{(nd-1-1)}{-0.25}{red}
					\idfigparam{(nd-1-2)}{0.25}{red}
					\idfigparam{(nd-1-3)}{0.25}{red}
				}
				\\ &
				\scaledfitstrings{
					\trivu{(0,0)}{blue}
					\idfigparam{(nd-1-1)}{-0.25}{blue}
					\idfigparam{(nd-1-2)}{0.25}{blue}
					\idfigparam{(nd-1-3)}{0.25}{blue}
				}
			\end{pmatrix}
			\endgroup
		}
		\\
		\begin{boxedstrings}
			\trivuparamcol{(0,0)}{0.5}{0.5}{green}{brown}{brown}
			\idfigparam{(nd-1-1)}{-0.25}{green}
			\idfigparam{(nd-1-2)}{0.25}{brown}
			\idfigparam{(nd-1-3)}{0.25}{brown}
		\end{boxedstrings}
		&\mapsto
		{
		\begingroup
		\setlength\arraycolsep{0pt}
		\def\arraystretch{1.2}
		\begin{pmatrix}
			\scaledfitstrings{
				\trivu{(0,0)}{red}
				\capfigparam{(0,-0.75)}{0.5}{0.4}{blue}
				\idfigparam{(nd-1-1)}{-0.25}{red}
				\idfigparam{(nd-1-2)}{0.25}{red}
				\idfigparam{(nd-1-3)}{0.25}{red}
			}
			\\ 
			\scaledfitstrings{
				\trivu{(0,0)}{blue}
				\capfigparam{(0,-0.75)}{0.5}{0.4}{red}
				\idfigparam{(nd-1-1)}{-0.25}{blue}
				\idfigparam{(nd-1-2)}{0.25}{blue}
				\idfigparam{(nd-1-3)}{0.25}{blue}
			}
		\end{pmatrix}
		\endgroup
		}
		\\
		\begin{boxedstrings}
			\trivuparamcol{(0,0)}{0.5}{0.5}{green}{orange}{green}
			\idfigparam{(nd-1-1)}{-0.25}{green}
			\idfigparam{(nd-1-2)}{0.25}{orange}
			\idfigparam{(nd-1-3)}{0.25}{green}
		\end{boxedstrings}
		&\mapsto
		{
		\begingroup
			\setlength\arraycolsep{2pt}
			\begin{pmatrix}
				\hspace{-0.1cm}
				\scaledfitstrings{
					\idfigparam{(0,1)}{1}{red}
				}
				\\ &
				-
				\hspace{-0.15cm}
				\scaledfitstrings{
					\idfigparam{(0,1)}{1}{blue}
				}
			\end{pmatrix}
			\endgroup
		}
		\\
		\begin{boxedstrings}
			\trivuparamcol{(0,0)}{0.5}{0.5}{brown}{green}{green}
			\idfigparam{(nd-1-1)}{-0.25}{brown}
			\idfigparam{(nd-1-2)}{0.25}{green}
			\idfigparam{(nd-1-3)}{0.25}{green}
		\end{boxedstrings}
		&\mapsto
		{
		\begingroup
		\setlength\arraycolsep{1pt}
		\begin{pmatrix}
			\hspace{-0.1cm}
			\begin{unboxedstrings}
				\idfigparam{(0,0)}{0.6}{red}
				\idfigparam{(0.25,0)}{0.6}{blue}
			\end{unboxedstrings}
			&
			\begin{unboxedstrings}
				\Xfigparam{(0,0)}{0.25}{0.25}{red}{blue}
				\idfigparam{(nd-1-1)}{-0.2}{red}
				\idfigparam{(nd-1-2)}{-0.2}{blue}
				\idfigparam{(nd-1-4)}{0.2}{red}
				\idfigparam{(nd-1-3)}{0.2}{blue}
			\end{unboxedstrings}
		\end{pmatrix}
		\endgroup
		}
		\\
		\begin{boxedstrings}
			\Xfig{(0,0)}{brown}{orange}
			\idfigparam{(nd-1-1)}{-0.25}{brown}
			\idfigparam{(nd-1-2)}{-0.25}{orange}
			\idfigparam{(nd-1-4)}{0.25}{brown}
			\idfigparam{(nd-1-3)}{0.25}{orange}
		\end{boxedstrings}
		&\mapsto
		\begin{boxedstrings}
			\idfig{(0,0)}{red}
			\idfig{(0.25,0)}{blue}
		\end{boxedstrings}
		\end{align}
	\end{minipage}
	\begin{minipage}[t]{0.45\textwidth}
	\begin{align}
		\begin{boxedstrings}
			\Xfig{(0,0)}{green}{orange}
			\idfigparam{(nd-1-1)}{-0.25}{green}
			\idfigparam{(nd-1-2)}{-0.25}{orange}
			\idfigparam{(nd-1-4)}{0.25}{green}
			\idfigparam{(nd-1-3)}{0.25}{orange}
		\end{boxedstrings}
		&\mapsto
		{
			\begingroup
			\setlength\arraycolsep{2pt}
			\begin{pmatrix}
				\hspace{-0.1cm}
				\scaledfitstrings{
					\idfigparam{(0,1)}{1}{red}
				}
				\\ &
				\scaledfitstrings{
					\idfigparam{(0,1)}{1}{blue}
				}
			\end{pmatrix}
			\endgroup
		}
		\\
		\begin{boxedstrings}
			\Xfig{(0,0)}{orange}{orange}
			\idfigparam{(nd-1-1)}{-0.25}{orange}
			\idfigparam{(nd-1-2)}{-0.25}{orange}
			\idfigparam{(nd-1-4)}{0.25}{orange}
			\idfigparam{(nd-1-3)}{0.25}{orange}
		\end{boxedstrings}
		&\mapsto
		\begin{boxedstrings}
			\path (0,0) -- (0.3,1);
		\end{boxedstrings}
		\\
		\begin{boxedstrings}
			\dotuparam{(0,0)}{0.6}{0.06}{green}
			\path (nd-1-1) ++ (0,1);
		\end{boxedstrings}
		&\mapsto
		{
			\begingroup
			\setlength\arraycolsep{0pt}
			\begin{pmatrix}
				\begin{unboxedstrings}
					\dotuparam{(0,0)}{0.5}{0.06}{red}
				\end{unboxedstrings}
				&
				\begin{unboxedstrings}
					\dotuparam{(0,0)}{0.5}{0.06}{blue}
				\end{unboxedstrings}
			\end{pmatrix}
			\endgroup
		}
		\\
		\begin{boxedstrings}
			\idfigparam{(0,1)}{0.5}{green}
			\idfigparam{(nd-1-2)}{0.5}{brown}
		\end{boxedstrings}
		&\mapsto
		{
			\begingroup
			\setlength\arraycolsep{0pt}
			\def\arraystretch{1.2}
			\begin{pmatrix}
				\scaledfitstrings{
					\idfigparam{(0,0)}{-1}{red}
					\dotuparam{(0.25,0)}{0.6}{0.06}{blue}
				}
				\\ 
				\scaledfitstrings{
					\idfigparam{(0.25,0)}{-1}{blue}
					\dotuparam{(0,0)}{0.6}{0.06}{red}
				}
			\end{pmatrix}
			\endgroup
			\hspace{0.1cm}
		}
		\\
		\begin{boxedstrings}
			\idfigparam{(0,1)}{0.5}{orange}
			\idfigparam{(nd-1-2)}{0.5}{green}
		\end{boxedstrings}
		&\mapsto
		{
			\begingroup
			\setlength\arraycolsep{0pt}
			\begin{pmatrix}
				\begin{unboxedstrings}
					\dotuparam{(0,0)}{0.5}{0.06}{red}
				\end{unboxedstrings}
				&
				-\hspace{-0.15cm}
				\begin{unboxedstrings}
					\dotuparam{(0,0)}{0.5}{0.06}{blue}
				\end{unboxedstrings}
			\end{pmatrix}
			\endgroup
		}
		\\
		\begin{boxedstrings}
			\dotuparam{(0,0)}{0.6}{0.06}{orange}
			\path (nd-1-1) ++ (0,1);
		\end{boxedstrings}
		&\mapsto
		\begin{boxedstrings}
			\polybox{(0,0.5)}{{$\alpha_s \! -\! \alpha_t$}}
			\path (0,0) ++ (0,1);
		\end{boxedstrings}
		\\
		\begin{boxedstrings}
			\dotuparam{(0,0)}{0.6}{0.06}{brown}
			\path (nd-1-1) ++ (0,1);
		\end{boxedstrings}
		&\mapsto
		\begin{boxedstrings}
			\dotuparam{(0,0)}{0.6}{0.06}{red}
			\dotuparam{(0.25,0)}{0.6}{0.06}{blue}
			\path (nd-1-1) ++ (0,1);
		\end{boxedstrings}
		\\
		\begin{boxedstrings}
			\polybox{(0,0.5)}{$f$}
			\path (0,0) ++ (0,1);
		\end{boxedstrings}
		&\mapsto
		\begin{boxedstrings}
			\polybox{(0,0.5)}{$f$}
			\path (0,0) ++ (0,1);
		\end{boxedstrings} 
	\text{ for }
	f\in\Rtau
	.
	\end{align}
	\end{minipage}
	\end{center}

}

\end{defi}

This functor is well defined since, by construction, the isotopy relations of $\Dcalpre$ also hold in $\Hcaltau$.

\subsection{Presentation}\label{subsection-presentation}

\begin{defi}\label{relations equiv category}
	
	Define the \textit{diagrammatic equivariant Hecke category of type $A_1\times A_1$}, which we denote by $\Dcal_{A_1\times A_1}$ (and abbreviate to $\Dcal$), to be the quotient of $\Dcalpre$ by the following relations.
	
	Barbell relations:
	\excludeDiagrams{1}
	{
		\begin{equation}\label{barbell rels}
			\begin{boxedstrings}
				\vbarb{(0,0.7)}{green}
				\path (0,0) ++ (0,1);
			\end{boxedstrings}
			=
			\begin{boxedstrings}
				\polybox{(0,0.5)}{$\alpha_s+\alpha_t$}
				\path (0,0) ++ (0,1);
			\end{boxedstrings},   
			\begin{boxedstrings}
				\vbarb{(0,0.7)}{brown}
				\path (0,0) ++ (0,1);
			\end{boxedstrings}
			=
			\begin{boxedstrings}
				\polybox{(0,0.5)}{$\alpha_s \alpha_t$}
				\path (0,0) ++ (0,1);
			\end{boxedstrings},   
			\begin{boxedstrings}
				\vbarb{(0,0.7)}{orange}
				\path (0,0) ++ (0,1);
			\end{boxedstrings}
			=
			\begin{boxedstrings}
				\polybox{(0,0.5)}{$(\alpha_s-\alpha_t)^2$}
				\path (0,0) ++ (0,1);
			\end{boxedstrings}.
		\end{equation}
	}
	
	Vertex definitions:
	\excludeDiagrams{1}
	{
		\begin{equation}\label{vertex definitions}
			\begin{boxedstrings}
				\trivuparamcol{(0,0)}{0.5}{0.5}{green}{brown}{brown}
				\idfigparam{(nd-1-1)}{-0.25}{green}
				\idfigparam{(nd-1-2)}{0.25}{brown}
				\idfigparam{(nd-1-3)}{0.25}{brown}
			\end{boxedstrings}
			   =
			\begin{boxedstrings}
				\Triangucol{(0,0)}{green}{green}{green}{green}{brown}{brown}
			\end{boxedstrings}
		,
			\begin{boxedstrings}
				\Xfig{(0,0)}{green}{orange}
				\idfigparam{(nd-1-1)}{-0.25}{green}
				\idfigparam{(nd-1-2)}{-0.25}{orange}
				\idfigparam{(nd-1-4)}{0.25}{green}
				\idfigparam{(nd-1-3)}{0.25}{orange}
			\end{boxedstrings}
			=
			\begin{boxedstrings}
				\trivdparamcol{(0,0)}{0.5}{0.5}{green}{orange}{green}
				\trivuparamcol{(nd-1-3)}{0.5}{0.5}{green}{orange}{green}
			\end{boxedstrings}
			,
		\end{equation}
		\begin{equation}\label{vertex definitions 2}
			\begin{boxedstrings}
				\Xfig{(0,0)}{orange}{orange}
				\idfigparam{(nd-1-1)}{-0.25}{orange}
				\idfigparam{(nd-1-2)}{-0.25}{orange}
				\idfigparam{(nd-1-4)}{0.25}{orange}
				\idfigparam{(nd-1-3)}{0.25}{orange}
			\end{boxedstrings}
		=
			\begin{boxedstrings}
				\idfigparam{(0,1)}{1}{orange}
				\idfigparam{(0.5,1)}{1}{orange}
			\end{boxedstrings}
			,
			\begin{boxedstrings}
				\Xfig{(0,0)}{brown}{orange}
				\idfigparam{(nd-1-1)}{-0.25}{brown}
				\idfigparam{(nd-1-2)}{-0.25}{orange}
				\idfigparam{(nd-1-4)}{0.25}{brown}
				\idfigparam{(nd-1-3)}{0.25}{orange}
			\end{boxedstrings}=
			-\frac{1}{2}
			\begin{boxedstrings}
				\Hfigparamcol{(0,1)}{0.5}{0.5}{brown}{orange}{green}{green}{green}
				\Hfigparamcol{(nd-1-3)}{0.5}{0.5}{green}{green}{green}{orange}{brown}
			\end{boxedstrings}.
		\end{equation}
		\begin{equation}\label{bivalent definition right}
		\begin{boxedstrings}
			\idfigparam{(0,1)}{0.5}{brown}
			\idfigparam{(nd-1-2)}{0.5}{green}
		\end{boxedstrings} =
		\begin{boxedstrings}
			\idfigparam{(0,1)}{0.5}{brown}
			\idfigparam{(nd-1-2)}{0.5}{green}
			\dotr{(nd-1-2)}{green}
		\end{boxedstrings}
		,
		\hspace{0.2cm}
		\begin{boxedstrings}
			\idfigparam{(0,1)}{0.5}{orange}
			\idfigparam{(nd-1-2)}{0.5}{green}
		\end{boxedstrings} =
		\begin{boxedstrings}
			\idfigparam{(0,1)}{0.5}{orange}
			\idfigparam{(nd-1-2)}{0.5}{green}
			\dotr{(nd-1-2)}{green}
		\end{boxedstrings}
		.
	\end{equation}
	
	}

	Circle and Needle relations:
	\excludeDiagrams{1}
	{
		\begin{equation}
			\label{needle relations}
				\begin{boxedstrings}
					\path (0,1) -- (0,0);
					\node (a) at (-0.2,0.5) {};
					\def\coltop{orange}
					\capfigparam{(a)}{0.4}{0.2}{\coltop}
					\cupfigparam{(a)}{0.4}{0.2}{\coltop}
				\end{boxedstrings}
				= 1
				,
				\hspace{0.2cm}
				\begin{boxedstrings}
					\node at (0,1) {};
					\node (a) at (-0.2,0.5) {};
					\def\coltop{green}
					\def\colbot{green}
					\capfigparam{(a)}{0.4}{0.2}{\coltop}
					\cupfigparam{(a)}{0.4}{0.2}{\coltop}
					\idfigparam {(nd-2-3)}{0.3}{\colbot}
				\end{boxedstrings}
				= 0
				,
				\hspace{0.2cm}
				\begin{boxedstrings}
					\node at (0,1) {};
					\node (a) at (-0.2,0.5) {};
					\def\coltop{green}
					\def\colbot{brown}
					\capfigparam{(a)}{0.4}{0.2}{\coltop}
					\cupfigparam{(a)}{0.4}{0.2}{\coltop}
					\idfigparam {(nd-2-3)}{0.3}{\colbot}
				\end{boxedstrings}
				= 0
				,
				\hspace{0.2cm}
				\begin{boxedstrings}
					\node at (0,1) {};
					\node (a) at (-0.2,0.5) {};
					\def\coltop{brown}
					\def\colbot{brown}
					\capfigparam{(a)}{0.4}{0.2}{\coltop}
					\cupfigparam{(a)}{0.4}{0.2}{\coltop}
					\idfigparam {(nd-2-3)}{0.3}{\colbot}
				\end{boxedstrings}
				= 0
				.
		\end{equation}
	}
	
	Unit relations:
	\excludeDiagrams{1}
	{

		\begin{equation}\label{one color unit rel}
			\begin{boxedstrings}
				\idfig{(1,1)}{green}
				\dotl{(nd-1-3)}{green}
			\end{boxedstrings}
			=
			\begin{boxedstrings}
				\idfig{(1,1)}{green}
			\end{boxedstrings}
		,
			\hspace{0.2cm}
			\begin{boxedstrings}
				\idfig{(1,1)}{brown}
				\dotl{(nd-1-3)}{brown}
			\end{boxedstrings}
			=
			\begin{boxedstrings}
				\idfig{(1,1)}{brown}
			\end{boxedstrings}  
			,
		\end{equation}
		
		\begin{equation}\label{bivalent definition left}
			\begin{boxedstrings}
				\idfigparam{(0,1)}{0.5}{brown}
				\idfigparam{(nd-1-2)}{0.5}{green}
				\dotl{(nd-1-2)}{green}
			\end{boxedstrings} 
			=
			\begin{boxedstrings}
				\idfigparam{(0,1)}{0.5}{brown}
				\idfigparam{(nd-1-2)}{0.5}{green}
			\end{boxedstrings} 
		,
			\begin{boxedstrings}
				\idfigparam{(0,1)}{0.5}{orange}
				\idfigparam{(nd-1-2)}{0.5}{green}
				\dotl{(nd-1-2)}{green}
			\end{boxedstrings} 
			=
			\begin{boxedstrings}
				\idfigparam{(0,1)}{0.5}{orange}
				\idfigparam{(nd-1-2)}{0.5}{green}
			\end{boxedstrings}
		,
		\end{equation}

		\begin{equation}\label{brown green unit rel}
			2
			\begin{boxedstrings}
				\idfig{(1,1)}{green}
				\dotl{(nd-1-3)}{brown}
			\end{boxedstrings}
			=
			\begin{boxedstrings}
				\dotd{(1,1)}{green}
				\dotu{(1,0)}{green}
			\end{boxedstrings}
			-
			\begin{boxedstrings}
				\idfigparam{(0,1)}{0.3}{green}
				\idfigparam{(nd-1-2)}{0.4}{orange}
				\idfigparam{(nd-2-2)}{0.3}{green}
			\end{boxedstrings},		
		\end{equation}

		\begin{equation}\label{orange green green =0 unit rel}
			\begin{boxedstrings}
				\idfigparam{(0,1)}{0.3}{orange}
				\idfigparam{(nd-1-2)}{0.4}{green}
				\idfigparam{(nd-2-2)}{0.3}{brown}
			\end{boxedstrings}
			=
			0 ,
		\end{equation}

		\begin{equation}\label{green orange unit rel}
			\begin{boxedstrings}
				\idfigparam{(0,0)}{-0.35}{orange}
				\idfigparam{(nd-1-2)}{-0.35}{green}
				\dotfig{(nd-2-2)}{green}
				\path (nd-1-1) ++ (0,1);
			\end{boxedstrings}
			=
			\begin{boxedstrings}
				\dotuparam{(0,0)}{0.4}{0.06}{orange}
				\path (nd-1-1) ++ (0,1);
			\end{boxedstrings}
			.
		\end{equation}
		Note that \eqref{bivalent definition right} and \eqref{bivalent definition left} are mirror images of each other (and we need them both).
	}
	
	Polynomial forcing relation:
	\excludeDiagrams{1}
	{
		\begin{equation}\label{poly forcing green 1}
			\begin{boxedstrings}
				\vbarb{(0.75,0.7)}{green}
				\idfig{(1,1)}{green}
			\end{boxedstrings}
			+
			\begin{boxedstrings}
				\idfig{(1,1)}{green}
				\dotr{(nd-1-3)}{orange}
			\end{boxedstrings}
			=
			\begin{boxedstrings}
				\dotd{(1,1)}{green}
				\dotu{(1,0)}{green}
			\end{boxedstrings}
			+
			\begin{boxedstrings}
				\idfigparam{(0,1)}{0.3}{green}
				\idfigparam{(nd-1-2)}{0.4}{orange}
				\idfigparam{(nd-2-2)}{0.3}{green}
			\end{boxedstrings}
			.
		\end{equation}
	
	}
	
	Sliding of the orange landing:
	\excludeDiagrams{1}
	{
		
		\begin{equation}\label{XY idempotent decomp}
			\begin{boxedstrings}
				\Ilongparam{(0,1)}{0.5}{1}{orange}{green}{green}{orange}{green}
			\end{boxedstrings}
			=
			\begin{boxedstrings}
				\idfigparam{(0,1)}{1}{orange}
				\idfigparam{(0.25,1)}{1}{green}
			\end{boxedstrings}
			,
		\end{equation}
		
		\begin{equation}\label{GGG orange}
			\begin{boxedstrings}
				\trivu{(0.25,0.75)}{green}
				\idfigparam{(nd-1-1)}{-0.25}{green}
				\idfigparam{(nd-1-2)}{0.25}{green}
				\idfigparam{(nd-1-3)}{0.25}{green}
				\draw[orange] (nd-1-1) -- (0.75,0.75) -- (0.75,1);
			\end{boxedstrings}
			=
			\begin{boxedstrings}
				\trivu{(0.25,0.75)}{green}
				\idfigparam{(nd-1-1)}{-0.25}{green}
				\idfigparam{(nd-1-2)}{0.25}{green}
				\idfigparam{(nd-1-3)}{0.25}{green}
				\draw[orange] (nd-1-3) -- (0.75,0.5) -- (0.75,1);
			\end{boxedstrings},
		\end{equation}
		
		\begin{equation}\label{GGB orange 1}
			\begin{boxedstrings}
				\trivuparamcol{(0.25,0.75)}{0.5}{0.5}{brown}{green}{green}
				\idfigparam{(nd-1-1)}{-0.25}{brown}
				\idfigparam{(nd-1-2)}{0.25}{green}
				\idfigparam{(nd-1-3)}{0.25}{green}
				\draw[orange] (nd-1-2) -- (0.25,0);
			\end{boxedstrings}
			=-
			\begin{boxedstrings}
				\trivuparamcol{(0.25,0.75)}{0.5}{0.5}{brown}{green}{green}
				\idfigparam{(nd-1-1)}{-0.25}{brown}
				\idfigparam{(nd-1-2)}{0.25}{green}
				\idfigparam{(nd-1-3)}{0.25}{green}
				\draw[orange] (nd-1-3) -- (0.25,0);
			\end{boxedstrings},
		\end{equation}
		
		\begin{equation}\label{GGB triv orange 2}
			\begin{boxedstrings}
				\trivuparamcol{(0.25,0.75)}{0.5}{0.5}{brown}{green}{green}
				\idfigparam{(nd-1-1)}{-0.25}{brown}
				\idfigparam{(nd-1-2)}{0.25}{green}
				\idfigparam{(nd-1-3)}{0.25}{green}
				\draw[orange] (nd-1-2) -- (0,0.5) -- (nd-1-1) -- (0.75,0.75) -- (0.75,1);
			\end{boxedstrings}
			=
			-
			\begin{boxedstrings}
				\trivuparamcol{(0.25,0.75)}{0.5}{0.5}{brown}{green}{green}
				\idfigparam{(nd-1-1)}{-0.25}{brown}
				\idfigparam{(nd-1-2)}{0.25}{green}
				\idfigparam{(nd-1-3)}{0.25}{green}
				\draw[orange] (nd-1-3) -- (0.75,0.5) -- (0.75,1);
			\end{boxedstrings},
		\end{equation}

	}
	
	Bigon and triangle relations:
	\excludeDiagrams{1}
	{
		\begin{equation}\label{bigon brown}
			\begin{boxedstrings}
				\node at (0,1) {};
				\node (a) at (-0.2,0.5) {};
				\def\coltop{green}
				\def\colbot{brown}
				\def\colside{orange}
				\capfigparam{(a)}{0.4}{0.2}{\coltop}
				\cupfigparam{(a)}{0.4}{0.2}{\coltop}
				\idfigparam {(0,0.3)}{0.3}{\colbot}
				\idfigparam {(0,0.7)}{-0.3}{\colbot}
			\end{boxedstrings}
			= 
			2
			\begin{boxedstrings}
				\idfig{(0,0)}{brown}
			\end{boxedstrings},
		\end{equation}
		
		\begin{equation}\label{bigon brown orange}
			\begin{boxedstrings}
				\Triangucol{(0,0)}{orange}{green}{green}{green}{brown}{brown}
			\end{boxedstrings}
			= 0
			.
		\end{equation}
	}
	
	H=I relations:
\excludeDiagrams{1}
{

	\begin{equation}\label{H=I one color}
		\begin{boxedstrings}
			\Hlongparam{(0,0)}{0.5}{1}{green}{green}{green}{green}{green}
		\end{boxedstrings}
		=
		\begin{boxedstrings}
			\Ilongparam{(0,0)}{0.5}{1}{green}{green}{green}{green}{green}
		\end{boxedstrings},
		\hspace{0.5cm}
		\begin{boxedstrings}
			\Hlongparam{(0,0)}{0.5}{1}{brown}{brown}{brown}{brown}{brown}
		\end{boxedstrings}
		=
		\begin{boxedstrings}
			\Ilongparam{(0,0)}{0.5}{1}{brown}{brown}{brown}{brown}{brown}
		\end{boxedstrings},
	\end{equation}

	\begin{equation}\label{H=I bicolor associativity}
		\begin{boxedstrings}
			\Hlongparam{(0,0)}{0.5}{1}{green}{brown}{brown}{brown}{green}
		\end{boxedstrings}
		=
		\begin{boxedstrings}
			\Ilongparam{(0,0)}{0.5}{1}{green}{brown}{brown}{brown}{green}
		\end{boxedstrings},
		\hspace{0.5cm}
		\begin{boxedstrings}
			\Hlongparam{(0,0)}{0.5}{1}{green}{brown}{brown}{brown}{brown}
		\end{boxedstrings}
		=
		\begin{boxedstrings}
			\Ilongparam{(0,0)}{0.5}{1}{green}{brown}{brown}{brown}{brown}
		\end{boxedstrings},
	\end{equation}

	\begin{equation}\label{H=I top brown}
		\begin{boxedstrings}
			\Hlongparam{(0,0)}{0.5}{1}{brown}{brown}{brown}{green}{green}
		\end{boxedstrings}
		=
		\begin{boxedstrings}
			\Ilongparam{(0,0)}{0.5}{1}{brown}{brown}{green}{green}{green}
		\end{boxedstrings} 
		+
		\begin{boxedstrings}
			\Ilongparam{(0,0)}{0.5}{1}{brown}{brown}{brown}{green}{green}
		\end{boxedstrings}    
	\end{equation}
	
	\begin{equation}\label{H=I L brown leg}
		2
		\begin{boxedstrings}
			\Hlongparam{(0,0)}{0.5}{1}{green}{green}{brown}{green}{brown}
		\end{boxedstrings}
		=
		2
		\begin{boxedstrings}
			\Ilongparam{(0,0)}{0.5}{1}{green}{green}{green}{green}{brown}
		\end{boxedstrings}
		+
		\begin{boxedstrings}
			\Ilongparam{(0,0)}{0.5}{1}{green}{green}{brown}{green}{brown}
		\end{boxedstrings}
		-
		\begin{boxedstrings}
			\Ilongparam{(0,0)}{0.5}{1}{green}{green}{brown}{green}{brown}
			\draw[orange] (nd-1-8) -- (nd-1-9);
		\end{boxedstrings}
	\end{equation}

	\begin{equation} \label{H=I middle brown} 
		2
		\begin{boxedstrings}
			\Hlongparam{(0,0)}{0.5}{1}{green}{green}{brown}{green}{green}
		\end{boxedstrings}
		=
		\begin{boxedstrings}
			\capfigparam{(0,0)}{0.5}{0.3}{green}
			\cupfigparam{(0,1)}{0.5}{0.3}{green}
		\end{boxedstrings}
		-
		\begin{boxedstrings}
			\capfigparam{(0,0)}{0.5}{0.3}{green}
			\cupfigparam{(0,1)}{0.5}{0.3}{green}
			\draw[orange] (nd-1-3) -- (nd-2-3);
		\end{boxedstrings}
		+
		\begin{boxedstrings}
			\Ilongparam{(0,0)}{0.5}{1}{green}{green}{brown}{green}{green}
		\end{boxedstrings}
		-
		\begin{boxedstrings}
			\Ilongparam{(0,0)}{0.5}{1}{green}{green}{brown}{green}{green}
			\draw[orange] (nd-1-8) -- (nd-1-9);
		\end{boxedstrings}
	\end{equation}

}
	
\end{defi}

\begin{lemma}\label{lemma F is well defined}
The relations in Definition \ref{relations equiv category}
hold after we apply the functor $\Fboldtilde:\Dcalpre\longrightarrow \Hcaltau$ 
(defined in Section \ref{subsection free category}).
Thus, $\Fboldtilde$ descends to an essentially surjective functor
$\Fbold:\Dcal\longrightarrow \Hcaltau$.
\end{lemma}
\begin{proof}
The fact that the relations hold comes down to diagrammatic calculations in $\Hcaltau$.
Let us show \eqref{poly forcing green 1},\eqref{bigon brown},\eqref{H=I L brown leg}, and \eqref{H=I middle brown}. We will leave the rest for the reader to verify.

\eqref{bigon brown} Since
\begin{align}
	\begin{boxedstrings}
		\trivuparamcol{(0,0)}{0.5}{0.5}{brown}{green}{green}
		\idfigparam{(nd-1-1)}{-0.25}{brown}
		\idfigparam{(nd-1-2)}{0.25}{green}
		\idfigparam{(nd-1-3)}{0.25}{green}
	\end{boxedstrings}
	&\mapsto
	{
		\begingroup
		\setlength\arraycolsep{1pt}
		\def\arraystretch{1.3}
		\begin{pmatrix}
			\hspace{-0.1cm}
			\begin{unboxedstrings}
				\idfigparam{(0,0)}{0.6}{red}
				\idfigparam{(0.25,0)}{0.6}{blue}
			\end{unboxedstrings}
			&
			\begin{unboxedstrings}
				\Xfigparam{(0,0)}{0.25}{0.25}{red}{blue}
				\idfigparam{(nd-1-1)}{-0.2}{red}
				\idfigparam{(nd-1-2)}{-0.2}{blue}
				\idfigparam{(nd-1-4)}{0.2}{red}
				\idfigparam{(nd-1-3)}{0.2}{blue}
			\end{unboxedstrings}
		\end{pmatrix}
		\endgroup
	}
\end{align}
and
\begin{align}
	\begin{boxedstrings}
		\trivdparamcol{(0,0)}{0.5}{0.5}{green}{green}{brown}
		\idfigparam{(nd-1-3)}{0.25}{brown}
		\idfigparam{(nd-1-2)}{-0.25}{green}
		\idfigparam{(nd-1-1)}{-0.25}{green}
	\end{boxedstrings}
	&\mapsto
	{
		\begingroup
		\setlength\arraycolsep{1pt}
		\def\arraystretch{1.3}
		\begin{pmatrix}
			\hspace{-0.2cm}
			\scaledstrings[0.75]{
				\idfigparam{(0,0)}{0.6}{red}
				\idfigparam{(0.25,0)}{0.6}{blue}
			}
			\\
			\hspace{-0.2cm}
			\scaledstrings[0.75]{
				\Xfigparam{(0,0)}{0.25}{0.25}{blue}{red}
				\idfigparam{(nd-1-1)}{-0.2}{blue}
				\idfigparam{(nd-1-2)}{-0.2}{red}
				\idfigparam{(nd-1-4)}{0.2}{blue}
				\idfigparam{(nd-1-3)}{0.2}{red}
			}
		\end{pmatrix}
		\endgroup
	}
	,
\end{align}
then
\begin{align}
	\begin{boxedstrings}
		\node at (0,1) {};
		\node (a) at (-0.2,0.5) {};
		\def\coltop{green}
		\def\colbot{brown}
		\def\colside{orange}
		\capfigparam{(a)}{0.4}{0.2}{\coltop}
		\cupfigparam{(a)}{0.4}{0.2}{\coltop}
		\idfigparam {(0,0.3)}{0.3}{\colbot}
		\idfigparam {(0,0.7)}{-0.3}{\colbot}
	\end{boxedstrings}
	&\mapsto
	{
		\begingroup
		\setlength\arraycolsep{1pt}
		\def\arraystretch{1.3}
		\begin{pmatrix}
			\hspace{-0.1cm}
			\begin{unboxedstrings}
				\idfigparam{(0,0)}{0.6}{red}
				\idfigparam{(0.25,0)}{0.6}{blue}
			\end{unboxedstrings}
			&
			\begin{unboxedstrings}
				\Xfigparam{(0,0)}{0.25}{0.25}{red}{blue}
				\idfigparam{(nd-1-1)}{-0.2}{red}
				\idfigparam{(nd-1-2)}{-0.2}{blue}
				\idfigparam{(nd-1-4)}{0.2}{red}
				\idfigparam{(nd-1-3)}{0.2}{blue}
			\end{unboxedstrings}
		\end{pmatrix}
		\endgroup
	}
	{
		\begingroup
		\setlength\arraycolsep{1pt}
		\def\arraystretch{1.3}
		\begin{pmatrix}
			\hspace{-0.2cm}
			\scaledstrings[0.75]{
				\idfigparam{(0,0)}{0.6}{red}
				\idfigparam{(0.25,0)}{0.6}{blue}
			}
			\\
			\hspace{-0.2cm}
			\scaledstrings[0.75]{
				\Xfigparam{(0,0)}{0.25}{0.25}{blue}{red}
				\idfigparam{(nd-1-1)}{-0.2}{blue}
				\idfigparam{(nd-1-2)}{-0.2}{red}
				\idfigparam{(nd-1-4)}{0.2}{blue}
				\idfigparam{(nd-1-3)}{0.2}{red}
			}
		\end{pmatrix}
		\endgroup
	}
	=
	2 
	\begin{boxedstrings}
		\idfigparam{(0,0)}{1}{red}
		\idfigparam{(0.25,0)}{1}{blue}
	\end{boxedstrings}
	.
\end{align}

\eqref{poly forcing green 1}
The left hand side of the equation is sent to
\begin{align}
	\begin{boxedstrings}
		\vbarb{(0.75,0.7)}{green}
		\idfig{(1,1)}{green}
	\end{boxedstrings}
	+
	\begin{boxedstrings}
		\idfig{(1,1)}{green}
		\dotr{(nd-1-3)}{orange}
	\end{boxedstrings}
	&\mapsto
	{
		\begingroup
		\setlength\arraycolsep{0pt}
		\def\arraystretch{1.3}
		\begin{pmatrix}
			\hspace{-0.2cm}
			\scaledstrings{
				\polyboxscale{(-1,0.5)}{$\alpha_s+\alpha_t$}{2}
				\draw[red, line width=1mm] (0,0)-- ++(0,1);
			}
			+
			\scaledstrings{
				\polyboxscale{(1,0.5)}{$\alpha_s-\alpha_t$}{2}
				\draw[red, line width=1mm] (0,0)-- ++(0,1);
			}
			\\
			&
			\scaledstrings{
				\polyboxscale{(-1,0.5)}{$\alpha_s+\alpha_t$}{2}
				\draw[blue, line width=1mm] (0,0)-- ++(0,1);
			}
			-
			\scaledstrings{
				\polyboxscale{(1,0.5)}{$\alpha_s-\alpha_t$}{2}
				\draw[blue, line width=1mm] (0,0)-- ++(0,1);
			}
		\end{pmatrix}
		\endgroup
	}
	\\
	&=
	{
		\begingroup
		\setlength\arraycolsep{0pt}
		\def\arraystretch{1.3}
		\begin{pmatrix}
			\hspace{-0.2cm}
			\scaledstrings{
				\polyboxscale{(-0.5,0.5)}{$\alpha_s$}{2}
				\draw[red, line width=1mm] (0,0)-- ++(0,1);
			}
			+
			\scaledstrings{
				\polyboxscale{(0.5,0.5)}{$\alpha_s$}{2}
				\draw[red, line width=1mm] (0,0)-- ++(0,1);
			}
			\\
			&
			\scaledstrings{
				\polyboxscale{(-0.5,0.5)}{$\alpha_t$}{2}
				\draw[blue, line width=1mm] (0,0)-- ++(0,1);
			}
			+
			\scaledstrings{
				\polyboxscale{(0.5,0.5)}{$\alpha_t$}{2}
				\draw[blue, line width=1mm] (0,0)-- ++(0,1);
			}
		\end{pmatrix}
		\endgroup
	}
	\\
	&=
	{
	\begingroup
	\setlength\arraycolsep{0pt}
	\def\arraystretch{1.3}
	2
	\begin{pmatrix}
		\hspace{-0.2cm}
		\scaledstrings{
			\tikzset{every path/.style={line width=1 mm}}
			\dotdparam{(0.75,1)}{0.3}{0.08}{red}
			\dotuparam{(0.75,0)}{0.3}{0.08}{red}
		}
		\\
		&
		\scaledstrings{
			\tikzset{every path/.style={line width=1 mm}}
			\dotdparam{(0.75,1)}{0.3}{0.08}{blue}
			\dotuparam{(0.75,0)}{0.3}{0.08}{blue}
		}
	\end{pmatrix}
	\endgroup
	}
	.
\end{align}

The right hand side of the equation is sent to
\begin{align}
	\begin{boxedstrings}
		\dotd{(1,1)}{green}
		\dotu{(1,0)}{green}
	\end{boxedstrings}
	+
	\begin{boxedstrings}
		\idfigparam{(0,1)}{0.3}{green}
		\idfigparam{(nd-1-2)}{0.4}{orange}
		\idfigparam{(nd-2-2)}{0.3}{green}
	\end{boxedstrings}
	&\mapsto
	{
		\begingroup
		\setlength\arraycolsep{0pt}
		\def\arraystretch{1.3}
		\begin{pmatrix}
			\hspace{-0.2cm}
			\scaledstrings{
				\tikzset{every path/.style={line width=1 mm}}
				\dotdparam{(0.75,1)}{0.3}{0.08}{red}
				\dotuparam{(0.75,0)}{0.3}{0.08}{red}
			}
			&
			\scaledstrings{
				\tikzset{every path/.style={line width=1 mm}}
				\dotdparam{(0.75,1)}{0.3}{0.08}{red}
				\dotuparam{(0.75,0)}{0.3}{0.08}{blue}
			}
			\\
			\hspace{-0.2cm}
			\scaledstrings{
				\tikzset{every path/.style={line width=1 mm}}
				\dotdparam{(0.75,1)}{0.3}{0.08}{blue}
				\dotuparam{(0.75,0)}{0.3}{0.08}{red}
			}
			&
			\scaledstrings{
				\tikzset{every path/.style={line width=1 mm}}
				\dotdparam{(0.75,1)}{0.3}{0.08}{blue}
				\dotuparam{(0.75,0)}{0.3}{0.08}{blue}
			}
		\end{pmatrix}
		+
		\begin{pmatrix}
			\hspace{-0.2cm}
			\scaledstrings{
				\tikzset{every path/.style={line width=1 mm}}
				\dotdparam{(0.75,1)}{0.3}{0.08}{red}
				\dotuparam{(0.75,0)}{0.3}{0.08}{red}
			}
			&
			-\hspace{-0.2cm}
			\scaledstrings{
				\tikzset{every path/.style={line width=1 mm}}
				\dotdparam{(0.75,1)}{0.3}{0.08}{red}
				\dotuparam{(0.75,0)}{0.3}{0.08}{blue}
			}
			\\
			-\hspace{-0.2cm}
			\scaledstrings{
				\tikzset{every path/.style={line width=1 mm}}
				\dotdparam{(0.75,1)}{0.3}{0.08}{blue}
				\dotuparam{(0.75,0)}{0.3}{0.08}{red}
			}
			&
			\scaledstrings{
				\tikzset{every path/.style={line width=1 mm}}
				\dotdparam{(0.75,1)}{0.3}{0.08}{blue}
				\dotuparam{(0.75,0)}{0.3}{0.08}{blue}
			}
		\end{pmatrix}
		\endgroup
	}
	=
	{
		\begingroup
		\setlength\arraycolsep{0pt}
		\def\arraystretch{1.3}
		2
		\begin{pmatrix}
			\hspace{-0.2cm}
			\scaledstrings{
				\tikzset{every path/.style={line width=1 mm}}
				\dotdparam{(0.75,1)}{0.3}{0.08}{red}
				\dotuparam{(0.75,0)}{0.3}{0.08}{red}
			}
			\\
			&
			\scaledstrings{
				\tikzset{every path/.style={line width=1 mm}}
				\dotdparam{(0.75,1)}{0.3}{0.08}{blue}
				\dotuparam{(0.75,0)}{0.3}{0.08}{blue}
			}
		\end{pmatrix}
		\endgroup
	}
	.
\end{align}

\eqref{H=I L brown leg}
The summands of the right hand side are sent to
\begin{align}
	\label{eqn mapsto Ileg}
	2
	\begin{boxedstrings}
		\Ilongparam{(0,0)}{0.5}{1}{green}{green}{green}{green}{brown}
	\end{boxedstrings}
	&\mapsto
	{
	\begingroup
	\setlength\arraycolsep{0pt}
	\def\arraystretch{1.2}
	2
	\begin{pmatrix}
		&
		\scaledstrings{
			\trivdparam{(0,1)}{0.5}{0.5}{red}
			\idfigparam{(nd-1-3)}{0.5}{red}
			\capfigparam{(0,0)}{0.5}{0.3}{blue}
		}
		\\
		\hspace{-0.2cm}
		\scaledstrings{
			\trivdparam{(0.25,1)}{0.5}{0.5}{blue}
			\idfigparam{(nd-1-3)}{0.5}{blue}
			\capfigparam{(0,0)}{0.3}{0.3}{red}
		}		
	\end{pmatrix}
	\endgroup
	}
	\\
	\label{eqn I send L}
	\begin{boxedstrings}
		\Ilongparam{(0,0)}{0.5}{1}{green}{green}{brown}{green}{brown}
	\end{boxedstrings}
	&\mapsto
	{
		\begingroup
		\setlength\arraycolsep{0pt}
		\def\arraystretch{1.2}
		\begin{pmatrix}
			\hspace{-0.2cm}
			\scaledstrings{
				\trivuparam{(0.25,0.5)}{0.5}{0.5}{red}
				\idfigparam{(nd-1-1)}{-0.5}{red}
				\idfigparam{(0.75,0)}{-1}{blue}
			}
			&
			\scaledstrings{
				\Xfigparam{(0.5,1)}{0.25}{0.5}{red}{blue}
				\trivuparam{(nd-1-3)}{1}{0.5}{blue}
				\idfigparam{(nd-1-4)}{0.5}{red}
			}
			\\
			\hspace{-0.2cm}
			\scaledstrings{
				\Xfigparam{(0.25,1)}{0.5}{0.5}{blue}{red}
				\trivuparam{(nd-1-3)}{0.5}{0.5}{red}
				\idfigparam{(nd-1-4)}{0.5}{blue}
			}		
			&
			\scaledstrings{
				\trivuparam{(0.5,0.5)}{1}{0.5}{blue}
				\idfigparam{(nd-1-1)}{-0.5}{blue}
				\idfigparam{(0.75,0)}{-1}{red}
			}
		\end{pmatrix}
		\endgroup
	}
	\\
	-
	\begin{boxedstrings}
		\Ilongparam{(0,0)}{0.5}{1}{green}{green}{brown}{green}{brown}
		\draw[orange] (nd-1-8) -- (nd-1-9);
	\end{boxedstrings}
	&\mapsto
	{
		\begingroup
		\setlength\arraycolsep{0pt}
		\def\arraystretch{1.2}
		\begin{pmatrix}
			-
			\hspace{-0.2cm}
			\scaledstrings{
				\trivuparam{(0.25,0.5)}{0.5}{0.5}{red}
				\idfigparam{(nd-1-1)}{-0.5}{red}
				\idfigparam{(0.75,0)}{-1}{blue}
			}
			&
			\scaledstrings{
				\Xfigparam{(0.5,1)}{0.25}{0.5}{red}{blue}
				\trivuparam{(nd-1-3)}{1}{0.5}{blue}
				\idfigparam{(nd-1-4)}{0.5}{red}
			}
			\\
			\hspace{-0.2cm}
			\scaledstrings{
				\Xfigparam{(0.25,1)}{0.5}{0.5}{blue}{red}
				\trivuparam{(nd-1-3)}{0.5}{0.5}{red}
				\idfigparam{(nd-1-4)}{0.5}{blue}
			}		
			&
			-
			\hspace{-0.2cm}
			\scaledstrings{
				\trivuparam{(0.5,0.5)}{1}{0.5}{blue}
				\idfigparam{(nd-1-1)}{-0.5}{blue}
				\idfigparam{(0.75,0)}{-1}{red}
			}
		\end{pmatrix}
		\endgroup
	}
	\label{eqn I orange to L}
\end{align}
Rotating equation \eqref{eqn I send L} we get
\begin{equation}
	2
	\begin{boxedstrings}
		\Hlongparam{(0,0)}{0.5}{1}{green}{green}{brown}{green}{brown}
	\end{boxedstrings}
	\mapsto
	{
		\begingroup
		\setlength\arraycolsep{0pt}
		\def\arraystretch{1.2}
		2
		\begin{pmatrix}
			&
			\scaledstrings{
				\trivdparam{(0,1)}{0.5}{0.5}{red}
				\idfigparam{(nd-1-3)}{0.5}{red}
				\capfigparam{(0,0)}{0.5}{0.3}{blue}
			}
			\\
			&
			\scaledstrings{
				\Xfigparam{(0.5,1)}{0.25}{0.5}{red}{blue}
				\trivuparam{(nd-1-3)}{1}{0.5}{blue}
				\idfigparam{(nd-1-4)}{0.5}{red}
			}
			\\
			\hspace{-0.2cm}
			\scaledstrings{
				\Xfigparam{(0.25,1)}{0.5}{0.5}{blue}{red}
				\trivuparam{(nd-1-3)}{0.5}{0.5}{red}
				\idfigparam{(nd-1-4)}{0.5}{blue}
			}	
			\\
			\hspace{-0.2cm}
			\scaledstrings{
				\trivdparam{(0.25,1)}{0.5}{0.5}{blue}
				\idfigparam{(nd-1-3)}{0.5}{blue}
				\capfigparam{(0,0)}{0.3}{0.3}{red}
			}
			&
		\end{pmatrix}
		\endgroup
	}
	,
\end{equation}
which is exactly the sum of equations \eqref{eqn mapsto Ileg}-\eqref{eqn I orange to L}.

\eqref{H=I middle brown}
The summands of the right hand side are sent to
\begin{align}
	\label{cupcap mapsto}
	\begin{boxedstrings}
		\capfigparam{(0,0)}{0.5}{0.3}{green}
		\cupfigparam{(0,1)}{0.5}{0.3}{green}
	\end{boxedstrings}
	&\mapsto
	{
		\begingroup
		\setlength\arraycolsep{2pt}
		\def\arraystretch{1.2}
		\begin{pmatrix}
			\scaledfitstrings{
				\capfigparam{(0,0)}{0.5}{0.3}{red}
				\cupfigparam{(0,1)}{0.5}{0.3}{red}
			}
			&
			\scaledfitstrings{
				\capfigparam{(0,0)}{0.5}{0.3}{blue}
				\cupfigparam{(0,1)}{0.5}{0.3}{red}
			}
			\\
			\scaledfitstrings{
				\capfigparam{(0,0)}{0.5}{0.3}{red}
				\cupfigparam{(0,1)}{0.5}{0.3}{blue}
			}
			&
			\scaledfitstrings{
				\capfigparam{(0,0)}{0.5}{0.3}{blue}
				\cupfigparam{(0,1)}{0.5}{0.3}{blue}
			}
		\end{pmatrix}
		\endgroup
	}
	\\
	-
	\begin{boxedstrings}
		\capfigparam{(0,0)}{0.5}{0.3}{green}
		\cupfigparam{(0,1)}{0.5}{0.3}{green}
		\draw[orange] (nd-1-3) -- (nd-2-3);
	\end{boxedstrings}
	&\mapsto
	{
		\begingroup
		\setlength\arraycolsep{2pt}
		\def\arraystretch{1.2}
		\begin{pmatrix}
			-\hspace{-0.15cm}
			\scaledfitstrings{
				\capfigparam{(0,0)}{0.5}{0.3}{red}
				\cupfigparam{(0,1)}{0.5}{0.3}{red}
			}
			&
			\scaledfitstrings{
				\capfigparam{(0,0)}{0.5}{0.3}{blue}
				\cupfigparam{(0,1)}{0.5}{0.3}{red}
			}
			\\
			\scaledfitstrings{
				\capfigparam{(0,0)}{0.5}{0.3}{red}
				\cupfigparam{(0,1)}{0.5}{0.3}{blue}
			}
			&
			-\hspace{-0.15cm}
			\scaledfitstrings{
				\capfigparam{(0,0)}{0.5}{0.3}{blue}
				\cupfigparam{(0,1)}{0.5}{0.3}{blue}
			}
		\end{pmatrix}
		\endgroup
	}
	\\ \label{I middle brown mapsto}
	\begin{boxedstrings}
		\Ilongparam{(0,0)}{0.5}{1}{green}{green}{brown}{green}{green}
	\end{boxedstrings}
	&\mapsto
	{
		\begingroup
		\setlength\arraycolsep{2pt}
		\def\arraystretch{1.2}
		\begin{pmatrix}
			\scaledfitstrings{
				\idfig{(0,0)}{red}
				\idfig{(0.25,0)}{blue}
			}
			&
			\scaledfitstrings{
				\Xfigparam{(0,0)}{0.25}{0.5}{blue}{red}
				\idfigparam{(nd-1-1)}{-0.25}{blue}
				\idfigparam{(nd-1-2)}{-0.25}{red}
				\idfigparam{(nd-1-4)}{0.25}{blue}
				\idfigparam{(nd-1-3)}{0.25}{red}
			}
			\\
			\scaledfitstrings{
			\Xfigparam{(0,0)}{0.25}{0.5}{red}{blue}
			\idfigparam{(nd-1-1)}{-0.25}{red}
			\idfigparam{(nd-1-2)}{-0.25}{blue}
			\idfigparam{(nd-1-4)}{0.25}{red}
			\idfigparam{(nd-1-3)}{0.25}{blue}
			}
			&
			\scaledfitstrings{
				\idfig{(0,0)}{blue}
				\idfig{(0.25,0)}{red}
			}
		\end{pmatrix}
		\endgroup
	}
	\\
	-
	\begin{boxedstrings}
		\Ilongparam{(0,0)}{0.5}{1}{green}{green}{brown}{green}{green}
		\draw[orange] (nd-1-8) -- (nd-1-9);
	\end{boxedstrings}
	&\mapsto
	{
		\begingroup
		\setlength\arraycolsep{2pt}
		\def\arraystretch{1.2}
		\begin{pmatrix}
			-\hspace{-0.15cm}
			\scaledfitstrings{
				\idfig{(0,0)}{red}
				\idfig{(0.25,0)}{blue}
			}
			&
			\scaledfitstrings{
				\Xfigparam{(0,0)}{0.25}{0.5}{blue}{red}
				\idfigparam{(nd-1-1)}{-0.25}{blue}
				\idfigparam{(nd-1-2)}{-0.25}{red}
				\idfigparam{(nd-1-4)}{0.25}{blue}
				\idfigparam{(nd-1-3)}{0.25}{red}
			}
			\\
			\scaledfitstrings{
				\Xfigparam{(0,0)}{0.25}{0.5}{red}{blue}
				\idfigparam{(nd-1-1)}{-0.25}{red}
				\idfigparam{(nd-1-2)}{-0.25}{blue}
				\idfigparam{(nd-1-4)}{0.25}{red}
				\idfigparam{(nd-1-3)}{0.25}{blue}
			}
			&
			-\hspace{-0.15cm}
			\scaledfitstrings{
				\idfig{(0,0)}{blue}
				\idfig{(0.25,0)}{red}
			}
		\end{pmatrix}
		\endgroup
	}
	\label{I mid brown orange mapsto}
\end{align}

Rotating equation \eqref{I middle brown mapsto}, we get
\begin{equation}
	2	
	\begin{boxedstrings}
		\Hlongparam{(0,0)}{0.5}{1}{green}{green}{brown}{green}{green}
	\end{boxedstrings}
	\mapsto	
	2
	{
		\begingroup 
		\setlength\arraycolsep{0pt}
	\begin{pmatrix}
	&&&
	\scaledstrings{
		\capfigparam{(0,0)}{0.5}{0.3}{blue}
		\cupfigparam{(0,1)}{0.5}{0.3}{red}
		}	
		\\
	&&
		\scaledstrings{
			\Xfigparam{(0,0)}{0.25}{0.5}{red}{blue}
		\idfigparam{(nd-1-1)}{-0.25}{red}
		\idfigparam{(nd-1-2)}{-0.25}{blue}
		\idfigparam{(nd-1-4)}{0.25}{red}
		\idfigparam{(nd-1-3)}{0.25}{blue}
		}
		\\
	&
		\scaledstrings{
		\Xfigparam{(0,0)}{0.25}{0.5}{blue}{red}
		\idfigparam{(nd-1-1)}{-0.25}{blue}
		\idfigparam{(nd-1-2)}{-0.25}{red}
		\idfigparam{(nd-1-4)}{0.25}{blue}
		\idfigparam{(nd-1-3)}{0.25}{red}
		}
		\\
		\scaledstrings{
			\capfigparam{(0,0)}{0.5}{0.3}{red}
			\cupfigparam{(0,1)}{0.5}{0.3}{blue}
		}
	\end{pmatrix}
	\endgroup
	}
	,
\end{equation}
which is exactly the sum of equations \eqref{cupcap mapsto}-\eqref{I mid brown orange mapsto}.

Since the objects $\onecal,\Xobj,\Yobj,\Zobj,$ and $\XZobj$ are in the image of $\Fbold$, by Corollary \ref{coro Htau is Krull schmidt},  $\Fbold$ is essentially surjective.
\end{proof}

We will now prove that the functor $\Fbold$ is an equivalence. This will require the use of Theorem \ref{diagram reduction},
as well as Theorem \ref{thm idempotent decompositions}.

\begin{prop}\label{prop F is fully faithful Hom to 1}
	The functor $\Fbold$ induces an isomorphism on
	$\Hom\left(\onecalbar,\onecalbar\right)$,
	$\Hom\left(\Xobjbar,\onecalbar\right)$, $\Hom\left(\Yobjbar,\onecalbar\right)$, $\Hom\left(\Zobjbar,\onecalbar\right)$, and $\Hom\left(\XZobjbar,\onecalbar\right)$.
	In particular, these are all free $\Rtau$-modules.
\end{prop}
\begin{proof}
	Recall that $\Ical=\{\onecal,\Xobj,\Yobj,\Zobj,\XZobj\}$ and $\Icalbar=\{\onecalbar,\Xobjbar,\Yobjbar,\Zobjbar,\XZobjbar\}$.
	Let $A\in\Ical$ and $\widehat{A}\in\Icalbar$.
	Proposition \ref{prop HomXto1 in Hcaltau} shows that $\Hom(A,\onecal)$ is a free $\Rtau$-module and gives an $\Rtau$-basis of it.
	In Theorem \ref{diagram reduction} we will give a generating set of $\Hom(\widehat{A},\onecalbar)$.
	We have that $\Fbold$ sends
	\begin{align}
		\begin{boxedstrings}
			\path (0,0) -- (0.3,1);
		\end{boxedstrings}
		&\mapsto
		\begin{boxedstrings}
			\path (0,0) -- (0.3,1);
		\end{boxedstrings},
		\\
		\begin{boxedstrings}
			\dotuparam{(0,0)}{0.6}{0.06}{orange}
			\path (nd-1-1) ++ (0,1);
		\end{boxedstrings}
		&\mapsto
		\begin{boxedstrings}
			\polybox{(0,0.5)}{{$\alpha_s \! -\! \alpha_t$}}
			\path (0,0) ++ (0,1);
		\end{boxedstrings},
		\\
		\begin{boxedstrings}
			\dotuparam{(0,0)}{0.6}{0.06}{green}
			\path (nd-1-1) ++ (0,1);
		\end{boxedstrings}
		&\mapsto
		\begin{pmatrix}
			\Dotup{red}
			&
			\Dotup{blue}
		\end{pmatrix}
		,
		\\
		\begin{boxedstrings}
			\trivdparamcol{(0,0)}{0.5}{0.7}{orange}{green}{green}
			\dotfig{(nd-1-1)}{orange}
			\dotfig{(nd-1-2)}{green}
			\path (nd-1-3) ++ (0,1);
		\end{boxedstrings}
		&\mapsto
		(\alpha_s-\alpha_t)
		\begin{pmatrix}
			\Dotup{red}
			&
			\Dotup{blue}
		\end{pmatrix}
		,
		\\
		\begin{boxedstrings}
			\dotuparam{(0,0)}{0.6}{0.06}{brown}
			\path (nd-1-1) ++ (0,1);
		\end{boxedstrings}
		&\mapsto
		\begin{boxedstrings}
			\dotuparam{(0,0)}{0.6}{0.06}{red}
			\dotuparam{(0.25,0)}{0.6}{0.06}{blue}
			\path (nd-1-1) ++ (0,1);
		\end{boxedstrings},
		\\
		\begin{boxedstrings}
			\dotuparam{(0,0)}{0.6}{0.06}{orange}
			\dotuparam{(0.25,0)}{0.6}{0.06}{brown}
			\path (nd-1-1) ++ (0,1);
		\end{boxedstrings}
		&\mapsto
		(\alpha_s-\alpha_t)
		\begin{boxedstrings}
			\dotuparam{(0,0)}{0.6}{0.06}{red}
			\dotuparam{(0.25,0)}{0.6}{0.06}{blue}
			\path (nd-1-1) ++ (0,1);
		\end{boxedstrings}.
	\end{align}
	That is, $\Fbold$ sends the generators in Theorem \ref{diagram reduction} to the basis elements in Proposition \ref{prop HomXto1 in Hcaltau}.
	Hence, $\Fbold$ induces an isomorphism on $\Hom(\widehat{A},\onecalbar)$.
\end{proof}

\begin{theorem}\label{thm F is an equivalence}
	The functor $\Fbold$ is an equivalence of categories.
\end{theorem}
\begin{proof} 
	We have to show that $\Fbold$ is full and faithful (we already have that $\Fbold$ is essentially surjective).
	By Theorem \ref{thm idempotent decompositions}, every object of $\Dcal$ is isomorphic to a direct sum of grading shifts of objects in $\Icalbar$.
	By Proposition \ref{prop F is fully faithful Hom to 1}, $\Fbold$ is bijective on $\Hom(\widehat{A},\onecalbar)$ 
	for all $A\in\Icalbar$.
	By Proposition \ref{prop criterion for equivalence},
	$\Fbold$ is an equivalence.
\end{proof}

	\subsection{Polynomial forcing relations}\label{subsection poly forcing 1}
	
	In this section we will show the basic polynomial forcing relations in our diagrammatic category.
	Corollary \ref{polynomial forcing under orange}
	will show that polynomial boxes slide across orange strands, so we will only show polynomial forcing relations for the green and brown strands. 
	The proof of Corollary \ref{polynomial forcing under orange} does not rely on
	the relations deduced here, so we may use it going forward.
	We will show the general polynomial forcing and needle relations 
	in Appendix \ref{subsection poly forcing 2}.

	Recall that $\Rtau=\kk[\alpha_s+\alpha_t, \alpha_s\alpha_t]$. 
	By \eqref{barbell rels}, the generators of $\Rtau$ correspond to the green and brown barbells.
	We also have the following algebraic dependence of the brown barbell in terms of the green and orange ones:
	\begin{equation}\label{polynomial dependence barbells}
		4
		\begin{boxedstrings}
			\vbarb{(0.75,0.7)}{brown}
			\path (0.75,0) -- (0.75,1);
		\end{boxedstrings}
		=
		\begin{boxedstrings}
			\polybox{(0,0.5)}{$4\alpha_s\alpha_t$}
			\path (0,0) ++ (0,1);
		\end{boxedstrings}
		=
		\begin{boxedstrings}
			\polybox{(0,0.5)}{$(\alpha_s+\alpha_t)^2$}
			\path (0,0) ++ (0,1);
		\end{boxedstrings}
		-
		\begin{boxedstrings}
			\polybox{(0,0.5)}{$(\alpha_s-\alpha_t)^2$}
			\path (0,0) ++ (0,1);
		\end{boxedstrings}
		=
		\begin{boxedstrings}
			\vbarb{(0.65,0.7)}{green}
			\vbarb{(0.85,0.7)}{green}
			\path (0.75,0) -- (0.75,1);
		\end{boxedstrings}
		-
		\begin{boxedstrings}
			\vbarb{(0.75,0.7)}{orange}
			\path (0.75,0) -- (0.75,1);
		\end{boxedstrings}.
	\end{equation}
	Altogether, we have the following proposition.
	\begin{prop}\label{prop polynomials are barbells}
		Any polynomial box
		\hspace{-0.4cm}
		\begin{fitunboxedstrings}
			\polybox{(0,0.5)}{$f$}
		\end{fitunboxedstrings}
		\hspace{-0.1cm},
		 for $f\in \Rtau$,
		 may be written as a polynomial
		on the green and brown barbells.
		Alternatively, 
		\hspace{-0.4cm}
		\begin{fitunboxedstrings}
			\polybox{(0,0.5)}{$f$}
		\end{fitunboxedstrings}
		\hspace{-0.1cm}
		may also be written as a polynomial on
		the green and orange barbells.
	\end{prop}
	
	The relation \eqref{poly forcing green 1} lets us `force' green barbels and orange dots across a green strand.
	By Proposition \ref{prop polynomials are barbells} (and Corollary \ref{polynomial forcing under orange}), we may `force' an entire polynomial across a green strand
	by `forcing' one green barbell or orange dot at a time.
	We will now show the rest of the forcing relations for green barbells,
	brown barbells, and orange dots.	
	
	\begin{prop}[Polynomial forcing]\label{polynomial forcing all}
		The following relations follow from the polynomial forcing relation
		\eqref{poly forcing green 1}
		and the bigon relations
		\eqref{bigon brown} and \eqref{bigon brown orange}.
		{

			\begin{equation}\label{poly forcing green 2}
				2 \left( \ 
				\begin{boxedstrings}
					\vbarb{(0.75,0.7)}{brown}
					\idfig{(1,1)}{green}
				\end{boxedstrings}
				+
				\begin{boxedstrings}
					\idfig{(1,1)}{green}
					\vbarb{(1.25,0.7)}{brown}
				\end{boxedstrings} \ 
				\right)
				=
				\begin{boxedstrings}
					\vbarb{(0.7,0.7)}{green}
					\dotd{(1,1)}{green}
					\dotu{(1,0)}{green}
				\end{boxedstrings}
				+
				\begin{boxedstrings}
					\idfigparam{(0,1)}{0.3}{green}
					\idfigparam{(nd-1-2)}{0.4}{orange}
					\idfigparam{(nd-2-2)}{0.3}{green}
					\vbarb{(-0.3,0.7)}{green}
				\end{boxedstrings}
				-
				\begin{boxedstrings}
					\idfigparam{(1,1)}{0.25}{green}
					\dotdparam{(nd-1-2)}{0.2}{0.06}{orange}
					\dotuparam{(1,0)}{0.25}{0.06}{green}
				\end{boxedstrings}
				-
				\begin{boxedstrings}
					\idfigparam{(1,0)}{-0.25}{green}
					\dotuparam{(nd-1-2)}{0.2}{0.06}{orange}
					\dotdparam{(1,1)}{0.25}{0.06}{green}
				\end{boxedstrings}
			\end{equation}

			\begin{equation}\label{poly forcing brown 1}
				\begin{boxedstrings}
					\vbarb{(0.75,0.7)}{green}
					\idfig{(1,1)}{brown}
				\end{boxedstrings}
				+
				\begin{boxedstrings}
					\idfig{(1,1)}{brown}
					\vbarb{(1.25,0.7)}{green}
				\end{boxedstrings}
				=
				2
				\begin{boxedstrings}
					\idfigparam{(0,1)}{0.3}{brown}
					\idfigparam{(nd-1-2)}{0.4}{green}
					\idfigparam{(nd-2-2)}{0.3}{brown}
				\end{boxedstrings}
			\end{equation}

			\begin{equation}\label{poly forcing brown 2}
				\begin{boxedstrings}
					\Xfig{(0,0)}{brown}{orange}
					\idfigparam{(nd-1-1)}{-0.25}{brown}
					\idfigparam{(nd-1-3)}{0.25}{orange}
					\idfigparam{(nd-1-4)}{0.25}{brown}
					\dotfig{(nd-1-2)}{orange}
				\end{boxedstrings}
				+
				\begin{boxedstrings}
					\idfig{(1,1)}{brown}
					\dotu{(0.75,0)}{orange}
				\end{boxedstrings}
				=
				-
				2
				\begin{boxedstrings}
					\idfigparam{(0,1)}{0.3}{brown}
					\idfigparam{(nd-1-2)}{0.4}{green}
					\idfigparam{(nd-2-2)}{0.3}{brown}
					\draw[orange] (nd-2-3) -- ++(-0.25,0) -- ++(0,-0.5);
				\end{boxedstrings}
			\end{equation}

			\begin{equation}\label{poly forcing brown 3}
				\begin{boxedstrings}
					\vbarb{(0.75,0.7)}{brown}
					\idfig{(1,1)}{brown}
				\end{boxedstrings}
				+
				\begin{boxedstrings}
					\idfigparam{(0,1)}{0.3}{brown}
					\idfigparam{(nd-1-2)}{0.4}{green}
					\idfigparam{(nd-2-2)}{0.3}{brown}
					\vbarb{(0.25,0.7)}{green}
				\end{boxedstrings}
				+
				\begin{boxedstrings}
					\idfigparam{(0,1)}{0.3}{brown}
					\idfigparam{(nd-1-2)}{0.4}{green}
					\idfigparam{(nd-2-2)}{0.3}{brown}
					\dotrparam{(nd-2-3)}{0.25}{0.06}{orange}
				\end{boxedstrings}
				=
				4
				\begin{boxedstrings}
					\dotd{(1,1)}{brown}
					\dotu{(1,0)}{brown}
				\end{boxedstrings}
				+
				\begin{boxedstrings}
					\vbarb{(1.25,0.7)}{brown}
					\idfig{(1,1)}{brown}
				\end{boxedstrings}
			\end{equation}

		}
	\end{prop}
	\begin{proof}
		We will show \eqref{poly forcing green 2} and \eqref{poly forcing brown 1}. The others are similar and will be left to the reader as an exercise.
		
		\eqref{poly forcing green 2}
		By \eqref{polynomial dependence barbells} we have
		\begin{equation}
			4
			\begin{boxedstrings}
				\vbarb{(0.75,0.7)}{brown}
				\idfig{(1,1)}{green}
			\end{boxedstrings}
			=
			\begin{boxedstrings}
				\vbarb{(0.65,0.7)}{green}
				\vbarb{(0.85,0.7)}{green}
				\idfig{(1,1)}{green}
			\end{boxedstrings}
			-
			\begin{boxedstrings}
				\vbarb{(0.75,0.7)}{orange}
				\idfig{(1,1)}{green}
			\end{boxedstrings},
		\end{equation}
		and
		\begin{equation}
			4
			\begin{boxedstrings}
				\vbarb{(0.75,0.7)}{brown}
				\idfig{(0.5,1)}{green}
			\end{boxedstrings}
			=
			\begin{boxedstrings}
				\vbarb{(0.65,0.7)}{green}
				\vbarb{(0.85,0.7)}{green}
				\idfig{(0.5,1)}{green}
			\end{boxedstrings}
			-
			\begin{boxedstrings}
				\vbarb{(0.75,0.7)}{orange}
				\idfig{(0.5,1)}{green}
			\end{boxedstrings}.
		\end{equation}
		
		Using \eqref{poly forcing green 1} (and orange landing slide relations) we get
		\excludeDiagrams{1}{
			\begin{equation}
				\begin{boxedstrings}
					\vbarb{(0.65,0.7)}{green}
					\vbarb{(0.85,0.7)}{green}
					\idfig{(1,1)}{green}
				\end{boxedstrings}
				=
				\begin{boxedstrings}
					\dotd{(1,1)}{green}
					\dotu{(1,0)}{green}
					\vbarb{(0.75,0.7)}{green}
				\end{boxedstrings}
				+
				\begin{boxedstrings}
					\idfigparam{(0,1)}{0.3}{green}
					\idfigparam{(nd-1-2)}{0.4}{orange}
					\idfigparam{(nd-2-2)}{0.3}{green}
					\vbarb{(-0.3,0.7)}{green}
				\end{boxedstrings}
				-
				\begin{boxedstrings}
					\idfig{(1,1)}{green}
					\dotr{(nd-1-3)}{orange}
					\vbarb{(0.75,0.7)}{green}
				\end{boxedstrings},
			\end{equation}
			\begin{equation}
				\begin{boxedstrings}
					\vbarb{(0.65,0.7)}{green}
					\vbarb{(0.85,0.7)}{green}
					\idfig{(0.5,1)}{green}
				\end{boxedstrings}
				=
				\begin{boxedstrings}
					\dotd{(1,1)}{green}
					\dotu{(1,0)}{green}
					\vbarb{(0.75,0.7)}{green}
				\end{boxedstrings}
				+
				\begin{boxedstrings}
					\idfigparam{(0,1)}{0.3}{green}
					\idfigparam{(nd-1-2)}{0.4}{orange}
					\idfigparam{(nd-2-2)}{0.3}{green}
					\vbarb{(-0.3,0.7)}{green}
				\end{boxedstrings}
				-
				\begin{boxedstrings}
					\idfig{(1,1)}{green}
					\dotl{(nd-1-3)}{orange}
					\vbarb{(1.25,0.7)}{green}
				\end{boxedstrings},
			\end{equation}
			\begin{equation}
				-
				\begin{boxedstrings}
					\vbarb{(0.75,0.7)}{orange}
					\idfig{(1,1)}{green}
				\end{boxedstrings}
				=
				-
				\begin{boxedstrings}
					\idfig{(1,1)}{green}
					\dotl{(1,0.25)}{orange}
					\dotl{(1,0.75)}{orange}
				\end{boxedstrings}
				=
				-
				\begin{boxedstrings}
					\idfigparam{(1,1)}{0.25}{green}
					\dotdparam{(nd-1-2)}{0.2}{0.06}{orange}
					\dotuparam{(1,0)}{0.25}{0.06}{green}
				\end{boxedstrings}
				-
				\begin{boxedstrings}
					\idfigparam{(1,0)}{-0.25}{green}
					\dotuparam{(nd-1-2)}{0.2}{0.06}{orange}
					\dotdparam{(1,1)}{0.25}{0.06}{green}
				\end{boxedstrings}
				+
				\begin{boxedstrings}
					\idfig{(1,1)}{green}
					\dotl{(nd-1-3)}{orange}
					\vbarb{(1.25,0.7)}{green}
				\end{boxedstrings},
			\end{equation}
			\begin{equation}
				-
				\begin{boxedstrings}
					\vbarb{(0.75,0.7)}{orange}
					\idfig{(0.5,1)}{green}
				\end{boxedstrings}
				=
				-
				\begin{boxedstrings}
					\idfig{(1,1)}{green}
					\dotr{(1,0.25)}{orange}
					\dotr{(1,0.75)}{orange}
				\end{boxedstrings}
				=
				-
				\begin{boxedstrings}
					\idfigparam{(1,1)}{0.25}{green}
					\dotdparam{(nd-1-2)}{0.2}{0.06}{orange}
					\dotuparam{(1,0)}{0.25}{0.06}{green}
				\end{boxedstrings}
				-
				\begin{boxedstrings}
					\idfigparam{(1,0)}{-0.25}{green}
					\dotuparam{(nd-1-2)}{0.2}{0.06}{orange}
					\dotdparam{(1,1)}{0.25}{0.06}{green}
				\end{boxedstrings}
				+
				\begin{boxedstrings}
					\idfig{(1,1)}{green}
					\dotr{(nd-1-3)}{orange}
					\vbarb{(0.75,0.7)}{green}
				\end{boxedstrings}.
			\end{equation}
			Combining the previous equations we get
			\begin{equation}
				4 \left( \ 
				\begin{boxedstrings}
					\vbarb{(0.75,0.7)}{brown}
					\idfig{(1,1)}{green}
				\end{boxedstrings}
				+
				\begin{boxedstrings}
					\idfig{(1,1)}{green}
					\vbarb{(1.25,0.7)}{brown}
				\end{boxedstrings} \ 
				\right)
				=
				2
				\begin{boxedstrings}
					\vbarb{(0.7,0.7)}{green}
					\dotd{(1,1)}{green}
					\dotu{(1,0)}{green}
				\end{boxedstrings}
				+2
				\begin{boxedstrings}
					\idfigparam{(0,1)}{0.3}{green}
					\idfigparam{(nd-1-2)}{0.4}{orange}
					\idfigparam{(nd-2-2)}{0.3}{green}
					\vbarb{(-0.3,0.7)}{green}
				\end{boxedstrings}
				-2
				\begin{boxedstrings}
					\idfigparam{(1,1)}{0.25}{green}
					\dotdparam{(nd-1-2)}{0.2}{0.06}{orange}
					\dotuparam{(1,0)}{0.25}{0.06}{green}
				\end{boxedstrings}
				-2
				\begin{boxedstrings}
					\idfigparam{(1,0)}{-0.25}{green}
					\dotuparam{(nd-1-2)}{0.2}{0.06}{orange}
					\dotdparam{(1,1)}{0.25}{0.06}{green}
				\end{boxedstrings}
			\end{equation}
		}
		
		\eqref{poly forcing brown 1}
		By \eqref{bigon brown},
		we have
		\begin{equation}
			2
			\begin{boxedstrings}
				\vbarb{(0.75,0.7)}{green}
				\idfig{(1,1)}{brown}
			\end{boxedstrings}
			=
			\begin{boxedstrings}
				\node at (0,1) {};
				\node (a) at (-0.2,0.5) {};
				\def\coltop{green}
				\def\colbot{brown}
				\def\colside{orange}
				\capfigparam{(a)}{0.4}{0.2}{\coltop}
				\cupfigparam{(a)}{0.4}{0.2}{\coltop}
				\idfigparam {(0,0.3)}{0.3}{\colbot}
				\idfigparam {(0,0.7)}{-0.3}{\colbot}
				\vbarb{(-0.5,0.7)}{green}
			\end{boxedstrings}
		\end{equation}
		and 
		\begin{equation}
			2
			\begin{boxedstrings}
				\vbarb{(1.25,0.7)}{green}
				\idfig{(1,1)}{brown}
			\end{boxedstrings}
			=
			\begin{boxedstrings}
				\node at (0,1) {};
				\node (a) at (-0.2,0.5) {};
				\def\coltop{green}
				\def\colbot{brown}
				\def\colside{orange}
				\capfigparam{(a)}{0.4}{0.2}{\coltop}
				\cupfigparam{(a)}{0.4}{0.2}{\coltop}
				\idfigparam {(0,0.3)}{0.3}{\colbot}
				\idfigparam {(0,0.7)}{-0.3}{\colbot}
				\vbarb{(0.5,0.7)}{green}
			\end{boxedstrings}.
		\end{equation}
		By \eqref{poly forcing green 1} (and orange landing slide relations) we have
		\begin{align}
			\begin{boxedstrings}
				\node at (0,1) {};
				\node (a) at (-0.2,0.5) {};
				\def\coltop{green}
				\def\colbot{brown}
				\def\colside{orange}
				\capfigparam{(a)}{0.4}{0.2}{\coltop}
				\cupfigparam{(a)}{0.4}{0.2}{\coltop}
				\idfigparam {(0,0.3)}{0.3}{\colbot}
				\idfigparam {(0,0.7)}{-0.3}{\colbot}
				\vbarb{(-0.35,0.7)}{green}
			\end{boxedstrings}
			&=
			\begin{boxedstrings}
				\idfigparam{(0,1)}{0.3}{brown}
				\idfigparam{(nd-1-2)}{0.4}{green}
				\idfigparam{(nd-2-2)}{0.3}{brown}
			\end{boxedstrings}
			+
			\begin{boxedstrings}
				\trivuparamcol{(1.25,1)}{0.5}{0.35}{brown}{green}{green}
				\trivdparamcol{(1,0.35)}{0.5}{0.35}{green}{green}{brown}
				\dotfig{(nd-1-2)}{green};
				\dotfig{(nd-2-1)}{green};
				\draw[green] (nd-1-3) -- (nd-2-2);
				\draw[orange] (1.125,0.725) -- (1.125,0.275);
			\end{boxedstrings}
			-
			\begin{boxedstrings}
				\trivuparamcol{(1.25,1)}{0.5}{0.35}{brown}{green}{green}
				\trivdparamcol{(1,0.35)}{0.5}{0.35}{green}{green}{brown}
				\draw[green] (nd-1-3) -- (nd-2-2);
				\draw[green] (nd-1-2) -- (nd-2-1);
				\dotr{(1,0.5)}{orange}
			\end{boxedstrings}
			\\
			&=
			2
			\begin{boxedstrings}
				\idfigparam{(0,1)}{0.3}{brown}
				\idfigparam{(nd-1-2)}{0.4}{green}
				\idfigparam{(nd-2-2)}{0.3}{brown}
			\end{boxedstrings}
			+
			\begin{boxedstrings}
				\trivuparamcol{(1.25,1)}{0.5}{0.35}{brown}{green}{green}
				\trivdparamcol{(1,0.35)}{0.5}{0.35}{green}{green}{brown}
				\draw[green] (nd-1-3) -- (nd-2-2);
				\draw[green] (nd-1-2) -- (nd-2-1);
				\dotl{(1.5,0.5)}{orange}
			\end{boxedstrings}
			.
		\end{align}
		Similarly,
		\begin{equation}
			\begin{boxedstrings}
				\node at (0,1) {};
				\node (a) at (-0.2,0.5) {};
				\def\coltop{green}
				\def\colbot{brown}
				\def\colside{orange}
				\capfigparam{(a)}{0.4}{0.2}{\coltop}
				\cupfigparam{(a)}{0.4}{0.2}{\coltop}
				\idfigparam {(0,0.3)}{0.3}{\colbot}
				\idfigparam {(0,0.7)}{-0.3}{\colbot}
				\vbarb{(0.35,0.7)}{green}
			\end{boxedstrings}
			=
			2
			\begin{boxedstrings}
				\idfigparam{(0,1)}{0.3}{brown}
				\idfigparam{(nd-1-2)}{0.4}{green}
				\idfigparam{(nd-2-2)}{0.3}{brown}
			\end{boxedstrings}
			-
			\begin{boxedstrings}
				\trivuparamcol{(1.25,1)}{0.5}{0.35}{brown}{green}{green}
				\trivdparamcol{(1,0.35)}{0.5}{0.35}{green}{green}{brown}
				\draw[green] (nd-1-3) -- (nd-2-2);
				\draw[green] (nd-1-2) -- (nd-2-1);
				\dotl{(1.5,0.5)}{orange}
			\end{boxedstrings}.
		\end{equation}
		Combining the last two equations, we get
		\begin{equation}
			2\left(
			\begin{boxedstrings}
				\vbarb{(0.75,0.7)}{green}
				\idfig{(1,1)}{brown}
			\end{boxedstrings}
			+
			\begin{boxedstrings}
				\idfig{(1,1)}{brown}
				\vbarb{(1.25,0.7)}{green}
			\end{boxedstrings}
			\right)
			=
			4
			\begin{boxedstrings}
				\idfigparam{(0,1)}{0.3}{brown}
				\idfigparam{(nd-1-2)}{0.4}{green}
				\idfigparam{(nd-2-2)}{0.3}{brown}
			\end{boxedstrings}
			.
		\end{equation}
	\end{proof}

	\subsection{Idempotent decompositions in $\Hcaltau$}\label{subsection idempotent decomposition}
	
	The following proposition gives explicit idempotent decompositions for the objects $\Yobjbar\otimes\Yobjbar$, $\Yobjbar\otimes \Zobjbar$ and
	$\Zobjbar\otimes \Zobjbar$. 
	
	\begin{prop}[Idempotent decompositions]\label{prop idemp decomp}
		\label{thm idempotent decompositions}
		The following are decompositions into orthogonal idempotents of $\Yobjbar\otimes \Yobjbar$, $\Zobjbar\otimes \Zobjbar$, and $\Yobjbar\otimes\Zobjbar$, that factor through their direct summands.
		\excludeDiagrams{1}
		{
			\begin{equation}\label{Y2 decomp}
				\begin{boxedstrings}
					\idfigparam{(0,2)}{1}{green}
					\idfigparam{(0.5,2)}{1}{green}
				\end{boxedstrings}
				=
				\frac{1}{2}
				\begin{boxedstrings}
					\idfigparam{(0,1)}{0.25}{green}
					\Ifig{(nd-1-2)}{green}
					\idfigparam{(nd-2-3)}{0.25}{green}
					\idfigparam{(nd-2-2)}{-0.25}{green}
					\idfigparam{(nd-2-4)}{0.25}{green}
					\hbarbparam{(0.125,0.875)}{0.25}{0.04}{green}
				\end{boxedstrings}
				+\frac{1}{2}
				\begin{boxedstrings}
					\idfigparam{(0,1)}{0.25}{green}
					\Ifig{(nd-1-2)}{green}
					\idfigparam{(nd-2-3)}{0.25}{green}
					\idfigparam{(nd-2-2)}{-0.25}{green}
					\idfigparam{(nd-2-4)}{0.25}{green}
					\dotrparam{(nd-3-3)}{0.25}{0.04}{orange}
				\end{boxedstrings}
				+\frac{1}{2}
				\begin{boxedstrings}
					\idfigparam{(0,1)}{0.25}{green}
					\Ifigcol{(nd-1-2)}{green}{green}{brown}{green}{green}
					\idfigparam{(nd-2-3)}{0.25}{green}
					\idfigparam{(nd-2-2)}{-0.25}{green}
					\idfigparam{(nd-2-4)}{0.25}{green}
				\end{boxedstrings}
				+\frac{1}{2}
				\begin{boxedstrings}
					\idfigparam{(0,1)}{0.25}{green}
					\Ifigcol{(nd-1-2)}{green}{green}{brown}{green}{green}
					\idfigparam{(nd-2-3)}{0.25}{green}
					\idfigparam{(nd-2-2)}{-0.25}{green}
					\idfigparam{(nd-2-4)}{0.25}{green}
					\idfigparam{(0.125,0.75)}{0.5}{orange}
				\end{boxedstrings}
				,	
			\end{equation}

			\begin{equation}\label{Z2 decomp}
				\begin{boxedstrings}
					\idfigparam{(0,2)}{1}{brown}
					\idfigparam{(0.5,2)}{1}{brown}
				\end{boxedstrings}
				=
				\frac{1}{4}
				\begin{boxedstrings}
					\idfigparam{(0,1)}{0.25}{brown}
					\Ifig{(nd-1-2)}{brown}
					\idfigparam{(nd-2-3)}{0.25}{brown}
					\idfigparam{(nd-2-2)}{-0.25}{brown}
					\idfigparam{(nd-2-4)}{0.25}{brown}
					\hbarbparam{(0.125,0.875)}{0.25}{0.04}{brown}
				\end{boxedstrings}
				+\frac{1}{4}
				\begin{boxedstrings}
					\idfigparam{(0,1)}{0.25}{brown}
					\Ifig{(nd-1-2)}{brown}
					\idfigparam{(nd-2-3)}{0.25}{brown}
					\idfigparam{(nd-2-2)}{-0.25}{brown}
					\idfigparam{(nd-2-4)}{0.25}{brown}
					\hbarbparam{(0.125,0.125)}{0.25}{0.04}{brown}
				\end{boxedstrings}
				+\frac{1}{8}
				\begin{boxedstrings}
					\idfigparam{(0,1)}{0.25}{brown}
					\Ifig{(nd-1-2)}{brown}
					\idfigparam{(nd-2-3)}{0.25}{brown}
					\idfigparam{(nd-2-2)}{-0.25}{brown}
					\idfigparam{(nd-2-4)}{0.25}{brown}
					\hbarbparam{(0.125,0.125)}{0.25}{0.04}{green}
					\hbarbparam{(0.125,0.875)}{0.25}{0.04}{green}
				\end{boxedstrings}
				-\frac{1}{8}
				\begin{boxedstrings}
					\idfigparam{(0,1)}{0.25}{brown}
					\Ifig{(nd-1-2)}{brown}
					\idfigparam{(nd-2-3)}{0.25}{brown}
					\idfigparam{(nd-2-2)}{-0.25}{brown}
					\idfigparam{(nd-2-4)}{0.25}{brown}
					\idfigparam{(0.125,0.75)}{0.5}{orange}
					\dotuparam{(nd-6-1)}{0.125}{0.04}{orange}
					\dotdparam{(nd-6-2)}{0.125}{0.04}{orange}
				\end{boxedstrings}
				,		
			\end{equation}

			\begin{equation}\label{YZ decomp}
				\begin{split}
					\begin{boxedstrings}
						\idfigparam{(0,2)}{1}{green}
						\idfigparam{(0.5,2)}{1}{brown}
					\end{boxedstrings}
					= &
					\frac{1}{8}
					\begin{boxedstrings}
						\idfigparam{(0,1)}{0.25}{green}
						\Ifigcol{(nd-1-2)}{green}{brown}{brown}{green}{brown}
						\idfigparam{(nd-2-3)}{0.25}{green}
						\idfigparam{(nd-2-2)}{-0.25}{brown}
						\idfigparam{(nd-2-4)}{0.25}{brown}
						\hbarbparam{(0.125,0.875)}{0.25}{0.04}{green}
					\end{boxedstrings}
					+ \frac{1}{8}
					\begin{boxedstrings}
						\idfigparam{(0,1)}{0.25}{green}
						\Ifigcol{(nd-1-2)}{green}{brown}{brown}{green}{brown}
						\idfigparam{(nd-2-3)}{0.25}{green}
						\idfigparam{(nd-2-2)}{-0.25}{brown}
						\idfigparam{(nd-2-4)}{0.25}{brown}
						\idfigparam{(0.125,0.75)}{0.5}{orange}
						\hbarbparam{(0.125,0.875)}{0.25}{0.04}{green}
					\end{boxedstrings}
					+\frac{1}{8}
					\begin{boxedstrings}
						\idfigparam{(0,1)}{0.25}{green}
						\Ifigcol{(nd-1-2)}{green}{brown}{brown}{green}{brown}
						\idfigparam{(nd-2-3)}{0.25}{green}
						\idfigparam{(nd-2-2)}{-0.25}{brown}
						\idfigparam{(nd-2-4)}{0.25}{brown}
						\hbarbparam{(0.125,0.125)}{0.25}{0.04}{green}
					\end{boxedstrings}
					+\frac{1}{8}
					\begin{boxedstrings}
						\idfigparam{(0,1)}{0.25}{green}
						\Ifigcol{(nd-1-2)}{green}{brown}{brown}{green}{brown}
						\idfigparam{(nd-2-3)}{0.25}{green}
						\idfigparam{(nd-2-2)}{-0.25}{brown}
						\idfigparam{(nd-2-4)}{0.25}{brown}
						\idfigparam{(0.125,0.75)}{0.5}{orange}
						\hbarbparam{(0.125,0.125)}{0.25}{0.04}{green}
					\end{boxedstrings}
					\\
					+ & \frac{1}{8}
					\begin{boxedstrings}
						\idfigparam{(0,1)}{0.25}{green}
						\Ifigcol{(nd-1-2)}{green}{brown}{brown}{green}{brown}
						\idfigparam{(nd-2-3)}{0.25}{green}
						\idfigparam{(nd-2-2)}{-0.25}{brown}
						\idfigparam{(nd-2-4)}{0.25}{brown}
						\dotrparam{(nd-1-3)}{0.25}{0.04}{orange}
					\end{boxedstrings}
					+ \frac{1}{8}
					\begin{boxedstrings}
						\idfigparam{(0,1)}{0.25}{green}
						\Ifigcol{(nd-1-2)}{green}{brown}{brown}{green}{brown}
						\idfigparam{(nd-2-3)}{0.25}{green}
						\idfigparam{(nd-2-2)}{-0.25}{brown}
						\idfigparam{(nd-2-4)}{0.25}{brown}
						\idfigparam{(0.125,0.75)}{0.5}{orange}
						\dotrparam{(nd-1-3)}{0.25}{0.04}{orange}
					\end{boxedstrings}
					+ \frac{1}{8}
					\begin{boxedstrings}
						\idfigparam{(0,1)}{0.25}{green}
						\Ifigcol{(nd-1-2)}{green}{brown}{brown}{green}{brown}
						\idfigparam{(nd-2-3)}{0.25}{green}
						\idfigparam{(nd-2-2)}{-0.25}{brown}
						\idfigparam{(nd-2-4)}{0.25}{brown}
						\dotrparam{(nd-3-3)}{0.25}{0.04}{orange}
					\end{boxedstrings}
					+ \frac{1}{8}
					\begin{boxedstrings}
						\idfigparam{(0,1)}{0.25}{green}
						\Ifigcol{(nd-1-2)}{green}{brown}{brown}{green}{brown}
						\idfigparam{(nd-2-3)}{0.25}{green}
						\idfigparam{(nd-2-2)}{-0.25}{brown}
						\idfigparam{(nd-2-4)}{0.25}{brown}
						\idfigparam{(0.125,0.75)}{0.5}{orange}
						\dotrparam{(nd-3-3)}{0.25}{0.04}{orange}
					\end{boxedstrings}
				\end{split}
				.
			\end{equation}
			In this equation, the idempotents are the sum of the terms in each column.
		}
		
	\end{prop}
	\begin{proof}
		We will show \eqref{Y2 decomp} and \eqref{YZ decomp}. Showing \eqref{Z2 decomp} is similar and relies on using an H=I relation and one of the polynomial forcing relations in Proposition \ref{polynomial forcing all}.
		We will also leave to the reader to check that the summands are, in fact, orthogonal idempotents.
		All of them are bigons that can be reduced with an $H=I$ relation 
		(see Appendix \ref{appendix-relations} for the full list of relations).
		
		\eqref{Y2 decomp}
		Start by using \eqref{poly forcing green 1} and \eqref{H=I one color} on
		\begin{equation}\label{Y2 part1}
			\begin{boxedstrings}
				\idfigparam{(0,1)}{0.25}{green}
				\Ifig{(nd-1-2)}{green}
				\idfigparam{(nd-2-3)}{0.25}{green}
				\idfigparam{(nd-2-2)}{-0.25}{green}
				\idfigparam{(nd-2-4)}{0.25}{green}
				\hbarbparam{(0.125,0.875)}{0.25}{0.04}{green}
			\end{boxedstrings}
			+
			\begin{boxedstrings}
				\idfigparam{(0,1)}{0.25}{green}
				\Ifig{(nd-1-2)}{green}
				\idfigparam{(nd-2-3)}{0.25}{green}
				\idfigparam{(nd-2-2)}{-0.25}{green}
				\idfigparam{(nd-2-4)}{0.25}{green}
				\dotrparam{(nd-3-3)}{0.25}{0.04}{orange}
			\end{boxedstrings}
			=
			\begin{boxedstrings}
				\idfigparam{(0,1)}{0.25}{green}
				\Hfig{(nd-1-2)}{green}
				\idfigparam{(nd-2-3)}{0.25}{green}
				\idfigparam{(nd-2-2)}{-0.25}{green}
				\idfigparam{(nd-2-4)}{0.25}{green}
				\hbarbparam{(0.125,0.75)}{0.25}{0.04}{green}
			\end{boxedstrings}
			+
			\begin{boxedstrings}
				\idfigparam{(0,1)}{0.25}{green}
				\Hfig{(nd-1-2)}{green}
				\idfigparam{(nd-2-3)}{0.25}{green}
				\idfigparam{(nd-2-2)}{-0.25}{green}
				\idfigparam{(nd-2-4)}{0.25}{green}
				\dotdparam{(nd-2-7)}{0.25}{0.04}{orange}
			\end{boxedstrings}
			=
			\begin{boxedstrings}
				\idfigparam{(0,2)}{1}{green}
				\idfigparam{(0.5,2)}{1}{green}
			\end{boxedstrings}
			+
			\begin{boxedstrings}
				\def\coltempa{green}
				\def\coltempb{orange}
				\Hfigparamcol{(0,1)}{0.5}{1}{\coltempa}{\coltempa}{\coltempb}{\coltempa}{\coltempa}
			\end{boxedstrings}.
		\end{equation}
		
		Now use \eqref{H=I middle brown} 
		\begin{equation}\label{Y2 part2}
			\begin{split}
				\begin{boxedstrings}
					\idfigparam{(0,1)}{0.25}{green}
					\Ifigcol{(nd-1-2)}{green}{green}{brown}{green}{green}
					\idfigparam{(nd-2-3)}{0.25}{green}
					\idfigparam{(nd-2-2)}{-0.25}{green}
					\idfigparam{(nd-2-4)}{0.25}{green}
				\end{boxedstrings}
				+
				\begin{boxedstrings}
					\idfigparam{(0,1)}{0.25}{green}
					\Ifigcol{(nd-1-2)}{green}{green}{brown}{green}{green}
					\idfigparam{(nd-2-3)}{0.25}{green}
					\idfigparam{(nd-2-2)}{-0.25}{green}
					\idfigparam{(nd-2-4)}{0.25}{green}
					\idfigparam{(0.125,0.75)}{0.5}{orange}
				\end{boxedstrings}
				= &
				{
					\frac{1}{2}\left(
					\begin{boxedstrings}
						\idfigparam{(0,2)}{1}{green}
						\idfigparam{(0.5,2)}{1}{green}
					\end{boxedstrings}
					-
					\begin{boxedstrings}
						\def\coltempa{green}
						\def\coltempb{orange}
						\Hfigparamcol{(0,1)}{0.5}{1}{\coltempa}{\coltempa}{\coltempb}{\coltempa}{\coltempa}
					\end{boxedstrings}
					+
					\begin{boxedstrings}
						\def\coltempa{green}
						\def\coltempb{brown}
						\Hfigparamcol{(0,1)}{0.5}{1}{\coltempa}{\coltempa}{\coltempb}{\coltempa}{\coltempa}
					\end{boxedstrings}
					-
					\begin{boxedstrings}
						\def\coltempa{green}
						\def\coltempb{brown}
						\Hfigparamcol{(0,1)}{0.5}{0.5}{\coltempa}{\coltempa}{\coltempb}{\coltempa}{\coltempa}
						\idfigparam{(nd-1-1)}{-0.25}{\coltempa}
						\idfigparam{(nd-1-2)}{-0.25}{\coltempa}
						\idfigparam{(nd-1-3)}{ 0.25}{\coltempa}
						\idfigparam{(nd-1-4)}{ 0.25}{\coltempa}
						\draw[orange] (nd-1-1) -- (nd-1-2);
					\end{boxedstrings}
					\right)
				}
				\\
				+ &
				{
					\frac{1}{2}\left(
					\begin{boxedstrings}
						\idfigparam{(0,1)}{1}{green}
						\idfigparam{(0.5,1)}{1}{green}
						\draw[orange] (0,0.75) -- (-.125,0.75) --
						(-.125,0.25) -- (0,0.25);
					\end{boxedstrings}
					-
					\begin{boxedstrings}
						\def\coltempa{green}
						\def\coltempb{orange}
						\Hfigparamcol{(0,1)}{0.5}{1}{\coltempa}{\coltempa}{\coltempb}{\coltempa}{\coltempa}
						\draw[orange] (0,0.75) -- (-.125,0.75) --
						(-.125,0.25) -- (0,0.25);
					\end{boxedstrings}
					+
					\begin{boxedstrings}
						\def\coltempa{green}
						\def\coltempb{brown}
						\Hfigparamcol{(0,1)}{0.5}{1}{\coltempa}{\coltempa}{\coltempb}{\coltempa}{\coltempa}
						\draw[orange] (0,0.75) -- (-.125,0.75) --
						(-.125,0.25) -- (0,0.25);
					\end{boxedstrings}
					-
					\begin{boxedstrings}
						\def\coltempa{green}
						\def\coltempb{brown}
						\Hfigparamcol{(0,0.75)}{0.5}{0.5}{\coltempa}{\coltempa}{\coltempb}{\coltempa}{\coltempa}
						\idfigparam{(nd-1-1)}{-0.25}{\coltempa}
						\idfigparam{(nd-1-2)}{-0.25}{\coltempa}
						\idfigparam{(nd-1-3)}{ 0.25}{\coltempa}
						\idfigparam{(nd-1-4)}{ 0.25}{\coltempa}
						\draw[orange] (nd-1-1) -- (nd-1-2);
						\draw[orange] (0,0.75) -- (-.125,0.75) --
						(-.125,0.25) -- (0,0.25);
					\end{boxedstrings}
					\right)
				}
				\\
				= &
				{
					\frac{1}{2}\left(
					\begin{boxedstrings}
						\idfigparam{(0,2)}{1}{green}
						\idfigparam{(0.5,2)}{1}{green}
					\end{boxedstrings}
					-
					\begin{boxedstrings}
						\def\coltempa{green}
						\def\coltempb{orange}
						\Hfigparamcol{(0,1)}{0.5}{1}{\coltempa}{\coltempa}{\coltempb}{\coltempa}{\coltempa}
					\end{boxedstrings}
					+
					\begin{boxedstrings}
						\def\coltempa{green}
						\def\coltempb{brown}
						\Hfigparamcol{(0,1)}{0.5}{1}{\coltempa}{\coltempa}{\coltempb}{\coltempa}{\coltempa}
					\end{boxedstrings}
					-
					\begin{boxedstrings}
						\def\coltempa{green}
						\def\coltempb{brown}
						\Hfigparamcol{(0,1)}{0.5}{0.5}{\coltempa}{\coltempa}{\coltempb}{\coltempa}{\coltempa}
						\idfigparam{(nd-1-1)}{-0.25}{\coltempa}
						\idfigparam{(nd-1-2)}{-0.25}{\coltempa}
						\idfigparam{(nd-1-3)}{ 0.25}{\coltempa}
						\idfigparam{(nd-1-4)}{ 0.25}{\coltempa}
						\draw[orange] (nd-1-1) -- (nd-1-2);
					\end{boxedstrings}
					\right)
				}
				\\
				+ &
				{
					\frac{1}{2} \left(
					\begin{boxedstrings}
						\idfigparam{(0,1)}{1}{green}
						\idfigparam{(0.5,1)}{1}{green}
					\end{boxedstrings}
					-
					\begin{boxedstrings}
						\def\coltempa{green}
						\def\coltempb{orange}
						\Hfigparamcol{(0,1)}{0.5}{1}{\coltempa}{\coltempa}{\coltempb}{\coltempa}{\coltempa}
					\end{boxedstrings}
					-
					\begin{boxedstrings}
						\def\coltempa{green}
						\def\coltempb{brown}
						\Hfigparamcol{(0,1)}{0.5}{1}{\coltempa}{\coltempa}{\coltempb}{\coltempa}{\coltempa}
					\end{boxedstrings}
					+
					\begin{boxedstrings}
						\def\coltempa{green}
						\def\coltempb{brown}
						\Hfigparamcol{(0,0.75)}{0.5}{0.5}{\coltempa}{\coltempa}{\coltempb}{\coltempa}{\coltempa}
						\idfigparam{(nd-1-1)}{-0.25}{\coltempa}
						\idfigparam{(nd-1-2)}{-0.25}{\coltempa}
						\idfigparam{(nd-1-3)}{ 0.25}{\coltempa}
						\idfigparam{(nd-1-4)}{ 0.25}{\coltempa}
						\draw[orange] (nd-1-1) -- (nd-1-2);
					\end{boxedstrings}
					\right)
				}
				\\
				= &
				{
					\begin{boxedstrings}
						\idfigparam{(0,2)}{1}{green}
						\idfigparam{(0.5,2)}{1}{green}
					\end{boxedstrings}
					-
					\begin{boxedstrings}
						\def\coltempa{green}
						\def\coltempb{orange}
						\Hfigparamcol{(0,1)}{0.5}{1}{\coltempa}{\coltempa}{\coltempb}{\coltempa}{\coltempa}
					\end{boxedstrings}
				}.
			\end{split}
		\end{equation}
		
		Combining \eqref{Y2 part1} and \eqref{Y2 part2} we get
		\begin{equation}
			2
			\begin{boxedstrings}
				\idfigparam{(0,2)}{1}{green}
				\idfigparam{(0.5,2)}{1}{green}
			\end{boxedstrings}
			=
			\begin{boxedstrings}
				\idfigparam{(0,1)}{0.25}{green}
				\Ifig{(nd-1-2)}{green}
				\idfigparam{(nd-2-3)}{0.25}{green}
				\idfigparam{(nd-2-2)}{-0.25}{green}
				\idfigparam{(nd-2-4)}{0.25}{green}
				\hbarbparam{(0.125,0.875)}{0.25}{0.04}{green}
			\end{boxedstrings}
			+
			\begin{boxedstrings}
				\idfigparam{(0,1)}{0.25}{green}
				\Ifig{(nd-1-2)}{green}
				\idfigparam{(nd-2-3)}{0.25}{green}
				\idfigparam{(nd-2-2)}{-0.25}{green}
				\idfigparam{(nd-2-4)}{0.25}{green}
				\dotrparam{(nd-3-3)}{0.25}{0.04}{orange}
			\end{boxedstrings}
			+
			\begin{boxedstrings}
				\idfigparam{(0,1)}{0.25}{green}
				\Ifigcol{(nd-1-2)}{green}{green}{brown}{green}{green}
				\idfigparam{(nd-2-3)}{0.25}{green}
				\idfigparam{(nd-2-2)}{-0.25}{green}
				\idfigparam{(nd-2-4)}{0.25}{green}
			\end{boxedstrings}
			+
			\begin{boxedstrings}
				\idfigparam{(0,1)}{0.25}{green}
				\Ifigcol{(nd-1-2)}{green}{green}{brown}{green}{green}
				\idfigparam{(nd-2-3)}{0.25}{green}
				\idfigparam{(nd-2-2)}{-0.25}{green}
				\idfigparam{(nd-2-4)}{0.25}{green}
				\idfigparam{(0.125,0.75)}{0.5}{orange}
			\end{boxedstrings}
			.\end{equation}
		
		\eqref{YZ decomp}
		Using \eqref{H=I top brown} and \eqref{poly forcing green 1} we get 
		\begin{align} \label{eqn ZY decomp 1}
			\begin{boxedstrings}
				\idfigparam{(0,1)}{0.25}{green}
				\Ifigcol{(nd-1-2)}{green}{brown}{brown}{green}{brown}
				\idfigparam{(nd-2-3)}{0.25}{green}
				\idfigparam{(nd-2-2)}{-0.25}{brown}
				\idfigparam{(nd-2-4)}{0.25}{brown}
				\hbarbparam{(0.125,0.875)}{0.25}{0.04}{green}
			\end{boxedstrings}
			+ 	
			\begin{boxedstrings}
				\idfigparam{(0,1)}{0.25}{green}
				\Ifigcol{(nd-1-2)}{green}{brown}{brown}{green}{brown}
				\idfigparam{(nd-2-3)}{0.25}{green}
				\idfigparam{(nd-2-2)}{-0.25}{brown}
				\idfigparam{(nd-2-4)}{0.25}{brown}
				\dotrparam{(nd-3-3)}{0.25}{0.04}{orange}
			\end{boxedstrings}
			&=
			\begin{boxedstrings}
				\Hfigparamcol{(0,1)}{0.5}{1}{green}{brown}{green}{green}{brown}
				\hbarbparam{(0.125,0.75)}{0.25}{0.04}{green}
			\end{boxedstrings}
			+ 	
			\begin{boxedstrings}
				\Hfigparamcol{(0,1)}{0.5}{1}{green}{brown}{green}{green}{brown}
				\dotrparam{(0,0.25)}{0.25}{0.04}{orange}
			\end{boxedstrings}
			+
			\begin{boxedstrings}
				\Hfigparamcol{(0,1)}{0.5}{1}{green}{brown}{brown}{green}{brown}
				\hbarbparam{(0.125,0.75)}{0.25}{0.04}{green}
			\end{boxedstrings}
			+ 	
			\begin{boxedstrings}
				\Hfigparamcol{(0,1)}{0.5}{1}{green}{brown}{brown}{green}{brown}
				\dotrparam{(0,0.25)}{0.25}{0.04}{orange}
			\end{boxedstrings}
			\\
			&=
			\begin{boxedstrings}
				\idfigparam{(0,1)}{1}{green}
				\idfigparam{(0.5,1)}{1}{brown}
				\dotlparam{(nd-2-3)}{0.25}{0.04}{green}
			\end{boxedstrings}
			+
			\begin{boxedstrings}
				\idfigparam{(0,1)}{1}{green}
				\idfigparam{(0.5,1)}{1}{brown}
				\draw[orange] (nd-1-3) -- ++(0.25,0);
				\draw[green] (nd-2-3) -- ++(-0.25,0);
			\end{boxedstrings}
			+
			\begin{boxedstrings}
				\Hfigparamcol{(0,1)}{0.5}{1}{green}{brown}{brown}{green}{brown}
				\hbarbparam{(0.125,0.75)}{0.25}{0.04}{green}
			\end{boxedstrings}
			+ 	
			\begin{boxedstrings}
				\Hfigparamcol{(0,1)}{0.5}{1}{green}{brown}{brown}{green}{brown}
				\dotrparam{(0,0.25)}{0.25}{0.04}{orange}
			\end{boxedstrings}
			\\
			&=
			2
			\begin{boxedstrings}
				\idfigparam{(0,1)}{1}{green}
				\idfigparam{(0.5,1)}{1}{brown}
			\end{boxedstrings}
			+
			\begin{boxedstrings}
				\Hfigparamcol{(0,1)}{0.5}{1}{green}{brown}{brown}{green}{brown}
				\hbarbparam{(0.125,0.75)}{0.25}{0.04}{green}
			\end{boxedstrings}
			+ 	
			\begin{boxedstrings}
				\Hfigparamcol{(0,1)}{0.5}{1}{green}{brown}{brown}{green}{brown}
				\dotrparam{(0,0.25)}{0.25}{0.04}{orange}
			\end{boxedstrings}
			.
		\end{align}
		
		Similarly,
		\begin{align}\label{eqn ZY decomp 2}
			\begin{boxedstrings}
				\idfigparam{(0,1)}{0.25}{green}
				\Ifigcol{(nd-1-2)}{green}{brown}{brown}{green}{brown}
				\idfigparam{(nd-2-3)}{0.25}{green}
				\idfigparam{(nd-2-2)}{-0.25}{brown}
				\idfigparam{(nd-2-4)}{0.25}{brown}
				\hbarbparam{(0.125,0.125)}{0.25}{0.04}{green}
			\end{boxedstrings}
			+
			\begin{boxedstrings}
				\idfigparam{(0,1)}{0.25}{green}
				\Ifigcol{(nd-1-2)}{green}{brown}{brown}{green}{brown}
				\idfigparam{(nd-2-3)}{0.25}{green}
				\idfigparam{(nd-2-2)}{-0.25}{brown}
				\idfigparam{(nd-2-4)}{0.25}{brown}
				\dotrparam{(nd-1-3)}{0.25}{0.04}{orange}
			\end{boxedstrings}
			&=
			2
			\begin{boxedstrings}
				\idfigparam{(0,1)}{1}{green}
				\idfigparam{(0.5,1)}{1}{brown}
			\end{boxedstrings}
			+
			\begin{boxedstrings}
				\Hfigparamcol{(0,1)}{0.5}{1}{green}{brown}{brown}{green}{brown}
				\hbarbparam{(0.125,0.25)}{0.25}{0.04}{green}
			\end{boxedstrings}
			+ 	
			\begin{boxedstrings}
				\Hfigparamcol{(0,1)}{0.5}{1}{green}{brown}{brown}{green}{brown}
				\dotrparam{(0,0.75)}{0.25}{0.04}{orange}
			\end{boxedstrings}
			,
			\\ \label{eqn ZY decomp 3}
			\begin{boxedstrings}
				\idfigparam{(0,1)}{0.25}{green}
				\Ifigcol{(nd-1-2)}{green}{brown}{brown}{green}{brown}
				\idfigparam{(nd-2-3)}{0.25}{green}
				\idfigparam{(nd-2-2)}{-0.25}{brown}
				\idfigparam{(nd-2-4)}{0.25}{brown}
				\idfigparam{(0.125,0.75)}{0.5}{orange}
				\hbarbparam{(0.125,0.875)}{0.25}{0.04}{green}
			\end{boxedstrings}
			+
			\begin{boxedstrings}
				\idfigparam{(0,1)}{0.25}{green}
				\Ifigcol{(nd-1-2)}{green}{brown}{brown}{green}{brown}
				\idfigparam{(nd-2-3)}{0.25}{green}
				\idfigparam{(nd-2-2)}{-0.25}{brown}
				\idfigparam{(nd-2-4)}{0.25}{brown}
				\idfigparam{(0.125,0.75)}{0.5}{orange}
				\dotrparam{(nd-3-3)}{0.25}{0.04}{orange}
			\end{boxedstrings}
			&=
			2
			\begin{boxedstrings}
				\idfigparam{(0,1)}{1}{green}
				\idfigparam{(0.5,1)}{1}{brown}
			\end{boxedstrings}
			-
			\begin{boxedstrings}
				\Hfigparamcol{(0,1)}{0.5}{1}{green}{brown}{brown}{green}{brown}
				\hbarbparam{(0.125,0.75)}{0.25}{0.04}{green}
			\end{boxedstrings}
			-	
			\begin{boxedstrings}
				\Hfigparamcol{(0,1)}{0.5}{1}{green}{brown}{brown}{green}{brown}
				\dotrparam{(0,0.25)}{0.25}{0.04}{orange}
			\end{boxedstrings}
			,
			\\ \label{eqn ZY decomp 4}
			\begin{boxedstrings}
				\idfigparam{(0,1)}{0.25}{green}
				\Ifigcol{(nd-1-2)}{green}{brown}{brown}{green}{brown}
				\idfigparam{(nd-2-3)}{0.25}{green}
				\idfigparam{(nd-2-2)}{-0.25}{brown}
				\idfigparam{(nd-2-4)}{0.25}{brown}
				\idfigparam{(0.125,0.75)}{0.5}{orange}
				\hbarbparam{(0.125,0.125)}{0.25}{0.04}{green}
			\end{boxedstrings}
			+
			\begin{boxedstrings}
				\idfigparam{(0,1)}{0.25}{green}
				\Ifigcol{(nd-1-2)}{green}{brown}{brown}{green}{brown}
				\idfigparam{(nd-2-3)}{0.25}{green}
				\idfigparam{(nd-2-2)}{-0.25}{brown}
				\idfigparam{(nd-2-4)}{0.25}{brown}
				\idfigparam{(0.125,0.75)}{0.5}{orange}
				\dotrparam{(nd-1-3)}{0.25}{0.04}{orange}
			\end{boxedstrings}
			&=
			2
			\begin{boxedstrings}
				\idfigparam{(0,1)}{1}{green}
				\idfigparam{(0.5,1)}{1}{brown}
			\end{boxedstrings}
			-
			\begin{boxedstrings}
				\Hfigparamcol{(0,1)}{0.5}{1}{green}{brown}{brown}{green}{brown}
				\hbarbparam{(0.125,0.25)}{0.25}{0.04}{green}
			\end{boxedstrings}
			-	
			\begin{boxedstrings}
				\Hfigparamcol{(0,1)}{0.5}{1}{green}{brown}{brown}{green}{brown}
				\dotrparam{(0,0.75)}{0.25}{0.04}{orange}
			\end{boxedstrings}
			.
		\end{align}
	
		Adding \eqref{eqn ZY decomp 1}-\eqref{eqn ZY decomp 4},
		we get
		\begin{equation}
			\begin{split}
				8
				\begin{boxedstrings}
					\idfigparam{(0,2)}{1}{green}
					\idfigparam{(0.5,2)}{1}{brown}
				\end{boxedstrings}
				= &
				\begin{boxedstrings}
					\idfigparam{(0,1)}{0.25}{green}
					\Ifigcol{(nd-1-2)}{green}{brown}{brown}{green}{brown}
					\idfigparam{(nd-2-3)}{0.25}{green}
					\idfigparam{(nd-2-2)}{-0.25}{brown}
					\idfigparam{(nd-2-4)}{0.25}{brown}
					\hbarbparam{(0.125,0.875)}{0.25}{0.04}{green}
				\end{boxedstrings}
				+ 
				\begin{boxedstrings}
					\idfigparam{(0,1)}{0.25}{green}
					\Ifigcol{(nd-1-2)}{green}{brown}{brown}{green}{brown}
					\idfigparam{(nd-2-3)}{0.25}{green}
					\idfigparam{(nd-2-2)}{-0.25}{brown}
					\idfigparam{(nd-2-4)}{0.25}{brown}
					\idfigparam{(0.125,0.75)}{0.5}{orange}
					\hbarbparam{(0.125,0.875)}{0.25}{0.04}{green}
				\end{boxedstrings}
				+
				\begin{boxedstrings}
					\idfigparam{(0,1)}{0.25}{green}
					\Ifigcol{(nd-1-2)}{green}{brown}{brown}{green}{brown}
					\idfigparam{(nd-2-3)}{0.25}{green}
					\idfigparam{(nd-2-2)}{-0.25}{brown}
					\idfigparam{(nd-2-4)}{0.25}{brown}
					\hbarbparam{(0.125,0.125)}{0.25}{0.04}{green}
				\end{boxedstrings}
				+
				\begin{boxedstrings}
					\idfigparam{(0,1)}{0.25}{green}
					\Ifigcol{(nd-1-2)}{green}{brown}{brown}{green}{brown}
					\idfigparam{(nd-2-3)}{0.25}{green}
					\idfigparam{(nd-2-2)}{-0.25}{brown}
					\idfigparam{(nd-2-4)}{0.25}{brown}
					\idfigparam{(0.125,0.75)}{0.5}{orange}
					\hbarbparam{(0.125,0.125)}{0.25}{0.04}{green}
				\end{boxedstrings}
				\\
				+ & 
				\begin{boxedstrings}
					\idfigparam{(0,1)}{0.25}{green}
					\Ifigcol{(nd-1-2)}{green}{brown}{brown}{green}{brown}
					\idfigparam{(nd-2-3)}{0.25}{green}
					\idfigparam{(nd-2-2)}{-0.25}{brown}
					\idfigparam{(nd-2-4)}{0.25}{brown}
					\dotrparam{(nd-1-3)}{0.25}{0.04}{orange}
				\end{boxedstrings}
				+ 
				\begin{boxedstrings}
					\idfigparam{(0,1)}{0.25}{green}
					\Ifigcol{(nd-1-2)}{green}{brown}{brown}{green}{brown}
					\idfigparam{(nd-2-3)}{0.25}{green}
					\idfigparam{(nd-2-2)}{-0.25}{brown}
					\idfigparam{(nd-2-4)}{0.25}{brown}
					\idfigparam{(0.125,0.75)}{0.5}{orange}
					\dotrparam{(nd-1-3)}{0.25}{0.04}{orange}
				\end{boxedstrings}
				+ 
				\begin{boxedstrings}
					\idfigparam{(0,1)}{0.25}{green}
					\Ifigcol{(nd-1-2)}{green}{brown}{brown}{green}{brown}
					\idfigparam{(nd-2-3)}{0.25}{green}
					\idfigparam{(nd-2-2)}{-0.25}{brown}
					\idfigparam{(nd-2-4)}{0.25}{brown}
					\dotrparam{(nd-3-3)}{0.25}{0.04}{orange}
				\end{boxedstrings}
				+ 
				\begin{boxedstrings}
					\idfigparam{(0,1)}{0.25}{green}
					\Ifigcol{(nd-1-2)}{green}{brown}{brown}{green}{brown}
					\idfigparam{(nd-2-3)}{0.25}{green}
					\idfigparam{(nd-2-2)}{-0.25}{brown}
					\idfigparam{(nd-2-4)}{0.25}{brown}
					\idfigparam{(0.125,0.75)}{0.5}{orange}
					\dotrparam{(nd-3-3)}{0.25}{0.04}{orange}
				\end{boxedstrings}
			\end{split}
			.
		\end{equation}
	\end{proof}

	The following is an analog of Lemma \ref{lemma tensor products in Htau} for $\Dcal$.
	\begin{theorem}
		We have the following isomorphisms: 
		\begin{equation}\label{gr group bar 1}
			\Xobjbar\otimes \Xobjbar\cong \onecal,  
		\end{equation}
		\begin{equation}\label{gr group bar 2}
			\Xobjbar\otimes \Yobjbar \cong \Yobjbar \cong \Yobjbar\otimes \Xobjbar ,
		\end{equation} 
		\begin{equation}\label{gr group bar 3}
			\Xobjbar\otimes \Zobjbar \cong \XZobjbar \cong \Zobjbar\otimes \Xobjbar,
		\end{equation} 
		\begin{equation}\label{gr group bar 4}
			\Yobjbar\otimes \Yobjbar \cong \Yobjbar[-1]\oplus \Yobjbar[1]\oplus \Zobjbar \oplus \XZobjbar ,
		\end{equation}
		\begin{equation}\label{gr group bar 5}
			\Yobjbar\otimes \Zobjbar\cong \Zobjbar[-1]\oplus \Zobjbar[1] \oplus \XZobjbar[-1] \oplus \XZobjbar[1] \cong \Zobjbar \otimes \Yobjbar,
		\end{equation}
		\begin{equation}\label{gr group bar 6}
			\Zobjbar\otimes \Zobjbar\cong \Zobjbar[-2]\oplus \Zobjbar \oplus \Zobjbar[2] \oplus \XZobjbar .
		\end{equation}
		
		Hence, every object in $\Dcal$ is a direct sum of grading shifts of objects in
		$\Icalbar$.
	\end{theorem}
	\begin{proof}
		\eqref{gr group bar 1}  The diagrams $\Capp{orange}$ and $\Cupp{orange}$ correspond to mutually inverse morphisms between $\Xobjbar\otimes \Xobjbar$ and $\onecal$.
		This is shown in equations \eqref{needle relations} and \eqref{X2 idempotent decomp}.
		
		\eqref{gr group bar 2} 
		The diagrams $\Trivup{green}{orange}{green}$ and $\Trivdown{orange}{green}{green}$ correspond to mutually inverse morphisms between $\Xobjbar\otimes \Yobjbar$ and $\Yobjbar$.
		To check this, note that \eqref{XY idempotent decomp} implies $\Trivup{green}{orange}{green}\circ \Trivdown{orange}{green}{green} = \iddoubletext{orange}{green}$ and
		\begin{equation}
			\begin{boxedstrings}
				\trivuparamcol{(0.25,1)}{0.5}{0.5}{green}{orange}{green}
				\trivdparamcol{(nd-1-2)}{0.5}{0.5}{orange}{green}{green}
			\end{boxedstrings}
			=
			\begin{boxedstrings}
				\idfigparam{(0.5,1)}{1}{green}
				\node (a) at (0,0.5) {};
				\def\coltop{orange}
				\capfigparam{(a)}{0.4}{0.2}{\coltop}
				\cupfigparam{(a)}{0.4}{0.2}{\coltop}
			\end{boxedstrings}
			=
			\begin{boxedstrings}
				\idfigparam{(0.5,1)}{1}{green}
			\end{boxedstrings}
			.
		\end{equation}
		
		\eqref{gr group bar 3}
		The diagrams $\Cross{brown}{orange}$ and $\Cross{orange}{brown}$ correspond to mutually inverse morphisms between $\Xobjbar\otimes \Zobjbar$ and $\Zobjbar\otimes \Xobjbar$.
		This is shown in \eqref{orange brown uncross}.
		
		The isomorphisms \eqref{gr group bar 4}, \eqref{gr group bar 4}, and \eqref{gr group bar 6} follow from Proposition \ref{prop idemp decomp}.
		
	\end{proof}

\subsection{Sliding Orange Strands}\label{subsection orange strand}
One of the tools we will need to prove Theorem \ref{diagram reduction}
in Section \ref{subsection-diagram reduction},
is to get orange strands, as much as possible, out of the way 
of the non-orange elements of our diagrams.
In this section we will prove all the tools needed for this.

Our first goal is to show that orange strands slide over any diagram.
To do this, we must show that the orange strand slides over all other strands, 
as well as over dots, trivalent vertices and polynomials.

The following lemma shows the interactions of an orange strand with other orange strands and dots. Its follows easily from the definition of 
$
\Cross{orange}{orange}
$
in
\eqref{vertex definitions 2} and the orange barbell in \ref{barbell rels}.
\begin{lemma}\label{lemma orange interaction with itself}
	The orange strand slides over an orange strand, dot, or barbell.
	That is, the following relations hold:
	\begin{equation}\label{X2 idempotent decomp}
		\begin{boxedstrings}
			\idfigparam{(0,1)}{1}{orange}
			\idfigparam{(0.5,1)}{1}{orange}
		\end{boxedstrings}
		=
		\begin{boxedstrings}
			\Xfig{(0,0)}{orange}{orange}
			\idfigparam{(nd-1-1)}{-0.25}{orange}
			\idfigparam{(nd-1-2)}{-0.25}{orange}
			\idfigparam{(nd-1-4)}{0.25}{orange}
			\idfigparam{(nd-1-3)}{0.25}{orange}
		\end{boxedstrings}
		=
		\begin{boxedstrings}
			\capfigparam{(0,0)}{0.5}{0.3}{orange}
			\cupfigparam{(0,1)}{0.5}{0.3}{orange}
		\end{boxedstrings}
		,
	\end{equation}
	
	\begin{equation}\label{orange jump dot orange rel}
		\begin{boxedstrings}
			\Xfig{(0,0)}{orange}{orange}
			\idfigparam{(nd-1-2)}{-0.25}{orange}
			\idfigparam{(nd-1-3)}{0.25}{orange}
			\idfigparam{(nd-1-4)}{0.25}{orange}
			\dotfig{(nd-1-1)}{orange}
		\end{boxedstrings}
		=
		\begin{boxedstrings}
			\idfig{(0.5,1)}{orange}
			\dotuparam{(0.25,0)}{0.5}{0.06}{orange}
		\end{boxedstrings}
		=
		\begin{boxedstrings}
			\idfigparam{(0,1)}{0.5}{orange}
			\cupfigparam{(nd-1-2)}{-0.4}{0.2}{orange}
			\capfigparam{(0,0)}{-0.4}{0.2}{orange}
			\dotfigparamcol{(nd-2-2)}{0.06}{orange}
		\end{boxedstrings}
		=
		\begin{boxedstrings}
			\dotdparam{(0.2,1)}{0.5}{0.06}{orange}
			\capfigparam{(0,0)}{0.4}{0.2}{orange}
		\end{boxedstrings}
		=
		\begin{boxedstrings}
			\idfigparam{(0,1)}{0.5}{orange}
			\cupfigparam{(nd-1-2)}{0.4}{0.2}{orange}
			\capfigparam{(0,0)}{0.4}{0.2}{orange}
			\dotfigparamcol{(nd-2-2)}{0.06}{orange}
		\end{boxedstrings}
		=
		\begin{boxedstrings}
			\idfig{(0.5,1)}{orange}
			\dotuparam{(0.75,0)}{0.5}{0.06}{orange}
		\end{boxedstrings}
		,
	\end{equation}
	
	\begin{equation}\label{orange dots merger}
		(\alpha_s-\alpha_t)^2
		\begin{boxedstrings}
			\idfig{(1,1)}{orange}
		\end{boxedstrings}
		=
		\begin{boxedstrings}
		\idfig{(1,1)}{orange}
		\vbarb{(0.75,0.7)}{orange}
		\end{boxedstrings}
		=	
		\begin{boxedstrings}
			\dotd{(1,1)}{orange}
			\dotu{(1,0)}{orange}
		\end{boxedstrings}
		=
		\begin{boxedstrings}
		\idfig{(1,1)}{orange}
		\vbarb{(1.25,0.7)}{orange}
		\end{boxedstrings}
		=
		\begin{boxedstrings}
			\idfig{(1,1)}{orange}
		\end{boxedstrings}
		(\alpha_s-\alpha_t)^2.
		\end{equation}
\end{lemma}

\begin{coro}
	Let $D$ be a diagram containing only orange strands and dots, with $k$ boundary strands.
	Then $D$ is an $\Rtau$-multiple of $\mathsf{Or}_k$, defined by
	\begin{equation}\label{ORk definition}
		\mathsf{Or}_k := 
		\begin{cases} 
			\begin{boxedstrings}
				\foreach \x in {0.1,0.3,0.7,0.9}
				{
					\draw[orange] (\x,0) -- ++(0,1);
				};
				\node (A) at (0.5,0.5) {$\overset{m}{...}$};
			\end{boxedstrings}
			&\mbox{if } k=2m \\
			&\\
			\begin{boxedstrings}
				\foreach \x in {0.1,0.3,0.7,0.9}
				{
					\draw[orange] (\x,0) -- ++(0,1);
				};
				\node (A) at (0.5,0.5) {$\overset{m}{...}$};
				\dotu{(-0.1,0)}{orange}
			\end{boxedstrings}
			& \mbox{if } k=2m+1 
		\end{cases}
		.
	\end{equation}
\end{coro}
\begin{proof}
	By \eqref{orange jump dot orange rel} and \eqref{orange dots merger},
	\begin{equation}
		\begin{boxedstrings}
			\foreach \x in {0.5,0.75}
			{
				\draw[orange] (\x,0) -- ++(0,1);
			};
			\dotu{(0.25,0)}{orange}
			\dotu{(1,0)}{orange}
		\end{boxedstrings}
		=
		\begin{boxedstrings}
			\foreach \x in {0.5,0.75}
			{
				\draw[orange] (\x,0) -- ++(0,1);
			};
			\draw[orange] (1,0) -- ++(0,0.5) -- ++(-0.75,0);
			\dotu{(0.25,0)}{orange}
			\dotfig{(0.25,0.5)}{orange}
		\end{boxedstrings}
		=
		\begin{boxedstrings}
			\foreach \x in {0.5,0.75}
			{
				\draw[orange] (\x,0) -- ++(0,1);
			};
			\draw[orange] (0.25,0) -- ++(0,0.5) -- ++(0.75,0) -- ++(0,-0.5);
			\vbarb{(0,0.7)}{orange}
		\end{boxedstrings}
		=
		(\alpha_s-\alpha_t)
		\begin{boxedstrings}
			\foreach \x in {0.25,0.5,0.75}
			{
				\draw[orange] (\x,0) -- ++(0,1);
			};
			\draw[orange] (0,0) -- ++(0,0.5) -- ++(1,0) -- ++(0,-0.5);
		\end{boxedstrings}
		.
	\end{equation}
	If $D$ has more than one orange dot, we can use this process to merge them into a factor of $(\alpha_s-\alpha_t)^2$. Hence, we may suppose that $D$ has at most one orange dot.
	Moreover, \eqref{orange jump dot orange rel} allows us to move that orange dot to the bottom left orange boundary strand.
	
	By \eqref{X2 idempotent decomp}, we may replace any $\Cross{orange}{orange}$ by
	$\iddoubletext[0.6]{orange}{orange}$. 
	Thus, we may suppose that $D$ has no orange crossings.
	
	We can also use Lemma \ref{lemma orange interaction with itself} 
	along with the orange circle relation in \eqref{needle relations},
	to slide any orange circle and eliminate any orange circle in $D$.
	Because of this, we can suppose further that $D$ does not have any orange circles.
	Therefore, the diagram $D$ is a non-crossing matching of its boundary points, 
	plus an orange boundary dot if $k$ is odd.
	
	Finally, \eqref{X2 idempotent decomp} allows us to transform any perfect matching into the one in \eqref{ORk definition}.
\end{proof}

Let us consider the relations 
\begin{equation}\label{orange slide horizontal green}
	\begin{fitboxedstrings}
		\path (0,0)--(1,1);
		\draw[green] (0,0.5) -- (1,0.5);
		\draw[orange] (0.25,1) -- (0.25,0.5);
		\draw[orange] (0.75,0) -- (0.75,0.5);
	\end{fitboxedstrings}
	=
	\begin{fitboxedstrings}
		\path (0,0)--(1,1);
		\draw[green] (0,0.5) -- (1,0.5);
		\draw[orange] (0.5,1) -- (0.5,0);
	\end{fitboxedstrings}
	=
	\begin{fitboxedstrings}
		\path (0,0)--(1,1);
		\draw[green] (0,0.5) -- (1,0.5);
		\draw[orange] (0.25,0) -- (0.25,0.5);
		\draw[orange] (0.75,1) -- (0.75,0.5);
	\end{fitboxedstrings}
	\hspace{0.1cm},
\end{equation}
which are rotations of \eqref{vertex definitions}.
This shows that we can split orange crossings and slide orange landings on opposite sides of a green strand past each other.
When we have two orange landings on the same side of a green strand we get, 
\begin{equation}\label{horizontal green cups}
	\begin{fitboxedstrings}
		\path (0,0)--(1,1);
		\draw[green] (0,0.5) -- (1,0.5);
		\foreach \x in {1,2}
		{
			\tikzmath{\z=\x*0.333;}
			\draw[orange] (\z,1) -- (\z,0.5);
		};
	\end{fitboxedstrings}
	=
	\begin{fitboxedstrings}
		\path (0,0)--(1,1);
		\draw[green] (0,0.5) -- (1,0.5);
		\cupfigparam{(0.333,1)}{0.333}{0.333}{orange}
	\end{fitboxedstrings}
	\hspace{0.1cm}.
\end{equation}
instead. This is a rotation of \eqref{XY idempotent decomp}.
Applying \eqref{orange slide horizontal green} and \eqref{horizontal green cups} repeatedly,
we have the following lemma.
\begin{lemma}\label{lemma orange landings green edge}
	Let $D$ be a diagram of the form
	\begin{equation}
		D
		=
		\begin{fitboxedstrings}
			\path (0,0)--(1,1);
			\draw[green] (0,0.5) -- (2,0.5);
			\node (A) at (0.5,0.25) {$\overset{m}{...}$};
			\node (B) at (1.5,0.75) {$\overset{k}{...}$};
			\foreach \x in {0.15,0.3,0.7,0.85}
			{
				\draw[orange] (\x,0) -- ++(0,0.5);
			};
			\foreach \x in {1.15,1.3,1.7,1.85}
			{
				\draw[orange] (\x,0.5) -- ++(0,0.5);
			};
		\end{fitboxedstrings}
		\hspace{0.2cm} ,
	\end{equation}
	with $m$ orange landings from one side and $k$ from the other.
	Then
	\begin{equation}
		D=
		\left\{
		\begin{array}{llll}
			\begin{fitboxedstrings}
				\path (0,0)--(1,1);
				\draw[green] (0,0.5) -- (2,0.5);
				\node (A) at (1,0.25) {$\overset{\frac{m}{2}}{...}$};
				\node (B) at (1,0.75) {$\overset{\frac{k}{2}}{...}$};
				\foreach \x in {0.4,1.3}
				{
					\cupfigparam{(\x,1)}{0.3}{0.3}{orange}
					\capfigparam{(\x,0)}{0.3}{0.3}{orange}
				};
			\end{fitboxedstrings} 
			\vspace{0.2cm}
			 & \text{ if } k,m\in 2\ZZ;
			&
			\begin{fitboxedstrings}
				\path (0,0)--(1,1);
				\draw[green] (0,0.5) -- (2,0.5);
				\node (A) at (1,0.25) {$\overset{\frac{m-1}{2}}{...}$};
				\node (B) at (1,0.75) {$\overset{\frac{k}{2}}{...}$};
				\foreach \x in {0.4,1.3}
				{
					\cupfigparam{(\x,1)}{0.3}{0.3}{orange}
					\capfigparam{(\x,0)}{0.3}{0.3}{orange}
				};
				\draw[orange] (0.2,0) -- ++(0,0.5);
			\end{fitboxedstrings}
			 & \text{ if } k\in 2\ZZ, m\notin 2\ZZ ;
			\\
			\begin{fitboxedstrings}
			\path (0,0)--(1,1);
			\draw[green] (0,0.5) -- (2,0.5);
			\node (A) at (1,0.25) {$\overset{\frac{m}{2}}{...}$};
			\node (B) at (1,0.75) {$\overset{\frac{k-1}{2}}{...}$};
			\foreach \x in {0.4,1.3}
			{
				\cupfigparam{(\x,1)}{0.3}{0.3}{orange}
				\capfigparam{(\x,0)}{0.3}{0.3}{orange}
			};
			\draw[orange] (0.2,0.5) -- ++(0,0.5);
			\end{fitboxedstrings}
			 & \text{ if } m\in 2\ZZ, k\notin 2\ZZ ;
			 &
			 \begin{fitboxedstrings}
			 	\path (0,0)--(1,1);
			 	\draw[green] (0,0.5) -- (2,0.5);
			 	\node (A) at (1,0.25) {$\overset{\frac{m-1}{2}}{...}$};
			 	\node (B) at (1,0.75) {$\overset{\frac{k-1}{2}}{...}$};
			 	\foreach \x in {0.4,1.3}
			 	{
			 		\cupfigparam{(\x,1)}{0.3}{0.3}{orange}
			 		\capfigparam{(\x,0)}{0.3}{0.3}{orange}
			 	};
			 	\draw[orange] (0.2,0) -- ++(0,1);
			 \end{fitboxedstrings}
			  & \text{ if } k,m\notin 2\ZZ;
		\end{array}
		\right.
		.
	\end{equation}
	
	Hence, we can transform any green edge with an arbitrary number of 
	orange crossings and landings into an edge that has at most one crossing or landing.
\end{lemma}

\begin{prop}
	The orange strand slides over the green and brown strands.
	That is, the following relations hold:
	\begin{equation}\label{orange green uncross}
		\begin{boxedstrings}
			\idfigparam{(0,1)}{1}{orange}
			\idfigparam{(0.25,1)}{1}{green}
		\end{boxedstrings}
		=
		\begin{boxedstrings}
			\Xfig{(0,0)}{orange}{green}
			\Xfig{(nd-1-3)}{green}{orange}
		\end{boxedstrings}
		,
	\end{equation}
	\begin{equation}\label{orange brown uncross}
		\begin{boxedstrings}
			\idfigparam{(0,1)}{1}{orange}
			\idfigparam{(0.25,1)}{1}{brown}
		\end{boxedstrings}
		=
		\begin{boxedstrings}
			\Xfig{(0,0)}{orange}{brown}
			\Xfig{(nd-1-3)}{brown}{orange}
		\end{boxedstrings}
		.
	\end{equation}
\end{prop}
\begin{proof}
	Using the orange circle in 
	\eqref{needle relations}, 
	\eqref{X2 idempotent decomp}, and
	\eqref{XY idempotent decomp} we can conclude that 
	\begin{equation}\label{orange green uncross proof}
		\begin{boxedstrings}
			\Xfig{(0,0)}{orange}{green}
			\Xfig{(nd-1-3)}{green}{orange}
		\end{boxedstrings}
		=
		\begin{boxedstrings}
			\Ilongparam{(0,1)}{0.5}{1}{orange}{green}{green}{orange}{green}
			\draw[orange] (nd-1-2) -- (nd-1-4);
		\end{boxedstrings}
		=
		\begin{boxedstrings}
			\idfigparam{(0,1)}{1}{orange}
			\idfigparam{(0.25,1)}{1}{green}
			\node (a) at (0.5,0.5) {};
			\def\coltop{orange}
			\capfigparam{(a)}{0.4}{0.2}{\coltop}
			\cupfigparam{(a)}{0.4}{0.2}{\coltop}
		\end{boxedstrings}
		=
		\begin{boxedstrings}
			\idfigparam{(0,1)}{1}{orange}
			\idfigparam{(0.25,1)}{1}{green}
		\end{boxedstrings}
		.
	\end{equation}
	
	Similarly, \eqref{GGB triv orange 2} and \eqref{bigon brown orange}
	allow us to write
	\begin{equation}\label{orange brown uncross proof}
		\begin{boxedstrings}
			\Xfig{(0,0)}{orange}{brown}
			\Xfig{(nd-1-3)}{brown}{orange}
			\idfigparam{(nd-1-1)}{-0.25}{orange}
			\idfigparam{(nd-2-3)}{ 0.25}{orange}
			\idfigparam{(nd-1-2)}{-0.25}{brown}
			\idfigparam{(nd-2-4)}{ 0.25}{brown}
		\end{boxedstrings}
		=
		\frac{1}{4}
		\begin{boxedstrings}
			\draw[orange] (0,1)-- ++(0,0.4);
			\capfigparam{(0,1)}{0.4}{0.2}{green}
			\idfigparam{(nd-1-2)}{0.7}{orange}
			\cupfigparam{(nd-1-1)}{0.4}{0.2}{green}
			\capfigparam{(nd-2-2)}{-0.4}{0.2}{green}
			\draw[orange](nd-4-2)-- ++(0,-0.4);
			\cupfigparam{(nd-2-2)}{-0.4}{0.2}{green}
			\idfigparam{(nd-5-3)}{0.2}{brown}
			\idfigparam{(nd-1-3)}{-0.2}{brown}
			\draw[brown] (nd-3-3)--(nd-4-3);
		\end{boxedstrings}
		=
		-
		\frac{1}{4}
		\begin{boxedstrings}
			\draw[orange] (0,1)-- ++(0,0.4);
			\capfigparam{(0,1)}{0.4}{0.2}{green}
			\draw[orange](nd-1-2)-- ++(0,-0.35) -- ++(-0.3,0) -- ++(0,-0.125);
			\cupfigparam{(nd-1-1)}{0.4}{0.2}{green}
			\capfigparam{(0,0.3)}{0.4}{0.2}{green}
			\draw[orange](nd-3-1)-- ++(0,-0.4);
			\cupfigparam{(0,0.3)}{0.4}{0.2}{green}
			\idfigparam{(nd-4-3)}{0.2}{brown}
			\idfigparam{(nd-1-3)}{-0.2}{brown}
			\draw[brown] (nd-2-3)--(nd-3-3);
		\end{boxedstrings}
		=
		\frac{1}{4}
		\begin{boxedstrings}
			\draw[orange] (0,1)-- ++(0,0.4);
			\capfigparam{(0,1)}{0.4}{0.2}{green}
			\draw[orange](0.1,0.525) -- ++(0,0.25);
			\cupfigparam{(nd-1-1)}{0.4}{0.2}{green}
			\capfigparam{(0,0.3)}{0.4}{0.2}{green}
			\draw[orange](nd-3-1)-- ++(0,-0.4);
			\cupfigparam{(0,0.3)}{0.4}{0.2}{green}
			\idfigparam{(nd-4-3)}{0.2}{brown}
			\idfigparam{(nd-1-3)}{-0.2}{brown}
			\draw[brown] (nd-2-3)--(nd-3-3);
		\end{boxedstrings}
		=
		\frac{1}{4}
		\begin{boxedstrings}
			\capfigparam{(0,1)}{0.4}{0.2}{green}
			\cupfigparam{(nd-1-1)}{0.4}{0.2}{green}
			\capfigparam{(0,0.3)}{0.4}{0.2}{green}
			\cupfigparam{(0,0.3)}{0.4}{0.2}{green}
			\idfigparam{(nd-4-3)}{0.2}{brown}
			\idfigparam{(nd-1-3)}{-0.2}{brown}
			\draw[brown] (nd-2-3)--(nd-3-3);
			\idfigparam{(-0.2,1.4)}{1.5}{orange}
		\end{boxedstrings}
		=
		\begin{boxedstrings}
			\idfigparam{(0,1)}{1.5}{orange}
			\idfigparam{(0.25,1)}{1.5}{brown}
		\end{boxedstrings}
		.
	\end{equation}
\end{proof}

As a consequence of \eqref{orange brown uncross}, we have
\begin{equation}\label{orange slides over orange brown crossing}
		\begin{fitboxedstrings}
			\path (0,0)--(1,1);
			\draw[brown] (0,0.5) -- (1,0.5);
			\idfigparam{(0.25,1)}{1}{orange}
			\idfigparam{(0.75,1)}{1}{orange}
		\end{fitboxedstrings}
		=
		\begin{fitboxedstrings}
			\path (0,0)--(1,1);
			\draw[brown] (0,0.5) -- (1,0.5);
			\cupfigparam{(0.25,1)}{0.5}{0.2}{orange}
			\capfigparam{(0.3,0.5)}{0.4}{0.2}{orange}
			\cupfigparam{(0.3,0.5)}{0.4}{0.2}{orange}
			\capfigparam{(0.25,0)}{0.5}{0.2}{orange}
		\end{fitboxedstrings}
		=
		\begin{fitboxedstrings}
			\path (0,0)--(1,1);
			\draw[brown] (0,0.5) -- (1,0.5);
			\cupfigparam{(0.25,1)}{0.5}{0.2}{orange}
			\capfigparam{(0.25,0)}{0.5}{0.2}{orange}
		\end{fitboxedstrings}
		\hspace{0.1cm}.
\end{equation}
Applying this relation repeatedly, we have the following lemma.
\begin{lemma}\label{lemma orange crossings brown edge}
	Let $D$ be a diagram of the form
	\begin{equation}
		D
		=
		\begin{fitboxedstrings}
			\path (0,0)--(1,1);
			\draw[brown] (0,0.5) -- (1,0.5);
			\node (A) at (0.5,0.25) {$\overset{k}{...}$};
			\foreach \x in {0.15,0.3,0.7,0.85}
			{
				\draw[orange] (\x,0) -- ++(0,1);
			};
		\end{fitboxedstrings}
		\hspace{0.2cm} ,
	\end{equation}
	with $k$ orange crossings.
	Then
	\begin{equation}
		D=
		\left\{
		\begin{array}{ll}
			\begin{fitboxedstrings}
				\path (0,0)--(1,1);
				\draw[brown] (0,0.5) -- (2,0.5);
				\node (A) at (1,0.25) {$\overset{\frac{k}{2}}{...}$};
				\node (B) at (1,0.75) {$\overset{\frac{k}{2}}{...}$};
				\foreach \x in {0.4,1.3}
				{
					\cupfigparam{(\x,1)}{0.3}{0.3}{orange}
					\capfigparam{(\x,0)}{0.3}{0.3}{orange}
				};
			\end{fitboxedstrings} 
			\vspace{0.2cm}
			& \text{ if } k\in 2\ZZ;
			\\
			\begin{fitboxedstrings}
				\path (0,0)--(1,1);
				\draw[brown] (0,0.5) -- (2,0.5);
				\node (A) at (1,0.25) {$\overset{\frac{k-1}{2}}{...}$};
				\node (B) at (1,0.75) {$\overset{\frac{k-1}{2}}{...}$};
				\foreach \x in {0.4,1.3}
				{
					\cupfigparam{(\x,1)}{0.3}{0.3}{orange}
					\capfigparam{(\x,0)}{0.3}{0.3}{orange}
				};
				\draw[orange] (0.2,0) -- ++(0,1);
			\end{fitboxedstrings}
			& \text{ if } k\notin 2\ZZ;
		\end{array}
		\right.
		.
	\end{equation}
	Hence, we can transform any brown edge with an arbitrary number of 
	orange crossings into an edge that has at most one crossing.
\end{lemma}

\begin{prop}\label{prop orange jumps over green brown dots}
	The orange strand slides over green and brown dots.
	That is, the following relations hold:
	\begin{equation}\label{green dot orange rel}
		\begin{boxedstrings}
			\idfig{(0.5,1)}{orange}
			\dotuparam{(0.75,0)}{0.5}{0.06}{green}
		\end{boxedstrings}
		=
		\begin{boxedstrings}
			\Xfig{(0,0)}{green}{orange}
			\idfigparam{(nd-1-2)}{-0.25}{orange}
			\idfigparam{(nd-1-3)}{0.25}{orange}
			\idfigparam{(nd-1-4)}{0.25}{green}
			\dotfig{(nd-1-1)}{green}
		\end{boxedstrings}
		,
	\end{equation}
	\begin{equation}\label{brown dot orange rel}
		\begin{boxedstrings}
			\idfig{(0.5,1)}{orange}
			\dotuparam{(0.75,0)}{0.5}{0.06}{brown}
		\end{boxedstrings}
		=
		\begin{boxedstrings}
			\Xfig{(0,0)}{brown}{orange}
			\idfigparam{(nd-1-2)}{-0.25}{orange}
			\idfigparam{(nd-1-3)}{0.25}{orange}
			\idfigparam{(nd-1-4)}{0.25}{brown}
			\dotfig{(nd-1-1)}{brown}
		\end{boxedstrings}.
	\end{equation}
\end{prop}
\begin{proof}
	For \eqref{green dot orange rel}, let us use \eqref{bivalent definition right}, 
	\eqref{bivalent definition left}, and \eqref{XY idempotent decomp} to write
	\begin{equation}
		\begin{boxedstrings}
			\Xfig{(0,0)}{green}{orange}
			\idfigparam{(nd-1-2)}{-0.25}{orange}
			\idfigparam{(nd-1-3)}{0.25}{orange}
			\idfigparam{(nd-1-4)}{0.25}{green}
			\dotfig{(nd-1-1)}{green}
		\end{boxedstrings}
		=
		\begin{boxedstrings}
			\dotuparam{(0.25,0)}{0.8}{0.06}{green}
			\draw[orange] (0,0) -- ++(0,0.4) -- ++ (0.25,0);
			\draw[orange] (0.5,1) -- ++(0,-0.4) -- ++ (-0.25,0);
		\end{boxedstrings}
		=
		\begin{boxedstrings}
			\idfigparam{(0.25,1)}{0.4}{orange}
			\idfigparam{(nd-1-2)}{0.6}{green}
			\draw[orange] (0,0) -- ++(0,0.4) -- ++ (0.25,0);
		\end{boxedstrings}
		=
		\begin{boxedstrings}
			\dotuparam{(0.25,0)}{0.8}{0.06}{green}
			\draw[orange] (0,0) -- ++(0,0.4) -- ++ (0.25,0);
			\draw[orange] (0,1) -- ++(0,-0.4) -- ++ (0.25,0);
		\end{boxedstrings}
		=
		\begin{boxedstrings}
			\idfig{(0.5,1)}{orange}
			\dotuparam{(0.75,0)}{0.5}{0.06}{green}
		\end{boxedstrings}
		.
	\end{equation}
	For \eqref{brown dot orange rel}, let us use \eqref{vertex definitions}, \eqref{brown green unit rel}, and \eqref{GGB unit rel} to write
	\begin{align}
		\begin{boxedstrings}
			\Xfig{(0,0)}{brown}{orange}
			\idfigparam{(nd-1-2)}{-0.375}{orange}
			\idfigparam{(nd-1-3)}{0.375}{orange}
			\idfigparam{(nd-1-4)}{0.375}{brown}
			\dotfig{(nd-1-1)}{brown}
		\end{boxedstrings}
		&=
		- 
		\frac{1}{2}
		\begin{boxedstrings}
			\Hfigparamcol{(0,1)}{0.5}{0.35}{brown}{orange}{green}{green}{green}
			\dotfigparamcol{(nd-1-1)}{0.06}{brown}
			\Hfigparamcol{(nd-1-3)}{0.5}{0.35}{green}{green}{green}{orange}{brown}
			\idfigparam{(nd-1-2)}{-0.275}{orange}
			\idfigparam{(nd-2-3)}{0.275}{orange}
			\idfigparam{(nd-2-4)}{0.275}{brown}
		\end{boxedstrings}	
		=
		- 
		\frac{1}{2}
		\begin{boxedstrings}
			\Hfigparamcol{(0,0.75)}{0.5}{0.5}{green}{green}{green}{orange}{brown}
			\idfigparam{(nd-1-1)}{-0.25}{green}
			\dotfigparamcol{(nd-2-2)}{0.06}{green}
			\idfigparam{(nd-1-2)}{-0.25}{green}
			\dotlparam{(nd-3-3)}{0.25}{0.06}{green}
			\idfigparam{(nd-1-3)}{0.25}{orange}
			\idfigparam{(nd-1-4)}{0.25}{brown}
			\idfigparam{(nd-3-2)}{-0.25}{orange}
		\end{boxedstrings}
		+
		\frac{1}{2}
		\begin{boxedstrings}
			\Hfigparamcol{(0,0.75)}{0.5}{0.5}{green}{green}{green}{orange}{brown}
			\draw[orange] (nd-1-1) -- ++(0.375,0) -- ++(0,0.25);
			\idfigparam{(nd-1-1)}{-0.25}{green}
			\dotfigparamcol{(nd-2-2)}{0.06}{green}
			\idfigparam{(nd-1-2)}{-0.25}{green}
			\dotlparam{(nd-3-2)}{0.25}{0.06}{green}
			\idfigparam{(nd-1-3)}{0.25}{orange}
			\idfigparam{(nd-1-4)}{0.25}{brown}
			\idfigparam{(nd-3-2)}{-0.25}{orange}
		\end{boxedstrings}
		\\
		&=
		\begin{boxedstrings}
			\Hfigparamcol{(0,0.75)}{0.5}{0.5}{green}{green}{green}{orange}{brown}
			\draw[orange] (nd-1-1) -- ++(0.375,0) -- ++(0,0.25);
			\idfigparam{(nd-1-1)}{-0.25}{green}
			\dotfigparamcol{(nd-2-2)}{0.06}{green}
			\idfigparam{(nd-1-2)}{-0.25}{green}
			\dotlparam{(nd-3-2)}{0.25}{0.06}{green}
			\idfigparam{(nd-1-3)}{0.25}{orange}
			\idfigparam{(nd-1-4)}{0.25}{brown}
			\idfigparam{(nd-3-2)}{-0.25}{orange}
		\end{boxedstrings}
		=
		\begin{boxedstrings}
			\trivdparamcol{(0,0.7)}{0.5}{0.7}{green}{green}{brown}
			\dotfig{(nd-1-1)}{green}
			\dotfig{(nd-1-2)}{green}
			\idfigparam{(-0.25,0)}{-1.25}{orange}
		\end{boxedstrings}
		=
		\begin{boxedstrings}
			\dotuparam{(0,0)}{0.6}{0.06}{brown}
			\idfigparam{(-0.25,0)}{-1.25}{orange}
		\end{boxedstrings}.
	\end{align}
\end{proof}

An easy consequence of Proposition \ref{prop orange jumps over green brown dots}
is that the orange strand slides over barbells. 
Writing polynomial boxes as polynomials in green and brown barbells, by Proposition \ref{prop polynomials are barbells},
 we have the following corollary.
\begin{coro}[Orange polynomial forcing]\label{polynomial forcing under orange}
	The orange strand slides over polynomial boxes.
	That is, 
	\begin{equation}
		\begin{boxedstrings}
			\idfigparam{(1,1)}{1}{orange}
			\polybox{(0.5,0.5)}{$f$}
		\end{boxedstrings}
		=
		\begin{boxedstrings}
		\idfigparam{(0,1)}{1}{orange}
		\polybox{(0.5,0.5)}{$f$}
		\end{boxedstrings} 
		\text{ for all } 
		f\in\Rtau.
	\end{equation}
\end{coro}

\begin{prop}\label{prop orange jumps through trivalents}
	The orange strand slides over trivalent vertices.
	That is, the following relations hold:
	
	\begin{center}
		\begin{minipage}[t]{0.4\textwidth}
			\begin{equation}\label{orange slide over orange green trivalent 1}
				\begin{boxedstrings}
					\trivuparamcol{(0.25,0.75)}{0.5}{0.5}{green}{green}{orange}
					\idfigparam{(nd-1-1)}{-0.25}{green}
					\idfigparam{(nd-1-2)}{0.25}{green}
					\idfigparam{(nd-1-3)}{0.25}{orange}
					\draw[orange] (-0.25,0) -- ++(0,0.25) -- (nd-1-1) -- (0.75,0.75) -- (0.75,1);
				\end{boxedstrings}
				=
				\begin{boxedstrings}
					\trivuparamcol{(0.25,0.75)}{0.5}{0.5}{green}{green}{orange}
					\idfigparam{(nd-1-1)}{-0.25}{green}
					\idfigparam{(nd-1-2)}{0.25}{green}
					\idfigparam{(nd-1-3)}{0.25}{orange}
					\draw[orange] (-0.25,0) -- ++(0,0.25) -- (nd-1-2) -- (nd-1-3) -- 
					++(0.25,0.25) -- (0.75,1);
				\end{boxedstrings}
				,
			\end{equation}
			
			\begin{equation}\label{orange slide over orange green trivalent 2}
				\begin{boxedstrings}
					\trivuparamcol{(0.25,0.75)}{0.5}{0.5}{green}{orange}{green}
					\idfigparam{(nd-1-1)}{-0.25}{green}
					\idfigparam{(nd-1-2)}{0.25}{orange}
					\idfigparam{(nd-1-3)}{0.25}{green}
					\draw[orange] (-0.25,0) -- ++(0,0.25) -- (nd-1-1) -- (0.75,0.75) -- (0.75,1);
				\end{boxedstrings}
				=
				\begin{boxedstrings}
					\trivuparamcol{(0.25,0.75)}{0.5}{0.5}{green}{orange}{green}
					\idfigparam{(nd-1-1)}{-0.25}{green}
					\idfigparam{(nd-1-2)}{0.25}{orange}
					\idfigparam{(nd-1-3)}{0.25}{green}
					\draw[orange] (-0.25,0) -- ++(0,0.25) -- (nd-1-2) -- (nd-1-3) -- 
					++(0.25,0.25) -- (0.75,1);
				\end{boxedstrings}
				,
			\end{equation}
			
			\begin{equation}\label{orange slide over orange green trivalent 3}
				\begin{boxedstrings}
					\trivuparamcol{(0.25,0.75)}{0.5}{0.5}{orange}{green}{green}
					\idfigparam{(nd-1-1)}{-0.25}{orange}
					\idfigparam{(nd-1-2)}{0.25}{green}
					\idfigparam{(nd-1-3)}{0.25}{green}
					\draw[orange] (-0.25,0) -- ++(0,0.25) -- (nd-1-1) -- (0.75,0.75) -- (0.75,1);
				\end{boxedstrings}
				=
				\begin{boxedstrings}
					\trivuparamcol{(0.25,0.75)}{0.5}{0.5}{orange}{green}{green}
					\idfigparam{(nd-1-1)}{-0.25}{orange}
					\idfigparam{(nd-1-2)}{0.25}{green}
					\idfigparam{(nd-1-3)}{0.25}{green}
					\draw[orange] (-0.25,0) -- ++(0,0.25) -- (nd-1-2) -- (nd-1-3) -- 
					++(0.25,0.25) -- (0.75,1);
				\end{boxedstrings}
				,
			\end{equation}
			
			\begin{equation}\label{orange jump through GGG}
				\begin{boxedstrings}
					\trivuparamcol{(0.25,0.75)}{0.5}{0.5}{green}{green}{green}
					\idfigparam{(nd-1-1)}{-0.25}{green}
					\idfigparam{(nd-1-2)}{0.25}{green}
					\idfigparam{(nd-1-3)}{0.25}{green}
					\draw[orange] (-0.25,0) -- ++(0,0.25) -- (nd-1-1) -- (0.75,0.75) -- (0.75,1);
				\end{boxedstrings}=
				\begin{boxedstrings}
					\trivuparamcol{(0.25,0.75)}{0.5}{0.5}{green}{green}{green}
					\idfigparam{(nd-1-1)}{-0.25}{green}
					\idfigparam{(nd-1-2)}{0.25}{green}
					\idfigparam{(nd-1-3)}{0.25}{green}
					\draw[orange] (-0.25,0) -- ++(0,0.25) -- (nd-1-3) -- (0.75,0.5) -- (0.75,1);
				\end{boxedstrings},
			\end{equation}
		\end{minipage}
		\begin{minipage}[t]{0.4\textwidth}
			\begin{equation}\label{orange jump through GGB}
				\begin{boxedstrings}
					\trivuparamcol{(0.25,0.75)}{0.5}{0.5}{brown}{green}{green}
					\idfigparam{(nd-1-1)}{-0.25}{brown}
					\idfigparam{(nd-1-2)}{0.25}{green}
					\idfigparam{(nd-1-3)}{0.25}{green}
					\draw[orange] (-0.25,0) -- ++(0,0.25) -- (nd-1-1) -- (0.75,0.75) -- (0.75,1);
				\end{boxedstrings}
				=
				\begin{boxedstrings}
					\trivuparamcol{(0.25,0.75)}{0.5}{0.5}{brown}{green}{green}
					\idfigparam{(nd-1-1)}{-0.25}{brown}
					\idfigparam{(nd-1-2)}{0.25}{green}
					\idfigparam{(nd-1-3)}{0.25}{green}
					\draw[orange] (-0.25,0) -- ++(0,0.25) -- (nd-1-3) -- (0.75,0.5) -- (0.75,1);
				\end{boxedstrings},
			\end{equation}
			
			\begin{equation}\label{orange jump through GBB}
				\begin{boxedstrings}
					\trivuparamcol{(0.25,0.75)}{0.5}{0.5}{green}{brown}{brown}
					\idfigparam{(nd-1-1)}{-0.25}{green}
					\idfigparam{(nd-1-2)}{0.25}{brown}
					\idfigparam{(nd-1-3)}{0.25}{brown}
					\draw[orange] (-0.25,0) -- ++(0,0.25) -- (nd-1-1) -- (0.75,0.75) -- (0.75,1);
				\end{boxedstrings}
				=
				\begin{boxedstrings}
					\trivuparamcol{(0.25,0.75)}{0.5}{0.5}{green}{brown}{brown}
					\idfigparam{(nd-1-1)}{-0.25}{green}
					\idfigparam{(nd-1-2)}{0.25}{brown}
					\idfigparam{(nd-1-3)}{0.25}{brown}
					\draw[orange] (-0.25,0) -- ++(0,0.25) -- (nd-1-3) -- (0.75,0.5) -- (0.75,1);
				\end{boxedstrings},
			\end{equation}
			
			\begin{equation}\label{orange jump through BBB}
				\begin{boxedstrings}
					\trivuparamcol{(0.25,0.75)}{0.5}{0.5}{brown}{brown}{brown}
					\idfigparam{(nd-1-1)}{-0.25}{brown}
					\idfigparam{(nd-1-2)}{0.25}{brown}
					\idfigparam{(nd-1-3)}{0.25}{brown}
					\draw[orange] (-0.25,0) -- ++(0,0.25) -- (nd-1-1) -- (0.75,0.75) -- (0.75,1);
				\end{boxedstrings}=
				\begin{boxedstrings}
					\trivuparamcol{(0.25,0.75)}{0.5}{0.5}{brown}{brown}{brown}
					\idfigparam{(nd-1-1)}{-0.25}{brown}
					\idfigparam{(nd-1-2)}{0.25}{brown}
					\idfigparam{(nd-1-3)}{0.25}{brown}
					\draw[orange] (-0.25,0) -- ++(0,0.25) -- (nd-1-3) -- (0.75,0.5) -- (0.75,1);
				\end{boxedstrings}
				.
			\end{equation}
		\end{minipage}
	\end{center}
\end{prop}
\begin{proof}
	The proof of \eqref{orange slide over orange green trivalent 1}-\eqref{orange slide over orange green trivalent 3} are all similar. We will only show \eqref{orange slide over orange green trivalent 1} and leave the rest as an exercise for the reader.
	By \eqref{GGB orange 1} and \eqref{orange slide horizontal green},
	\begin{equation}\label{orange slide over orange green trivalent 1 proof}
		\begin{boxedstrings}
			\trivuparamcol{(0.25,0.75)}{0.5}{0.5}{green}{green}{orange}
			\idfigparam{(nd-1-1)}{-0.25}{green}
			\idfigparam{(nd-1-2)}{0.25}{green}
			\idfigparam{(nd-1-3)}{0.25}{orange}
			\draw[orange] (-0.25,0) -- ++(0,0.25) -- (nd-1-1) -- (0.75,0.75) -- (0.75,1);
		\end{boxedstrings}
		=
		\begin{boxedstrings}
			\trivuparamcol{(0.25,0.75)}{0.5}{0.5}{green}{green}{orange}
			\idfigparam{(nd-1-1)}{-0.25}{green}
			\idfigparam{(nd-1-2)}{0.25}{green}
			\idfigparam{(nd-1-3)}{0.25}{orange}
			\draw[orange] (-0.25,0) -- ++(0,0.25) -- (nd-1-2);
			\draw[orange] (nd-1-1) -- (0.75,0.75) -- (0.75,1);
		\end{boxedstrings}
		=
		\begin{boxedstrings}
			\trivuparamcol{(0.25,0.75)}{0.5}{0.5}{green}{green}{orange}
			\idfigparam{(nd-1-1)}{-0.25}{green}
			\idfigparam{(nd-1-2)}{0.25}{green}
			\idfigparam{(nd-1-3)}{0.25}{orange}
			\draw[orange] (-0.25,0) -- ++(0,0.25) -- (nd-1-2) -- (nd-1-3) -- 
			++(0.25,0.25) -- (0.75,1);
		\end{boxedstrings}
		.
	\end{equation}

	For \eqref{orange jump through GGB}, let us use \eqref{GGB orange 1} and \eqref{GGB triv orange 2} to write
	\begin{equation}
		\begin{boxedstrings}
			\trivuparamcol{(0.25,0.75)}{0.5}{0.5}{brown}{green}{green}
			\idfigparam{(nd-1-1)}{-0.25}{brown}
			\idfigparam{(nd-1-2)}{0.25}{green}
			\idfigparam{(nd-1-3)}{0.25}{green}
			\draw[orange] (-0.25,0) -- ++(0,0.25) -- (nd-1-3) -- (0.75,0.5) -- (0.75,1);
		\end{boxedstrings}
		=
		-
		\begin{boxedstrings}
			\trivuparamcol{(0.25,0.75)}{0.5}{0.5}{brown}{green}{green}
			\idfigparam{(nd-1-1)}{-0.25}{brown}
			\idfigparam{(nd-1-2)}{0.25}{green}
			\idfigparam{(nd-1-3)}{0.25}{green}
			\draw[orange] (-0.25,0) -- ++(0,0.25) -- (nd-1-2);
			\draw[orange] (nd-1-3) -- (0.75,0.5) -- (0.75,1);
		\end{boxedstrings}
		=
		\begin{boxedstrings}
			\trivuparamcol{(0.25,0.75)}{0.5}{0.5}{brown}{green}{green}
			\idfigparam{(nd-1-1)}{-0.25}{brown}
			\idfigparam{(nd-1-2)}{0.25}{green}
			\idfigparam{(nd-1-3)}{0.25}{green}
			\draw[orange] (0.125,0.375) -- ++(0,0.25) -- (nd-1-1) -- (0.75,0.75) -- (0.75,1);
			\draw[orange] (-0.25,0) -- ++(0,0.25) -- (nd-1-2);
		\end{boxedstrings}
		=
		\begin{boxedstrings}
			\trivuparamcol{(0.25,0.75)}{0.5}{0.5}{brown}{green}{green}
			\idfigparam{(nd-1-1)}{-0.25}{brown}
			\idfigparam{(nd-1-2)}{0.25}{green}
			\idfigparam{(nd-1-3)}{0.25}{green}
			\draw[orange] (-0.25,0) -- ++(0,0.25) -- (nd-1-1) -- (0.75,0.75) -- (0.75,1);
		\end{boxedstrings}.
	\end{equation}
	The proof of \eqref{orange jump through GGG} is similar.
	
	For \eqref{orange jump through GBB}, use \eqref{vertex definitions}, \eqref{orange jump through GGG}, and \eqref{orange jump through GGB} to write
	\begin{equation}
		\begin{boxedstrings}
			\trivuparamcol{(0.25,0.75)}{0.5}{0.5}{green}{brown}{brown}
			\idfigparam{(nd-1-1)}{-0.25}{green}
			\idfigparam{(nd-1-2)}{0.25}{brown}
			\idfigparam{(nd-1-3)}{0.25}{brown}
			\draw[orange] (-0.25,0) -- ++(0,0.25) -- (nd-1-1) -- (0.75,0.75) -- (0.75,1);
		\end{boxedstrings}
		=
		\begin{boxedstrings}
			\Triangucol{(0.5,1)}{green}{green}{green}{green}{brown}{brown}
			\draw[orange] (0,0) -- ++(0,0.5) -- (0.5,0.875) -- ++(0.5,0) -- ++(0,0.125);
		\end{boxedstrings}
		=
		\begin{boxedstrings}
			\Triangucol{(0.5,1)}{green}{green}{green}{green}{brown}{brown}
			\draw[orange] (0,0) -- ++(0,0.5) -- ++(1,0) -- ++(0,0.5);
		\end{boxedstrings}
		=
		\begin{boxedstrings}
			\Triangucol{(0.5,1)}{green}{green}{green}{green}{brown}{brown}
			\draw[orange] (0,0) -- ++(0,0.125) -- ++(1,0) -- ++(0,0.875);
		\end{boxedstrings}
		=
		\begin{boxedstrings}
			\trivuparamcol{(0.25,0.75)}{0.5}{0.5}{green}{brown}{brown}
			\idfigparam{(nd-1-1)}{-0.25}{green}
			\idfigparam{(nd-1-2)}{0.25}{brown}
			\idfigparam{(nd-1-3)}{0.25}{brown}
			\draw[orange] (-0.25,0) -- ++(0,0.25) -- (nd-1-3) -- (0.75,0.5) -- (0.75,1);
		\end{boxedstrings}
		.
	\end{equation}
	
	To show \eqref{orange jump through BBB} we must show the relation
	\begin{equation}\label{BBB triangle}
		\begin{boxedstrings}
			\Triangucol{(0,0)}
			{brown}{green}{green}{brown}{brown}{brown}
		\end{boxedstrings}
		=
		2
		\begin{boxedstrings}
			\trivuparamcol{(0,0)}{0.5}{0.5}{brown}{brown}{brown}
			\idfigparam{(nd-1-1)}{-0.25}{brown}
			\idfigparam{(nd-1-2)}{0.25}{brown}
			\idfigparam{(nd-1-3)}{0.25}{brown}
		\end{boxedstrings}
	\end{equation}
	and then use relations \eqref{orange jump through GGG} and \eqref{orange jump through GBB}.
	We leave the verification of \eqref{BBB triangle} as an exercise to the reader
	(use \eqref{H=I brown bottom feet}).
\end{proof}

As bivalent vertices and crossings are defined in terms of trivalent vertices and dots,
The following corollary follows directly from Proposition \ref{prop orange jumps over green brown dots} and Proposition \ref{prop orange jumps through trivalents}.
\begin{coro}
	The orange strand slides over trivalent vertices and crossings.
\end{coro}

\begin{notation}
	From now on, we will use purple strands
	$
	\begin{miniunboxedstrings}
		\idfigparam{(0,0)}{1}{purple}
	\end{miniunboxedstrings}
	$
	as a shorthand for places where we can place a brown or green strand indistinctively, 
	and black strands
	$
	\begin{miniunboxedstrings}
		\idfigparam{(0,0)}{1}{black}
	\end{miniunboxedstrings}
	$
	for places where we can place an orange, brown, or green strand indistinctively.
\end{notation}

Adding all the results in this section we have the following theorem.
\begin{theorem}\label{thm orange sliding}
	The orange strand slides over any diagram. 
	That is, for a diagram $D$ with $k$ boundary strands
	\begin{equation}
		\begin{boxedstrings}
			\draw[dashed] (0.75,0.625) rectangle ++(1,0.5);
			\node (A) at (1.25,0.875) {$D$};
			\node (B) at (1.25,0.25) {$\overset{k}{...}$};
			\draw[orange] (0.625,0)-- ++(0,0.5) -- ++(1.25,0) -- ++(0,0.75) -- ++(-1.25,0) -- ++(0,0.25);
			\foreach \x in {0.9,1.05,1.45,1.6}
			{
				\draw (\x,0) -- ++(0,0.625);
			};
		\end{boxedstrings}
		=
		\begin{boxedstrings}
			\draw[dashed] (0.75,0.625) rectangle ++(1,0.5);
			\node (A) at (1.25,0.875) {$D$};
			\node (B) at (1.25,0.25) {$\overset{k}{...}$};
			\idfigparam{(0.625,1.5)}{1.5}{orange}
			\foreach \x in {0.9,1.05,1.45,1.6}
			{
					\draw (\x,0) -- ++(0,0.625);
			};
		\end{boxedstrings}
		.
	\end{equation}
	In particular, when $k=0$,
	\begin{equation}
		\begin{boxedstrings}
			\draw[dashed] (0,0.25) rectangle ++(0.5,0.5);
			\node (A) at (0.25,0.5) {$D$};
			\idfigparam{(-0.25,1)}{1}{orange}
		\end{boxedstrings}
		=
		\begin{boxedstrings}
			\draw[dashed] (0,0.25) rectangle ++(0.5,0.5);
			\node (A) at (0.25,0.5) {$D$};
			\idfigparam{(0.75,1)}{1}{orange}
		\end{boxedstrings}
		.
	\end{equation}
\end{theorem}

\begin{defi}\label{defi encircling}
	Let $D$ be a diagram.
	Define the \textit{orange encircling of $D$},
	$\encircO(D)$, to be the diagram
	obtained from $D$ in the following way:
	\begin{enumerate}
		\item Start with an annulus $A$ consisting of an orange circle around its center.
		\item For each boundary strand of $D$,
		 add (sequentially) to $A$ a strand of the same color and
		 an orange crossing connecting the outter and inner boundary
		 of $A$.
		\item Paste $D$ in the center of $A$ (noting that its inner boundary is the same as the boundary of $D$).
	\end{enumerate}
	That is, add an orange circle encircling $D$.
	Diagrammatically, 
	\begin{equation}\label{encircle diagram}
			\encircO(D)=
			\begin{unboxedstrings}
				\node[circle,draw,dashed,
				minimum size = 2.5cm] (p) at (0,0) {};
				\node[circle,draw,dashed,
				minimum size = 1.5cm] (q) at (0,0) {$D$};
				\foreach \x in {0,1,2,...,11}
				{
					\tikzmath{\z=360/10*\x-90;}
					\draw[black] (p.\z) -- (q.\z);
				};
			\node[circle,draw,minimum size = 20mm, orange] (r) at (0,0) {};
			\end{unboxedstrings}
		,
	\end{equation}
	where the number and color of the boundary strands corresponds exactly with the boundary of $D$.
	In particular, if $D$ has empty boundary then
	\begin{equation}\label{encircle diagram 2}
		\encircO(D)=
		\begin{unboxedstrings}
			\node[circle,draw,dashed,
			minimum size = 2.5cm] (p) at (0,0) {};
			\node[circle,draw,dashed,
			minimum size = 1.5cm] (q) at (0,0) {$D$};
			\node[circle,draw,minimum size = 20mm, orange] (r) at (0,0) {};
		\end{unboxedstrings}
		.
	\end{equation}
	We say that a diagram $D$ is an 
	\textit{orange 
		 encircled diagram}
	if there exists a diagram $D'$ contained in $D$ such that $\encircO(D')$ 
		is exactly $D$.
\end{defi}

Using all the slide relations shown in Proposition \ref{prop orange jumps through trivalents} along with the slide relations 
\eqref{orange jump dot orange rel}-\eqref{orange brown uncross},
we can conclude that the orange strands slide over any diagram.
Hence, we have the following corollary.
\begin{coro}[Circle Sliding]\label{coro circle sliding}
	Let $D$ be a diagram with no orange boundary strands.
	Then $\encircO(D)=D$. 
	Hence, erasing closed orange circles from a diagram $D$ does not change the underlying morphism in $\Dcal$.
\end{coro}
\begin{proof}
	Let $D$ be a diagram without any orange boundary strands.
	By Theorem \ref{thm orange sliding}, we can slide the orange circle on the outside 
	of $\encircO(D)$ to a region where it does not intersect $D$,
	and then use the orange circle relation to erase it.
	Diagrammatically,
	\begin{equation}
		\encircO(D)=
		\begin{unboxedstrings}
			\node[circle,draw,dashed,
			minimum size = 2.5cm] (p) at (0,0) {};
			\node[circle,draw,dashed,
			minimum size = 1cm] (q) at (0,0) {$D$};
			\foreach \x in {0,1,2,...,5}
			{
				\tikzmath{\z=360/6*\x-90;}
				\draw[black] (p.\z) -- (q.\z);
			};
			\node[circle,draw,minimum size = 20mm, orange] (r) at (0,0) {};
		\end{unboxedstrings}
		=
		\begin{unboxedstrings}
			\node[circle,draw,dashed,
			minimum size = 2.5cm] (p) at (0,0) {};
			\node[circle,draw,dashed,
			minimum size = 1cm] (q) at (0,0) {$D$};
			\foreach \x in {0,1,2,...,5}
			{
				\tikzmath{\z=360/6*\x-90;}
				\draw[black] (p.\z) -- (q.\z);
			};
			\node[circle,draw,minimum size = 5mm, orange] (r) at (-0.85,0) {};
		\end{unboxedstrings}
	= D.
	\end{equation}
\end{proof}

	\begin{defi}\label{crossbar diagram}
		A \textit{crossbar diagram} is a diagram made exclusively of two trivalent vertices joined by a single strand and where the other 4 strands are in the boundary. That is, a diagram of the form
		\begin{equation}
			\begin{boxedstrings}
				\Hlongparam{(0,0)}{0.5}{1}{black}{black}{black}{black}{black}
			\end{boxedstrings}.
		\end{equation}
		A crossbar made only with GB trivalent vertices will be called a \textit{GB crossbar}. That is, a diagram of the form
		\begin{equation}
			\begin{boxedstrings}
				\Hlongparam{(0,0)}{0.5}{1}{purple}{purple}{purple}{purple}{purple}
			\end{boxedstrings}
		.
		\end{equation}
		
		A \textit{GB comb} is a diagram made by starting with a GB crossbar and any adding number of non-crossing orange strands,
		each intersecting the crossbar once on the middle edge (at a crossing or landing).
		Diagrammatically, a \textit{GB comb} either of the form
		\begin{equation}
			\begin{boxedstrings}
				\Hfigparamcol{(0,1)}{1}{1}{purple}{purple}{brown}{purple}{purple}
				\foreach \x in {0.3,0.7}
				{
					\draw[orange] (\x,0) -- ++(0,1);
				};
				\node (A) at (0.5,0.25) {$\overset{k}{...}$};
			\end{boxedstrings}
			,
		\end{equation}
		for some $k\in \NN$, or obtained from
		\begin{equation}
			\begin{boxedstrings}
				\Hlongparam{(0,1)}{2.2}{1}{purple}{purple}{green}{purple}{purple}
				\foreach \x in {2.5,7.5}
				{
					\tikzmath{\z=\x*0.1+1.3;}
					\draw[orange] (\z,1) -- (\z,0.5);
				};
				\node (A) at (1.8,0.75) {$\overset{k}{...}$};
				\foreach \x in {2.5,7.5}
				{
					\tikzmath{\z=\x*0.1-0.1;}
					\draw[orange] (\z,0) -- (\z,0.5);
				};
				\node (B) at (0.4,0.25) {$\overset{m}{...}$};
				\foreach \x in {2.5,7.5}
				{
					\tikzmath{\z=\x*0.1+0.6;}
					\draw[orange] (\z,0) -- (\z,1);
				};
				\node (C) at (1.1,0.75) {$\overset{\ell}{...}$};
			\end{boxedstrings}
			,
		\end{equation}
		for some $k,\ell,m\in\NN$, by permuting its orange crossings and landings.
		A \textit{single GB comb} is a GB comb only has one orange crossing or landing.
		
	\end{defi}

	An important part of the proof of Lemma \ref{cycle reduction} requires
	us to apply and $H=I$ relation on GB cycles (until they eventually become needles).
	To be able to apply an $H=I$ relation, we need a diagram containing a GB crossbar.
	To make a GB comb into a diagram containing a GB crossbar,
	we need to remove all the orange crossings and landings out of its central edge.
	The following lemma shows that we can do for most GB combs.
	\begin{lemma}\label{GB comb sliding}
		The following relations hold:
		\begin{equation}\label{purple crossbar slide crossing}
			\begin{boxedstrings}
				\Hlongparam{(0,1)}{0.5}{1}{purple}{purple}{purple}{purple}{purple}
				\idfigparam{(0.25,1)}{1}{orange}
			\end{boxedstrings}
			=
			\begin{boxedstrings}
				\Hlongparam{(0,1)}{0.5}{1}{purple}{purple}{purple}{purple}{purple}
				\draw[orange] (0.25,1) -- (0.25, 0.75) --
				(-0.25,0.75) -- (-0.25,0.25) -- 
				(0.25,0.25) -- (0.25,0);
			\end{boxedstrings}
			.
		\end{equation}
		
		\begin{equation}\label{purple crossbar slide landing}
			\begin{boxedstrings}
				\Hlongparam{(0,1)}{0.5}{1}{green}{purple}{green}{purple}{purple}
				\idfigparam{(0.25,1)}{0.5}{orange}
			\end{boxedstrings}
			=
			\pm
			\begin{boxedstrings}
				\Hlongparam{(0,1)}{0.5}{1}{green}{purple}{green}{purple}{purple}
				\draw[orange] (0.25,1) -- (0.25, 0.75) --
				(0,0.75);
			\end{boxedstrings},
			\text{ and }
			\begin{boxedstrings}
				\Hlongparam{(0,1)}{0.5}{1}{green}{purple}{green}{purple}{purple}
				\idfigparam{(0.25,0)}{-0.5}{orange}
			\end{boxedstrings}
			=
			\pm
			\begin{boxedstrings}
				\Hlongparam{(0,1)}{0.5}{1}{green}{purple}{green}{purple}{purple}
				\draw[orange] (0.25,0) -- (0.25, 0.25) --
				(-0.25,0.25) -- (-0.25,0.75) -- (0,0.75);
			\end{boxedstrings},
		\end{equation}
		where the sign is positive if the bottom left strand is green and negative it its brown.
		Hence, any single GB comb that is not of the form
		\begin{equation}
			\begin{boxedstrings}
				\Hlongparam{(0,1)}{0.5}{1}{brown}{brown}{green}{brown}{brown}
				\idfigparam{(0.25,1)}{0.5}{orange}
			\end{boxedstrings} 
		\end{equation}
		(or its rotation), is equal to a diagram 
		(with exactly two GB trivalent vertices) 
		containing a GB crossbar.		
	\end{lemma} 
	
	\begin{rema}
		Let $D$ be a $GB$ comb.
		By Lemma \ref{lemma orange landings green edge} and Lemma \ref{lemma orange crossings brown edge},
		every GB comb can be transformed into the union of a single GB comb or GB crossbar $E$ and some number of orange cups and caps.
		By Lemma \ref{GB comb sliding}, if $E$ is not of the form 
		$
			\begin{miniunboxedstrings}
				\Hlongparam{(0,0.8)}{0.5}{0.8}{brown}{brown}{green}{brown}{brown}
				\idfigparam{(0.25,0.8)}{0.4}{orange}
			\end{miniunboxedstrings}
			, 
		$
		equals a diagram containing a GB crossbar.
	\end{rema}

	The only tool we have left to define is the restriction to one or more colors. This operation erases all polynomials and edges that are not of that color and is defined as follows.
	
	\begin{defi}\label{GB Restriction}
		Let $\Res_{GB}$ be the operation diagrams of $\Dcal$ sending all the generators on green and brown colors to themselves and sending
		\begin{equation}
			\begin{gathered}
				\begin{boxedstrings}
					\polybox{(0,0.5)}{$f$}
					\path (0,0) ++ (0,1);
				\end{boxedstrings}
				\mapsto
				\begin{boxedstrings}
					\path (0,0) ++ (0.5,1);
				\end{boxedstrings},
				\begin{boxedstrings}
					\dotuparam{(0,0)}{0.6}{0.06}{orange}
					\path (nd-1-1) ++ (0,1);
				\end{boxedstrings}
				\mapsto
				\begin{boxedstrings}
					\path (0,0) ++ (0.5,1);
				\end{boxedstrings},
				\begin{boxedstrings}
					\idfig{(0,0)}{orange}
				\end{boxedstrings}
				\mapsto
				\begin{boxedstrings}
					\path (0,0) ++ (0.5,1);
				\end{boxedstrings}
				,
				\\
				\begin{boxedstrings}
					\trivuparamcol{(0,0)}{0.5}{0.5}{green}{orange}{green}
					\idfigparam{(nd-1-1)}{-0.25}{green}
					\idfigparam{(nd-1-2)}{0.25}{orange}
					\idfigparam{(nd-1-3)}{0.25}{green}
				\end{boxedstrings}
				\mapsto
				\begin{boxedstrings}
					\idfig{(0,0)}{green}
				\end{boxedstrings}
				,
				\begin{boxedstrings}
					\Xfig{(0,0)}{brown}{orange}
					\idfigparam{(nd-1-1)}{-0.25}{brown}
					\idfigparam{(nd-1-2)}{-0.25}{orange}
					\idfigparam{(nd-1-4)}{0.25}{brown}
					\idfigparam{(nd-1-3)}{0.25}{orange}
				\end{boxedstrings}
				\mapsto
				\begin{boxedstrings}
					\idfig{(0,0)}{brown}
				\end{boxedstrings}
				,
				\begin{boxedstrings}
					\Xfig{(0,0)}{green}{orange}
					\idfigparam{(nd-1-1)}{-0.25}{green}
					\idfigparam{(nd-1-2)}{-0.25}{orange}
					\idfigparam{(nd-1-4)}{0.25}{green}
					\idfigparam{(nd-1-3)}{0.25}{orange}
				\end{boxedstrings}
				\mapsto
				\begin{boxedstrings}
					\idfig{(0,0)}{green}
				\end{boxedstrings}
				,
				\begin{boxedstrings}
					\idfigparam{(0,1)}{0.5}{orange}
					\idfigparam{(nd-1-2)}{0.5}{green}
				\end{boxedstrings}
				\mapsto
				\begin{boxedstrings}
					\dotuparam{(0,0)}{0.6}{0.06}{green}
					\path (nd-1-1) ++ (0,1);
				\end{boxedstrings}
				.
			\end{gathered}
		\end{equation}
		This operation takes a diagram and erases all orange strands, orange dots and polynomials.
		We will call this operation the \textit{green-brown restriction}. 
		We will refer to the connected components of $\ResGB(D)$ as its \textit{GB components}.
		
		Let $\Res_{G}$ be the diagrams of $\Dcal$ sending all the green generators to themselves and sending
		\begin{equation}
			\begin{gathered}
				\begin{boxedstrings}
					\dotuparam{(0,0)}{0.6}{0.06}{orange}
					\path (nd-1-1) ++ (0,1);
				\end{boxedstrings}
				\mapsto
				\begin{boxedstrings}
					\path (0,0) ++ (0.5,1);
				\end{boxedstrings}
				,
				\begin{boxedstrings}
					\idfig{(0,0)}{orange}
				\end{boxedstrings}
				\mapsto
				\begin{boxedstrings}
					\path (0,0) ++ (0.5,1);
				\end{boxedstrings}
				,
				\begin{boxedstrings}
					\dotuparam{(0,0)}{0.6}{0.06}{brown}
					\path (nd-1-1) ++ (0,1);
				\end{boxedstrings}
				\mapsto
				\begin{boxedstrings}
					\path (0,0) ++ (0.5,1);
				\end{boxedstrings}
				,
				\\
				\begin{boxedstrings}
					\idfig{(0,0)}{brown}
				\end{boxedstrings}
				\mapsto
				\begin{boxedstrings}
					\path (0,0) ++ (0.5,1);
				\end{boxedstrings}
				,
				\begin{boxedstrings}
					\trivu{(0,0)}{brown}
					\idfigparam{(nd-1-1)}{-0.25}{brown}
					\idfigparam{(nd-1-2)}{0.25}{brown}
					\idfigparam{(nd-1-3)}{0.25}{brown}
				\end{boxedstrings}
				\mapsto
				\begin{boxedstrings}
					\path (0,0) ++ (0.5,1);
				\end{boxedstrings}
				,
				\begin{boxedstrings}
					\Xfig{(0,0)}{brown}{orange}
					\idfigparam{(nd-1-1)}{-0.25}{brown}
					\idfigparam{(nd-1-2)}{-0.25}{orange}
					\idfigparam{(nd-1-4)}{0.25}{brown}
					\idfigparam{(nd-1-3)}{0.25}{orange}
				\end{boxedstrings}
				\mapsto
				\begin{boxedstrings}
					\path (0,0) ++ (0.5,1);
				\end{boxedstrings}
				,
				\\
				\begin{boxedstrings}
					\idfigparam{(0,1)}{0.5}{orange}
					\idfigparam{(nd-1-2)}{0.5}{green}
				\end{boxedstrings}
				\mapsto
				\begin{boxedstrings}
					\dotuparam{(0,0)}{0.6}{0.06}{green}
					\path (nd-1-1) ++ (0,1);
				\end{boxedstrings}
				,
				\begin{boxedstrings}
					\trivuparamcol{(0,0)}{0.5}{0.5}{brown}{brown}{green}
					\idfigparam{(nd-1-1)}{-0.25}{brown}
					\idfigparam{(nd-1-2)}{0.25}{brown}
					\idfigparam{(nd-1-3)}{0.25}{green}
				\end{boxedstrings}.
				\mapsto
				\begin{boxedstrings}
					\dotuparam{(0,0)}{0.6}{0.06}{green}
					\path (nd-1-1) -- ++(0,1);
				\end{boxedstrings}
				,
				\begin{boxedstrings}
					\idfigparam{(0,1)}{0.5}{brown}
					\idfigparam{(nd-1-2)}{0.5}{green}
				\end{boxedstrings}
				\mapsto
				\begin{boxedstrings}
					\dotuparam{(0,0)}{0.6}{0.06}{green}
					\path (nd-1-1) -- ++(0,1);
				\end{boxedstrings}
				,
				\\		
				\begin{boxedstrings}
					\trivuparamcol{(0,0)}{0.5}{0.5}{green}{brown}{green}
					\idfigparam{(nd-1-1)}{-0.25}{green}
					\idfigparam{(nd-1-2)}{0.25}{brown}
					\idfigparam{(nd-1-3)}{0.25}{green}
				\end{boxedstrings}
				\mapsto
				\begin{boxedstrings}
					\idfig{(0,0)}{green}
				\end{boxedstrings}
				,
				\begin{boxedstrings}
					\trivuparamcol{(0,0)}{0.5}{0.5}{green}{orange}{green}
					\idfigparam{(nd-1-1)}{-0.25}{green}
					\idfigparam{(nd-1-2)}{0.25}{orange}
					\idfigparam{(nd-1-3)}{0.25}{green}
				\end{boxedstrings}
				\mapsto
				\begin{boxedstrings}
					\idfig{(0,0)}{green}
				\end{boxedstrings}
				,
				\begin{boxedstrings}
					\Xfig{(0,0)}{green}{orange}
					\idfigparam{(nd-1-1)}{-0.25}{green}
					\idfigparam{(nd-1-2)}{-0.25}{orange}
					\idfigparam{(nd-1-4)}{0.25}{green}
					\idfigparam{(nd-1-3)}{0.25}{orange}
				\end{boxedstrings}
				\mapsto
				\begin{boxedstrings}
					\idfig{(0,0)}{green}
				\end{boxedstrings}
				,
				\\
				\begin{boxedstrings}
					\polybox{(0,0.5)}{$f$}
					\path (0,0) ++ (0,1);
				\end{boxedstrings}
				\mapsto
				\begin{boxedstrings}
					\path (0,0) ++ (0.5,1);
				\end{boxedstrings}
				,.
			\end{gathered}
		\end{equation}
		This operation takes a diagram and erases all orange or brown strands, all orange or brown dots, and all polynomials.
		We will call this operation the \textit{green restriction}.
		We will refer to the connected components of $\Res_{G}(D)$ as its \textit{green components}.
		
		For a given diagram $D$, let $\nuGB{D}$ (resp. $\nuG{D}$) be the embedding of the graph of $\ResGB(D)$ (resp. $\Res_G(D)$) into $D$.
		
		Let $S$ be a subset of the graph of $\ResGB(D)$ (resp. $\Res_G(D)$)
		and $t$ a point in $D$ corresponding to an orange crossing or landing.
		We say \textit{$t$ is over $S$} if $t\in \nuGB{D}(S)$ (resp. $t\in \nuG{D}(S)$).
	\end{defi}
	\begin{rema}
		The operations $\Res_{G}$ and $\ResGB$ does not define a functor. 
		For instance,
		\begin{equation}
			\Res_{G}
			\left(
			\begin{boxedstrings}
				\vbarb{(0,0.7)}{green}
				\path (0,0) ++ (0,1);
			\end{boxedstrings}
			\right)
			=
			\begin{boxedstrings}
				\vbarb{(0,0.7)}{green}
				\path (0,0) ++ (0,1);
			\end{boxedstrings}
		\end{equation}
		but
		\begin{equation}
			\Res_{G}
			\left(
			\begin{boxedstrings}
				\polybox{(0,0.5)}{$\alpha_s$}
				\path (0,0) ++ (0,1);
			\end{boxedstrings}
			\right)
			=
			\begin{boxedstrings}
				\path (0,0) ++ (0.5,1);
			\end{boxedstrings}
		.
		\end{equation}
		These operations are well defined for a specific diagram, but do not preserve the relations of $\Dcal$.
	\end{rema}
	
	Having defined the green components of a diagram, we can state the following proposition.
	\begin{prop}\label{prop only one landing in each green component}
		Any diagram $D$ equals a diagram $E$ where each green component of $E$ has at most one orange landing over it.
	\end{prop}
	\begin{proof}
		Note that the orange landing slide relations 
		\eqref{orange jump through GGG} and \eqref{GGG orange}-\eqref{GGB triv orange 2} only change the number of
		orange crossings. 
		Because of this, we can slide all the orange landings in the same connected components to a single green edge without creating any new orange landings.
		We can finally use \eqref{orange slide horizontal green} and \eqref{horizontal green cups} to change any pair of orange landings into a crossing or a cup/cap.
	\end{proof}
	
	We will use the following lemma to deal with diagrams that have two or more orange boundary strands.
	\begin{lemma}\label{lemma orange boundary strands}
		Let $D$ be a diagram with two or more orange boundary strands and at most one green or brown boundary strand.
		Then there is a diagram $E$ such that $D=E$ and diagrams $O,N$ contained in $E$
		satisfying:
		\begin{itemize}
			\item $E=O\circ N$.
			\item $\ResGB(D)=\ResGB(E)=\ResGB(N)$.
			\item $N$ has at most one orange boundary strand.
			\item $O$ is a rotation of $\mathsf{Or}_{2k}$ for some $k\in\NN$.
		\end{itemize}
		That is, we can make $D$ into a diagram $E$ that factors into a diagram of the form $\mathsf{Or}_{2k}$ and a diagram with at most one orange boundary strand.
	\end{lemma}
	\begin{proof}
		Let us only consider the case where $\ResGB(D)$ has only one boundary strand, as the case where $\ResGB(D)$ has empty boundary is similar.
		We will use relation \eqref{X2 idempotent decomp} in $D$ to build a diagram $E$, containing a diagram $N$ that has at most one orange boundary strand.
		Note that we can slide any pair of nearby orange boundary strands next to each other	and merge them into 
		$
		\begin{miniunboxedstrings}
			\cupfigparam{(0,1)}{0.5}{0.4}{orange}
			\capfigparam{(0,0)}{0.5}{0.4}{orange}
		\end{miniunboxedstrings}
		$.
		If $D$ has exactly two orange boundary strands, we write
		\begin{equation}\label{orange cups circle diagram}
			D=
			\begin{boxedstrings}
				\draw[dashed] (0,0) rectangle ++(0.75,0.75);
				\node at (0.375,0.375) {$D$};
				\draw[purple] (0.375,0) -- ++(0,-0.5);
				\draw[orange] (0.25,0.75) -- ++(0,0.5);
				\draw[orange] (0.5,0.75) -- ++(0,0.5);
			\end{boxedstrings}
			=
			\begin{boxedstrings}
				\draw[dashed] (0,0) rectangle ++(0.75,0.75);
				\node at (0.375,0.375) {$D$};
				\draw[purple] (0.375,0) -- ++(0,-0.5);
				\capfigparam{(0.25,0.75)}{0.25}{0.20}{orange}
				\cupfigparam{(0.25,1.25)}{0.25}{0.20}{orange}
			\end{boxedstrings}
			.
		\end{equation}
		If we set $E$ to be the diagram on the right in \eqref{orange cups circle diagram}, then we can set
		\begin{equation}
			N :=
			\begin{boxedstrings}
				\draw[dashed] (0,0) rectangle ++(0.75,0.75);
				\node at (0.375,0.375) {$D$};
				\draw[purple] (0.375,0) -- ++(0,-0.5);
				\capfigparam{(0.25,0.75)}{0.25}{0.20}{orange}
				\path (0.25,1.25) -- (0.25,0);
			\end{boxedstrings}
			.
		\end{equation}
		Now it becomes clear that $E$ and $N$ satisfy all the conditions required.
		
		If $D$ has more than two orange strands, we can apply \eqref{orange cups circle diagram} until we are left with at most one orange strands coming out of $D$.
		Let $E$ be the diagram obtained at the end of this process,
		and define $N$ to be the diagram contained in $E$
		that only excludes the edges connecting orange boundary points of $E$
		and the (at most two) boundary edges of $D$ that were not changed.
		Then $N$ has at most one orange boundary strand and $\ResGB(D)=\ResGB(E)=\ResGB(N)$.
		Hence, $E$ and $N$ satisfy all the conditions needed.
	\end{proof}

\subsection{Reduction of Diagrams}\label{subsection-diagram reduction}
In Section \ref{subsection-presentation}, we defined the functor $\Fbold:\Fcal\longrightarrow \Hcaltau$ and showed it
is an equivalence as long as Theorem \ref{diagram reduction} holds.
The goal of this section is to carefully prove this theorem with the use of diagrammatic techniques.

\begin{theorem}[Diagram Reduction]\label{diagram reduction}\phantom{.}
The spaces $\Hom\!\left(\onecalbar,\onecalbar\right)$, $\Hom\!\left(\Xobjbar,\onecalbar\right)$,
$\Hom\!\left(\Yobjbar,\onecalbar\right)$, $\Hom\!\left(\Zobjbar,\onecalbar\right)$,
and $\Hom\!\left(\XZobjbar,\onecalbar\right)$ have the following $\Rtau$-spanning sets:
\begin{enumerate}[(a)]
    \item $\Hom\!\left(\onecalbar,\onecalbar\right)$ is spanned by the empty diagram $\emptydiagram$.

    \item $\Hom\!\left(\Xobjbar,\onecalbar\right)$ is spanned by    
    $\Dotup{orange} $.
    
    \item $\Hom\!\left(\Yobjbar,\onecalbar\right)$ is spanned by 
    $\Dotup{green}$ and
    $
    \begin{miniunboxedstrings}
    	\idfigparam{(1,0)}{-0.3}{green}
    	\dotuparam{(nd-1-2)}{0.3}{0.06}{orange}
    	\path (nd-1-1) -- ++(0,-0.2);
    \end{miniunboxedstrings}
    $
    .

    \item $\Hom\!\left(\Zobjbar,\onecalbar\right)$ is spanned by 
    $ \Dotup{brown} $.

    \item $\Hom\!\left(\XZobjbar,\onecalbar\right)$ is spanned by 
    $\Dotup{orange}\hspace{-0.2cm} \Dotup{brown}
    $
    .
\end{enumerate}

\end{theorem}

\begin{defi}
We say a diagram $D$ is a \textit{tree} if the complement of the underlying graph is connected. That is, and doesn't contain any cycles (it may still have polynomial boxes).

We say that $D$ is a \textit{tree without boundary} if $D$ is a tree and doesn't have any boundary strands (e.g. a barbell).
We say $D$ is a \textit{based tree} if $D$ is a tree and contains a single boundary strand.

We say that $D$ is a \textit{forest} if all of its connected components are trees.
We say that $D$ is a \textit{based forest} if all of its connected components are either based trees or trees without boundary.

We say that $D$ is a \textit{GB tree (resp. based tree, forest, based forest)} if $\Res_{GB}(D)$ is a tree (resp. based tree, forest, based forest).
\end{defi}

Note that a based forest does not partition the strip $[0,1]\times [0,1]$ into more than one region. Hence, every based forest is an $\Rtau$-multiple of a based forest with no polynomials.

\begin{lemma}\label{tree reduction}
	Let $D$ be a GB tree.
	\begin{enumerate}[(a)]
		\item If $D$ has no boundary strands, $D$ is in the $\Rtau$-span of
		$
		\left\{\emptydiagram \right\}
		$.
		
		\item If the boundary of $D$ is one orange strand, $D$ is in the $\Rtau$-span of 
		$
		\left\{\Dotup{orange}\right\}
		$.
		
		\item If the boundary of $D$ is one green strand, $D$ is in the $\Rtau$-span of 
		$\left\{\Dotup{green}, 
		\begin{miniunboxedstrings}
			\idfigparam{(1,0)}{-0.3}{green}
			\dotuparam{(nd-1-2)}{0.3}{0.06}{orange}
			\path (nd-1-1) -- ++(0,-0.2);
		\end{miniunboxedstrings}
		\right\}
		$.
		
		\item If the boundary of $D$ is one green strand and one orange strand, $D$ is in the $\Rtau$-span of 
		$\left\{
		\begin{miniunboxedstrings}
			\idfigparam{(1,0)}{-0.3}{green}
			\dotuparam{(nd-1-2)}{0.3}{0.06}{green}
			\draw[orange] (nd-1-2) -- ++(-0.25,0) -- ++(0,-0.3);
			\path (nd-1-1) -- ++(0,-0.2);
		\end{miniunboxedstrings}
		,
		\Dotup{orange}\hspace{-0.2cm} \Dotup{green}
		\right\}
		$
		.
		
		\item If the boundary of $D$ is one brown strand, $D$ is in the $\Rtau$-span of 
		$\left\{\Dotup{brown}\right\}$.

		\item If the boundary of $D$ is one brown strand and one orange strand, $D$ is in the $\Rtau$-span of  
		$\left\{	\Dotup{orange}\hspace{-0.2cm} \Dotup{brown}
		\right\}
		$.
	\end{enumerate}
\end{lemma}
\begin{proof}
	
	Let $D$ be a green-brown tree whose boundary is either empty,
	an orange strand,
	a green strand, a brown strand,
	an orange and a green strand,
	or an orange and a brown strand.
	By Proposition \ref{prop only one landing in each green component}, we may assume that there is at most one orange landing over each green component of $D$.
	Note that $\ResGB(D)$ has only one region (because it has no cycles),
	so we may suppose that $D$ has no polynomial boxes (as they can be pulled out).
	Also, by \eqref{orange dots merger}, we may suppose $D$ has at most one orange dot. 
	By Corollary \ref{coro circle sliding}, we may assume $D$ has no orange circles.

	This proof will go by induction on the number of trivalent vertices
	of $\ResGB(D)$, after replacing any bivalent vertices by a corresponding trivalent vertex with a dot. 
	When $\ResGB(D)$ has no trivalent vertices, then $\ResGB(D)$ is either empty, 
	$
	\Dotup{green}
	$,
	or
	$
	\Dotup{brown}
	$.
	If $\ResGB(D)$ is empty then $D$ can't have any orange landings.
	Since we assumed that $D$ has at most one orange dot and not orange circles,
	it must be a polynomial or an $\Rtau$-multiple of
	$\Dotup{orange}$.
	Similarly, if $\ResGB(D)$ is a green or a brown dot then $D$
	must be an $\Rtau$-multiple of
	{
		$\Dotup{green}$	,
		$
		\begin{miniunboxedstrings}
			\idfigparam{(1,0)}{-0.3}{green}
			\dotuparam{(nd-1-2)}{0.3}{0.06}{orange}
			\path (nd-1-1) -- ++(0,-0.2);
		\end{miniunboxedstrings}
		$
		,
		$
		\begin{miniunboxedstrings}
			\idfigparam{(1,0)}{-0.3}{green}
			\dotuparam{(nd-1-2)}{0.3}{0.06}{green}
			\draw[orange] (nd-1-2) -- ++(-0.25,0) -- ++(0,-0.3);
			\path (nd-1-1) -- ++(0,-0.2);
		\end{miniunboxedstrings}
		$
		,
		$
		\Dotup{orange}\hspace{-0.2cm} \Dotup{green}
		$
		,
		$\Dotup{brown}$	,
		or 
		$
		\Dotup{orange}\hspace{-0.2cm} \Dotup{brown}
		$.
	}
	This constitutes our base case.

	Now suppose that $\ResGB(D)$ has at least one trivalent vertex.
	Then, the tree $\ResGB(D)$ must have a trivalent vertex with two dots (leaves).
	Let $T$ be a closed neighborhood (in $D$) that only includes this trivalent vertex.
	By Proposition \ref{prop orange jumps through trivalents}, we may slide out any orange crossings on $T$.
	By our previous assumptions, $T$ has at most one orange landing and at most one orange dot.
	Thus, we may assume that $T$ is one of the following:
	\excludeDiagrams{1}
	{
		\begin{equation}
			\begin{boxedstrings}
				\trivdparamcol{(0,0)}{0.5}{0.7}{green}{green}{green}
				\dotfig{(nd-1-1)}{green}
				\dotfig{(nd-1-2)}{green}
				\path (nd-1-3) ++ (0,1);
			\end{boxedstrings}
			=
			\begin{boxedstrings}
				\dotuparam{(0,0)}{0.4}{0.06}{green}
				\path (nd-1-1) ++ (0,1);
			\end{boxedstrings},
		\end{equation}

		\begin{equation}
			\begin{boxedstrings}
				\trivdparamcol{(0,0.7)}{0.5}{0.7}{green}{green}{green}
				\dotfig{(nd-1-1)}{green}
				\dotfig{(nd-1-2)}{green}
				\path (nd-1-3) ++ (0,1);
				\draw[orange] (0,0) -- (0,0.35) -- (0.125,0.525);
			\end{boxedstrings}
			=
			\begin{boxedstrings}
				\dotuparam{(0.25,0)}{0.6}{0.06}{green}
				\draw[orange] (0,0) -- (0,0.25) -- (0.25,0.25);
				\path (nd-1-1) ++ (0,1);
			\end{boxedstrings}
			,
		\end{equation}

		\begin{equation}
			\begin{boxedstrings}
				\trivdparamcol{(0,0)}{0.5}{0.7}{brown}{brown}{brown}
				\dotfig{(nd-1-1)}{brown}
				\dotfig{(nd-1-2)}{brown}
				\path (nd-1-3) ++ (0,1);
			\end{boxedstrings}
			=
			\begin{boxedstrings}
				\dotuparam{(0,0)}{0.4}{0.06}{brown}
				\path (nd-1-1) ++ (0,1);
			\end{boxedstrings},
		\end{equation}
		
		\begin{equation}
			\begin{boxedstrings}
				\trivdparamcol{(0,0)}{0.5}{0.7}{green}{green}{brown}
				\dotfig{(nd-1-1)}{green}
				\dotfig{(nd-1-2)}{green}
				\path (nd-1-3) ++ (0,1);
			\end{boxedstrings}
			= 
			\begin{boxedstrings}
				\trivdparamcol{(0,0)}{0.5}{0.7}{green}{brown}{brown}
				\dotfig{(nd-1-1)}{green}
				\dotfig{(nd-1-2)}{brown}
				\path (nd-1-3) ++ (0,1);
			\end{boxedstrings}
			=2 \ 
			\begin{boxedstrings}
				\dotuparam{(0,0)}{0.4}{0.06}{brown}
				\path (nd-1-1) ++ (0,1);
			\end{boxedstrings},
		\end{equation}

		\begin{equation}
			\begin{boxedstrings}
				\trivdparamcol{(0,0.7)}{0.5}{0.7}{green}{green}{brown}
				\dotfig{(nd-1-1)}{green}
				\dotfig{(nd-1-2)}{green}
				\path (nd-1-3) ++ (0,1);
				\draw[orange] (0,0) -- (0,0.35) -- (0.125,0.525);
			\end{boxedstrings}
			=
			\begin{boxedstrings}
				\trivdparamcol{(0,0.7)}{0.5}{0.7}{green}{brown}{brown}
				\dotfig{(nd-1-1)}{green}
				\dotfig{(nd-1-2)}{brown}
				\path (nd-1-3) ++ (0,1);
				\draw[orange] (0,0) -- (0,0.35) -- (0.125,0.525);
			\end{boxedstrings}
			=
			0,      
		\end{equation}
		
		\begin{equation}
			2 \ 
			\begin{boxedstrings}
				\trivdparamcol{(0,0)}{0.5}{0.7}{brown}{brown}{green}
				\dotfig{(nd-1-1)}{brown}
				\dotfig{(nd-1-2)}{brown}
				\path (nd-1-3) ++ (0,1);
			\end{boxedstrings}
			= 
			2 \ 
			\begin{boxedstrings}
				\trivdparamcol{(0,0)}{0.5}{0.7}{green}{brown}{green}
				\dotfig{(nd-1-1)}{green}
				\dotfig{(nd-1-2)}{brown}
				\path (nd-1-3) ++ (0,1);
			\end{boxedstrings}
			= 
			\begin{boxedstrings}
				\vbarb{(0.75,0.7)}{green}
				\dotuparam{(1,0)}{0.4}{0.06}{green}
				\path (nd-2-1) ++ (0,1);
			\end{boxedstrings}
			-
			\begin{boxedstrings}
				\idfigparam{(1,0)}{-0.25}{green}
				\dotuparam{(nd-1-2)}{0.2}{0.06}{orange}
				\path (nd-1-1) ++ (0,1);
			\end{boxedstrings}, 
		\end{equation}
		
		\begin{equation}
			2 \ 
			\begin{boxedstrings}
				\trivdparamcol{(0,0.7)}{0.5}{0.7}{green}{brown}{green}
				\dotfig{(nd-1-1)}{green}
				\dotfig{(nd-1-2)}{brown}
				\path (nd-1-3) ++ (0,1);
				\draw[orange] (0,0) -- (0,0.35) -- (0.125,0.525);
			\end{boxedstrings}
			=
			-\ 
			\begin{boxedstrings}
				\dotuparam{(0.25,0)}{0.6}{0.06}{green}
				\draw[orange] (0,0) -- (0,0.25) -- (0.25,0.25);
				\path (nd-1-1) ++ (0,1);
				\vbarb{(-0.25,0.7)}{green}
			\end{boxedstrings}
			+
			\begin{boxedstrings}
				\dotuparam{(0,0)}{0.4}{0.06}{orange}
				\dotuparam{(0.25,0)}{0.4}{0.06}{brown}
				\path (nd-1-1) ++ (0,1);
			\end{boxedstrings}.
		\end{equation}
	}  
	
	In each case, we write that trivalent vertex $T$ as a sum of diagrams with no GB trivalent vertices.
	After replacing $T$ by the corresponding expression with no GB trivalent vertices,
	we are left with a diagram $D'$ with one less GB trivalent vertex.
	By induction the result follows.
\end{proof}

\begin{coro}\label{forest based reduction}
	Let $D$ be a GB based forest. Then $D= f_1 D_1 + ... + f_k D_k$ for some $f_i\in \Rtau$ and diagrams $D_i$
	and such that $\ResGB(D_i)$ is a union of boundary dots.
	In particular, if $\ResGB(D)$ has empty boundary, then $D= f \mathsf{Or}_k$ for some $f\in\Rtau$ and $k$ is the number of orange boundary strands of $D$.
\end{coro}
\begin{proof}
	It suffices to prove this when $D$ is a GB based tree, since any based forest can be split into finitely many based trees.
	
	By Lemma \ref{lemma orange boundary strands}, we may suppose that 
	$D$ is the composition of a diagram $N$ with at most one orange boundary strand,
	and a diagram $O$ that is a rotation of $\mathsf{Or}_{2k}$ for some $k\in\NN$.
	By Lemma \ref{tree reduction}, 
	$N$ is in the $\Rtau$-span of the sets
	$
	\left\{\emptydiagram \right\}
	$,
	$
	\left\{\Dotup{orange}\right\}
	$,
	$\left\{\Dotup{green}, 
	\begin{miniunboxedstrings}
		\idfigparam{(1,0)}{-0.3}{green}
		\dotuparam{(nd-1-2)}{0.3}{0.06}{orange}
		\path (nd-1-1) -- ++(0,-0.2);
	\end{miniunboxedstrings}
	\right\}
	$
	,
	$\left\{
	\begin{miniunboxedstrings}
		\idfigparam{(1,0)}{-0.3}{green}
		\dotuparam{(nd-1-2)}{0.3}{0.06}{green}
		\draw[orange] (nd-1-2) -- ++(-0.25,0) -- ++(0,-0.3);
		\path (nd-1-1) -- ++(0,-0.2);
	\end{miniunboxedstrings}
	,
	\Dotup{orange}\hspace{-0.2cm} \Dotup{green}
	\right\}
	$
	,
	$\left\{\Dotup{brown}\right\}$	,
	or 
	$\left\{	\Dotup{orange}\hspace{-0.2cm} \Dotup{brown}
	\right\}
	$.
	Let $N= f_1 D_1 + f_2 D_2$ where $D_1'$ and $D_2'$ are both in one of the generating sets before and $f_1,f_2\in\Rtau$, so that $D= f_1 D_1\circ O + f_2 D_2\circ O$.
	Since $\ResGB(D_1\circ O)=\ResGB(D_1)$ and $\ResGB(D_2\circ O)=\ResGB(D_2)$
	are either empty or a boundary dot,	the result follows.
	
	Suppose now that $D$ is a based GB forest with $k$ orange boundary strands, such that $\ResGB(D)$ has empty boundary.
	Applying our result to $D$, we can write $D= f_1 D_1+...+f_k D_k$ for some $f_i\in \Rtau$ and diagrams $D_i$ such that $\ResGB(D_i)$ is the empty diagram for all $i$.
	Since each $D_i$ is entirely made of orange strands and dots, each $D_i$ must be an $\Rtau$-multiple of $\mathsf{Or}_k$.
	Hence, $D= f \mathsf{Or}_k$ for some $f\in\Rtau$.
\end{proof}

\begin{defi}
An \textit{$n$-cycle with spokes} is a pair of diagrams $(D,D')$ where $D'$ is contained in the interior of $D$ and the annulus
obtained by excising $D'$ from $D$ forms an $n$-cycle
entirely made by:
\begin{enumerate}
    \item Bivalent vertices.
    \item Trivalent vertices with a strand to the boundary of $D$,
    which we will call an \textit{outside spoke}.
    \item Trivalent vertices with a strand to the boundary of $D'$, which we will cal an \textit{inside spoke}.
\end{enumerate}
This implies that the boundary of $D'$ corresponds to the inside spokes. We will refer to $D'$ as the \textit{inside} of $(D,D')$.
Diagramatically,
\excludeDiagrams{1}
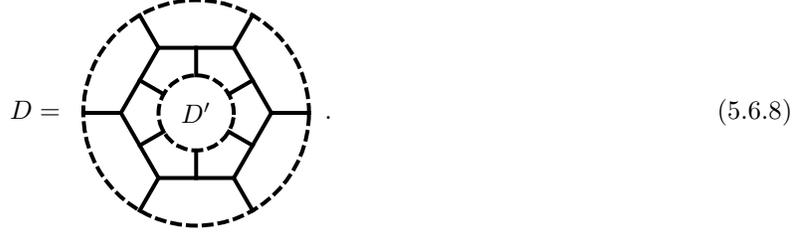
{
\begin{equation}
		D=	
	\begin{unboxedstrings}
		\node[circle,draw,dashed,
		minimum size = 3cm] (p) at (0,0) {};
		\node[circle,
		minimum size = 2cm] (q) at (0,0) {};
		\node[circle,draw, dashed,
		minimum size = 1cm] (t) at (0,0) {};
		\node at (0,0) {$D'$};
		\foreach \x in {0,1,2,...,5}
		{
			\tikzmath{\y=\x+1; \z=2*\x+1;}
			\draw[black] (p.360/6*\x) -- (q.360/6*\x);
			\draw[black] (q.360/6*\x) -- (q.360/6*\y);
			\draw[black] ($0.5*(q.360/6*\x)+0.5*(q.360/6*\y)$) -- (t.360/12*\z);
		};
	\end{unboxedstrings}
	.
\end{equation}
}

We note that when $k=0$, $D$ is of the form
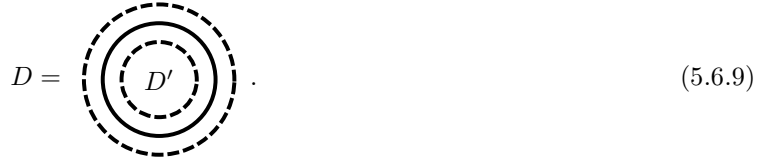
\excludeDiagrams{1}{
	\begin{equation}
		D=	
		\begin{unboxedstrings}
			\node[circle,draw,dashed,
			minimum size = 2cm] (p) at (0,0) {};
			\node[circle,draw,
			minimum size = 1.5cm] (q) at (0,0) {};
			\node[circle,draw, dashed,
			minimum size = 1cm] (t) at (0,0) {};
			\node at (0,0) {$D'$};
		\end{unboxedstrings}
		.
	\end{equation}
}

We say that our cycle with spokes is \textit{empty} if $D'$ is empty.
In particular, $(D,D')$ has no inside spokes.

We say that our cycle with spokes is \textit{innermost} if $D'$ is a based forest.
Abusing notation, we say that $D$ is a cycle with spokes
(without specifying the interior)
if there exists a diagram $D'$ contained in $D$ such that $(D,D')$ is a cycle with spokes.
We say that $D$ is a \textit{GB cycle with spokes 
(resp. empty/innermost cycle with spokes)}
if $(\Res_{GB}(D)$ is a cycle with spokes
(resp. empty/innermost cycle with spokes).
\end{defi}

The argument to prove Theorem \ref{diagram reduction} will go by induction,
first on the total number of GB bivalent and trivalent vertices,
and then on the number of GB components of our diagram 
(given it has the same total number of GB bivalent and trivalent vertices).
For our base case, we need to prove the theorem for diagrams with at most one GB component and at most one GB bivalent or trivalent vertex.
This corresponds to the following lemma.

\begin{lemma}\label{lemma base case}
	Let $D$ be a diagram satisfying the following conditions:
	\begin{enumerate}[(a)]
		\item $D$ is not a GB forest.
		\item $D$ has at most one GB
		and at most one orange boundary strand.
		\item $D$ has at most one trivalent vertex.
		\item $D$ has at most one GB component.
		\item $D$ has no orange circles.
		\item there is at most one orange landing over each green component of $D$.
		\item $D$ has no polynomial boxes in its (only) outer region region.
		\item $D$ has no orange boundary dots and has at most one orange dot
		in its (only) enclosed region.
	\end{enumerate}
	If $D$ has no trivalent vertices, then $D$ is of the form
	
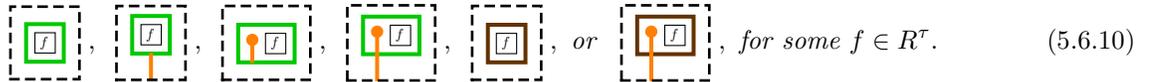
\begin{align}\label{eqn cases circles}
		\begin{boxedstrings}
			\draw[green] (-0.25,0.25) rectangle ++(0.5,0.5);
			\path (0,0) -- ++(0,1);
			\polyboxscale{(0,0.5)}{$f$}{0.9}
		\end{boxedstrings}
		,
		\hspace{0.1cm}
		\begin{boxedstrings}
			\draw[green] (-0.25,0.35) rectangle ++(0.5,0.5);
			\idfigparam {(0,0.35)}{0.35}{orange}
			\path (nd-1-2) -- ++(0,1);
			\polyboxscale{(0,0.6)}{$f$}{0.9}
		\end{boxedstrings}
		,
		\hspace{0.1cm}
		\begin{boxedstrings}
			\node at (0,1) {};
			\draw[green] (0.25,0.25) rectangle ++(-0.75,0.5);			
			\path (0,0) -- ++(0,1);
			\polyboxscale{(0,0.5)}{$f$}{0.9}
			\dotu{(-0.3125,0.25)}{orange}
		\end{boxedstrings}
		,
		\hspace{0,1cm}
		\begin{boxedstrings}
			\node at (0,1) {};
			\draw[green] (0.25,0.35) rectangle ++(-0.75,0.5);			
			\path (0,0) -- ++(0,1);
			\polyboxscale{(0,0.6)}{$f$}{0.9}
			\dotu{(-0.3125,0.35)}{orange}
			\idfigparam {(nd-1-1)}{0.35}{orange}	
		\end{boxedstrings}
		,
		\hspace{0.1cm}
		\begin{boxedstrings}
			\node at (0,1) {};
			\draw[brown] (-0.25,0.25) rectangle ++(0.5,0.5);
			\path (0,0) -- ++(0,1);
			\polyboxscale{(0,0.5)}{$f$}{0.9}
		\end{boxedstrings}
		,
		\text{ or }
		\hspace{0.1cm}
		\begin{boxedstrings}
			\node at (0,1) {};
			\draw[brown] (0.25,0.35) rectangle ++(-0.75,0.5);			
			\path (0,0) -- ++(0,1);
			\polyboxscale{(0,0.6)}{$f$}{0.9}
			\dotu{(-0.3125,0.35)}{orange}
			\idfigparam {(-0.3125,0.35)}{0.35}{orange}	
		\end{boxedstrings}
		,
		\text{ for some }
		f\in\Rtau.
	\end{align}
	If $D$ has exactly one trivalent vertex, then $D$ is (up to rotation) of the form
	
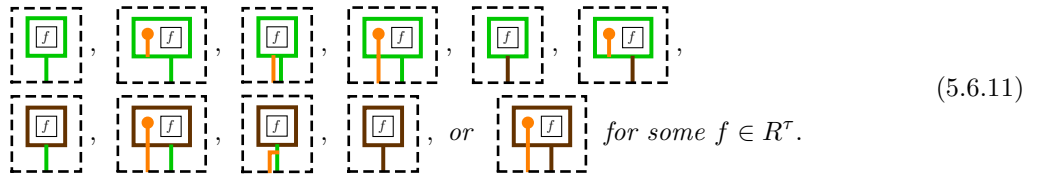
\begin{align}\label{eqn cases needles} 
		\begin{split}
			\begin{boxedstrings}
				\draw[green] (-0.25,0.35) rectangle ++(0.5,0.5);
				\idfigparam {(0,0.35)}{0.35}{green}
				\path (nd-1-2) -- ++(0,1);
				\polyboxscale{(0,0.6)}{$f$}{0.9}
			\end{boxedstrings}
			,  \hspace{0.1cm}
			\begin{boxedstrings}
				\node at (0,1) {};
				\draw[green] (0.25,0.35) rectangle ++(-0.75,0.5);
				\idfigparam {(0,0.35)}{0.35}{green}				
				\path (nd-1-2) -- ++(0,1);
				\polyboxscale{(0,0.6)}{$f$}{0.9}
				\dotu{(-0.3125,0.35)}{orange}
			\end{boxedstrings}
			, \hspace{0.1cm}
			\begin{boxedstrings}
				\draw[green] (-0.25,0.35) rectangle ++(0.5,0.5);
				\idfigparam {(0.05,0)}{-0.35}{green}
				\idfigparam {(-0.05,0)}{-0.35}{orange}
				\path (nd-1-1) -- ++(0,1);
				\polyboxscale{(0,0.6)}{$f$}{0.9}
			\end{boxedstrings}
			&, \hspace{0.1cm}
			\begin{boxedstrings}
				\node at (0,1) {};
				\draw[green] (0.25,0.35) rectangle ++(-0.75,0.5);
				\idfigparam {(0,0.35)}{0.35}{green}				
				\path (nd-1-2) -- ++(0,1);
				\polyboxscale{(0,0.6)}{$f$}{0.9}
				\dotu{(-0.3125,0.35)}{orange}
				\idfigparam {(nd-2-1)}{0.35}{orange}	
			\end{boxedstrings}
			, \hspace{0.1cm}
			\begin{boxedstrings}
				\node at (0,1) {};
				\draw[green] (-0.25,0.35) rectangle ++(0.5,0.5);
				\idfigparam {(0,0.35)}{0.35}{brown}
				\path (nd-1-2) -- ++(0,1);
				\polyboxscale{(0,0.6)}{$f$}{0.9}
			\end{boxedstrings}
			,  \hspace{0.1cm}
			\begin{boxedstrings}
				\node at (0,1) {};
				\draw[green] (0.25,0.35) rectangle ++(-0.75,0.5);
				\idfigparam {(0,0.35)}{0.35}{brown}				
				\path (nd-1-2) -- ++(0,1);
				\polyboxscale{(0,0.6)}{$f$}{0.9}
				\dotu{(-0.3125,0.35)}{orange}
			\end{boxedstrings}
			,
			\\
			\begin{boxedstrings}
				\node at (0,1) {};
				\draw[brown] (-0.25,0.35) rectangle ++(0.5,0.5);
				\idfigparam {(0,0.35)}{0.35}{green}
				\path (nd-1-2) -- ++(0,1);
				\polyboxscale{(0,0.6)}{$f$}{0.9}
			\end{boxedstrings}
			, \hspace{0.1cm}
			\begin{boxedstrings}
				\node at (0,1) {};
				\draw[brown] (0.25,0.35) rectangle ++(-0.75,0.5);
				\idfigparam {(0,0.35)}{0.35}{green}				
				\path (nd-1-2) -- ++(0,1);
				\polyboxscale{(0,0.6)}{$f$}{0.9}
				\dotu{(-0.3125,0.35)}{orange}
				\idfigparam {(-0.3125,0.35)}{0.35}{orange}	
			\end{boxedstrings}
			, \hspace{0.1cm}
			\begin{boxedstrings}
				\node at (0,1) {};
				\draw[brown] (-0.25,0.35) rectangle ++(0.5,0.5);
				\idfigparam {(0,0.35)}{0.35}{green}
				\path (nd-1-2) -- ++(0,1);
				\polyboxscale{(0,0.6)}{$f$}{0.9}
				\draw[orange] (0,0.25) -- ++(-0.1,0) -- (-0.1,0);
			\end{boxedstrings}
			&,\hspace{0.1cm}
			\begin{boxedstrings}
				\node at (0,1) {};
				\draw[brown] (-0.25,0.35) rectangle ++(0.5,0.5);
				\idfigparam {(0,0.35)}{0.35}{brown}
				\path (nd-1-2) -- ++(0,1);
				\polyboxscale{(0,0.6)}{$f$}{0.9}
			\end{boxedstrings}
			,
			\text{ or } \hspace{0.1cm}
			\begin{boxedstrings}
				\node at (0,1) {};
				\draw[brown] (0.25,0.35) rectangle ++(-0.75,0.5);
				\idfigparam {(0,0.35)}{0.35}{brown}				
				\path (nd-1-2) -- ++(0,1);
				\polyboxscale{(0,0.6)}{$f$}{0.9}
				\dotu{(-0.3125,0.35)}{orange}
				\idfigparam {(-0.3125,0.35)}{0.35}{orange}	
			\end{boxedstrings}
			\text{ for some }
			f\in\Rtau.
		\end{split}
	\end{align}
		
	All of these diagrams are in the $\Rtau$-span of
	{
		\begin{equation}\label{eqn spanning diagrams}
			\left\{\emptydiagram \right\}
			,		 
			\left\{\Dotup{orange}\right\}		 
			,
			\left\{\Dotup{green}, 
			\begin{miniunboxedstrings}
				\idfigparam{(1,0)}{-0.3}{green}
				\dotuparam{(nd-1-2)}{0.3}{0.06}{orange}
				\path (nd-1-1) -- ++(0,-0.2);
			\end{miniunboxedstrings}
			\right\}		 
			,
			\left\{
			\begin{miniunboxedstrings}
				\idfigparam{(1,0)}{-0.3}{green}
				\dotuparam{(nd-1-2)}{0.3}{0.06}{green}
				\draw[orange] (nd-1-2) -- ++(-0.25,0) -- ++(0,-0.3);
				\path (nd-1-1) -- ++(0,-0.2);
			\end{miniunboxedstrings}
			,
			\Dotup{orange}\hspace{-0.2cm} \Dotup{green}
			\right\}		 
			,
			\left\{\Dotup{brown}\right\} 
			,
			\text{ or }
			\left\{	\Dotup{orange}\hspace{-0.2cm} \Dotup{brown}
			\right\}
			.
		\end{equation}
	}
\end{lemma}
\begin{proof}
	Since $D$ is not a GB forest, $\ResGB(D)$ must be a circle or a needle.
	The diagrams listed, correspond to the all possible circles and needles, possibly adding an orange boundary strand or an orange dot.
	Proposition \ref{prop general circle rels} shows that 
	the diagrams in \eqref{eqn cases circles} are in the span 
	of the generators in \eqref{eqn spanning diagrams}.
	Proposition \ref{prop general needle rels} shows that 
	the diagrams in \eqref{eqn cases needles} are in the span 
	of the generators in \eqref{eqn spanning diagrams}.
\end{proof}

We now want to show that any empty GB cycle with spokes $D$ is equal to an $\Rtau$-combination of GB forests.
Before we do the general case, let our initial diagram be
the green $4$-cycle and apply the green $H=I$ relation repeatedly:
\begin{equation}
	\begin{boxedstrings}
		\trivuparam{(0,1)}{0.5}{0.5}{green}
		\trivdparam{(nd-1-2)}{0.5}{0.5}{green}
		\draw[green] (nd-1-2)--(-0.5,0.25)--(-0.5,0);
		\draw[green] (nd-1-3)--(0.5,0.25)--(0.5,0);
	\end{boxedstrings}
=
	\begin{boxedstrings}
	\trivuparam{(0,1.25)}{0.5}{0.5}{green}
	\trivuparam{(-0.25,0.5)}{0.5}{0.5}{green}
	\trivdparam{(nd-1-2)}{0.5}{0.5}{green}
	\draw[green] (nd-1-2) -- (nd-2-1);
	\draw[green] (nd-1-2) -- (0,0.5) -- (nd-1-3);
	\draw[green] (nd-3-3) -- (0.25,0);
	\end{boxedstrings}
=
	\begin{boxedstrings}
	\trivuparam{(0,1)}{0.5}{0.5}{green}
	\trivuparam{(nd-1-2)}{0.5}{0.5}{green}
	\draw[green] (nd-1-3) -- (0.25,0);
	\draw[green] (nd-1-1) -- (0,1.25);
	\draw[green] (nd-1-1) -- (-0.25,1);
	\draw[color=green] (-0.4,1) circle [radius=0.15];
	\end{boxedstrings}.
\end{equation}
Observe how we end up getting a tree with a needle attached to it.
In general, when we apply the $H=I$ relation to a green $n$-cycle with spokes we get 
	\excludeDiagrams{1}{
		\begin{equation}\label{cyle with spokes H=I}
			\begin{unboxedstrings}
				\node[circle,draw,dashed,
				minimum size = 2cm] (p) at (0,0) {};
				\node[green, circle,
				minimum size = 1cm] (q) at (0,0) {};
				\foreach \x in {1,2,...,7}
				{
					\tikzmath{\y=360/9*(\x+1)+90; \z=360/9*\x+90;}
					\draw[green] (p.\z) -- (q.\z);
					\draw[green] (q.\z) -- (q.\y);
				};
			\draw[green] (p.360*8/9+90) -- (q.360*8/9+90);
			\draw[green,dashed] (q.360*1/9+90) -- (q.360*8/9+90);
			\node (A) at (0,0.7) {$\overset{n}{...}$};
			\end{unboxedstrings}
		=
			\begin{unboxedstrings}
			\node[circle,draw,dashed,
			minimum size = 2cm] (p) at (0,0) {};
			\node[green, circle,
			minimum size = 1cm] (q) at (0,0) {};
			\foreach \x in {1,2,6,7}
			{
				\tikzmath{\y=360/9*(\x+1)+90; \z=360/9*\x+90;}
				\draw[green] (p.\z) -- (q.\z);
				\draw[green] (q.\z) -- (q.\y);
			};
			\draw[green] (p.360*8/9+90) -- (q.360*8/9+90);
			\draw[green,dashed] (q.360*1/9+90) -- (q.360*8/9+90);
			\node (A) at (0,0.7) {$\overset{n-1}{...}$};
			\node (B) at (q.270) {};
			\node (C) at ($(B)-(0,0.2)$) {};
			\draw[green] (p.360*3/9+90) -- (q.360*3/9+90);
			\draw[green] (q.360*3/9+90) -- (q.270) -- (q.360*6/9+90);
			\draw[green] (p.360*4/9+90) -- (C) -- (p.360*5/9+90);
			\draw[green] (q.270) -- (C);
		\end{unboxedstrings}
		.
		\end{equation}
	}
The right hand side becomes a cycle with spokes of size $n-1$ attached to a trivalent vertex. If we apply the $H=I$ relation repeatedly, we will eventually get a tree and a needle (as in the case for $n=4$).
That is the main idea for the proof of the following lemma.

\begin{lemma}\label{cycle reduction}
Let $D$ be an empty GB cycle with spokes. 
Then $D= D_1+...+ D_k$ for some diagrams $D_i$ such that $\ResGB(D_i)$ is a forest with fewer trivalent vertices than $\ResGB(D)$, or the same number fewer GB components and fewer GB components. 
\end{lemma}
\begin{proof}
Let $D$ be an empty GB cycle with spokes.
We can assume that $D$ has no orange circles, that $D$ has at most one orange dot inside it and that every green connected component of $D$ has at most one orange landing.

We can assume that $D$ has no orange strands splitting the interior of its cycle into 2 different regions.
That is because we can slide any orange strand outside the cycle. Here are the three possible cases.

\excludeDiagrams{1}
{
\begin{equation}
\begin{unboxedstrings}
    \node[circle,draw,dashed,
    minimum size = 2.5cm] (p) at (0,0) {};
    \node[circle,
    minimum size = 1.5cm] (q) at (0,0) {};
        \foreach \x in {0,1,2,...,9}
    {
    \tikzmath{\y=\x+1;}
    \draw[purple] (p.360/10*\x) -- (q.360/10*\x);
    \draw[purple] (q.360/10*\x) -- (q.360/10*\y);
    };
    \draw[orange] (p.90)--(0,1);
    \draw[orange] (p.270)--(0,-1);
    \draw[orange] (0,1)--(0,-1);
\end{unboxedstrings}
=
\begin{unboxedstrings}
    \node[circle,draw,dashed,
    minimum size = 2.5cm] (p) at (0,0) {};
    \node[circle,
    minimum size = 1.5cm] (q) at (0,0) {};
        \foreach \x in {0,1,2,...,9}
    {
    \tikzmath{\y=\x+1;}
    \draw[purple] (p.360/10*\x) -- (q.360/10*\x);
    \draw[purple] (q.360/10*\x) -- (q.360/10*\y);
    };
    \draw[orange] (p.90)--(0,1);
    \draw[orange] (p.270)--(0,-1);
    \draw[orange] (0,1) arc (90:270:1);
\end{unboxedstrings}
\end{equation}
\begin{equation}
\begin{unboxedstrings}
    \node[circle,draw,dashed,
    minimum size = 2.5cm] (p) at (0,0) {};
    \node[circle,
    minimum size = 1.5cm] (q) at (0,0) {};
        \foreach \x in {0,1,2,...,9}
    {
    \tikzmath{\y=\x+1;}
    \draw[purple] (p.360/10*\x) -- (q.360/10*\x);
    \draw[purple] (q.360/10*\x) -- (q.360/10*\y);
    };
    \draw[orange] (p.90)--($0.5*(q.252)+0.5*(q.288)$);
\end{unboxedstrings}
=
\begin{unboxedstrings}
    \node[circle,draw,dashed,
    minimum size = 2.5cm] (p) at (0,0) {};
    \node[circle,
    minimum size = 1.5cm] (q) at (0,0) {};
        \foreach \x in {0,1,2,...,9}
    {
    \tikzmath{\y=\x+1;}
    \draw[purple] (p.360/10*\x) -- (q.360/10*\x);
    \draw[purple] (q.360/10*\x) -- (q.360/10*\y);
    };
    \draw[orange] (p.90)-- (0,1) arc (90:270:1) -- ($0.5*(q.252)+0.5*(q.288)$);
\end{unboxedstrings}
\end{equation}
\begin{equation}
\begin{unboxedstrings}
    \node[circle,draw,dashed,
    minimum size = 2.5cm] (p) at (0,0) {};
    \node[circle,
    minimum size = 1.5cm] (q) at (0,0) {};
        \foreach \x in {0,1,2,...,9}
    {
    \tikzmath{\y=\x+1;}
    \draw[purple] (p.360/10*\x) -- (q.360/10*\x);
    \draw[purple] (q.360/10*\x) -- (q.360/10*\y);
    };
    \draw[orange] ($0.5*(q.72)+0.5*(q.108)$)-- 
    ($0.5*(q.252)+0.5*(q.288)$);
\end{unboxedstrings}
=
\begin{unboxedstrings}
    \node[circle,draw,dashed,
    minimum size = 2.5cm] (p) at (0,0) {};
    \node[circle,
    minimum size = 1.5cm] (q) at (0,0) {};
        \foreach \x in {0,1,2,...,9}
    {
    \tikzmath{\y=\x+1;}
    \draw[purple] (p.360/10*\x) -- (q.360/10*\x);
    \draw[purple] (q.360/10*\x) -- (q.360/10*\y);
    };
    \draw[orange] ($0.5*(q.72)+0.5*(q.108)$)-- (0,1) arc (90:270:1)
     -- ($0.5*(q.252)+0.5*(q.288)$);
\end{unboxedstrings}
\end{equation}
}

The proof will go by induction on the size of the cycle of $\ResGB(D)$.
We showed our base case in Lemma \ref{lemma base case}.

Now let $\ResGB(D)$ is a cycle of size $2$ or greater, writing any 
GB bivalent vertices in $D$ by their definition as a trivalent vertex with a dot on the outside of the cycle.
We will now make use of all the $H=I$ relations in Appendix \ref{appendix-H=I-relations}.

In most cases, when we apply an $H=I$ relation to the outside cycle of $D$ we get that $D=D_1+... +D_r$ where each $D_i$ satisfies that $\ResGB(D_i)$ is an empty $n-1$-cycle with spokes attached to a trivalent vertex 
(unlike equation \eqref{cyle with spokes H=I} we get a sum of diagrams of this form).
The one exception is \eqref{H=I middle brown},
\begin{equation} 
	2
	\begin{boxedstrings}
		\Hlongparam{(0,0)}{0.5}{1}{green}{green}{brown}{green}{green}
	\end{boxedstrings}
	=
	\begin{boxedstrings}
		\capfigparam{(0,0)}{0.5}{0.3}{green}
		\cupfigparam{(0,1)}{0.5}{0.3}{green}
	\end{boxedstrings}
	-
	\begin{boxedstrings}
		\capfigparam{(0,0)}{0.5}{0.3}{green}
		\cupfigparam{(0,1)}{0.5}{0.3}{green}
		\draw[orange] (nd-1-3) -- (nd-2-3);
	\end{boxedstrings}
	+
	\begin{boxedstrings}
		\Ilongparam{(0,0)}{0.5}{1}{green}{green}{brown}{green}{green}
	\end{boxedstrings}
	-
	\begin{boxedstrings}
		\Ilongparam{(0,0)}{0.5}{1}{green}{green}{brown}{green}{green}
		\draw[orange] (nd-1-8) -- (nd-1-9);
	\end{boxedstrings}
\end{equation}
whose first two summands break the cycle, making them trees right away with $2$ less trivalent vertices on green and brown.
Thus, if we show that our diagram $D$ equals a GB cycle with spokes $E$ containing a GB crossbar (to be able to apply an $H=I$ relation), we are done.

If the central cycle has a brown edge, we can use 
Lemma \ref{GB comb sliding} to get $E$.
Otherwise, the central cycle in $D$ is made entirely of green edges, all in the same green component.
Since we assumed there is at most one orange landing over this green component, we can apply Lemma \ref{GB comb sliding} to an edge that doesn't have an orange landing.
This concludes the proof.
\end{proof}

\begin{coro}\label{innermost cycle reduction}
	Let $D$ be an innermost RG cycle with spokes.
	Then $D=E_1+...+E_k$ for diagrams $E_i$ such that each $\ResGB(E_i)$ is a forest with fewer trivalent vertices than $\ResGB(D)$.	
\end{coro}
\begin{proof}
	Let $D$ be a diagram such that $\ResGB(D)$ is an innermost cycle with spokes and let $D'$ be the inside of $D$.
	Since $\ResGB(D')$ is a based forest, by Corollary \ref{forest based reduction}, we can write $D'= D_1'+...+D_m'$, where each $\ResGB(D_j')$ is a union of boundary dots.
	Thus, $D=D_1+...+D_m$, where each $\ResGB(D_j)$ is a cycle with spokes whose inside is a union of boundary dots.
	As for previous proofs, suppose that $D_1,...,D_m$ have at most one
	orange landing over each of their green components and have no orange circles.
	
	We now want to apply unit relations to the dots in the inside of each $D_j$.
	There are three types of unit relations:
	\begin{itemize}
		\item \eqref{bivalent definition right},\eqref{one color unit rel},\eqref{bivalent definition left},
		\eqref{GBB unit rel 1}, and \eqref{GBB unit rel 2},
		that merge a dot into a green or brown strand, or into a green-brown bivalent vertex.
		After using one of these relations on $D_j$, we end up with a GB cycle with spokes whose interior has one fewer boundary dot.
		If $D_j$ only has boundary dots of these types, we can use our relations to make $D_j$ into an empty GB cycle with spokes.
		
		\item \eqref{brown green unit rel}, that splits the trivalent vertex into an $\Rtau$-combination 
		of 
		$
		\begin{miniunboxedstrings}
			\dotd{(1,1)}{green}
			\dotu{(1,0)}{green}
			\path (nd-2-3) -- ++(0,-0.2);
		\end{miniunboxedstrings}
		$
		and
		$
		\begin{miniunboxedstrings}
			\idfigparam{(0,1)}{0.3}{green}
			\idfigparam{(nd-1-2)}{0.4}{orange}
			\idfigparam{(nd-2-2)}{0.3}{green}
			\path (nd-1-1) -- ++(0,-0.2);
		\end{miniunboxedstrings}
		$
		.
		This relation transforms $D_j$ into a GB tree, as it breaks the cycle with spokes (with at least one trivalent vertex less than $D$). Note that using this more than once will result in a forest with more than one GB tree.
		
		\item  \eqref{GBB unit rel 3 orange}, which makes the diagram equal to zero.
	\end{itemize}

	In the last two cases, $D_j$ becomes a forests right away.
	In the first case, instead, we can use Lemma \ref{cycle reduction} to 
	write $D_j=E_1+...+E_k$ for some diagrams $E_1,...,E_k$, where each $\ResGB(E_i)$ is a forest and has fewer trivalent vertices or fewer GB components than $\ResGB(D_j)$. 
	This completes the proof.
\end{proof}

\begin{lemma}\label{existence of cycle}
Let $D$ be a diagram such that $D$ is not a GB forest.
Then there is a diagram $N$ contained in $D$ such that $N$ is an innermost GB cycle with spokes.
\end{lemma}
\begin{proof}
	Let $\mathbf{C}$ be the set of all GB cycles with spokes contained in $D$.
	The set $\mathbf{C}$ is partially ordered by containment.
	Since $D$ is not a GB forest, the set $\mathbf{C}$ is non-empty.
	Let $(C,C')$ be a minimal element of $\mathbf{C}$.
	By minimality, $C'$ must be a GB forest.
	We claim that $C'$ is a GB based forest and $(C,C')$ is an innermost GB cycle with spokes.
	
	Let us now suppose that $\ResGB(C')$ is a GB forest, but not a based forest.
	This means that $\ResGB(C')$ contains a green-brown tree with at least 2 boundary strands connecting to $C$. 
	Let $T$ be such tree. Let $P$ be a diagram inside $C$ containing $T$ and also containing all the edges in the outer cycle connecting the two boundary strands of $T$, but excluding at least one trivalent vertex in the outer cycle of $C$.
	Diagrammatically
	\excludeDiagrams{1}
	{
		\begin{equation}
			C=
			\begin{unboxedstrings}
				\node[circle,draw,dashed,
				minimum size = 3.5cm] (p) at (0,0) {};
				\node[circle,
				minimum size = 3cm] (q) at (0,0) {};
				\node[circle,draw, dashed,
				minimum size = 2cm] (t) at (0,0) {};
				\node[circle,draw, dashed,
				minimum size = 0.75cm] (r) at (0,-0.25) {$T$};
				\node at (0.45,0.45) {$C'$};
				\foreach \x in {0,1,2,...,5}
				{
					\tikzmath{\y=\x+1; \z=2*\x+1;}
					\draw[black] (p.360/6*\x) -- (q.360/6*\x);
					\draw[purple] (q.360/6*\x) -- (q.360/6*\y);
				};
				\foreach \x in {0,1,2,4}
				{
					\tikzmath{\y=\x+1; \z=2*\x+1;}
					\draw[black] ($0.5*(q.360/6*\x)+0.5*(q.360/6*\y)$) -- (t.360/12*\z);
				};
				\foreach \x in {3,5}
				{
					\tikzmath{\y=\x+1; \z=2*\x+1;\w=180*(-\x+5)/2;};
					\draw[purple] ($0.5*(q.360/6*\x)+0.5*(q.360/6*\y)$) -- (t.360/12*\z)
					-- (r.\w);
				};
			\end{unboxedstrings},
			\text{ and }
			P=
			\begin{unboxedstrings}
				\draw[dashed] (-1.3,0.25) rectangle (1.3,-1.65);
				\path[clip] (-1.28,0.25) rectangle (1.28,-1.65);
				\node[circle,
				minimum size = 4cm] (p) at (0,0) {};
				\node[circle,
				minimum size = 3cm] (q) at (0,0) {};
				\node[circle,draw, dashed,
				minimum size = 2cm] (t) at (0,0) {};
				\node[circle,draw, dashed,
				minimum size = 0.75cm] (r) at (0,-0.25) {$T$};
				\node at (0.45,0.45) {$C'$};
				\foreach \x in {0,1,2,...,5}
				{
					\tikzmath{\y=\x+1; \z=2*\x+1;}
					\draw[black] (p.360/6*\x) -- (q.360/6*\x);
					\draw[purple] (q.360/6*\x) -- (q.360/6*\y);
				};
				\foreach \x in {0,1,2,4}
				{
					\tikzmath{\y=\x+1; \z=2*\x+1;}
					\draw[black] ($0.5*(q.360/6*\x)+0.5*(q.360/6*\y)$) -- (t.360/12*\z);
				};
				\foreach \x in {3,5}
				{
					\tikzmath{\y=\x+1; \z=2*\x+1;\w=180*(-\x+5)/2;};
					\draw[purple] ($0.5*(q.360/6*\x)+0.5*(q.360/6*\y)$) -- (t.360/12*\z)
					-- (r.\w);
				};
			\end{unboxedstrings}
			\hspace{0.1cm}
			.
		\end{equation}
	}
	But $P$ contains a GB cycle with spokes and this contradicts the minimality of $C$.
\end{proof}

Now we will prove the theorem.
\begin{proof}[Proof of Theorem \ref{diagram reduction}]
Let $D$ be a diagram such that $\ResGB(D)$ only has at most one boundary strand, and such that $D$ has at most one orange boundary strand.
The proof will go by induction on the total number of bivalent and trivalent vertices of $\ResGB(D)$ first,
and then on the number of GB components of $D$.
As before, suppose that $D$ has no orange circles, no barbells,
no more than two orange dots in each region,
and no more than one orange landing over each green component.
If $\ResGB(D)$ is a GB forest,
we can apply Lemma \ref{forest based reduction},
so suppose that $D$ is not a forest.
Since $D$ is not a forest, our base case corresponds to Lemma \ref{lemma base case}.

Let us suppose that $\ResGB(D)$ has more than one trivalent vertex.
Since $D$ is not a GB forest,
by Lemma \ref{existence of cycle},
$D$ has an innermost cycle with spokes.
Let $C$ be an innermost cycle with spokes contained in $D$.
Use Corollary \ref{innermost cycle reduction} to write $C=E_1+...+E_k$ for diagrams $E_i$,
such that $\ResGB(E_i)$ is a forest with fewer trivalent vertices than $\ResGB(C)$, or with the same number of trivalent vertices but fewer
GB components.
Hence $D=D_1+...+D_k$, where $D_i$ is the diagram obtained by replacing $C$ for $E_i$.
Each $D_i$ either has fewer trivalent vertices than $D$, or has the same number of trivalent vertices as $D$ but has fewer
GB components.
Hence we can apply our induction hypothesis to all the $D_i$'s.
This concludes our proof.
\end{proof}

\begin{appendices}

\setcounter{visible}{1}

\section{Relations in $\Hcaltau$}\label{appendix-relations}

\subsection{Additional Needle and Unit relations}\label{appendix-unit-relations}

\begin{prop}\label{prop needle relations}
	The needle relations
	\begin{equation}
		\label{needle relations 2}
		\begin{boxedstrings}
			\node at (0,1) {};
			\node (a) at (-0.2,0.5) {};
			\def\coltop{brown}
			\def\colbot{green}
			\capfigparam{(a)}{0.4}{0.2}{\coltop}
			\cupfigparam{(a)}{0.4}{0.2}{\coltop}
			\idfigparam {(0,0)}{-0.3}{\colbot}
		\end{boxedstrings}
		= 0
		,
		\hspace{0.2cm}
		\begin{boxedstrings}
			\node at (0,1) {};
			\node (a) at (-0.2,0.5) {};
			\def\coltop{green}
			\def\colbot{green}
			\def\colside{orange}
			\capfigparam{(a)}{0.4}{0.2}{\coltop}
			\cupfigparam{(a)}{0.4}{0.2}{\coltop}
			\idfigparam {(0.05,0)}{-0.3}{\colbot}
			\idfigparam {(-0.05,0)}{-0.3}{\colside}
		\end{boxedstrings}
		=0
		,
		\hspace{0.1cm}
		\text{ and }
		\hspace{0.1cm}
		\begin{boxedstrings}
			\node at (0,1) {};
			\node (a) at (-0.2,0.5) {};
			\def\coltop{green}
			\def\colbot{brown}
			\def\colside{orange}
			\capfigparam{(a)}{0.4}{0.2}{\coltop}
			\cupfigparam{(a)}{0.4}{0.2}{\coltop}
			\idfigparam {(0.05,0)}{-0.3}{\colbot}
			\idfigparam {(-0.05,0)}{-0.3}{\colside}
		\end{boxedstrings}
		=0
	\end{equation}
	hold.
\end{prop}
\begin{proof}
	Checking these is straightforward.
	By \eqref{H=I bicolor associativity} and \eqref{needle relations},
	\begin{align}
		\begin{boxedstrings}
			\node at (0,1) {};
			\node (a) at (-0.2,0.5) {};
			\def\coltop{brown}
			\def\colbot{green}
			\capfigparam{(a)}{0.4}{0.2}{\coltop}
			\cupfigparam{(a)}{0.4}{0.2}{\coltop}
			\idfigparam {(nd-2-3)}{0.3}{\colbot}
		\end{boxedstrings}
		&=
		\begin{boxedstrings}
			\node at (0,1) {};
			\node (a) at (-0.2,0.5) {};
			\def\coltop{brown}
			\def\colbot{green}
			\capfigparam{(a)}{0.4}{0.2}{\coltop}
			\cupfigparam{(a)}{0.4}{0.2}{\coltop}
			\idfigparam {(nd-2-3)}{0.3}{\colbot}
			\dotuparam{(0,0.7)}{0.15}{0.06}{\coltop}
		\end{boxedstrings}
		=
		\begin{boxedstrings}
			\node at (0,1) {};
			\node (a) at (-0.2,0.7) {};
			\def\coltop{brown}
			\def\colbot{green}
			\capfigparam{(a)}{0.4}{0.2}{\coltop}
			\cupfigparam{(a)}{0.4}{0.2}{\coltop}
			\idfigparam {(nd-2-3)}{0.3}{\coltop}
			\idfigparam {(nd-3-2)}{0.2}{\colbot}
		\end{boxedstrings}
		=
		0.
	\end{align}
	
	By \eqref{GGG orange} and \eqref{needle relations},
	\begin{equation}
		\begin{boxedstrings}
			\node at (0,1) {};
			\node (a) at (-0.2,0.5) {};
			\def\coltop{green}
			\def\colbot{green}
			\def\colside{orange}
			\capfigparam{(a)}{0.4}{0.2}{\coltop}
			\cupfigparam{(a)}{0.4}{0.2}{\coltop}
			\idfigparam {(0.05,0)}{-0.3}{\colbot}
			\idfigparam {(-0.05,0)}{-0.3}{\colside}
		\end{boxedstrings}
		=
		\begin{boxedstrings}
			\node at (0,1) {};
			\node (a) at (-0.2,0.5) {};
			\def\coltop{green}
			\def\colbot{green}
			\def\colside{orange}
			\capfigparam{(a)}{0.4}{0.2}{\coltop}
			\cupfigparam{(a)}{0.4}{0.2}{\coltop}
			\idfigparam {(0,0)}{-0.3}{\colbot}
			\draw[orange] (nd-3-3) -- ++(-0.1,0) -- ++(0,-0.15);
		\end{boxedstrings}
		=
		0.
	\end{equation}
	
	By \eqref{GGB orange 1} and Theorem \ref{thm orange sliding},
	\begin{equation}
		\begin{boxedstrings}
			\node at (0,1) {};
			\node (a) at (-0.2,0.5) {};
			\def\coltop{green}
			\def\colbot{brown}
			\def\colside{orange}
			\capfigparam{(a)}{0.4}{0.2}{\coltop}
			\cupfigparam{(a)}{0.4}{0.2}{\coltop}
			\idfigparam {(0.05,0)}{-0.3}{\colbot}
			\idfigparam {(-0.05,0)}{-0.3}{\colside}
		\end{boxedstrings}
		=
		-
		\begin{boxedstrings}
			\node at (0,1) {};
			\node (a) at (-0.2,0.5) {};
			\def\coltop{green}
			\def\colbot{brown}
			\capfigparam{(a)}{0.4}{0.2}{\coltop}
			\cupfigparam{(a)}{0.4}{0.2}{\coltop}
			\idfigparam {(0,0)}{-0.3}{\colbot}
			\draw[orange] (nd-1-2) -- ++(0,-0.35) -- (nd-3-3) -- ++(-0.1,0) -- ++(0,-0.15);
		\end{boxedstrings}
		=
		-
		\begin{boxedstrings}
			\node at (0,1) {};
			\node (a) at (-0.2,0.5) {};
			\def\coltop{green}
			\def\colbot{brown}
			\capfigparam{(a)}{0.4}{0.2}{\coltop}
			\cupfigparam{(a)}{0.4}{0.2}{\coltop}
			\idfigparam {(0,0)}{-0.3}{\colbot}
			\draw[orange] (nd-1-2) -- ++(0,0.3) -- ++(-0.5,0) -- ++(0,-0.55) --
			++(0.2,0) -- ++(0,-0.25);
		\end{boxedstrings}
		=
		-
		\begin{boxedstrings}
			\node at (0,1) {};
			\node (a) at (-0.2,0.5) {};
			\def\coltop{green}
			\def\colbot{brown}
			\def\colside{orange}
			\capfigparam{(a)}{0.4}{0.2}{\coltop}
			\cupfigparam{(a)}{0.4}{0.2}{\coltop}
			\idfigparam {(0.05,0)}{-0.3}{\colbot}
			\idfigparam {(-0.05,0)}{-0.3}{\colside}
		\end{boxedstrings}
		.
	\end{equation}
	Hence,
	\begin{equation}
		2
		\begin{boxedstrings}
			\node at (0,1) {};
			\node (a) at (-0.2,0.5) {};
			\def\coltop{green}
			\def\colbot{brown}
			\def\colside{orange}
			\capfigparam{(a)}{0.4}{0.2}{\coltop}
			\cupfigparam{(a)}{0.4}{0.2}{\coltop}
			\idfigparam {(0.05,0)}{-0.3}{\colbot}
			\idfigparam {(-0.05,0)}{-0.3}{\colside}
		\end{boxedstrings}
		=0.
	\end{equation}
\end{proof}

\begin{prop}\label{prop unit relations}
	The unit relations
	\begin{equation}\label{GBB unit rel 1}
		\begin{boxedstrings}
			\idfigparam{(0,1)}{0.5}{brown}
			\idfigparam{(nd-1-2)}{0.5}{brown}
			\dotl{(nd-1-2)}{green}
		\end{boxedstrings} 
		=
		2
		\begin{boxedstrings}
			\idfigparam{(0,1)}{1}{brown}
		\end{boxedstrings} 
		,
	\end{equation}
	\begin{equation}\label{GBB unit rel 3 orange}
		\begin{boxedstrings}
			\idfigparam{(0,1)}{1}{brown}
			\draw[orange] (-0.25,0) -- ++(0,0.5);
			\draw[green] (nd-1-3)-- ++(-0.25,0);
		\end{boxedstrings} 
		= 0
		,
	\end{equation}
	\begin{equation}\label{GBB unit rel 2}
		\begin{boxedstrings}
			\idfigparam{(0,1)}{0.5}{brown}
			\idfigparam{(nd-1-2)}{0.5}{green}
			\dotl{(nd-1-2)}{brown}
		\end{boxedstrings} 
		=
		\begin{boxedstrings}
			\idfigparam{(0,1)}{0.5}{brown}
			\idfigparam{(nd-1-2)}{0.5}{green}
		\end{boxedstrings} 
		,
	\end{equation}
	and
	\begin{equation}\label{GGB unit rel}
		\begin{boxedstrings}
			\idfigparam{(1,0)}{-0.25}{brown}
			\dotuparam{(nd-1-2)}{0.2}{0.06}{green}
			\path (nd-1-1) -- ++(0,1);
		\end{boxedstrings}
		=
		2
		\begin{boxedstrings}
			\dotuparam{(0,0)}{0.4}{0.06}{brown}
			\path (0,0) -- ++(0,1);
		\end{boxedstrings}	
	\end{equation}	
	hold.
\end{prop}
\begin{proof}
	\eqref{GBB unit rel 1}
	By \eqref{bigon brown} and \eqref{one color unit rel} we have
	\begin{equation}
		\begin{boxedstrings}
			\idfigparam{(0,1)}{0.5}{brown}
			\idfigparam{(nd-1-2)}{0.5}{brown}
			\dotl{(nd-1-2)}{green}
		\end{boxedstrings} 
		=
		\begin{boxedstrings}
			\node at (0,1) {};
			\node (a) at (-0.2,0.5) {};
			\def\coltop{green}
			\def\colbot{brown}
			\capfigparam{(a)}{0.4}{0.2}{\coltop}
			\cupfigparam{(a)}{0.4}{0.2}{\coltop}
			\dotl{(a)}{green}
			\idfigparam {(0,0.3)}{0.3}{\colbot}
			\idfigparam {(0,0.7)}{-0.3}{\colbot}
		\end{boxedstrings}
		=
		\begin{boxedstrings}
			\node at (0,1) {};
			\node (a) at (-0.2,0.5) {};
			\def\coltop{green}
			\def\colbot{brown}
			\capfigparam{(a)}{0.4}{0.2}{\coltop}
			\cupfigparam{(a)}{0.4}{0.2}{\coltop}
			\idfigparam {(0,0.3)}{0.3}{\colbot}
			\idfigparam {(0,0.7)}{-0.3}{\colbot}
		\end{boxedstrings}
		=
		2
		\begin{boxedstrings}
			\idfigparam{(0,1)}{1}{brown}
		\end{boxedstrings} 
		.
	\end{equation}
	
	\eqref{GBB unit rel 3 orange}
	By \eqref{bigon brown orange},
	\begin{equation}
		\begin{boxedstrings}
			\idfigparam{(0,1)}{1}{brown}
			\draw[orange] (-0.25,0) -- ++(0,0.5);
			\draw[green] (nd-1-3)-- ++(-0.25,0);
		\end{boxedstrings} 
		=
		\begin{boxedstrings}
			\Triangucol{(0,0)}{brown}{green}{green}{green}{orange}{brown}
		\end{boxedstrings}
		= 0
		.
	\end{equation}
	
	\eqref{GBB unit rel 2}
	By \eqref{brown green unit rel} and \eqref{GGB orange 1}
	\begin{equation}
		2
		\begin{boxedstrings}
			\idfigparam{(0,1)}{0.5}{brown}
			\idfigparam{(nd-1-2)}{0.5}{green}
			\dotl{(nd-1-2)}{brown}
		\end{boxedstrings} 
		=
		2
		\begin{boxedstrings}
			\node at (0,1) {};
			\node (a) at (-0.2,0.5) {};
			\def\coltop{green}
			\def\colbot{brown}
			\capfigparam{(a)}{0.4}{0.2}{\coltop}
			\cupfigparam{(a)}{0.4}{0.2}{\coltop}
			\dotl{(a)}{brown}
			\idfigparam {(0,0.3)}{0.3}{\coltop}
			\idfigparam {(0,0.7)}{-0.3}{\colbot}
		\end{boxedstrings}
		=
		\begin{boxedstrings}
			\idfigparam{(0,1)}{0.5}{brown}
			\idfigparam{(nd-1-2)}{0.5}{green}
			\dotl{(nd-1-2)}{green}
		\end{boxedstrings}
		-
		\begin{boxedstrings}
			\idfigparam{(0,1)}{0.5}{brown}
			\idfigparam{(nd-1-2)}{0.5}{green}
			\draw[orange] (nd-2-3) -- ++(-0.25,0) -- ++(0,0.25);
			\draw[green] (nd-1-2) -- ++(-0.25,0);
		\end{boxedstrings}	
		=
		2
		\begin{boxedstrings}
			\idfigparam{(0,1)}{0.5}{brown}
			\idfigparam{(nd-1-2)}{0.5}{green}
			\dotl{(nd-1-2)}{green}
		\end{boxedstrings}	
		=
		2
		\begin{boxedstrings}
			\idfigparam{(0,1)}{0.5}{brown}
			\idfigparam{(nd-1-2)}{0.5}{green}
		\end{boxedstrings} 
		.
	\end{equation}

	\eqref{GGB unit rel}
	By \eqref{bigon brown} and \eqref{GGB orange 1} we have
	\begin{equation}
		2
		\begin{boxedstrings}
			\dotuparam{(0,0)}{0.6}{0.06}{brown}
			\path (nd-1-3) -- (0,1.25);
		\end{boxedstrings}
		=
		\begin{boxedstrings}
			\node at (0,1) {};
			\node (a) at (-0.2,0.5) {};
			\def\coltop{green}
			\def\colbot{brown}
			\def\colside{orange}
			\capfigparam{(a)}{0.4}{0.2}{\coltop}
			\cupfigparam{(a)}{0.4}{0.2}{\coltop}
			\idfigparam {(0,0.3)}{0.3}{\colbot}
			\idfigparam {(0,0.7)}{-0.3}{\colbot}
			\dotfigparamcol{(nd-4-2)}{0.06}{\colbot}
			\path (nd-1-3) -- (0,1.25);
		\end{boxedstrings}
		=
		\begin{boxedstrings}
			\trivdparamcol{(0,0.7)}{0.5}{0.7}{green}{green}{brown}
			\dotfig{(nd-1-1)}{green}
			\dotfig{(nd-1-2)}{green}
			\path (nd-1-3) -- ++(0,1.25);
		\end{boxedstrings}
		-
		\begin{boxedstrings}
			\draw[orange] (0.125,0.525) -- ++(0.25,0);
			\trivdparamcol{(0,0.7)}{0.5}{0.7}{green}{green}{brown}
			\dotfig{(nd-1-1)}{green}
			\dotfig{(nd-1-2)}{green}
			\path (nd-1-3) -- ++(0,1.25);
		\end{boxedstrings}
		=
		2
		\begin{boxedstrings}
			\trivdparamcol{(0,0.7)}{0.5}{0.7}{green}{green}{brown}
			\dotfig{(nd-1-1)}{green}
			\dotfig{(nd-1-2)}{green}
			\path (nd-1-3) -- ++(0,1.25);
		\end{boxedstrings}.
		= 2
		\begin{boxedstrings}
			\idfigparam{(1,0)}{-0.35}{brown}
			\dotuparam{(nd-1-2)}{0.25}{0.06}{green}
			\path (nd-1-1) -- ++(0,1.25);
		\end{boxedstrings}
	\end{equation}
\end{proof}

\subsection{Additional $H=I$ relations}\label{appendix-H=I-relations}

The $H=I$ relations in Definition \ref{relations equiv category} are:
\excludeDiagrams{1}
{

	\begin{equation}\label{H=I one color2}
		\begin{boxedstrings}
			\Hlongparam{(0,0)}{0.5}{1}{green}{green}{green}{green}{green}
		\end{boxedstrings}
		=
		\begin{boxedstrings}
			\Ilongparam{(0,0)}{0.5}{1}{green}{green}{green}{green}{green}
		\end{boxedstrings},
		\hspace{0.5cm}
		\begin{boxedstrings}
			\Hlongparam{(0,0)}{0.5}{1}{brown}{brown}{brown}{brown}{brown}
		\end{boxedstrings}
		=
		\begin{boxedstrings}
			\Ilongparam{(0,0)}{0.5}{1}{brown}{brown}{brown}{brown}{brown}
		\end{boxedstrings},
	\end{equation}

	\begin{equation}\label{H=I bicolor associativity2}
		\begin{boxedstrings}
			\Hlongparam{(0,0)}{0.5}{1}{green}{brown}{brown}{brown}{green}
		\end{boxedstrings}
		=
		\begin{boxedstrings}
			\Ilongparam{(0,0)}{0.5}{1}{green}{brown}{brown}{brown}{green}
		\end{boxedstrings},
		\hspace{0.5cm}
		\begin{boxedstrings}
			\Hlongparam{(0,0)}{0.5}{1}{green}{brown}{brown}{brown}{brown}
		\end{boxedstrings}
		=
		\begin{boxedstrings}
			\Ilongparam{(0,0)}{0.5}{1}{green}{brown}{brown}{brown}{brown}
		\end{boxedstrings},
	\end{equation}

	\begin{equation}\label{H=I top brown2}
		\begin{boxedstrings}
			\Hlongparam{(0,0)}{0.5}{1}{brown}{brown}{brown}{green}{green}
		\end{boxedstrings}
		=
		\begin{boxedstrings}
			\Ilongparam{(0,0)}{0.5}{1}{brown}{brown}{green}{green}{green}
		\end{boxedstrings} 
		+
		\begin{boxedstrings}
			\Ilongparam{(0,0)}{0.5}{1}{brown}{brown}{brown}{green}{green}
		\end{boxedstrings}    
	\end{equation}
	
	\begin{equation}\label{H=I L brown leg2}
		2
		\begin{strings}
			\Hlongparam{(0,0)}{0.5}{1}{green}{green}{brown}{green}{brown}
		\end{strings}
		=
		2
		\begin{strings}
			\Ilongparam{(0,0)}{0.5}{1}{green}{green}{green}{green}{brown}
		\end{strings}
		+
		\begin{strings}
			\Ilongparam{(0,0)}{0.5}{1}{green}{green}{brown}{green}{brown}
		\end{strings}
		-
		\begin{strings}
			\Ilongparam{(0,0)}{0.5}{1}{green}{green}{brown}{green}{brown}
			\draw[orange] (nd-1-8) -- (nd-1-9);
		\end{strings}
	\end{equation}

	\begin{equation} \label{H=I middle brown2} 
		2
		\begin{boxedstrings}
			\Hlongparam{(0,0)}{0.5}{1}{green}{green}{brown}{green}{green}
		\end{boxedstrings}
		=
		\begin{boxedstrings}
			\capfigparam{(0,0)}{0.5}{0.3}{green}
			\cupfigparam{(0,1)}{0.5}{0.3}{green}
		\end{boxedstrings}
		-
		\begin{boxedstrings}
			\capfigparam{(0,0)}{0.5}{0.3}{green}
			\cupfigparam{(0,1)}{0.5}{0.3}{green}
			\draw[orange] (nd-1-3) -- (nd-2-3);
		\end{boxedstrings}
		+
		\begin{boxedstrings}
			\Ilongparam{(0,0)}{0.5}{1}{green}{green}{brown}{green}{green}
		\end{boxedstrings}
		-
		\begin{boxedstrings}
			\Ilongparam{(0,0)}{0.5}{1}{green}{green}{brown}{green}{green}
			\draw[orange] (nd-1-8) -- (nd-1-9);
		\end{boxedstrings}
	\end{equation}

}

We wish to show the rest of the $H=I$ relations, since we need all of them to prove Theorem \ref{diagram reduction}.

\begin{lemma}
	The following follow from \eqref{H=I top brown} and all other non-$H=I$ relations.
	\excludeDiagrams{0}
	{
		\begin{equation}\label{H=I side brown}
			2
			\begin{strings}
				\Hlongparam{(0,0)}{0.5}{1}{green}{brown}{green}{green}{brown}
			\end{strings}
			=
			\begin{strings}
				\Ilongparam{(0,0)}{0.5}{1}{green}{brown}{brown}{green}{brown}
			\end{strings}
			+
			\begin{strings}
				\Ilongparam{(0,0)}{0.5}{1}{green}{brown}{brown}{green}{brown}
				\draw[orange] (nd-1-8) -- (nd-1-9);
			\end{strings}
		\end{equation}
		
		\begin{equation}\label{H=I T brown}
			2
			\begin{strings}
				\Hlongparam{(0,0)}{0.5}{1}
				{green}{brown}{brown}{green}{brown}
			\end{strings}
			=
			\begin{strings}
				\Ilongparam{(0,0)}{0.5}{1}
				{green}{brown}{brown}{green}{brown}
			\end{strings}
			-
			\begin{strings}
				\Ilongparam{(0,0)}{0.5}{1}
				{green}{brown}{brown}{green}{brown}
				\draw[orange] (nd-1-8) -- (nd-1-9);
			\end{strings}
		\end{equation}
	}
	
	The relation
	{
		\begin{equation}\label{H=I brown arms and legs}
			2
			\begin{strings}
				\Hlongparam{(0,0)}{0.5}{1}{brown}{brown}{green}{brown}{brown}
			\end{strings}
			=
			\begin{strings}
				\Ilongparam{(0,0)}{0.5}{1}{brown}{brown}{brown}{brown}{brown}
				\hbarbparam{(blBarb)}{0.25}{0.04}{green}
			\end{strings}
			+
			\begin{strings}
				\Ilongparam{(0,0)}{0.5}{1}{brown}{brown}{brown}{brown}{brown}
				\hbarbparam{(tlBarb)}{0.25}{0.04}{green}
			\end{strings}
		\end{equation}
	}	
	also requires \eqref{H=I one color}.
\end{lemma}
\begin{proof}
	By \eqref{H=I top brown} we get
	\begin{equation}\label{eq top brown orange line}
		\begin{boxedstrings}
			\Hlongparam{(0,1)}{0.5}{1}{brown}{brown}{brown}{green}{green}
			\draw[orange] (0,0.25)--(0.5,0.25);
		\end{boxedstrings}
		=
		\begin{boxedstrings}
			\Ilongparam{(0,1)}{0.5}{1}{brown}{brown}{green}{green}{green}
			\draw[orange] (0,0.125)--(0.5,0.125);
		\end{boxedstrings} 
		+
		\begin{boxedstrings}
			\Ilongparam{(0,1)}{0.5}{1}{brown}{brown}{brown}{green}{green}
			\draw[orange] (0,0.125)--(0.5,0.125);
		\end{boxedstrings}  
		=  
		\begin{boxedstrings}
			\Ilongparam{(0,1)}{0.5}{1}{brown}{brown}{green}{green}{green}
		\end{boxedstrings} 
		-
		\begin{boxedstrings}
			\Ilongparam{(0,1)}{0.5}{1}{brown}{brown}{brown}{green}{green}
		\end{boxedstrings}	.
	\end{equation}
	Adding and subtracting 
	\eqref{H=I top brown} and \eqref{eq top brown orange line}
	we get 
		\begin{equation}
		\begin{boxedstrings}
			\Hlongparam{(0,1)}{0.5}{1}{brown}{brown}{brown}{green}{green}
		\end{boxedstrings}
		+
		\begin{boxedstrings}
			\Hlongparam{(0,1)}{0.5}{1}{brown}{brown}{brown}{green}{green}
			\draw[orange] (0,0.25)--(0.5,0.25);
		\end{boxedstrings}
		=
		2
		\begin{boxedstrings}
			\Ilongparam{(0,1)}{0.5}{1}{brown}{brown}{green}{green}{green}
		\end{boxedstrings} 	
	\end{equation}
	and
	\begin{equation}
		\begin{boxedstrings}
			\Hlongparam{(0,1)}{0.5}{1}{brown}{brown}{brown}{green}{green}
		\end{boxedstrings}
		-
		\begin{boxedstrings}
			\Hlongparam{(0,1)}{0.5}{1}{brown}{brown}{brown}{green}{green}
			\draw[orange] (0,0.25)--(0.5,0.25);
		\end{boxedstrings}
		=
		2
		\begin{boxedstrings}
			\Ilongparam{(0,1)}{0.5}{1}{brown}{brown}{brown}{green}{green}
		\end{boxedstrings}.	
	\end{equation}
	These are rotated versions of \eqref{H=I side brown} and
	 \eqref{H=I T brown} respectively.
	 
	 For \eqref{H=I brown arms and legs},
	 \begin{align}
	 		\begin{strings}
	 			\Ilongparam{(0,0)}{0.5}{1}{brown}{brown}{brown}{brown}{brown}
	 			\hbarbparam{(blBarb)}{0.25}{0.04}{green}
	 		\end{strings}
	 		+
	 		\begin{strings}
	 			\Ilongparam{(0,0)}{0.5}{1}{brown}{brown}{brown}{brown}{brown}
	 			\hbarbparam{(tlBarb)}{0.25}{0.04}{green}
	 		\end{strings}
	 		&=
	 		\begin{strings}
	 			\Hlongparam{(0,1)}{0.5}{1}{brown}{brown}{brown}{brown}{brown}
	 			\hbarbparam{(blBarb)}{0.25}{0.04}{green}
	 		\end{strings}
	 		+
	 		\begin{strings}
	 			\Hlongparam{(0,1)}{0.5}{1}{brown}{brown}{brown}{brown}{brown}
	 			\hbarbparam{(tlBarb)}{0.25}{0.04}{green}
	 		\end{strings}
	 	\\
	 	&= 
	 	2
	 	\begin{strings}
	 		\Hmidfigparamcol{(0,1)}{1}{1}{brown}{brown}{brown}{white}{brown}{brown}
	 		\dotuparam{(nd-1-7)}{0.25}{0.05}{green}
	 		\dotuparam{(nd-1-8)}{0.25}{0.05}{green}
	 		\draw[green] (nd-1-7)-- ++(0,-0.25) -- ++(0.5,0) -- (nd-1-8);
	 	\end{strings}.
	 \end{align}
	Applying \eqref{H=I T brown} on the left side of the diagram 
	and using \eqref{bigon brown orange} to get rid of the middle term we get,
	\begin{equation}
		2
		\begin{strings}
			\Hmidfigparamcol{(0,1)}{1}{1}
			{brown}{brown}{brown}{white}{brown}{brown}
			\dotuparam{(nd-1-7)}{0.25}{0.05}{green}
			\dotuparam{(nd-1-8)}{0.25}{0.05}{green}
			\draw[green] (nd-1-7)-- ++(0,-0.25) -- ++(0.5,0) -- (nd-1-8);
		\end{strings}
		=
		\begin{strings}
			\renewcommand{\tempcolora}{brown}
			\renewcommand{\tempcolorb}{brown}
			\renewcommand{\tempcolorc}{green}
			\renewcommand{\tempcolord}{green}
			\renewcommand{\tempcolore}{brown}
			\renewcommand{\tempcolorf}{brown}
			\renewcommand{\tempcolorg}{brown}
			\Hmidfigparamsetcol{(0,1)}{1}{1}
			\dotrparam{(0,0.75)}{0.5}{0.05}{green}
			\dotuparam{(nd-1-8)}{0.25}{0.05}{green}
		\end{strings}
		+
		\begin{strings}
			\renewcommand{\tempcolora}{brown}
			\renewcommand{\tempcolorb}{brown}
			\renewcommand{\tempcolorc}{green}
			\renewcommand{\tempcolord}{green}
			\renewcommand{\tempcolore}{brown}
			\renewcommand{\tempcolorf}{brown}
			\renewcommand{\tempcolorg}{brown}
			\Hmidfigparamsetcol{(0,1)}{1}{1}
			\dotrparam{(0,0.75)}{0.5}{0.05}{green}
			\draw[orange] (nd-1-7) -- ++(0,0.25);
			\dotuparam{(nd-1-8)}{0.25}{0.05}{green}
		\end{strings}
		=
		2
		\begin{strings}
			\renewcommand{\tempcolora}{brown}
			\renewcommand{\tempcolorb}{brown}
			\renewcommand{\tempcolorc}{green}
			\renewcommand{\tempcolord}{green}
			\renewcommand{\tempcolore}{brown}
			\renewcommand{\tempcolorf}{brown}
			\renewcommand{\tempcolorg}{brown}
			\Hmidfigparamsetcol{(0,1)}{1}{1}
			\dotuparam{(nd-1-8)}{0.25}{0.05}{green}
		\end{strings}.
	\end{equation}
	
	Aplying the same procedure on the right side,
	\begin{equation}
		2
		\begin{strings}
			\renewcommand{\tempcolora}{brown}
			\renewcommand{\tempcolorb}{brown}
			\renewcommand{\tempcolorc}{green}
			\renewcommand{\tempcolord}{green}
			\renewcommand{\tempcolore}{brown}
			\renewcommand{\tempcolorf}{brown}
			\renewcommand{\tempcolorg}{brown}
			\Hmidfigparamsetcol{(0,1)}{1}{1}
			\dotuparam{(nd-1-8)}{0.25}{0.05}{green}
		\end{strings}
		=
		\begin{strings}
			\Hlongparam{(0,1)}{0.5}{1}{brown}{brown}{green}{brown}{brown}
			\dotlparam{(0.5,0.75)}{0.15}{0.04}{green}
		\end{strings}
		=
		2
		\begin{strings}
			\Hlongparam{(0,1)}{0.5}{1}{brown}{brown}{green}{brown}{brown}
		\end{strings}.
	\end{equation}
\end{proof}

\begin{lemma}
	The following follows from \eqref{H=I L brown leg} and all other non-$H=I$ relations.
	\begin{equation}\label{H=I brown leg}
		2
		\begin{boxedstrings}
			\Hlongparam{(0,0)}{0.5}{1}{green}{green}{green}{green}{brown}
		\end{boxedstrings}
		=
		\begin{boxedstrings}
			\Ilongparam{(0,0)}{0.5}{1}{green}{green}{brown}{green}{brown}
		\end{boxedstrings}
		+
		\begin{boxedstrings}
			\Ilongparam{(0,0)}{0.5}{1}{green}{green}{brown}{green}{brown}
			\draw[orange] (nd-1-8) -- (nd-1-9);
		\end{boxedstrings}
	\end{equation}
\end{lemma}
\begin{proof}
	By \eqref{H=I L brown leg} we have
	\begin{align}
		2
		\begin{strings}
			\Hlongparam{(0,0)}{0.5}{1}{green}{green}{brown}{green}{brown}
			\draw[orange] (tlsDot) -- (trsDot);
		\end{strings}
		&=
		2
		\begin{strings}
			\Ilongparam{(0,0)}{0.5}{1}{green}{green}{green}{green}{brown}
			\draw[orange] (tlsDot) -- (trsDot);
		\end{strings}
		+
		\begin{strings}
			\Ilongparam{(0,0)}{0.5}{1}{green}{green}{brown}{green}{brown}
			\draw[orange] (tlsDot) -- (trsDot);
		\end{strings}
		-
		\begin{strings}
			\Ilongparam{(0,0)}{0.5}{1}{green}{green}{brown}{green}{brown}
			\draw[orange] (tlsDot) -- (trsDot);
			\draw[orange] (nd-1-8) -- (nd-1-9);
		\end{strings}
		\\ \label{eqn orange Ls}
		&=
		2
		\begin{strings}
			\Ilongparam{(0,0)}{0.5}{1}{green}{green}{green}{green}{brown}
		\end{strings}
		-
		\begin{strings}
			\Ilongparam{(0,0)}{0.5}{1}{green}{green}{brown}{green}{brown}
		\end{strings}
		+
		\begin{strings}
			\Ilongparam{(0,0)}{0.5}{1}{green}{green}{brown}{green}{brown}
			\draw[orange] (nd-1-8) -- (nd-1-9);
		\end{strings}
	\end{align}
	Adding \eqref{H=I L brown leg} and \eqref{eqn orange Ls} we get
	\begin{equation}
		2
		\begin{strings}
			\Hlongparam{(0,0)}{0.5}{1}{green}{green}{brown}{green}{brown}
		\end{strings}
		+
		2
		\begin{strings}
			\Hlongparam{(0,0)}{0.5}{1}{green}{green}{brown}{green}{brown}
			\draw[orange] (tlsDot) -- (trsDot);
		\end{strings}
		=
		4
		\begin{strings}
			\Ilongparam{(0,0)}{0.5}{1}{green}{green}{green}{green}{brown}
		\end{strings},
	\end{equation}
	which is a rotation of \eqref{H=I L brown leg}.
\end{proof}

\begin{lemma}

We have the following relations:

\begin{equation}\label{H=I brown bottom feet}
	4
	\begin{strings}
		\Hlongparam{(0,0)}{0.5}{1}{green}{green}{green}{brown}{brown}
	\end{strings}
	=
	\begin{strings}
		\Ilongparam{(0,0)}{0.5}{1}{green}{green}{green}{brown}{brown}
		\hbarbparam{(blBarb)}{0.25}{0.04}{green}
	\end{strings}
	+
	\begin{strings}
		\Ilongparam{(0,0)}{0.5}{1}{green}{green}{green}{brown}{brown}
		\hbarbparam{(tlBarb)}{0.25}{0.04}{green}
	\end{strings}
	+
	\begin{strings}
		\Ilongparam{(0,0)}{0.5}{1}{green}{green}{green}{brown}{brown}
		\draw[orange] (trDot) -- (0.375,-0.5) -- (nd-1-7);
		\dotuparam{(trDot)}{0.125}{0.04}{orange}
	\end{strings}
	+
	\begin{strings}
		\Ilongparam{(0,0)}{0.5}{1}{green}{green}{green}{brown}{brown}
		\draw[orange] (brDot) -- (0.375,-0.5) -- (nd-1-7);
		\dotdparam{(brDot)}{0.125}{0.04}{orange}
	\end{strings}
\end{equation}

\begin{equation}\label{H=I 3 brown edges}
    4
    \begin{strings}
        \Hlongparam{(0,0)}{0.5}{1}{green}{brown}{green}{brown}{brown}
    \end{strings}
    =
    \begin{strings}
        \Ilongparam{(0,0)}{0.5}{1}{green}{brown}{brown}{brown}{brown}
        \hbarbparam{(blBarb)}{0.25}{0.04}{green}
    \end{strings}
    +
    \begin{strings}
        \Ilongparam{(0,0)}{0.5}{1}{green}{brown}{brown}{brown}{brown}
        \hbarbparam{(tlBarb)}{0.25}{0.04}{green}
    \end{strings}
    +
    \begin{strings}
        \Ilongparam{(0,0)}{0.5}{1}{green}{brown}{brown}{brown}{brown}
        \dotuparam{(tlDot)}{0.125}{0.04}{orange}
    \end{strings}
    +
    \begin{strings}
        \Ilongparam{(0,0)}{0.5}{1}{green}{brown}{brown}{brown}{brown}
        \dotdparam{(brDot)}{0.125}{0.04}{orange}
        \draw[orange] (brDot) -- (0.375,-0.5) -- (nd-1-7) -- (0.125,-0.5) -- (nd-1-8);
    \end{strings}
\end{equation}

\begin{equation}\label{H=I opposite brown edges (idemp decomp)}
\begin{split}
    8
    \begin{strings}
        \Hlongparam{(0,0)}{0.5}{1}{green}{brown}{green}{brown}{green}
    \end{strings}
    =&
    \begin{strings}
        \Ilongparam{(0,0)}{0.5}{1}{green}{brown}{brown}{brown}{green}
        \hbarbparam{(blBarb)}{0.25}{0.04}{green}
    \end{strings}
    +
    \begin{strings}
        \Ilongparam{(0,0)}{0.5}{1}{green}{brown}{brown}{brown}{green}
        \draw[orange] (tlDot) -- (0.125,-0.5) -- (nd-1-7) -- (0.375,-0.5) -- (brDot);
        \hbarbparam{(blBarb)}{0.25}{0.04}{green}
    \end{strings}
    +
    \begin{strings}
        \Ilongparam{(0,0)}{0.5}{1}{green}{brown}{brown}{brown}{green}
        \hbarbparam{(tlBarb)}{0.25}{0.04}{green}
    \end{strings}
    +
    \begin{strings}
        \Ilongparam{(0,0)}{0.5}{1}{green}{brown}{brown}{brown}{green}
        \draw[orange] (tlDot) -- (0.125,-0.5) -- (nd-1-7) -- (0.375,-0.5) -- (brDot);
        \hbarbparam{(tlBarb)}{0.25}{0.04}{green}
    \end{strings}
    \\
    \ 
    +&
    \begin{strings}
        \Ilongparam{(0,0)}{0.5}{1}{green}{brown}{brown}{brown}{green}
        \dotdparam{(brDot)}{0.125}{0.04}{orange}
    \end{strings}
    +
    \begin{strings}
        \Ilongparam{(0,0)}{0.5}{1}{green}{brown}{brown}{brown}{green}
        \draw[orange] (tlDot) -- (0.125,-0.5) -- (nd-1-7) -- (0.375,-0.5) -- (brDot);
        \dotdparam{(brDot)}{0.125}{0.04}{orange}
    \end{strings}
    +
    \begin{strings}
        \Ilongparam{(0,0)}{0.5}{1}{green}{brown}{brown}{brown}{green}
        \dotuparam{(tlDot)}{0.125}{0.04}{orange}
    \end{strings}
    +
    \begin{strings}
        \Ilongparam{(0,0)}{0.5}{1}{green}{brown}{brown}{brown}{green}
        \draw[orange] (tlDot) -- (0.125,-0.5) -- (nd-1-7) -- (0.375,-0.5) -- (brDot);
        \dotuparam{(tlDot)}{0.125}{0.04}{orange}
    \end{strings}
\end{split}
\end{equation}

\end{lemma}
\begin{proof}
	We will show \eqref{H=I brown bottom feet} and leave the rest to the reader.
	From \ref{Y2 decomp} and its rotation we get
	\begin{align}\label{eqn lemma step1}
		2
		\begin{strings}
			\Hlongparam{(0,0)}{0.5}{1.25}{green}{green}{green}{brown}{brown}
		\end{strings}
		&=
		\begin{strings}
			\Ilongparam{(0,1)}{0.5}{1}{green}{green}{green}{green}{green}
			\hbarbparam{(tlBarb)}{0.25}{0.04}{green}
			\Hfigparamcol{(nd-1-3)}{0.5}{0.5}
			{green}{green}{green}{brown}{brown}
		\end{strings}
		+
		\begin{strings}
			\Ilongparam{(0,1)}{0.5}{1}{green}{green}{green}{green}{green}
			\dotrparam{(blsDot)}{0.25}{0.04}{orange}
			\Hfigparamcol{(nd-1-3)}{0.5}{0.5}
			{green}{green}{green}{brown}{brown}
		\end{strings}
		+
		\begin{strings}
			\Ilongparam{(0,1)}{0.5}{1}{green}{green}{brown}{green}{green}
			\Hfigparamcol{(nd-1-3)}{0.5}{0.5}
			{green}{green}{green}{brown}{brown}
		\end{strings}
		+
		\begin{strings}
			\Ilongparam{(0,1)}{0.5}{1}{green}{green}{brown}{green}{green}
			\draw[orange] (nd-1-8)--(nd-1-9);
			\Hfigparamcol{(nd-1-3)}{0.5}{0.5}
			{green}{green}{green}{brown}{brown}
		\end{strings},
		\\
		\label{eqn lemma step2}
		2
		\begin{strings}
			\Hlongparam{(0,0)}{0.5}{1.25}{green}{green}{green}{brown}{brown}
		\end{strings}
		&=
		\begin{strings}
			\Ilongparam{(0,1)}{0.5}{1}{green}{green}{green}{green}{green}
			\hbarbparam{(blBarb)}{0.25}{0.04}{green}
			\Hfigparamcol{(nd-1-3)}{0.5}{0.5}
			{green}{green}{green}{brown}{brown}
		\end{strings}
		+
		\begin{strings}
			\Ilongparam{(0,1)}{0.5}{1}{green}{green}{green}{green}{green}
			\dotrparam{(tlsDot)}{0.25}{0.04}{orange}
			\Hfigparamcol{(nd-1-3)}{0.5}{0.5}
			{green}{green}{green}{brown}{brown}
		\end{strings}
		+
		\begin{strings}
			\Ilongparam{(0,1)}{0.5}{1}{green}{green}{brown}{green}{green}
			\Hfigparamcol{(nd-1-3)}{0.5}{0.5}
			{green}{green}{green}{brown}{brown}
		\end{strings}
		+
		\begin{strings}
			\Ilongparam{(0,1)}{0.5}{1}{green}{green}{brown}{green}{green}
			\draw[orange] (nd-1-8)--(nd-1-9);
			\Hfigparamcol{(nd-1-3)}{0.5}{0.5}
			{green}{green}{green}{brown}{brown}
		\end{strings}.
	\end{align}
	
	Using \eqref{H=I middle brown}, we can show that
	\begin{equation}
		\begin{strings}
			\Ilongparam{(0,1)}{0.5}{1}{green}{green}{brown}{green}{green}
			\Hfigparamcol{(nd-1-3)}{0.5}{0.5}
			{green}{green}{green}{brown}{brown}
		\end{strings}
		= 0
		\hspace{0.2cm}
		\text{and}
		\hspace{0.2cm}
		\begin{strings}
			\Ilongparam{(0,1)}{0.5}{1}{green}{green}{brown}{green}{green}
			\draw[orange] (nd-1-8)--(nd-1-9);
			\Hfigparamcol{(nd-1-3)}{0.5}{0.5}
			{green}{green}{green}{brown}{brown}
		\end{strings}
		=0.
	\end{equation}
	
	Adding \eqref{eqn lemma step1} and \eqref{eqn lemma step2} we get
	\begin{equation}
		4
		\begin{strings}
			\Hlongparam{(0,0)}{0.5}{1.25}{green}{green}{green}{brown}{brown}
		\end{strings}
		=
		\begin{strings}
			\Ilongparam{(0,1)}{0.5}{1.25}{green}{green}{green}{brown}{brown}
			\hbarbparam{(tlBarb)}{0.25}{0.04}{green}
		\end{strings}
		+
		\begin{strings}
			\Ilongparam{(0,1)}{0.5}{1.25}{green}{green}{green}{brown}{brown}
			\dotrparam{(tlsDot)}{0.25}{0.04}{orange}
		\end{strings}
		+
		\begin{strings}
			\Ilongparam{(0,1)}{0.5}{1}{green}{green}{green}{green}{green}
			\hbarbparam{(blBarb)}{0.25}{0.04}{green}
			\Hfigparamcol{(nd-1-3)}{0.5}{0.5}
			{green}{green}{green}{brown}{brown}
		\end{strings}
		+
		\begin{strings}
			\Ilongparam{(0,1)}{0.5}{1}{green}{green}{green}{green}{green}
			\dotrparam{(blsDot)}{0.25}{0.04}{orange}
			\Hfigparamcol{(nd-1-3)}{0.5}{0.5}
			{green}{green}{green}{brown}{brown}
		\end{strings}
		.
	\end{equation}
	It remains to show that
	\begin{equation}\label{eqn target sum of triangles}
		\begin{strings}
			\Ilongparam{(0,1)}{0.5}{1}{green}{green}{green}{green}{green}
			\hbarbparam{(blBarb)}{0.25}{0.04}{green}
			\Hfigparamcol{(nd-1-3)}{0.5}{0.5}
			{green}{green}{green}{brown}{brown}
		\end{strings}
		+
		\begin{strings}
			\Ilongparam{(0,1)}{0.5}{1}
			{green}{green}{green}{green}{green}
			\dotrparam{(blsDot)}{0.25}{0.04}{orange}
			\Hfigparamcol{(nd-1-3)}{0.5}{0.5}
			{green}{green}{green}{brown}{brown}
		\end{strings}
		=
		\begin{strings}
			\Ilongparam{(0,1)}{0.5}{1.25}
			{green}{green}{green}{brown}{brown}
			\hbarbparam{(blBarb)}{0.25}{0.04}{green}
		\end{strings}
		+
		\begin{strings}
			\Ilongparam{(0,1)}{0.5}{1.25}
			{green}{green}{green}{brown}{brown}
			\dotrparam{(blsDot)}{0.25}{0.04}{orange}
		\end{strings}.
	\end{equation}

	Using \eqref{H=I brown leg},
	\begin{align}
		2
		\begin{strings}
			\Ilongparam{(0,1)}{0.5}{1}{green}{green}{green}{green}{green}
			\hbarbparam{(blBarb)}{0.25}{0.04}{green}
			\Hfigparamcol{(nd-1-3)}{0.5}{0.5}
			{green}{green}{green}{brown}{brown}
		\end{strings}
		=&
		\begin{strings}
			\node (a) at (-0.2,0.3) {};
			\def\coltop{green}
			\def\colbot{brown}
			\capfigparam{(a)}{0.4}{0.2}{\coltop}
			\cupfigparam{(a)}{0.4}{0.2}{\coltop}
			\idfigparam {(0,0.1)}{0.1}{\colbot}
			\idfigparam {(0,0.6)}{0.1}{\colbot}
			\hbarbparam{(-0.075,0.3)}{0.15}{0.03}{green}
			\Ifigparamcol{(0,1.05)}{0.5}{0.45}
			{green}{green}{green}{brown}{brown}
			\idfigparam{(nd-6-1)}{-0.2}{green}
			\idfigparam{(nd-6-2)}{-0.2}{green}
			\idfigparam{(nd-6-4)}{0.6}{brown}
		\end{strings}
		+
		\begin{strings}
			\node (a) at (-0.2,0.3) {};
			\def\coltop{green}
			\def\colbot{brown}
			\capfigparam{(a)}{0.4}{0.2}{\coltop}
			\cupfigparam{(a)}{0.4}{0.2}{\coltop}
			\idfigparam {(0,0.1)}{0.1}{\colbot}
			\idfigparam {(0,0.6)}{0.1}{\colbot}
			\hbarbparam{(-0.075,0.3)}{0.15}{0.03}{green}
			\Ifigparamcol{(0,1.05)}{0.5}{0.45}
			{green}{green}{green}{brown}{brown}
			\draw[orange] (a) -- ++(-0.15,0) 
			-- ++(0,0.525) -- (nd-6-7);
			\idfigparam{(nd-6-1)}{-0.2}{green}
			\idfigparam{(nd-6-2)}{-0.2}{green}
			\idfigparam{(nd-6-4)}{0.6}{brown}
		\end{strings}
		\\
		\begin{split}
		=&	- 
		\begin{strings}
			\node (a) at (-0.2,0.3) {};
			\def\coltop{green}
			\def\colbot{brown}
			\capfigparam{(a)}{0.4}{0.2}{\coltop}
			\cupfigparam{(a)}{0.4}{0.2}{\coltop}
			\idfigparam {(0,0.1)}{0.1}{\colbot}
			\idfigparam {(0,0.6)}{0.1}{\colbot}
			\Ifigparamcol{(0,1.05)}{0.5}{0.45}
			{green}{green}{green}{brown}{brown}
			\idfigparam{(nd-5-1)}{-0.2}{green}
			\idfigparam{(nd-5-2)}{-0.2}{green}
			\idfigparam{(nd-5-4)}{0.6}{brown}
			\dotlparam{(a)}{0.15}{0.03}{orange}
		\end{strings}
		+
		\begin{strings}
			\trivuparamcol{(0,0.6)}{0.4}{0.2}
			{brown}{green}{green}
			\idfigparam{(nd-1-3)}{0.2}{green}
			\trivdparamcol{(nd-2-2)}{-0.4}{0.2}
			{green}{green}{brown}	
			\dotfigparamcol{(nd-1-2)}{0.03}{green}
			\dotfigparamcol{(nd-3-2)}{0.03}{green}		
			\Ifigparamcol{(0,1.05)}{0.5}{0.45}
			{green}{green}{green}{brown}{brown}
			\idfigparam{(nd-4-1)}{-0.2}{green}
			\idfigparam{(nd-4-2)}{-0.2}{green}
			\idfigparam{(nd-4-4)}{0.6}{brown}
		\end{strings}
		+
		\begin{strings}
			\trivuparamcol{(0,0.6)}{0.4}{0.2}
			{brown}{green}{green}
			\idfigparam{(nd-1-3)}{0.2}{green}
			\trivdparamcol{(nd-2-2)}{-0.4}{0.2}
			{green}{green}{brown}	
			\dotfigparamcol{(nd-1-2)}{0.03}{green}
			\dotfigparamcol{(nd-3-2)}{0.03}{green}
			\draw[orange] (-0.1,0.155)-- ++(0,0.29);		
			\Ifigparamcol{(0,1.05)}{0.5}{0.45}
			{green}{green}{green}{brown}{brown}
			\idfigparam{(nd-4-1)}{-0.2}{green}
			\idfigparam{(nd-4-2)}{-0.2}{green}
			\idfigparam{(nd-4-4)}{0.6}{brown}
		\end{strings}
		\\
		& +
		\begin{strings}
			\node (a) at (-0.1,0.3) {};
			\def\coltop{green}
			\def\colbot{brown}
			\capfigparam{(a)}{0.2}{0.1}{\coltop}
			\cupfigparam{(a)}{0.2}{0.1}{\coltop}
			\idfigparam {(0,0.2)}{0.2}{\colbot}
			\idfigparam {(0,0.6)}{0.2}{\colbot}
			\Ifigparamcol{(0,1.05)}{0.5}{0.45}
			{green}{green}{green}{brown}{brown}
			\draw[orange] (nd-5-7)-- ++(-0.125,0) --(nd-5-9);
			\dotdparam{(nd-5-9)}{0.15}{0.03}{orange}
			\idfigparam{(nd-5-1)}{-0.2}{green}
			\idfigparam{(nd-5-2)}{-0.2}{green}
			\idfigparam{(nd-5-4)}{0.6}{brown}
		\end{strings}
		-
		\begin{strings}
			\trivuparamcol{(0,0.6)}{0.4}{0.2}
			{brown}{green}{green}
			\idfigparam{(nd-1-3)}{0.2}{green}
			\trivdparamcol{(nd-2-2)}{-0.4}{0.2}
			{green}{green}{brown}	
			\dotfigparamcol{(nd-1-2)}{0.03}{green}
			\dotfigparamcol{(nd-3-2)}{0.03}{green}		
			\Ifigparamcol{(0,1.05)}{0.5}{0.45}
			{green}{green}{green}{brown}{brown}
			\draw[orange] (nd-5-7)-- ++(-0.125,0)
			--(nd-4-9) -- ++(0,-0.15);
			\idfigparam{(nd-4-1)}{-0.2}{green}
			\idfigparam{(nd-4-2)}{-0.2}{green}
			\idfigparam{(nd-4-4)}{0.6}{brown}
		\end{strings}
		-
		\begin{strings}
			\trivuparamcol{(0,0.6)}{0.4}{0.2}
			{brown}{green}{green}
			\idfigparam{(nd-1-3)}{0.2}{green}
			\trivdparamcol{(nd-2-2)}{-0.4}{0.2}
			{green}{green}{brown}	
			\dotfigparamcol{(nd-1-2)}{0.03}{green}
			\dotfigparamcol{(nd-3-2)}{0.03}{green}
			\draw[orange] (-0.1,0.155)-- ++(0,0.29);		
			\Ifigparamcol{(0,1.05)}{0.5}{0.45}
			{green}{green}{green}{brown}{brown}
			\draw[orange] (nd-5-7)-- ++(-0.125,0)
			--(nd-4-9) -- ++(0,-0.15);
			\idfigparam{(nd-4-1)}{-0.2}{green}
			\idfigparam{(nd-4-2)}{-0.2}{green}
			\idfigparam{(nd-4-4)}{0.6}{brown}
		\end{strings}
		\end{split}
		\\
		=&
		2 
		\begin{strings}
			\Ilongparam{(0,0)}{0.5}{1.25}{green}{green}{green}{brown}{brown}
			\draw[orange]  (nd-1-7)-- ++(0.125,0)
			--(brDot) ;
			\dotdparam{(brDot)}{0.125}{0.04}{orange}
		\end{strings}
		+
		2
		\begin{strings}
			\trivuparamcol{(0,0.6)}{0.4}{0.2}
			{brown}{green}{green}
			\idfigparam{(nd-1-3)}{0.2}{green}
			\trivdparamcol{(nd-2-2)}{-0.4}{0.2}
			{green}{green}{brown}	
			\dotfigparamcol{(nd-1-2)}{0.03}{green}
			\dotfigparamcol{(nd-3-2)}{0.03}{green}		
			\Ifigparamcol{(0,1.05)}{0.5}{0.45}
			{green}{green}{green}{brown}{brown}
			\idfigparam{(nd-4-1)}{-0.2}{green}
			\idfigparam{(nd-4-2)}{-0.2}{green}
			\idfigparam{(nd-4-4)}{0.6}{brown}
		\end{strings}
		-
		2
		\begin{strings}
			\trivuparamcol{(0,0.6)}{0.4}{0.2}
			{brown}{green}{green}
			\idfigparam{(nd-1-3)}{0.2}{green}
			\trivdparamcol{(nd-2-2)}{-0.4}{0.2}
			{green}{green}{brown}	
			\dotfigparamcol{(nd-1-2)}{0.03}{green}
			\dotfigparamcol{(nd-3-2)}{0.03}{green}		
			\Ifigparamcol{(0,1.05)}{0.5}{0.45}
			{green}{green}{green}{brown}{brown}
			\draw[orange] (nd-5-7)-- ++(-0.125,0)
			--(nd-4-9) -- ++(0,-0.15);
			\idfigparam{(nd-4-1)}{-0.2}{green}
			\idfigparam{(nd-4-2)}{-0.2}{green}
			\idfigparam{(nd-4-4)}{0.6}{brown}
		\end{strings}.
	\end{align}
	
	Hence,
	\begin{equation}\label{eqn triangle 1}
		\begin{strings}
			\Ilongparam{(0,1)}{0.5}{1}{green}{green}{green}{green}{green}
			\hbarbparam{(blBarb)}{0.25}{0.04}{green}
			\Hfigparamcol{(nd-1-3)}{0.5}{0.5}
			{green}{green}{green}{brown}{brown}
		\end{strings}
		=
		\begin{strings}
			\Ilongparam{(0,0)}{0.5}{1.25}{green}{green}{green}{brown}{brown}
			\draw[orange]  (nd-1-7)-- ++(0.125,0)
			--(brDot) ;
			\dotdparam{(brDot)}{0.125}{0.04}{orange}
		\end{strings}
		+
		\begin{strings}
			\trivuparamcol{(0,0.6)}{0.4}{0.2}
			{brown}{green}{green}
			\idfigparam{(nd-1-3)}{0.2}{green}
			\trivdparamcol{(nd-2-2)}{-0.4}{0.2}
			{green}{green}{brown}	
			\dotfigparamcol{(nd-1-2)}{0.03}{green}
			\dotfigparamcol{(nd-3-2)}{0.03}{green}		
			\Ifigparamcol{(0,1.05)}{0.5}{0.45}
			{green}{green}{green}{brown}{brown}
			\idfigparam{(nd-4-1)}{-0.2}{green}
			\idfigparam{(nd-4-2)}{-0.2}{green}
			\idfigparam{(nd-4-4)}{0.6}{brown}
		\end{strings}
		-
		\begin{strings}
			\trivuparamcol{(0,0.6)}{0.4}{0.2}
			{brown}{green}{green}
			\idfigparam{(nd-1-3)}{0.2}{green}
			\trivdparamcol{(nd-2-2)}{-0.4}{0.2}
			{green}{green}{brown}	
			\dotfigparamcol{(nd-1-2)}{0.03}{green}
			\dotfigparamcol{(nd-3-2)}{0.03}{green}		
			\Ifigparamcol{(0,1.05)}{0.5}{0.45}
			{green}{green}{green}{brown}{brown}
			\draw[orange] (nd-5-7)-- ++(-0.125,0)
			--(nd-4-9) -- ++(0,-0.15);
			\idfigparam{(nd-4-1)}{-0.2}{green}
			\idfigparam{(nd-4-2)}{-0.2}{green}
			\idfigparam{(nd-4-4)}{0.6}{brown}
		\end{strings}.
	\end{equation}
	
	A similar calculation yields
	\begin{equation}\label{eqn triangle 2}
		\begin{strings}
			\Ilongparam{(0,1)}{0.5}{1}
			{green}{green}{green}{green}{green}
			\dotrparam{(blsDot)}{0.25}{0.04}{orange}
			\Hfigparamcol{(nd-1-3)}{0.5}{0.5}
			{green}{green}{green}{brown}{brown}
		\end{strings}
		=
		\begin{strings}
			\Ilongparam{(0,0)}{0.5}{1.25}{green}{green}{green}{brown}{brown}
			\hbarbparam{(blBarb)}{0.25}{0.04}{green}
		\end{strings}
		-
		\begin{strings}
			\trivuparamcol{(0,0.6)}{0.4}{0.2}
			{brown}{green}{green}
			\idfigparam{(nd-1-3)}{0.2}{green}
			\trivdparamcol{(nd-2-2)}{-0.4}{0.2}
			{green}{green}{brown}	
			\dotfigparamcol{(nd-1-2)}{0.03}{green}
			\dotfigparamcol{(nd-3-2)}{0.03}{green}		
			\Ifigparamcol{(0,1.05)}{0.5}{0.45}
			{green}{green}{green}{brown}{brown}
			\idfigparam{(nd-4-1)}{-0.2}{green}
			\idfigparam{(nd-4-2)}{-0.2}{green}
			\idfigparam{(nd-4-4)}{0.6}{brown}
		\end{strings}
		+
		\begin{strings}
			\trivuparamcol{(0,0.6)}{0.4}{0.2}
			{brown}{green}{green}
			\idfigparam{(nd-1-3)}{0.2}{green}
			\trivdparamcol{(nd-2-2)}{-0.4}{0.2}
			{green}{green}{brown}	
			\dotfigparamcol{(nd-1-2)}{0.03}{green}
			\dotfigparamcol{(nd-3-2)}{0.03}{green}		
			\Ifigparamcol{(0,1.05)}{0.5}{0.45}
			{green}{green}{green}{brown}{brown}
			\draw[orange] (nd-5-7)-- ++(-0.125,0)
			--(nd-4-9) -- ++(0,-0.15);
			\idfigparam{(nd-4-1)}{-0.2}{green}
			\idfigparam{(nd-4-2)}{-0.2}{green}
			\idfigparam{(nd-4-4)}{0.6}{brown}
		\end{strings}.
	\end{equation}
	
	Adding \eqref{eqn triangle 1} and \eqref{eqn triangle 2} we get \eqref{eqn target sum of triangles}.

\end{proof}

\subsection{General polynomial forcing and needle relations}\label{subsection poly forcing 2}

	\begin{defi}
	Let $f\in R$. Define 
	\begin{align}
		\Sym_\tau(f):=& 
		\frac{f + \tau(f)}{2}
		\\
		\Alt_\tau(f):=&
		\frac{f - \tau(f)}{2}
	\end{align}
	to be the \textit{$\tau$-invariant} and \textit{$\tau$-anti invariant parts of $f$}.
\end{defi}

Notice that $\Sym_\tau(f)\in \Rtau$, $\Alt_\tau(f)\in\Rantitau$, and
\begin{equation}
	f=\Sym_\tau(f)+\Alt_\tau(f).
\end{equation}

The following polynomial forcing formulas will be shown by induction by `forcing'
one green barbell or orange dot at a time.
\begin{prop}[General polynomial forcing]\label{prop general polynomial forcing}
	Let $f\in \Rtau$.
	The following relations hold:
	\begin{align}
		\begin{split}
			\begin{boxedstrings}
				\idfigparam{(0,1)}{1}{green}
			\end{boxedstrings}
			f
			&=
			\Sym_\tau(s(f))
			\begin{boxedstrings}
				\idfigparam{(0,1)}{1}{green}
			\end{boxedstrings}
			+
			\frac{\Alt_\tau(s(f))}{\alpha_s-\alpha_t}
			\begin{boxedstrings}
				\idfigparam{(0,1)}{1}{green}
				\dotl{(nd-1-3)}{orange}
			\end{boxedstrings}
			\\
			&+
			\frac{\Sym(\partial_s(f))}{2}
			\left(
			\begin{boxedstrings}
				\dotd{(1,1)}{green}
				\dotu{(1,0)}{green}
			\end{boxedstrings}
			+
			\begin{boxedstrings}
				\idfigparam{(0,1)}{0.3}{green}
				\idfigparam{(nd-1-2)}{0.4}{orange}
				\idfigparam{(nd-2-2)}{0.3}{green}
			\end{boxedstrings}
			\right)
			+\frac{\Alt(\partial_s(f))}{2(\alpha_s-\alpha_t)}
			\left(
			\begin{boxedstrings}
				\idfigparam{(1,1)}{0.25}{green}
				\dotdparam{(nd-1-2)}{0.2}{0.06}{orange}
				\dotuparam{(1,0)}{0.25}{0.06}{green}
			\end{boxedstrings}
			+
			\begin{boxedstrings}
				\idfigparam{(1,0)}{-0.25}{green}
				\dotuparam{(nd-1-2)}{0.2}{0.06}{orange}
				\dotdparam{(1,1)}{0.25}{0.06}{green}
			\end{boxedstrings}			
			\right)
			,
		\end{split}
		\\
		\begin{split}
			\begin{boxedstrings}
				\idfigparam{(0,1)}{1}{green}
				\dotr{(nd-1-3)}{orange}
			\end{boxedstrings}
			f
			&=
			\Sym_\tau(s(f(\alpha_s-\alpha_t)))
			\begin{boxedstrings}
				\idfigparam{(0,1)}{1}{green}
			\end{boxedstrings}
			+
			\frac{\Alt_\tau(s(f(\alpha_s-\alpha_t)))}{\alpha_s-\alpha_t}
			\begin{boxedstrings}
				\idfigparam{(0,1)}{1}{green}
				\dotl{(nd-1-3)}{orange}
			\end{boxedstrings}
			\\
			&+
			\frac{\Sym(\partial_s(f(\alpha_s-\alpha_t)))}{2}
			\left(
			\begin{boxedstrings}
				\dotd{(1,1)}{green}
				\dotu{(1,0)}{green}
			\end{boxedstrings}
			+
			\begin{boxedstrings}
				\idfigparam{(0,1)}{0.3}{green}
				\idfigparam{(nd-1-2)}{0.4}{orange}
				\idfigparam{(nd-2-2)}{0.3}{green}
			\end{boxedstrings}
			\right)
			+\frac{\Alt(\partial_s(f(\alpha_s-\alpha_t)))}{2(\alpha_s-\alpha_t)}
			\left(
			\begin{boxedstrings}
				\idfigparam{(1,1)}{0.25}{green}
				\dotdparam{(nd-1-2)}{0.2}{0.06}{orange}
				\dotuparam{(1,0)}{0.25}{0.06}{green}
			\end{boxedstrings}
			+
			\begin{boxedstrings}
				\idfigparam{(1,0)}{-0.25}{green}
				\dotuparam{(nd-1-2)}{0.2}{0.06}{orange}
				\dotdparam{(1,1)}{0.25}{0.06}{green}
			\end{boxedstrings}			
			\right)
			,
		\end{split}
	\end{align}
	\begin{align}
		\begin{boxedstrings}
			\idfigparam{(0,1)}{1}{brown}
		\end{boxedstrings}
		f
		=&
		st(f)
		\begin{boxedstrings}
			\idfigparam{(0,1)}{1}{brown}
		\end{boxedstrings}
		+
		\Sym_\tau(t(\partial_s(f)))
		\begin{boxedstrings}
			\idfigparam{(0,1)}{0.3}{brown}
			\idfigparam{(nd-1-2)}{0.4}{green}
			\idfigparam{(nd-2-2)}{0.3}{brown}
		\end{boxedstrings}
		+
		\frac{\Alt_\tau(t(\partial_s(f)))}{\alpha_s-\alpha_t}
		\begin{boxedstrings}
			\idfigparam{(0,1)}{0.3}{brown}
			\idfigparam{(nd-1-2)}{0.4}{green}
			\idfigparam{(nd-2-2)}{0.3}{brown}
			\dotl{(nd-2-3)}{orange}
		\end{boxedstrings}
		+
		\partial_{st}(f)
		\begin{boxedstrings}
			\dotd{(0,1)}{brown}
			\dotu{(0,0)}{brown}
		\end{boxedstrings}
		,
		\\
		\begin{split}
			\begin{boxedstrings}
				\idfigparam{(0,1)}{1}{brown}
				\idfigparam{(-0.25,0.5)}{0.5}{orange}
				\dotr{(nd-1-3)}{orange}
				\draw[orange] (nd-2-1) -- (nd-1-3);
			\end{boxedstrings}
			f
			=&
			-
			st(f)
			\begin{boxedstrings}
				\idfigparam{(0,1)}{1}{brown}
				\dotu{(-0.25,0)}{orange}
			\end{boxedstrings}
			+
			\Sym_\tau(t(\partial_s(f(\alpha_s-\alpha_t))))
			\begin{boxedstrings}
				\idfigparam{(0,1)}{0.3}{brown}
				\idfigparam{(nd-1-2)}{0.4}{green}
				\idfigparam{(nd-2-2)}{0.3}{brown}
				\draw[orange] (-0.25,0) -- ++(0,0.5) -- (nd-2-3);
			\end{boxedstrings}
			\\
			+&
			\frac{\Alt_\tau(t(\partial_s(f(\alpha_s-\alpha_t))))}{\alpha_s-\alpha_t}
			\begin{boxedstrings}
				\idfigparam{(0,1)}{0.3}{brown}
				\idfigparam{(nd-1-2)}{0.4}{green}
				\idfigparam{(nd-2-2)}{0.3}{brown}
				\dotu{(-0.25,0)}{orange}
			\end{boxedstrings}
			+
			\frac{\partial_{st}(f(\alpha_s-\alpha_t))}{\alpha_s-\alpha_t}
			\begin{boxedstrings}
				\dotd{(0,1)}{brown}
				\dotu{(0,0)}{brown}
				\dotu{(-0.25,0)}{orange}
			\end{boxedstrings}
			,
		\end{split}
	\end{align}
\end{prop}

Applying the general polynomial forcing relations to a needle and circle diagrams and then applying the corresponding needle relation to any empty needle, we get the following propositions. 
The proof is a straightforward calculation and is left as an exercise to the reader.
\begin{prop}[General needle relations]\label{prop general needle rels}
	Let $f\in\Rtau$. 
	The following relations hold:
	\begin{align}
		\begin{boxedstrings}
			\draw[green] (-0.25,0.35) rectangle ++(0.5,0.5);
			\idfigparam {(0,0.35)}{0.35}{green}
			\path (nd-1-2) -- ++(0,1);
			\polyboxscale{(0,0.6)}{$f$}{0.9}
		\end{boxedstrings}
		&=
		\Sym_\tau(\partial_s(f))
		\begin{boxedstrings}
			\dotu{(1,0)}{green}
			\path (nd-1-1) -- ++(0,1);
		\end{boxedstrings}
		+
		\frac{\Alt_\tau(\partial_s(f))}{\alpha_s-\alpha_t}
		\begin{boxedstrings}
			\idfigparam{(1,0)}{-0.25}{green}
			\dotuparam{(nd-1-2)}{0.2}{0.06}{orange}
			\path (nd-1-1) -- ++(0,1);
		\end{boxedstrings}
		,
		\\
		\begin{boxedstrings}
			\node at (0,1) {};
			\draw[green] (0.25,0.35) rectangle ++(-0.75,0.5);
			\idfigparam {(0,0.35)}{0.35}{green}				
			\path (nd-1-2) -- ++(0,1);
			\polyboxscale{(0,0.6)}{$f$}{0.9}
			\dotu{(-0.3125,0.35)}{orange}
		\end{boxedstrings}
		&=
		\Sym_\tau(\partial_s(f(\alpha_s-\alpha_t)))
		\begin{boxedstrings}
			\dotu{(1,0)}{green}
			\path (nd-1-1) -- ++(0,1);
		\end{boxedstrings}
		+
		\frac{\Alt_\tau(\partial_s(f(\alpha_s-\alpha_t)))}{\alpha_s-\alpha_t}
		\begin{boxedstrings}
			\idfigparam{(1,0)}{-0.25}{green}
			\dotuparam{(nd-1-2)}{0.2}{0.06}{orange}
			\path (nd-1-1) -- ++(0,1);
		\end{boxedstrings}
		,
		\\
		\begin{boxedstrings}
			\draw[green] (-0.25,0.35) rectangle ++(0.5,0.5);
			\idfigparam {(0.05,0)}{-0.35}{green}
			\idfigparam {(-0.05,0)}{-0.35}{orange}
			\path (nd-1-1) -- ++(0,1);
			\polyboxscale{(0,0.6)}{$f$}{0.9}
		\end{boxedstrings}
		&=
		\Sym_\tau(\partial_s(f))
		\begin{boxedstrings}
			\idfigparam{(1,0)}{-0.25}{green}
			\dotuparam{(nd-1-2)}{0.2}{0.06}{green}
			\draw[orange] (nd-1-2) -- ++(-0.25,0) -- ++(0,-0.25);
			\path (nd-1-1) -- ++(0,1);
		\end{boxedstrings}	
		+
		\frac{\Alt_\tau(\partial_s(f))}{\alpha_s-\alpha_t}
		\begin{boxedstrings}
			\dotu{(1,0)}{green}
			\dotu{(0.75,0)}{orange}
			\path (nd-1-1) -- ++(0,1);
		\end{boxedstrings}
		,
		\\
		\begin{boxedstrings}
			\node at (0,1) {};
			\draw[green] (0.25,0.35) rectangle ++(-0.75,0.5);
			\idfigparam {(0,0.35)}{0.35}{green}				
			\path (nd-1-2) -- ++(0,1);
			\polyboxscale{(0,0.6)}{$f$}{0.9}
			\dotu{(-0.3125,0.35)}{orange}
			\idfigparam {(nd-2-1)}{0.35}{orange}	
		\end{boxedstrings}
		&=
		\Sym_\tau(\partial_s(f(\alpha_s-\alpha_t)))
		\begin{boxedstrings}
			\idfigparam{(1,0)}{-0.25}{green}
			\dotuparam{(nd-1-2)}{0.2}{0.06}{green}
			\draw[orange] (nd-1-2) -- ++(-0.25,0) -- ++(0,-0.25);
			\path (nd-1-1) -- ++(0,1);
		\end{boxedstrings}	
		+
		\frac{\Alt_\tau(\partial_s(f(\alpha_s-\alpha_t)))}{\alpha_s-\alpha_t}
		\begin{boxedstrings}
			\dotu{(1,0)}{green}
			\dotu{(0.75,0)}{orange}
			\path (nd-1-1) -- ++(0,1);
		\end{boxedstrings}
		,
		\\
		\begin{boxedstrings}
			\node at (0,1) {};
			\draw[green] (-0.25,0.35) rectangle ++(0.5,0.5);
			\idfigparam {(0,0.35)}{0.35}{brown}
			\path (nd-1-2) -- ++(0,1);
			\polyboxscale{(0,0.6)}{$f$}{0.9}
		\end{boxedstrings}
		&= 0, \hspace{0.2cm}
		\begin{boxedstrings}
			\node at (0,1) {};
			\draw[green] (0.25,0.35) rectangle ++(-0.75,0.5);
			\idfigparam {(0,0.35)}{0.35}{brown}				
			\path (nd-1-2) -- ++(0,1);
			\polyboxscale{(0,0.6)}{$f$}{0.9}
			\dotu{(-0.3125,0.35)}{orange}
		\end{boxedstrings}
		=0,
		\\
		\begin{boxedstrings}
			\node at (0,1) {};
			\draw[brown] (-0.25,0.35) rectangle ++(0.5,0.5);
			\idfigparam {(0,0.35)}{0.35}{green}
			\path (nd-1-2) -- ++(0,1);
			\polyboxscale{(0,0.6)}{$f$}{0.9}
		\end{boxedstrings}
		&=
		\frac{\partial_{st}(f)}{2}
		\left(
		(\alpha_s+\alpha_t)
		\begin{boxedstrings}
			\dotuparam{(1,0)}{0.25}{0.06}{green}
			\path (nd-1-1) -- ++(0,1);
		\end{boxedstrings}
		+
		\begin{boxedstrings}
			\idfigparam{(1,0)}{-0.25}{green}
			\dotuparam{(nd-1-2)}{0.2}{0.06}{orange}
			\path (nd-1-1) -- ++(0,1);
		\end{boxedstrings}			
		\right)
		,
		\\
		\begin{boxedstrings}
			\node at (0,1) {};
			\draw[brown] (0.25,0.35) rectangle ++(-0.75,0.5);
			\idfigparam {(0,0.35)}{0.35}{green}				
			\path (nd-1-2) -- ++(0,1);
			\polyboxscale{(0,0.6)}{$f$}{0.9}
			\dotu{(-0.3125,0.35)}{orange}
			\idfigparam {(-0.3125,0.35)}{0.35}{orange}	
		\end{boxedstrings}
		&=
		\frac{\partial_{st}(f(\alpha_s-\alpha_t))(\alpha_s-\alpha_t)}{2}
		\begin{boxedstrings}
			\idfigparam{(1,0)}{-0.25}{green}
			\dotuparam{(nd-1-2)}{0.2}{0.06}{green}
			\draw[orange] (nd-1-2) -- ++(-0.25,0) -- ++(0,-0.25);
			\path (nd-1-1) -- ++(0,1);
		\end{boxedstrings}	
		+
		\frac{(\alpha_s+\alpha_t)\partial_{st}(f(\alpha_s-\alpha_t))}{2(\alpha_s-\alpha_t)}
		\begin{boxedstrings}
			\dotu{(0.75,0)}{orange}
			\dotu{(1,0)}{green}
			\path (nd-1-1) -- ++(0,1);
		\end{boxedstrings}	
		,
		\\
		\begin{boxedstrings}
			\node at (0,1) {};
			\draw[brown] (-0.25,0.35) rectangle ++(0.5,0.5);
			\idfigparam {(0,0.35)}{0.35}{brown}
			\path (nd-1-2) -- ++(0,1);
			\polyboxscale{(0,0.6)}{$f$}{0.9}
		\end{boxedstrings}
		&=
		\partial_{st}(f)			
		\begin{boxedstrings}
			\dotu{(1,0)}{brown}
			\path (nd-1-1) -- ++(0,1);
		\end{boxedstrings}
		, 
		\hspace{0.2cm}
		\begin{boxedstrings}
			\node at (0,1) {};
			\draw[brown] (0.25,0.35) rectangle ++(-0.75,0.5);
			\idfigparam {(0,0.35)}{0.35}{brown}				
			\path (nd-1-2) -- ++(0,1);
			\polyboxscale{(0,0.6)}{$f$}{0.9}
			\dotu{(-0.3125,0.35)}{orange}
			\idfigparam {(-0.3125,0.35)}{0.35}{orange}	
		\end{boxedstrings}
		=
		\frac{\partial_{st}(f(\alpha_s-\alpha_t))}{\alpha_s-\alpha_t}			
		\begin{boxedstrings}
			\dotu{(1,0)}{brown}
			\dotu{(0.75,0)}{orange}
			\path (nd-1-1) -- ++(0,1);
		\end{boxedstrings}
		.
	\end{align}
\end{prop}

\begin{prop}[General circle relations]\label{prop general circle rels}
	Let $f\in\Rtau$. 
	The following relations hold:
	\begin{align}
		\begin{boxedstrings}
			\draw[green] (-0.25,0.25) rectangle ++(0.5,0.5);
			\path (0,0) -- ++(0,1);
			\polyboxscale{(0,0.5)}{$f$}{0.9}
		\end{boxedstrings}
		&=
		2 \Sym_\tau(\partial_s(f)) (\alpha_s+\alpha_t)
		,
		&
		\begin{boxedstrings}
			\draw[green] (-0.25,0.35) rectangle ++(0.5,0.5);
			\idfigparam {(0,0.35)}{0.35}{orange}
			\path (nd-1-2) -- ++(0,1);
			\polyboxscale{(0,0.6)}{$f$}{0.9}
		\end{boxedstrings}
		&=
		2 \Sym_\tau(\partial_s(f))
		\begin{boxedstrings}
			\dotu{(1,0)}{orange}
			\path (nd-1-1) -- ++(0,1);
		\end{boxedstrings}
		,
		\\
		\begin{boxedstrings}
			\node at (0,1) {};
			\draw[green] (0.25,0.25) rectangle ++(-0.75,0.5);			
			\path (0,0) -- ++(0,1);
			\polyboxscale{(0,0.5)}{$f$}{0.9}
			\dotu{(-0.3125,0.25)}{orange}
		\end{boxedstrings}
		&=
		2 \Sym_\tau(\partial_s(f(\alpha_s-\alpha_t))) (\alpha_s+\alpha_t)
		,
		&
		\begin{boxedstrings}
			\node at (0,1) {};
			\draw[green] (0.25,0.35) rectangle ++(-0.75,0.5);			
			\path (0,0) -- ++(0,1);
			\polyboxscale{(0,0.6)}{$f$}{0.9}
			\dotu{(-0.3125,0.35)}{orange}
			\idfigparam {(nd-1-1)}{0.35}{orange}	
		\end{boxedstrings}
		&=
		2 \Sym_\tau(\partial_s(f(\alpha_s-\alpha_t)))
		\begin{boxedstrings}
			\dotu{(1,0)}{orange}
			\path (nd-1-1) -- ++(0,1);
		\end{boxedstrings}
		,
		\\
		\begin{boxedstrings}
			\node at (0,1) {};
			\draw[brown] (-0.25,0.25) rectangle ++(0.5,0.5);
			\path (0,0) -- ++(0,1);
			\polyboxscale{(0,0.5)}{$f$}{0.9}
		\end{boxedstrings}
		&=
		\partial_{st}(f)	
		(\alpha_s \alpha_t)
		, 
		&
		\begin{boxedstrings}
			\node at (0,1) {};
			\draw[brown] (0.25,0.35) rectangle ++(-0.75,0.5);			
			\path (0,0) -- ++(0,1);
			\polyboxscale{(0,0.6)}{$f$}{0.9}
			\dotu{(-0.3125,0.35)}{orange}
			\idfigparam {(-0.3125,0.35)}{0.35}{orange}	
		\end{boxedstrings}
		&=
		\frac{\partial_{st}(f(\alpha_s-\alpha_t))}{\alpha_s-\alpha_t}	
		(\alpha_s \alpha_t)		
		\begin{boxedstrings}
			\dotu{(0.75,0)}{orange}
			\path (nd-1-1) -- ++(0,1);
		\end{boxedstrings}
		.
	\end{align}
\end{prop}

\end{appendices}

\printbibliography

\end{document}